\documentclass[reqno]{tran-l}

\usepackage{hyperref}
\usepackage{graphicx}
\usepackage{tikz}
\usetikzlibrary{snakes}
\usepackage{caption}
\usepackage{amssymb, mathtools, IEEEtrantools}
\usepackage{multirow}
\usepackage{float}
\usepackage{cancel,soul,ulem}
\usepackage{subcaption}
\usepackage{suetterl}

\usepackage{rotating}

\usepackage[section]{placeins}
\usepackage{longtable}
\newtheorem{theorem}{Theorem}
\newtheorem{lemma}[theorem]{Lemma}
\newtheorem{proposition}[theorem]{Proposition}

\theoremstyle{definition}
\newtheorem{definition}[theorem]{Definition}

\usetikzlibrary{decorations.pathreplacing,calligraphy,patterns}
\theoremstyle{remark}
\newtheorem{remark}[theorem]{Remark}

\tikzstyle{line}=[draw]

\def\ZZ{\mathbb{Z}}
\def\QQ{\mathbb{Q}}

\def\TT{\mathbb{T}}
\def\SP{\operatorname{Sp}}
\def\Sp{\operatorname{Sp}}
\def\SL{\operatorname{SL}}
\def\sL{\operatorname{sl}}
\def\GL{\operatorname{GL}}
\def\PSL{\operatorname{PSL}}
\def\Id{\operatorname{Id}}

\DeclareMathOperator{\KoZoMon}{\Gamma^{(0)}}
\DeclareMathOperator{\hol}{hol}

\DeclareMathOperator{\Thor}{T_{hor}}
\DeclareMathOperator{\Tver}{T_{ver}}
\DeclareMathOperator{\Tdiag}{T_{diag}}
\DeclareMathOperator{\Tadiag}{T_{adiag}}
\DeclareMathOperator{\Affplus}{Aff^{+}}
\DeclareMathOperator{\VG}{\Gamma}
\DeclareMathOperator{\Gal}{Gal}
\DeclareMathOperator{\Aut}{Aut}
\DeclareMathOperator{\Lie}{Lie}

\newlength{\wdth}

\numberwithin{equation}{section}

\newcommand{\uproman}[1]{\uppercase\expandafter{\romannumeral #1\relax}}
\newcommand{\mc}{\mathcal}

\newcommand{\F}{\mc{F}}

\newcommand{\K}{\mc{K}}

\newcommand{\C}{\mathbb{C}}
\newcommand{\R}{\mathbb{R}}

\newcommand{\N}{\mathbb{N}}
\newcommand{\Q}{\mathbb{Q}}
\newcommand{\Z}{\mathbb{Z}}

\allowdisplaybreaks
\begin{document}

\title[Arithmeticity of Kontsevich--Zorich monodromies of origamis]{Arithmeticity of the Kontsevich--Zorich monodromies of certain families of origamis}

\author[E. B.]{Etienne Bonnafoux}
\address{Centre de Math\'ematiques Laurent Schwartz, CNRS (UMR 7640), \'Ecole Polytechnique, 91128 Palaiseau, France.}
\email{etienne.bonnafoux@polytechnique.edu}

\author[M. K.]{Manuel Kany}
\address{Department of Mathematics and Computer Science, Saarland University, 66123 Saarbr\"ucken, Germany}
\email{kany@math.uni-sb.de}

\author[P. K.]{Pascal Kattler}
\address{Department of Mathematics and Computer Science, Saarland University, 66123 Saarbr\"ucken, Germany}
\email{kattler@math.uni-sb.de}

\author[C. M.]{Carlos Matheus}
%\address{Coll\`ege de France, 3 Rue d'Ulm, Paris CEDEX 05, France.}
\address{Centre de Math\'ematiques Laurent Schwartz, CNRS (UMR 7640), \'Ecole Polytechnique, 91128 Palaiseau, France.}
\email{carlos.matheus@math.cnrs.fr}
%\thanks{???}

%    Information for fourth author
\author[R. N.]{Rogelio Ni\~no Hern\'andez}
\address{Centro de Ciencias Matem\'aticas, UNAM Campus Morelia, Antigua Carretera a P\'atzcuaro \#8701, Col. Ex Hacienda San Jos\'e de la Huerta, C.P. 58089, Morelia, Michoac\'an, M\'exico}
\email{rnino@matmor.unam.mx}
%\thanks{???}
%    Information for fifth author
\author[M. S.-M.]{Manuel Sedano-Mendoza}
\address{Centro de Ciencias Matem\'aticas, UNAM Campus Morelia, Antigua Carretera a P\'atzcuaro \#8701, Col. Ex Hacienda San Jos\'e de la Huerta, C.P. 58089, Morelia, Michoac\'an, M\'exico}
\email{msedano@matmor.unam.mx}
%\thanks{???}
%    Information for third author
\author[F. V.]{Ferr\'an Valdez}
\address{Centro de Ciencias Matem\'aticas, UNAM Campus Morelia, Antigua Carretera a P\'atzcuaro \#8701, Col. Ex Hacienda San Jos\'e de la Huerta, C.P. 58089, Morelia, Michoac\'an, M\'exico}
\email{ferran@matmor.unam.mx}
%    Information for third author
\author[G. W.-S.]{Gabriela Weitze-Schmith\"{u}sen}
\address{Department of Mathematics and Computer Science, Saarland University, 66123 Saarbr\"ucken, Germany}
\email{weitze@math.uni-sb.de}

\date{November 15, 2024}

\begin{abstract}
The variations of Hodge structures of weight one associated to square-tiled surfaces naturally generate interesting subgroups of integral symplectic matrices called Kontsevich--Zorich monodromies. In this paper, we show that arithmetic groups are frequent among the Kontsevich--Zorich monodromies of square-tiled surfaces of low genera $g$.   
\end{abstract}

\maketitle

%\tableofcontents

%%%%%%%%%%%%%%%%%%%%%%%%%%%%%%%%%%%%%%%%%%
%%%%%%%%%%%% Section 1 %%%%%%%%%%%%%%%%%%%
%%%%%%%%%%%%%%%%%%%%%%%%%%%%%%%%%%%%%%%%%%

\section{Introduction}

A subgroup $\Gamma\subset \GL(n,\mathbb{Z})$ with Zariski closure $G$ is called \textit{arithmetic}, resp., \textit{thin} (in the sense of Sarnak \cite{Sarnak}) if the index of $\Gamma$ in $G(\mathbb{Z})$ is finite, respectively infinite. From the point of view of number theory, the problems driven by thin matrix groups possess an ``extra'' degree of difficulty in comparison with the questions involving arithmetic matrix groups. Partly motivated by this fact, several authors tried to identify how often one meets thin matrix groups in certain geometric situations, for example:
\begin{itemize}
\item certain Calabi--Yau threefolds form 14 families whose moduli spaces are isomorphic to $\overline{\mathbb{C}}\setminus\{0,1,\infty\}$, so that one gets 14 examples of subgroups of $\textrm{Sp}(4,\mathbb{Z})$ (with full Zariski closures) by looking at the corresponding variations of Hodge structures; in this context, Brav and Thomas \cite{BT} showed that 7 families lead to thin matrix groups, and Singh and Venkataramana \cite{Singh}, \cite{SV} proved that the remaining 7 families lead to arithmetic groups;
\item the setting of the previous paragraph can be significantly extended by looking at the monodromy groups generated by hypergeometric differential equations, and, in this direction, many new examples of thin matrix groups were found by Fuchs, Meiri and Sarnak \cite{FMS}, Filip and Fougeron \cite{FF}, \cite{F}, among other authors.
\end{itemize}

In this paper we are interested in the Kontsevich--Zorich monodromies of square-tiled surfaces i.e., the matrix groups associated to the actions on the first homology groups of affine homeomorphisms of square-tiled surfaces. Equivalently, Kontsevich--Zorich monodromies can be seen as the variations of Hodge structures along the closed $\SL(2,\mathbb{R})$-orbits spanned by integral points in moduli spaces of translation surfaces. In this direction, we do not have examples of thin Kontsevich--Zorich monodromies with the largest possible Zariski closures (despite a recent effort by Hubert--Matheus \cite{HM}). In fact, our two main results below partly explain why it might not be easy to find such examples among square-tiled surfaces of genera three and four\footnote{We restrict ourselves to the higher genus case because an observation of M. M\"oller (which is discussed in details in Appendix \ref{a.ArithmShadVee} below) asserts the arithmeticity of the Kontsevich--Zorich monodromy of any genus two square-tiled surface.} with a single conical singularity.

\begin{theorem}\label{t.A} There are infinitely many square-tiled surfaces of genus three with a single conical singularity whose Kontsevich--Zorich monodromies are arithmetic.
\end{theorem}

In particular, conditionally to a conjecture by Delecroix and Leli\`evre (whose statement is recalled in the next section), this theorem says that a positive proportion (at least $1/8$) of the $\SL(2,\mathbb{R})$-orbits of square-tiled surfaces of genus three with a single conical singularity have arithmetic Kontsevich--Zorich monodromies. For more precise formulations of Theorem \ref{t.A} and the following Theorem~\ref{t.B}, see the statements of Theorems \ref{t.H4odd-families1to6}, \ref{t.H4hyp-family1} and \ref{KoZoGenus4} below.

\begin{theorem}\label{t.B} For each $10\leq n\leq 260$ which is divisible by $5$, there exists a square-tiled surface of genus four with a single conical singularity tiled by $n$ squares whose Kontsevich--Zorich monodromy is arithmetic. 
\end{theorem}

\textbf{Structure of this article}. The proof of these results occupy the remainder of this text. More concretely, we quickly review in Section \ref{s.preliminaries} the aspects of the theory of square-tiled surfaces entering into the statements of Theorems \ref{t.A} and \ref{t.B}, and the relevant strategy towards the arithmeticity of subgroups of symplectic matrices. After that, in Section \ref{s.H4odd}, we prove Theorem \ref{t.H4odd-families1to6}, yielding a precise version of Theorem \ref{t.A} in the context of the so-called odd component of $\mathcal{H}(4)$. Subsequently, we complete in Section \ref{s.H4hyp} the discussion of Theorem \ref{t.A} by showing a statement, namely Theorem \ref{t.H4hyp-family1}, giving a precise version of Theorem \ref{t.A} in the context of the so-called hyperelliptic component of $\mathcal{H}(4)$. Next, we establish Theorem \ref{t.B} in Section \ref{s.H6}. Once the main results are proved, we include in Section \ref{s.computationals} some numerical experiments about the indices of the Kontsevich--Zorich  monodromies (in the integral lattices in their Zariski closures) of some square-tiled surfaces in genus two and we describe two curious examples in genera three and four: In particular, concerning the genus two case, the list of such indices seems to take only two values ($1$ or $3$) for square-tiled surfaces in $\mathcal{H}(2)$, while many values (including $1$, $3$, $4$, $6$, $12$, $24$) seem to be taken for square-tiled surfaces in $\mathcal{H}(1,1)$. Finally, we complete the article with two appendices: in Appendix \ref{a.examplePrym}, we briefly compute the Kontsevich--Zorich monodromy of a square-tiled surface of genus three in the Prym locus (but unfortunately we are not able to infer whether arithmeticity or thinness should be typically expected in this special locus of $\mathcal{H}^{odd}(4)$), and in Appendix \ref{a.ArithmShadVee}, we explain the result of M\"oller that the Kontsevich--Zorich  monodromy of any square-tiled surface of genus two is always arithmetic. 

\textbf{Acknowledgments}

We thank Martin M\"oller for allowing us to include his unpublished result about the arithmeticity of the Kontsevich--Zorich monodromy of square-tiled surfaces of genus two in Appendix \ref{a.ArithmShadVee} and for his helpful and unwavering support. R. Ni\~no would like to thank CONACYT's Ph.d. grant.
M. Sedano would like to thank UNAM-DGAPA's postdoctoral grant. 
F.Valdez would like to thank the following grants: CONACYT Ciencia B\'asica CB-2016 283960 and UNAM PAPIIT IN-101422. 
The working group of G. Weitze-Schmith\"usen would like to thank the Deutsche Forschungsgemeinschaft (DFG, German Research Foundation) for their funding in Project-ID 286237555 – TRR 195.

%\begin{figure}[htb!]
%\input{MPY-S-2.pdf_tex}
%\caption{Quadratic homoclinic tangency associated to a periodic point.}\label{f.2}
%\end{figure}

%%%%%%%%%%%%%%%%%%%%%%%%%%%%%%%%%%%%%%%%%
%%%%%%%%%%%% Section 2 %%%%%%%%%%%%%%%%%%
%%%%%%%%%%%%%%%%%%%%%%%%%%%%%%%%%%%%%%%%%

\section{Preliminaries}\label{s.preliminaries}

We briefly summarize what we need about translation surfaces and origamis and refer the reader not familiar with these topics to introductory texts about translation surfaces e.g. in  \cite{Masur2006, Zorich2006} and about origamis e.g. in  \cite{DissGabi, zmiaikou2011,M2022}.
A \textit{finite translation surface $(M,\omega)$} is a Riemann surface together with a holomorphic differential $\omega$. The latter defines a \textit{translation structure},  i.e. an atlas such that all transition maps are translations, on $M\backslash Z$ where $Z = \{x_1, \ldots, x_r\}$ is the set of zeroes of $\omega$.
The points $x_i$ are cone angle singularities of angle $(\mu_i + 1)2\pi$, where $\mu_i$ is the order of the zero. We consider the \textit{stratum} $\mathcal{H}(\mu_1,\ldots, \mu_r)$, i.e. the space of all finite translation surfaces with $r$ singularities of order $\mu_1$, \ldots, $\mu_r$. By Riemann Roch we have $\mu_1 + \ldots + \mu_r = 2g-2$.
An \textit{origami} or \textit{square-tiled surface} (cf. Figure~\ref{fig:O35}) is a finite translation surface $(M,\omega)$, where $\omega$ is obtained as pull back of the differential $dz$ on the flat torus $\mathbb{T}^2=\mathbb{C}/(\mathbb{Z}\oplus i\mathbb{Z})$ via a ramified covering $\pi\colon M\to\mathbb{T}^2$  which is not branched outside $0\in\mathbb{T}^2$.  Two origamis given by coverings $\pi_1$ and $\pi_2$ are \textit{equivalent} if there is a homeomorphism $h$ such that $\pi_2 = h \circ \pi_1$.  In the whole article, we assume that the origamis are \textit{reduced}, i.e. that the group of relative periods of $\omega$ is $\mathbb{Z}\oplus i \mathbb{Z}$. Recall that the fundamental group of the once-punctured torus $\mathbb{T}^2\backslash\{0\}$ equals the free group $F_2$. Let $x$ and $y$ be the closed paths on $\mathbb{T}^2$ in direction ${1 \choose 0}$ and ${0 \choose 1}$. Then, for an origami $(M,\omega)$, we consider the monodromy action $m: \pi_1(\mathbb{T}^2\backslash\{0\}) \to S_N$ of the associated unramified covering $\pi\colon M\backslash \pi^{-1}(0)\to\mathbb{T}^2\backslash\{0\}$. Hence, the origami is coded by the pair of permutations $(h = m(x),v = m(y))\in S_N\times S_N$, where $N$ is the degree of $\pi$. Two origamis $(M,\omega)$ and $(M',\omega')$ are equivalent if and only if the two corresponding permutations $(h,v)$ and $(h', v')$  are simultaneously conjugated, i.e. there exits a permutation $\sigma$ such that $h' = \sigma\circ h \circ \sigma^{-1}$ and $v' = \sigma\circ v \circ \sigma^{-1}$.  If the origami is defined by gluing a finite collection of squares which are labeled by $1$, \ldots, $N$  then $h$ gives the right neighbors and $v$ gives the upper neighbors of each square. Choosing a different labeling of the squares corresponds to the simultaneous conjugation with some $\sigma \in S_N$. 
$\SL(2,\Z)$ acts on the set of equivalence-classes of origamis in the following way. For $A \in \SL(2,\Z)$ the affine homeomorphism $f_A: \R^2 \to \R^2, {x \choose y} \mapsto A\cdot {x \choose y}$ descends to $\overline{f}_A: \TT^2 \to \TT^2$. Here we use the standard identification $\R^2 \cong \C$. For the origami $O = (X,\mu)$ defined by the covering $\pi: M \to \TT^2$ then $A\cdot O$ is the origami $(M_A,\omega_A)$ defined by the covering $\overline{f}_A \circ \pi:M \to \TT^2$. In turns out that the generators
\begin{equation}\label{TheMatrixSandT}
  T =\begin{pmatrix} 1 & 1\\0 & 1\end{pmatrix} \mbox{ and } S = \begin{pmatrix} 0 & -1 \\1 & 0\end{pmatrix}
\end{equation}
of $\SL(2,\Z)$ act on the associated permutations as the Nielsen transformations by $(h,v)\mapsto (h,vh^{-1})$ and $(h,v)\mapsto (v^{-1},h)$ (cf.\cite[Section 2]{Weitze2015}).  It follows in particular that the group $G=\langle h, v\rangle\subset S_N$ called the \textit{monodromy group} is an invariant in the $\SL(2,\mathbb{Z})$-orbit of the origami. 
%
%From the discussion in the previous paragraph, it is clear that the isomorphism class of the group $G=\langle h, v\rangle\subset S_N$ generated by a pair of permutations $(h,v)$ coding an origami $(M,\omega)$ is an invariant of its $\SL(2,\mathbb{Z})$-orbit.
%We consider $(M,\omega)$ endowed with the translation structure defined by $\omega$ on $M$ outside its zeroes.
%We refer the reader to \cite{DissGabi}, \cite[Chapters 1, 2.2]{zmiaikou2011} and \cite{M2022} for more explanations about origamis and to \cite{Masur2006} and \cite{Zorich2006} for introductory texts about translation surfaces. 

\subsection{Kontsevich--Zorich monodromy of origamis}\label{KZmonOrigami}
An \textit{affine homeomorphism} $f$ of a reduced origami $(M,\omega)$ is an orientation-preserving homeomorphism of $M$ which is affine in the local charts of the corresponding translation structure. The affine homeomorphisms of $(M,\omega)$ form a group called \textit{the affine group} $\Affplus(M,\omega)$. In this situation, the linear part $A = Df$ of $f$ is an element of $\SL(2,\mathbb{Z})$ and the derivative map $\Affplus(M,\omega) \to \SL(2,\Z), f \mapsto A = Df$ is a group homomorphism.   Its image is a finite-index subgroup of $\SL(2,\mathbb{Z})$ called the \textit{Veech group} $\VG(M,\omega)$. An affine homeomorphism whose derivative is the identity matrix is called \textit{translation}. Hence, if the identity map is the only translation of $(M,\omega)$ then $\Affplus(M,\omega)$ is isomorphic to $\VG(M,\omega)$. The Veech group of the torus $\mathbb{T}^2$ is $\SL(2,\Z)$. The group $\SL(2,\Z)$ acts on the set of origamis: For $A \in \SL(2,\Z)$ we send an origami $(M,\omega)$ given by the covering map $\pi$ to the origami $A\cdot (M,\omega)$ given by the covering map $\pi \circ f_A$, where $f_A$ is the affine homeomorphism of $\mathbb{T}^2$ with derivative $A$ fixing $0 \in \mathbb{T}^2$. See e.g. \cite{Vorobets96, MoellerAff09} for an introduction and a survey about Veech groups.

The first homology group $H_1(M,\mathbb{Q})$ of a reduced origami has a splitting
  \begin{equation}\label{splitting}
  H_1(M,\mathbb{Q}) = H_1^{st}(M,\mathbb{Q})\oplus H_1^{(0)}(M,\mathbb{Q}),
  \end{equation}
which is respected by the natural action of the affine homeomorphisms of $(M,\omega)$ via pushforward (cf.~\cite[Section 1]{HM}, \cite[Section 2.6]{M2022}). In concrete terms, $H_1^{(0)}(M,\mathbb{Q})$ is the kernel of $\pi_*:H_1(M,\mathbb{Q})\to H_1(\mathbb{T}^2,\mathbb{Q})$, and $H_1^{st}(M,\mathbb{Q})$ is the orthogonal complement of $H_1^{(0)}(M,\mathbb{Q})$ with respect to the symplectic intersection form $\Omega$ on $H_1(M,\mathbb{Q})$. The factor $H_1^{(0)}(M,\mathbb{Q})$ is called the \textit{non-tautological part} of $H_1(M,\mathbb{Q})$. Indeed, the splitting is well-defined over $\Z$ as $H_1(M,\mathbb{Z}) = H_1^{st}(M,\mathbb{Z})\oplus H_1^{(0)}(M,\mathbb{Z})$. Each square in an origami defines two relative cycles $h_i$ and $v_i$ given by the bottom horizontal and left vertical sides respectively. These cycles define a basis for $H_1^{st}(M,\mathbb{Q})$ given by $\sigma=\sum_i h_i$ and $\zeta=\sum_i v_i$. 
%In our cases the elements $h_i$ and $v_i$ also define a base for  $H_1(M,\mathbb{Q})$ and  $H_1^{(0)}(M,\mathbb{Q})$. In the first space each element of this base is given by $\sigma_m = \sum h_{i_l}$ and $\zeta_k=\sum v_{i_p}$, for some subsets of indices $l$ and $p$. For the latter space the elements of this base are of the form  $\sigma_i-\lambda\sigma_j$ and $\zeta_r-\kappa\zeta_s$, for some $i \neq j, r \neq s$ and $\lambda,\kappa \neq 0$. 
An affine homeomorphism $A$ of $(M,\omega)$ acts on $H_1^{st}(M,\mathbb{Z})\simeq\mathbb{Z}^2$ via the induced action of the derivative $DA\in \textrm{SL}(2,\mathbb{Z})$ on $\mathbb{Z}^2$. This action has been intensively studied in the context of translation surfaces. Understanding the action of the group of affine homeomorphisms $\textrm{Aff}^+(M,\omega)$ on $H_1^{(0)}(M,\mathbb{Z})$ is much more delicate. This goal leads us to the main object of this work.

\begin{definition}
  We denote by $\SP(H_1^{(0)}(M,\mathbb{Z}),\Omega)$ the symplectic group of $H_1^{(0)}(M,\mathbb{Z})$ with respect to the intersection form $\Omega$.
  %, i.e.
  %\[\SP(H_1^{(0)}(M,\mathbb{Z}),\Omega) = \{\Phi \in \GL(H_1(M,\mathbb{Z}))|\;\forall c_1,c_2 \in H_1^{(0)}(M,\mathbb{Z}) \Omega(\Phi(c_1),\Phi(c_2)) = \Omega(v_1,v_2)\}.\]
  Since the action of the affine group on the homology respects the intersection form $\Omega$ and the splitting (\ref{splitting}) is invariant,
  one obtains a representation
%  The subgroup\footnote{More precisely, the action on $H_1^{(0)}(M,\mathbb{Z})$ of affine homeomorphisms generates a subgroup $H\leq\textrm{SL}(2g-2,\Z)$ for which there is a finite index subgroup $\tilde{H}\leq H$ in $\textrm{Sp}(H_1^{(0)}(M,\mathbb{Z}))$. See the discussion in Section~\ref{ssec:example-genus-3} for a precise example. Passing to the finite index subgroup has no effect on any of the results presented in this text.} 
  %of $\textrm{Sp}(H_1^{(0)}(M,\mathbb{Z}))\simeq \textrm{Sp}_{\Omega}(2g-2,\mathbb{Z})$ generated by the action of the group of affine homeomorphisms of $(M,\omega)$ on $H_1^{(0)}(M,\mathbb{Q})$, is called the \textit{Kontsevich--Zorich monodromy} of $(M,\omega)$. In other words,
  \[
  \rho\colon \textrm{Aff}^+(M,\omega) \longrightarrow \textrm{Sp}(H_1^{(0)}(M,\mathbb{Z}),\Omega).
  \]
  Its image is called the \textit{Kontsevich--Zorich monodromy} of $(M,\omega)$ and is denoted by $\Gamma^{(0)}(M,\omega)$ or just $\Gamma^{(0)}$.
  %if it is evident which origami we consider.
\end{definition}

An origami $(M,\omega)$ has \textit{arithmetic}, respectively \textit{thin} Kontsevich--Zorich  monodromy (in the sense of Sarnak \cite{Sarnak}) if its Kontsevich--Zorich monodromy has finite, respectively infinite index in $G(\mathbb{Z})$, where $G$ is the Zariski closure of its Kontsevich--Zorich monodromy. Since for two origamis $(M,\omega)$ and $A \cdot (M,\omega)$ in the same $\SL(2,\Z)$-orbit  their  Kontsevich--Zorich monodromies are conjugated, either both are arithmetic or both are not arithmetic.
Hence it suffices to restrict to $\SL(2,\Z)$-orbits.

\subsection{Zariski density in symplectic groups}\label{ss.Prasad-Rapinchuk-MMY}

%\commf{This section has been rewritten. One definition, two lemmas and a remark have been added to organize the discussion.}

In this section we present in Lemma~\ref{lemma:criteria-zariski-density} the main criterion that we will use for showing that the Kontsevich-Zorich monodromy groups of origamis are Zariski dense in the ambient symplectic group. % $\textrm{Sp}(H_1^{(0)}(M,\mathbb{Z}),\Omega)$ for the intersection form $$.
In the rest of this section $\Omega$ is always a symplectic form on $\mathbb{Q}^{2d}$ taking integral values on the lattice $\mathbb{Z}^{2d}$. We denote $\textrm{Sp}_{\Omega}(2d,\mathbb{Z})$ for the symplectic group of $\mathbb{Z}^{2d}$ with respect to this form $\Omega$. A key ingredient of the criterion is to consider matrices with the property \textit{Galois-pinching} (cf. Definition~\ref{def:galois-pinching}) introduced in \cite[Section 4]{MMY}.
    %Let $\Omega$ be a symple}})
%In order to prove arithmeticity of a subgroup $\Gamma\subset \textrm{Sp}_{\Omega}(2d,\mathbb{Z})$, the first thing in our line of attack is to check whether $\Gamma$ is Zariski dense in $\textrm{Sp}_{\Omega}(2d,\mathbb{R})$. In what follows we present the main criteria we use to validate Zariski density. 

\begin{definition}
    \label{def:galois-pinching}
    %Let $\Omega$ be a symplectic form on $\mathbb{Q}^{2d}$ taking integral values on the lattice $\mathbb{Z}^{2d}$.
    A matrix $A\in \textrm{Sp}_{\Omega}(2d,\mathbb{Z})$ is called \textit{Galois-pinching} if the following holds:
\begin{enumerate}
\item
The characteristic polynomial of the matrix $A$ is irreducible over $\mathbb{Q}$.
\item
All the roots of the matrix $A$ are real numbers.
\item
The Galois group of the characteristic polynomial of $A$ is the largest possible, namely, isomorphic to the hyperoctahedral group of order $2^d\,d!$. Here we view $\textrm{Gal}$ as the centralizer of the involution $\lambda\to\lambda^{-1}$ on the set of roots.
\end{enumerate}
\end{definition}

We are particularly interested in determining Zariski density for Kontsevich-Zorich monodromies of genus three surfaces. In this case $H_1^{(0)}(M,\mathbb{Q})$ is isomorphic to $\mathbb{Q}^4$ and the following lemma, whose proof can be consulted in \cite[\S 6.7]{MMY}, provides a simple criterion for the Galois-pinching property. 

\begin{lemma}
    \label{lemma:galois-pinching-criterion}
Let $A \in \textrm{Sp}_{\Omega}(4,\mathbb{Z})$ and $x^4+ax^3+bx^2+ax+1\in \mathbb{Z}[x]$ be its characteristic polynomial.
Then $A$ is Galois-pinching whenever $-a-4>0$, $b+2a+2>0$, and the three discriminants 
  \[
  \Delta_1=a^2-4b+8>0, \qquad \Delta_2=(b+2+2a)(b+2-2a)
  \]
and $\Delta_1\Delta_2$ are not squares.
\end{lemma}

%According to a theorem of Prasad and Rapinchuk \cite[Theorem 9.10]{PR}, the Zariski closure of a subgroup $\Gamma\subset \textrm{Sp}_{\Omega}(2d,\mathbb{Z})$ containing a Galois-pinching matrix $A$\footnote{Note that a Galois-pinching matrix $A \in \textrm{Sp}_{\Omega}(4,\mathbb{Z})$ is generic in the sense of \cite{PR}} and an infinite order matrix $B$ not commuting with $A$, is either $\textrm{Sp}_{\Omega}(2d,\mathbb{R})$ or a proper subgroup generated by the long roots, which after conjugation, is the subgroup of block diagonal matrices, isomorphic to $\textrm{SL}(2,\mathbb{R})^d$ (c.f. \cite[Theorem 1.5]{IR} or the discussion in Section 6.3 in~\cite{Bainbridge-Habegger-Moeller-2016}). Especially $A$ and $B$ would share a common proper invariant subspace. Combining this fact with \cite[Prop. 4.3]{MMY}, we deduce that $(B-\Id)(\mathbb{R}^{2d})$ is a Lagrangian subspace. We almost showed the following:
The main tool  for the density criterion is that, by Theorem 1.5 in \cite{IR},  the Zariski closure of a subgroup $\Gamma\subset \textrm{Sp}_{\Omega}(2d,\mathbb{Z})$ containing a Galois-pinching matrix $A$ and an infinite order matrix $B$ not commuting with $A$, is either $\textrm{Sp}_{\Omega}(2d,\mathbb{R})$ or a proper subgroup conjugated to the subgroup of block diagonal matrices, isomorphic to $\textrm{SL}(2,\mathbb{R})^d$. This statement follows from  Theorem 9.10 in  \cite{PR} in the following way. We apply this theorem to $G =\textrm{Sp}_{\Omega}(2d)$.  Firstly, it turns out that $g = A$ being Galois-pinching implies that $A$ is generic (cf. \cite[Def. 9.4]{PR}). To see this consider the maximal torus $T$ for $A$ which is equal to the centralizer $C_{G}(A)$ of $A$ in $G$.  Its splitting field $K_T$, i.e. the minimal field extension of $\QQ$ such that $A$ is diagonalisable over $K_T$, equals $\QQ(\lambda_1,\ldots, \lambda_{2d})$, where $\lambda_1, \ldots, \lambda_{2d}$ are the eigenvalues of $A$.  Furthermore, one considers the faithful action \mbox{$\Theta_T:\Gal(K_T/K) \to \Aut(\Phi(G,T))$} of $\Gal(K_T/K)$ on the root system $\Phi(G,T)$ of $T$ induced by the natural action of $\Gal(K_T/K)$ on the character group $X(T)$ of $T$. Recall that the image of $\Theta_T$ in general lies in the Weyl group $W(G,T)$ (cf. \cite[Prop 2.1]{JKZ}). The order of  $W(G,T)$ for the symplectic group $G$ is $2^d\cdot d!$. By the definition of Galois pinching (cf. Def.~\ref{def:galois-pinching})  the order of $\Gal(K_T/K)$ also equals $2^d\cdot d!$. Hence the image of $\Theta_T$ equals the Weyl-Group $W(G,T)$ which is precisely the definition for $T$ and then for $A$ to  be generic (cf.~\cite[Section 9, Def. 9.4]{PR}). Secondly, now Theorem 9.10 in \cite{PR} tells us that the Zariski closure of $\Gamma$ is either the full group $G =\textrm{Sp}_{\Omega}(2d)$ or $G_T^{>}$. The latter is by definition the subgroup of $G$ generated by the torus $T$ and the 1-parameter family of unipotent subgroups $U_{\alpha}$, where $\alpha$ runs through the set $\Phi_{>}(G,T)$ of all long roots. Finally, one observes that since $G =\textrm{Sp}_{\Omega}(2d)$ the group $G_T^{>}$ is isomorphic to $\SL(2)^d$. This can be seen in the following way. We consider the associated Lie algebras rather than the Lie groups. We then have $\Lie(G_T^{>}) \cong \mathfrak{h} \bigoplus_{\alpha \in {\Phi^+}_{>}(G,T)} (\mathfrak{g}_{\alpha} \oplus \mathfrak{g}_{-\alpha})$, where $ {\Phi^+}_{>}(G,T)$ denotes the set of positive long roots and $\mathfrak{h}$,  $\mathfrak{g}_{\alpha}$ and $\mathfrak{g}_{-\alpha}$  are the Lie algebras associated to $T$, $U_{\alpha}$ and $U_{-\alpha}$.  In the case of $G =\textrm{Sp}_{\Omega}(2d)$ one can check for example by direct computation with explicit matrices that $\mathfrak{h} \cong \bigoplus_{\alpha \in \Phi^+_>(G,T)} \mathfrak{h}_{\alpha}$ where $\mathfrak{h}_{\alpha} = [\mathfrak{g}_{\alpha},\mathfrak{g}_{-\alpha}]$. We finally use that  $\mathfrak{g}_{\alpha} \oplus \mathfrak{g}_{-\alpha} \oplus \mathfrak{h}_{\alpha} \cong \sL_2$ (cf. \cite[Thm. 2.22]{Milne}) and obtain the claim. 
We now almost have  Lemma~\ref{lemma:criteria-zariski-density}. 

%(c.f. \cite[Theorem 1.5]{IR} or the discussion in Section 6.3 in~\cite{Bainbridge-Habegger-Moeller-2016

%\footnote{Note that a Galois-pinching matrix $A \in \textrm{Sp}_{\Omega}(4,\mathbb{Z})$ is generic in the sense of \cite{PR}}
%and an infinite order matrix $B$ not commuting with $A$, is either $\textrm{Sp}_{\Omega}(2d,\mathbb{R})$ or a proper subgroup generated by the long roots, which after conjugation, is the subgroup of block diagonal matrices, isomorphic to $\textrm{SL}(2,\mathbb{R})^d$ (c.f. \cite[Theorem 1.5]{IR} or the discussion in Section 6.3 in~\cite{Bainbridge-Habegger-Moeller-2016}).
%Especially $A$ and $B$ would share a common proper invariant subspace. Combining this fact with \cite[Prop. 4.3]{MMY}, we deduce that $(B-\Id)(\mathbb{R}^{2d})$ is a Lagrangian subspace. We almost showed the following:

%\commrefs{\sout{that theorem does not state anything of the sort, but only talks about a general situation of Zariski closure (in terms of long and short roots). It seems to need a lot of work to deduce this statement from Theorem (9.10) of Prasad-Rapinchuk.}}

\begin{lemma}
    \label{lemma:criteria-zariski-density}
Let $\Omega$ be a symplectic form on $\mathbb{Q}^{2d}$ taking integral values on the lattice $\mathbb{Z}^{2d}$. Suppose that $\Gamma\subset \textrm{Sp}_{\Omega}(2d,\mathbb{Z})$ contains a  Galois-pinching matrix $A$ and an unipotent matrix $B\neq \Id$ such that $(B-\Id)(\mathbb{R}^{2d})$ is not a Lagrangian subspace, then $\Gamma$ is Zariski dense in $\textrm{Sp}_{\Omega}(2d,\mathbb{R})$.
\end{lemma}

\begin{proof}
  Firstly, we realize as follows that $B$ has infinite order and the matrices $A$ and $B$ do not commute. The first statement directly follows from $B$ being a unipotent matrix $\neq \Id$. For the second statement let us assume that $A$ and $B$ commute. Since $B$ is unipotent $\neq \Id$ it follows that  the eigenspace of $B$ is a common invariant proper subspace for $A$ and $B$. By \cite[Prop. 4.3]{MMY} this however implies that $(B - \Id)(R^{2d})$ is a Langrangian subspace of $\R^{2d}$ in contradiction to our assumption. Hence $A$ and $B$ do not commute. Now, by the discussion before this lemma it follows from Theorem 9.10 in \cite{PR} that the group generated by $A$ and $B$ is either the whole group  $G =\textrm{Sp}_{\Omega}(2d,\R)$ or conjugated to the subgroup of block diagonal matrices, isomorphic to $\textrm{SL}(2,\mathbb{R})^d$. In the second case again $A$ and $B$ share a common proper invariant subspace and hence again by\cite[Prop. 4.3]{MMY} this would imply that $(B - \Id)(\mathbb{R}^{2d})$ is a Langrangian subspace of $\R^{2d}$ giving a contradiction.
%  and that $B$ has infinite order. 
%The only two things which are left to show are that the matrices $A$ and $B$ do not commute and that $B$ has infinite order. The unipotent matrix $B$ cannot commute with $A$ because a standard argument from linear algebra shows that in this situation $A$ and $B$ would share a common proper invariant subspace which would immediately follow in a contradiction with \cite[Prop. 4.3]{MMY}. Furthermore $B$ is of infinite order since it is a nontrivial unipotent matrix over a field of characteristic zero.
\end{proof}

\begin{remark}
    \label{rmk:criteria-zariski-density}
In Lemma~\ref{lemma:criteria-zariski-density}, the dimension of a Lagrangian subspace equals $d$. Thus if the dimension of $(B-\Id)(\mathbb{R}^{2d})$ is not $d$ then it is not a Lagrangian subspace.
\end{remark}

%\comma{I added to the Remark \ref{rmk:criteria-zariski-density} the clarification that the unipotent matrix $B$ is precisely the non-commuting, infinite order element, necessary for Prasad-Rapinchuk. Just in case it wasn't completely clear.}

\subsection{Arithmeticity of subgroups of symplectic matrices}\label{ss.Sing-Venkataramana} 
For proving arithmeticity we will use the powerful criterion due to Sing and Venkataramana presented in Theorem~\ref{thm:sing-venkataramana}.

\begin{theorem}(\cite[Theorem 1.2]{SV})
    \label{thm:sing-venkataramana}
 Let $\Omega$ be a symplectic form on $\mathbb{Q}^{2d}$ taking integral values on the lattice $\mathbb{Z}^{2d}$. Suppose that $\Gamma\subset \Sp_{\Omega}(2d,\mathbb{Z})$ is Zariski dense and there are three transvections $T_n\in \Gamma$, $n=1,2,3$, such that $(T_n-\Id)(\mathbb{Z}^{2d}) = \mathbb{Z}w_n$ satisfies $\Omega(w_1,w_2)\neq 0$ and $w_1,w_2, w_3$ are linearly independent. Then
$\Gamma$ is arithmetic if the group generated by the restrictions of $T_n$, $n=1,2,3$, to $W=\mathbb{Q}w_1\oplus\mathbb{Q}w_2\oplus\mathbb{Q}w_3$ contains an element of the unipotent radical of $\Sp_{\Omega}(W)$
\end{theorem}

%Singh and Venkataramana showed that $\Gamma$ is arithmetic if the group generated by the restrictions of $T_n$ to $W=\mathbb{Q}w_1\oplus\mathbb{Q}w_2\oplus\mathbb{Q}w_3$ contains an element of the unipotent radical of $\Sp_{\Omega}(W)$ (cf. \cite[Theorem 1.2]{SV}).

In the following we describe the unipotent radical of $\Sp_{\Omega}(W)$, where $W$ as defined in Theorem~\ref{thm:sing-venkataramana}. Recall that any three-dimensional vector space with a non-trivial alternating form has a null space $\langle e \rangle \leq W$. We choose %\textcolor{purple}{
$\{w_1,w_2,e\}$%}
as basis of $W$ with $w_1$, $w_2$ from Theorem~\ref{thm:sing-venkataramana}.%\commref{\sout{You use a different order for the basis in the rest of the article (see for instance p. 10). It would be better to use a unique convention for the whole paper.}} \commf{I changed the basis to make it match our calculations later.  Please check the rest of this section makes sense.}.
 With respect to this basis the symplectic group  is described as
    \[  \Sp(W) = \left\{ \left(\begin{array}{ccc}
        a & b & 0 \\
        c & d & 0 \\
        x & y & \lambda
    \end{array}\right) \in \GL_3(\mathbb{Q}) : \lambda \neq 0, \  ad - bc = 1 \right\} \cong (\mathbb{Q}^* \times \SL_2(\mathbb{Q})) \ltimes \mathbb{Q}^2.    \]
As $\Sp(2,\mathbb{Q}) = \SL_2(\mathbb{Q})$ is a simple factor, the radical subgroup of $\Sp(W)$, being the greatest normal, solvable subgroup, is 
    \[  R(\Sp(W)) = \left\{ \left(\begin{array}{ccc}
        1 & 0 & 0 \\
        0 & 1 & 0 \\
        x & y & \lambda
    \end{array}\right) \in \GL_3(\mathbb{Q}) : \lambda \neq 0 \right\} \cong \mathbb{Q}^* \ltimes \mathbb{Q}^2    \]
and the unipotent radical is
    \begin{equation}\label{NiceFormUnipotent}  U(\Sp(W)) = \left\{ \left(\begin{array}{ccc}
        1 & 0 & 0 \\
        0 & 1 & 0 \\
        x & y & 1
    \end{array}\right) \in \GL_3(\mathbb{Q}) : \right\} \cong  \mathbb{Q}^2.    \end{equation}
We will be dealing with symplectic matrices $A \in \Sp_{\Omega}(2d,\mathbb{Z})$ which preserve the element $e$, so their restriction lies in the subgroup 
    \[  \left\{ \left(\begin{array}{ccc}
        a & b & 0 \\
        c & b & 0 \\
        x & y & 1
    \end{array}\right) \in \GL_3(\mathbb{Q}) : \lambda \neq 0, \  ad - bc = 1 \right\} \cong \SL_2(\mathbb{Q}) \ltimes \mathbb{Q}^2.    \]
There is a small error made in \cite{SV}, or perhaps a difference in convention: The authors stated that in fact $\Sp(W)\cong \SL_2(\mathbb{Q}) \ltimes \mathbb{Q}^2$. Nevertheless this is a harmless error and does not affect their result, as they were also considering elements fixing the element $e$.

\begin{remark} Interestingly enough, Singh and Venkataramana also discuss in \cite[\S 2]{SV} an alternative arithmeticity criterion closer to a recent work of Benoist and Miquel \cite{BM} which was used by Hubert and Matheus \cite{HM} to establish the arithmeticity\footnote{Actually, they initially thought that this specific origami had good chances to possess a thin Kontsevich--Zorich  monodromy.} of the Kontsevich--Zorich monodromy of a specific example origami of genus three. However, it seems hard to push this method to produce infinte families of examples of arithmetic Kontsevich--Zorich monodromies.
\end{remark}

%\begin{figure}[htb!]
%\includegraphics[scale=0.53]{picture7.pdf}
%\caption{The parabolic core $c(P_k)$ of $P_k$ belongs to the grey region inside $P_k$.}\label{f.16}
%\end{figure}

\subsection{Two-cylinder decomposition and transvections} 
    \label{cyn-decompos-transvec}
    
%\commf{This section has been rewritten. We take the comments of the referee into consideration.}    

In this paragraph we now assume that $M$ is an origami which decomposes into two cylinders for some direction $\theta \in \Z^2$. In this setting we will develop a criterion for the  Kontsevich--Zorich monodromy $\Gamma^{(0)}$ of $M$ to be Zariski dense in $\Sp_{\Omega}(2d,\R)$, see Proposition~\ref{prop:density}.

Let us first briefly recall the definition of moduli of cylinders. After that we want to explain how to associate an affine Dehn multitwist to any direction $\theta$ in which  $M$ decomposes into cylinders with commensurable moduli. This also works for general translation surfaces (cf.\cite[Lemma 3.9]{Vorobets96}). %\textcolor{purple}{(xxx Add reference)}.\commman{For example Lemma 3.9 in the survey of Vorobets \cite{Vorobets96}?}.
%First of all without loss of generality we may assume that the direction $\thata$ is horizontal, since we may consider the surface $A\cdot M$ \todo{Check this notation} instead of $M$ with some matrix $A \in \SL(2,\Z)$ which maps the direction $\theta$ to the horizontal direction.
%
Consider any horizontal flat cylinder $C$ of height $h$ and width $w$. It is stabilized by the affine action of the matrix  $A := \left(\begin{smallmatrix}1&\frac{w}{h}\\0&1\end{smallmatrix}\right)$.  The quotient $h/w$ is called the \textit{modulus of $C$}. The matrix acts as an affine homeomorphism which fixes the boundary of $C$ pointwise. If we extend this homeomorphism from the cylinder to the whole surface by the identity we obtain a Dehn twist of $M$ about the core curve of the cylinder $C$. If the cylinder is in direction $\theta$ rather than in horizontal direction, we can do the same construction and just have to replace the matrix $A$ by $A_{\theta}$, where we obtain $A_{\theta}$ form $A$ by conjugating by a rotation $R_{\theta}$ by the angle $\theta$ which maps ${1 \choose 0}$ to a vector in direction $\theta$. \\
  Suppose now that $M$ decomposes in some direction $\theta$  into cylinders $\{C_i\}_{i=1}^k$ of moduli $(\mu_i)_{i=1}^k$. Moreover suppose that the inverse moduli are commensurable i.e., that there are $k_i\in\mathbb{Z}_{>0}$ and some $L \in \R$ such that $L=k_i\,\mu_i^{-1}$ for all $i=1,\ldots, k$. Then the conjugate of the matrix $\left(\begin{smallmatrix}1&L\\0&1\end{smallmatrix}\right)$ with $R_{\theta}$ acts as $k_i$-th power of a Dehn twist on the cylinder $C_i$. The product of all these powers of Dehn twists is a Dehn multitwist and an affine homeomorphism of $M$ with derivative  $\left(\begin{smallmatrix}1&L\\0&1\end{smallmatrix}\right)$. We say it is an {\it affine Dehn multitwist} of $M$. It maps each cylinder to itself and fixes the boundaries of the cylinders pointwise.
%      Moreover, $D$ acts on each cylinder $C_i$ as the $k_i$-th power of an affine Dehn twists because for every $i=1,\ldots,k$ the relation 
%  \[
%   \begin{pmatrix}1& L        \\ 0 & 1 \end{pmatrix}=
%   \begin{pmatrix}1& \mu_i^{-1}\\0 & 1 \end{pmatrix}^{k_i}
%  \]
%holds.\\
We call the quantity $L$ the \textit{strength} of the multitwist $D$ and $k_i$ its {\it multiplicity for the cylinder $C_i$}. Observe that affine Dehn multitwists in a given direction are unique up to powers.
The action of $D$ as above on $H_1(M,\mathbb{Q})$ is given by:
\begin{equation} 
    \label{eq:action-multitwist-homology}
D_*= \Id + \sum_{i=1}^k k_i\,\Omega(\cdot,\gamma_i)\,\gamma_i,
\end{equation}
where $\gamma_i$ is the core curve of the cylinder $C_i$ and $\Omega(\cdot,\cdot)$ is the algebraic intersection form on $H_1(M,\mathbb{Z})$.

Now, we return to the case that $M$ is an origami coming by definition with a covering $\pi:M \to \mathbb{T}^2$. Then $M$ decomposes in direction $v$ in cylinders precisely for rational directions i.e., $v$ is a multiple of a vector in $\Z^2$.  In this case the moduli of the cylinders are always commensurable. One can then express the length of cylinders in terms of the covering data.
\begin{definition}
  Let $\pi:M \to \mathbb{T}^2$ be an origami and $v = (v_x ,\, v_y) \in \Z^2$ primitive i.e., $v_x$ and $v_y$ are coprime. Let $w$ be the width and $h$ the height of a cylinder on $M$ in direction $v$. Let us consider the cylinder on $\mathbb{T}$ in direction $v$.  Its width equals $|v|$ and its height equals $h_0 = \frac{1}{|v|}$.  Then we call $f = w/|v|$ the \textit{combinatorial width} and $\tilde{h} = h/h_0 = h\cdot |v|$ the \textit{combinatorial height} of the cylinder on $M$.
\end{definition}

We show in the following that if we have two cylinders in the cylinder decomposition in direction $v$ then the associated Dehn multitwist $D$ defines a transvection on $H_1^{(0)}(M,\mathbb{Q})$.

\begin{remark}\label{rem:multiplicites}
  Suppose that the direction $v \in \Z^2$  decomposes the origami $M$ into two cylinders $C_1$ and $C_2$ of combinatorial width $f_1$ and $f_2$ and of height $h_1$ and $h_2$ with core curves $\gamma_1$ and $\gamma_2$, respectively.
  \begin{enumerate}
  \item
    We write 
    \begin{equation} \label{MultiplicitiesForTwoCylinders}
      \frac{h_1}{h_2} \cdot \frac{f_2}{f_1} = \frac{k_1}{k_2} \mbox{ with } k_1, k_2 \in \Z  \mbox{ coprime }.
    \end{equation}
    Then we obtain an affine Dehn multitwist in direction $v$ with multiplicities $k_1$ and $k_2$ for the cylinders $C_1$ and $C_2$.  It is of strength
    $L = \frac{k_1w_1}{h_1} = \frac{k_2w_2}{h_2}$, where $w_1 = |v|\cdot f_1$ and $w_2 = |v|\cdot f_2$ are the widths of the cylinders.
  \item
    If $\beta \in  H_1^{(0)}(M,\mathbb{\QQ})$ then
    \begin{equation}\label{RelIntersectionNumbers}
      h_1\,\Omega(\beta,\gamma_1)+h_2\,\Omega(\beta,\gamma_2)=0
    \end{equation}
  \end{enumerate}
\end{remark}

\begin{proof}
  (1)
  By the description of the affine Dehn multitwist associated to a cylinder direction above we have to find a common multiple $L$ of the inverse moduli $\mu_1^{-1}$ and $\mu_2^{-1}$ of the two cylinders. From (\ref{MultiplicitiesForTwoCylinders}) it follows that $k_1\cdot \frac{f_1}{h_1} = k_2\cdot \frac{f_2}{h_2}$. By definition, we have that $\mu_i^{-1} = \frac{w_i}{h_i} = \frac{|v|\cdot f_i}{h_i}$. Hence we obtain the desired $L$ as
  \[
  L = k_1\cdot\mu_1^{-1} = k_1\cdot \frac{|v|\cdot f_1}{h_1} =  k_2\cdot \frac{|v|\cdot f_2}{h_2} =k_2\cdot\mu_2^{-1}
  \]
  and the multiplicities are $k_1$ and $k_2$.\\
  (2)  Let $\delta_1$ be a simple closed geodesic curve on  $\mathbb{T}^2$ in direction $v$.  Observe that $\Omega(\pi_*(\beta),\delta_1) = \tilde{h}_1\cdot \Omega(\beta,\gamma_1) + \tilde{h}_2\cdot \Omega(\beta,\gamma_2)$, where $\tilde{h}_1$ and $\tilde{h}_2$ are the combinatorial heights of $C_1$ and $C_2$. On the other hand by definition of $H_1^{(0)}(M,\mathbb{Q})$ we have $\Omega(\pi_*(\beta),\delta_1) = 0$. This shows the claim. 
%  Let $\delta_1$ be a simple closed geodesic curve on  $\mathbb{T}^2$. We also denote the corresonding homology class in  $H_1(\mathbb{T}^2,\ZZ)$ by $\detla$ and extend it to a basis $\{\delta_1, \delta_2\}$ of $H_1(\mathbb{T}^2,\ZZ)$. Then $\pi_*(\beta) = a\cdot \delta$ Since $\beta \in  H_1^{(0)}(M,\mathbb{Q})$, i.e. $\pi_*(\beta) = 0 \in H_1(\mathbb{T}^2,\ZZ)$.\pi_1^*$
\end{proof}

\begin{lemma}\label{lemma:transvection}
  Suppose that  $M$ is an origami of genus $g \geq 2$ such that the translation flow in direction $\theta$ decomposes $M$ into two cylinders and let $D$ be the associated affine Dehn multitwist as described above. Then the induced automorphism $D_*| H_1^{(0)}(M,\QQ)$ is a transvection.
\end{lemma}
\begin{proof}
Let  $C_1$ and $C_2$  be the two cylinders in direction $\theta$, let $h_1$ and $h_2$ be their heights, $w_1$ and $w_2$ their widths and $L =  \frac{k_1w_1}{h_1} = \frac{k_2w_2}{h_2}$ the strength of the associated Dehn multitwist. Then:
%Firstly, observe that for every $\beta\in H_1^{(0)}(M,\mathbb{Q})$ we have the relation $h_1\,\Omega(\beta,\gamma_1)+h_2\,\Omega(\beta,\gamma_2)=0$ (if this was not the case $\pi_*\beta\neq 0$). Hence:
\begin{align} 
\begin{split}\label{eq:transvection-general-formula}
D_*| H_1^{(0)}(M,\mathbb{Q}) &\stackrel{\ref{eq:action-multitwist-homology}}{=}  \Id + k_1\,\Omega(\cdot,\gamma_1)\,\gamma_1+k_2\,\Omega(\cdot,\gamma_2)\,\gamma_2  \\
      &\stackrel{\ref{RelIntersectionNumbers}}{=}  \Id + \Omega(\cdot,\gamma_2)\,(k_2\,\gamma_2-\frac{k_1\,h_2}{h_1}\gamma_1)\\
      &=  \Id + \Omega(\cdot, \gamma_2)\,Y,
\end{split}
\end{align}
where
\begin{align*}
 Y    &=   k_2\,\gamma_2-\frac{k_1\,h_2}{h_1}\,\gamma_1 
      =  \frac{h_2}{w_2}\,L\,\gamma_2-\frac{h_2}{w_1}\,L\,\gamma_1.
\end{align*}

The core curves $\gamma_1$ and $\gamma_2$ are parallel and of lengths $w_1$ and $w_2$, respectively, thus $Y\in H_1^{(0)}(M,\mathbb{Q})$ and $D_*$ restricted to $H_1^{(0)}(M,\mathbb{Q})$ is a transvection.
\end{proof}

%\begin{remark}
%    \label{rmk:adhoc-simplicity}
%While working with a 2-cylinder decomposition, it is important to note that there exists no canonical method to compute an explicit expression for $D_*| H_1^{(0)}(M,\mathbb{Q})$. In order to obtain an explicit computation of $(\ref{eq:transvection-general-formula})$, we employ \emph{ad hoc} methods for the sake of simplicity.
%\end{remark}

\begin{proposition}\label{prop:density}
  Suppose that $M$ is an origami of genus $g \geq 3$ which decomposes for  some direction $\theta \in \Z^2$ into two cylinders. If the Kontsevich--Zorich monodromy $\KoZoMon$ contains a Galois-pinching matrix $A$, then  $\KoZoMon$ is Zariski dense in  $\Sp_{\Omega}(2d,\R)$, where $d = g-1$.
\end{proposition}

\begin{proof} 
  It follows from Lemma~\ref{lemma:transvection} that $\KoZoMon$ contains a transvection $B$. In particular $B$ is a non-trivial unipotent matrix and $(B-\Id)(\mathbb{R}^{2d})$  has dimension one. Now, Zariski density follows from Lemma~\ref{lemma:criteria-zariski-density} and Remark~\ref{rmk:criteria-zariski-density} since $d \geq 2$.
\end{proof}

\subsection{Origamis in $\mathcal{H}(4)$}

We now restrict to origamis in the stratum $\mathcal{H}(4)$, i.e. to origamis with a single singularity of order 4, which are then of genus 3.
%The moduli space of translation surfaces $(M,\omega)$ of compact Riemann surfaces $M$ of genus three equipped with Abelian differentials $\omega$ possessing a single zero (of order four) is denoted by $\mathcal{H}(4)$.
After Kontsevich and Zorich \cite{KZ}, $\mathcal{H}(4)$ has two connected components $\mathcal{H}^{hyp}(4)$ and $\mathcal{H}^{odd}(4)$. $(M,\omega)\in\mathcal{H}(4)$ belongs to $\mathcal{H}^{hyp}(4)$ if and only if it admits an \textit{anti-automorphism}, i.e. an affine homeomorphism with linear part $-\Id$, which has eight fixed points. $(M,\omega)\in\mathcal{H}^{odd}(4)$ can also have an anti-automorphism which then has four fixed points. In this case we say that $(M,\omega)$ belongs to the \textit{Prym locus} of $\mathcal{H}^{odd}(4)$.

%zzz
\subsubsection{$\SL(2,\mathbb{Z})$-orbits in $\mathcal{H}(4)$ and their invariants}

%Recall that the subset of origamis can be organized into $\SL(2,\mathbb{Z})$-orbits: Indeed, any origami $(M,\pi^*(dz))$, obtaines as a Torus-covering $\pi:M\to\mathbb{T}^2$, is coded by a pair of permutations $(h,v)\in S_N\times S_N$, where $N$ is the degree of $\pi$, modulo simultaneous permutations (i.e., $(h,v)\equiv (\phi h\phi^{-1},\phi v\phi^{-1})$), and the usual parabolic generators of $\SL(2,\mathbb{Z})$ act on pairs of permutations via the Nielsen transformations $(h,v)\mapsto (h,vh^{-1})$ and $(h,v)\mapsto (hv^{-1},v)$.

%From the discussion in the previous paragraph, it is clear that the isomorphism class of the group $G=\langle h, v\rangle\subset S_N$ generated by a pair of permutations $(h,v)$ coding an origami $(M,\omega)$ is an invariant of its $\SL(2,\mathbb{Z})$-orbit. This isomorphism class is called the \textit{monodromy}\footnote{It should not be confused with the Kontsevich--Zorich monodromy!} of $(M,\omega)$.
As it was shown by Zmiaikou (see also \cite[Prop. 6.1 and 6.2]{MMY}), the monodromy group of a reduced origami $(M,\omega)\in\mathcal{H}(4)$ is $A_N$ or $S_N$ if the degree of the corresponding covering $\pi:M\to\mathbb{T}^2$ is $N\geq 7$. If a reduced origami $(M,\omega)\in\mathcal{H}(4)$ possesses an anti-automorphism $\iota$, then the fixed points of $\iota$ project under $\pi$ on the $2$-torsion (half-integer) points of $\mathbb{T}^2$, and we obtain the list $(l_0,[l_1,l_2,l_3])$, where 
\begin{align*}
l_0 = \#(\textrm{Fix}(\iota)\cap\pi^{-1}(0))-1,  &\qquad l_1=\#(\textrm{Fix}(\iota)\cap\pi^{-1}(1/2))\\
l_2=\#(\textrm{Fix}(\iota)\cap\pi^{-1}(i/2)),    &\qquad l_3=\#(\textrm{Fix}(\iota)\cap\pi^{-1}((1+i)/2)).
\end{align*} 
Here $[l_1,l_2,l_3]$ denotes the unordered triple. Since $\SL(2,\mathbb{Z})$ acts on $\mathbb{T}^2$ by fixing $0\in\mathbb{T}^2$ and permuting the other three $2$-torsion points, one has that $(l_0,[l_1,l_2,l_3])$ is also an invariant of the $\SL(2,\mathbb{Z})$-orbit of $(M,\omega)$ called its HLK invariant (after the work of Hubert--Leli\`evre and Kani on genus two origamis).

\subsubsection{Delecroix--Leli\`evre conjecture}{\label{ss:Delecroix-Lelievre-conjecture}}

After performing extensive numerical experiments using SageMath, Delecroix and Leli\`evre conjectured that the monodromy group and the HLK invariants suffice to classify $\SL(2,\mathbb{Z})$-orbits of reduced origamis $(M,\omega)$ in $\mathcal{H}(4)$ tiled by $N>8$ squares (i.e., the covering $\pi:M\to\mathbb{T}^2$ has degree $N>8$). More precisely, their conjecture says:

\begin{itemize}
\item 
Outside of the Prym locus of $\mathcal{H}^{odd}(4)$, there are two $\SL(2,\mathbb{Z})$-orbits distinguished by the values $A_N$ or $S_N$ of their monodromy groups;
\item 
in $\mathcal{H}^{hyp}(4)$, if $N$ is odd, then there are four $\SL(2,\mathbb{Z})$-orbits distinguished by the values of their HLK invariants;
\item 
in $\mathcal{H}^{hyp}(4)$, if $N$ is even, then there are three $\SL(2,\mathbb{Z})$-orbits distinguished by the values of their HLK invariants.
\end{itemize}

\begin{remark} Concerning the Prym locus of $\mathcal{H}^{odd}(4)$, the analogue of the Delecroix--Leli\`evre conjecture was established by Lanneau and Nguyen \cite{LN}: in this situation, the HLK invariant is a complete invariant of $\SL(2,\mathbb{Z})$-orbits (taking one or two values depending on the parity of $N$).
\end{remark}

Closing this section, let us observe that, conditionally to the Delecroix--Leli\`evre conjecture, the main results of \cite{MMY} together with the discussion in \S\ref{ss.Prasad-Rapinchuk-MMY} above allow to conclude the Zariski density in $\textrm{Sp}(H_1^{(0)}(M,\mathbb{Q}))$ of the Kontsevich--Zorich monodromy of all but finitely many reduced origamis in $(M,\omega)\in \mathcal{H}(4)$ outside the Prym locus \footnote{Note that if $(M,\omega)\in\mathcal{H}^{odd}(4)$ is an origami in the Prym locus, then $H_1^{(0)}(M,\mathbb{Q})$ decomposes into two subbundles which are respected by all affine homeomorphisms. In particular, the Kontsevich--Zorich monodromy of an origami in the Prym locus of $\mathcal{H}^{odd}(4)$ is included in a product $\SL(2,\mathbb{R})\times \SL(2,\mathbb{R})$.} (cf. Appendix~\ref{a.examplePrym}).

%%.TEX FILES INPUT HERE%%

\subsection{Generic frame for the proof of arithmeticity}
    \label{ssec:general-algo}
    
    Based on the theory presented in the last sections, to prove arithmeticity of Kontsevich-Zorich monodromies of origamis we will proceed in the following way:
\begin{itemize}
  \item 
  We first show that  the Kontsevich-Zorich monodromy $\KoZoMon$ of the given origami 
  $\mathcal{O}$ is Zariski dense in  $\SP(H_1^{(0)}(\mathcal{O},\R))$.
  \item 
  We then find three 2-cylinder directions  $\theta_1$, $\theta_2$ and $\theta_3$ of $\mathcal{O}$. 
  As explained in Section~\ref{cyn-decompos-transvec} this gives us three  affine Dehn multitwists 
  $T_1$, $T_2$, $T_3$ whose multiplicities can be computed by Remark~\ref{rem:multiplicites}. 
  By Lemma ~\ref{lemma:transvection} the induced actions $\tilde{D}_1$, $\tilde{D}_2$, $\tilde{D}_3$ 
  of these multitwists on homology have restrictions $D_i = \tilde{D}_i| {H_1^{(0)}(M,\QQ)}$ 
  on the non tautological part of homology that are transvections. 
  Moreover, by (\ref{eq:transvection-general-formula}) in the same Lemma we have 
    \[
    (D_*-\Id)| {H_1^{(0)}(M,\mathbb{Z})}=\mathbb{Z}Y.
    \]
  \item 
  For each transvection $D_i$ we obtain from (\ref{eq:transvection-general-formula}) a generator 
  $X_i \in H_1^{(0)}(M,\Z)$ of the image of $D_i - \Id$.
  We check that $\Omega(X_i,X_j) \neq 0$ for some $i,j=1,2,3$.
  We then show that $X_1, X_2, X_3$ are linearly independent and consider the three-dimensional space 
  $W = <X_1, X_2, X_3>$.
  \item
   For some $i,j \in \{1,2,3\}, i\neq j$, we find an element $e \in W$ in the radical of the intersection form $\Omega$ such that $e$ is not contained 
   in the subspace $<X_i, X_j>$. Then $\{X_i, X_j, e\}$ is a basis of $W$. 
   We then compute the transition matrices $A_1$, $A_2$, $A_3$ of the linear maps $D_1|_W$, $D_2|_W$ and $D_3|_W$ 
   with respect to this basis. Finally, we find a matrix $A$ in the group generated by $A_1, A_2$ and $A_3$ 
   such that $A$ has the form given in (\ref{NiceFormUnipotent}). 
   Now, it follows from Theorem~\ref{thm:sing-venkataramana} that the Kontsevich-Zorich monodromy 
   $\KoZoMon$ is arithmetic.
\end{itemize}

%%%%%%%%%%%%%%%%%%%%%%%%%%%%%%%%%%%%%%%%%%%%%%%%%%%%%%%
%%%%%%%%%%%%%% Section 3 %%%%%%%%%%%%%%%%%%%%%%%%%%%%%%
%%%%%%%%%%%%%%%%%%%%%%%%%%%%%%%%%%%%%%%%%%%%%%%%%%%%%%%

\section{Arithmeticity of Kontsevich--Zorich monodromies in $\mathcal{H}^{odd}(4)$}\label{s.H4odd}

This section is divided in two parts. In the first part we describe seven infinite families of origamis in $\mathcal{H}^{odd}(4)$. In the second part we use the approach described in Section~\ref{ssec:general-algo} to show that all but finitely many elements in each family have arithmetic Kontsevich--Zorich monodromy. 

\subsection{The seven families}\label{TheFamilies}
The families we describe in this subsection are grouped in two groups. The first group stems from~\cite{MMY}, Chapter 7, and it contains six of the seven families. We recall parts of that chapter in the following. In Figure \ref{familyMMY} we depict an origami which depends on six positive real parameters $H_1$, $H_2$, $H_3$, $V_1$, $V_2$, $V_3$. Table~\ref{tabla1} describes nine disjoint families of
origamis according to these parameters. In this table we assume $n\geq 1$. Furthermore, $N$ is the number of squares and  ``$\rm Mon$'' the monodromy group (see Proposition 6.2 in~\cite{MMY} for details) of the corresponding origami. In the column with heading ``Origami'' we have listed the name that we gave to  the corresponding origami. Observe that by considering the number of squares and the monodromy one directly sees that these nine families are disjoint. Even more their $\SL(2,\Z)$-orbits are disjoint. In particular, the naming of the origamis is consistent.
\begin{table}[!h]
  \centering
\begin{tabular}{||c c c c c c c c c||}
 \hline
Origami & N & $H_1$ & $H_2$ & $H_3$ & $V_1$ & $V_2$ & $V_3$ & Mon\\ [0.5ex]
 \hline\hline
 $O_{3n+8,S_N}$ & 3n+8 & 1 & 2 & 1 & 1 & 2 & 3n & $S_N$ \\
 \hline
 $O_{3n+10,S_N}$ &3n+10 & 1 & 3 & 1 & 1 & 2 & 3n & $S_N$ \\
 \hline
 $O_{3n+12,S_N}$&3n+12 & 1 & 4 & 1 & 1 & 2 & 3n & $S_N$ \\
 \hline
 $O_{6n+13,A_N}$&6n+13 & 2 & 4 & 1 & 1 & 2 & 6n & $A_N$ \\
 \hline
 $O_{6n+14,A_N}$&6n+14 & 2 & 3 & 1 & 1 & 2 & 6n+3 & $A_N$ \\
 \hline
$O_{6n+17,A_N}$ &6n+17 & 2 & 6 & 1 & 1 & 2 & 6n & $A_N$ \\
 \hline
 $O_{6n+18,A_N}$&6n+18 & 2 & 5 & 1 & 1 & 2 & 6n+3 & $A_N$ \\
 \hline
 $O_{6n+21,A_N}$&6n+21 & 2 & 8 & 1 & 1 & 2 & 6n & $A_N$\\
 \hline
 $O_{6n+22,A_N}$&6n+22 & 2 & 7 & 1 & 1 & 2 & 6n+3 & $A_N$\\
 \hline
\end{tabular}
\caption{Parameters for families of origamis \label{tabla1}}
\end{table}

\begin{figure}[!h]
\centering
\begin{tikzpicture}
\draw[color=red] (0.5,0) -- (0.5,2);
\draw[color=red] (2,0) -- (2,1);
\draw[color=red] (5.5,1) -- (5.5,-3);
\draw[color=blue] (0,1.5) -- (1,1.5);
\draw[color=blue] (0,0.3) -- (6,0.3);
\draw[color=blue] (4,-.5) -- (6,-.5);

\draw (0,0) rectangle (1,1) node[pos=0.5,left=12pt,black]{$H_3$};
\draw (0,1) rectangle (1,2) node[pos=0.5,left=12pt,black]{$H_1$} node[pos=0.5,above=12pt,black]{$V_1$};
\draw (1,0) rectangle (4,1)  node[pos=0.5,above=12pt,black]{$V_3$};
\draw (4,0) rectangle (5,1);
\draw (5,0) rectangle (6,1);
\draw (4,0) rectangle (6,-3) node[pos=0.5,right=28pt,black]{$H_2$} node[pos=0.5,below=42pt,black]{$V_2$};

\draw [dashed] (1.5,0.5) -- (3.5,0.5);

% Label below the red line
\node at (0.5,-0.2) {$\zeta_3$};
\node at (2,-0.2) {$\zeta_1$};
\node at (5.5,-3.2) {$\zeta_2$};
\node at (1.3,1.5) {$\sigma_1$};
\node at (6.3,0.3) {$\sigma_2$};
\node at (6.3,-.5) {$\sigma_3$};

\end{tikzpicture}
\caption{Family prototypes for primitive origamis in $\mathcal{H}^{odd}(4)$}
\label{familyMMY}
\end{figure}

In Table~\ref{tabla2} we introduce six subfamilies $\mathcal{F}_1$, \ldots, $\mathcal{F}_6$ of families from Table~\ref{tabla1} which, as we will show later, have all arithmetic Kontsevich--Zorich monodromy except possibly for finitely many $k\geq 1$.
\begin{table}[!h]
  \centering
\begin{tabular}{ |c|c| }
 \hline
$\mathcal{F}_1: \mathcal{O}_{3n+8,S_N}$ with $n=20k-2$   &  $\mathcal{F}_2: \mathcal{O}_{3n+10,S_N}$ with $n=4k-3$ \\[0.5ex]
 \hline
 $\mathcal{F}_3: \mathcal{O}_{3n+12,S_N}$ with $n=144k-3$ & $\mathcal{F}_4: \mathcal{O}_{6n+14,A_N}$ with $n=4k-2$ \\[0.5ex]
 \hline
 $\mathcal{F}_5: \mathcal{O}_{6n+18,A_N}$ with $n=3k-1$  & $\mathcal{F}_6: \mathcal{O}_{6n+22}$ with $n=4k-1$ \\
\hline
\end{tabular}
\caption{\label{tabla2}}
\end{table}

%In other terms, we shall establish in Subsections \ref{ss.H4odd-family1}, \ref{ss.H4odd-family2}, \ref{ss.H4odd-family3}, \ref{ss.H4odd-family4}, \ref{ss.H4odd-family5} and \ref{ss.H4odd-family6} below the following result: 
%
%\begin{theorem}\label{t.H4odd-families1to6} There exists an integer $k_0\geq 1$ such that the Kontsevich--Zorich monodromy of an origami $\mathcal{O}$ is arithmetic whenever:
%\begin{itemize}
%    \item[(i)] $\mathcal{O}\in\mathcal{F}_{3n+8}$ and $n=20k-2$ with $k\geq k_0$, or 
%    \item[(ii)] $\mathcal{O}\in\mathcal{F}_{3n+10}$ and $n=4k-3$ with $k\geq k_0$, or    
%    \item[(iii)] $\mathcal{O}\in\mathcal{F}_{3n+12}$ and $n=144k-3$ with $k\geq k_0$, or 
%    \item[(iv)] $\mathcal{O}\in\mathcal{F}_{6n+14}$ and $n=4k-2$ with $k\geq k_0$, or 
%    \item[(v)] $\mathcal{O}\in\mathcal{F}_{6n+18}$ and $n=3k-1$ with $k\geq k_0$, or
%    \item[(vi)] $\mathcal{O}\in\mathcal{F}_{6n+22}$ and $n=4k-1$ with $k\geq k_0$.
%\end{itemize}
%\end{theorem}

We finally introduce in the following the last family $\mathcal{F}_7$. Let $\mathcal{O}'_{K,N}$ be the origami associated to the pair of permutations:
$$h=(1)(2)\dots(K-2)(K-1,K)(K+1,K+2,\dots, K+N)$$ 
and 
$$v=(K+1,K-1,K-2,\dots, 1)(K+2,K)(K+3)(K+4)\dots(K+N).$$
Note that $\mathcal{O}'_{K,N}\in\mathcal{H}^{odd}(4)$ has monodromy $A_{K+N}$ or $S_{K+N}$ depending if $K$ and $N$ are both even or not. Now, the family $\mathcal{F}_7$ consists by definition of the origamis:
\begin{table}[!h]
  \centering
\begin{tabular}{ |c| }
\hline
$\mathcal{F}_7 : \mathcal{O}'_{K,N}$ with $K=3n$, $N=5n$ and $n\geq 1$\\
\hline
\end{tabular}
\caption{\label{tabla3}}
\end{table}

\begin{figure}[h]
\centering
\begin{tikzpicture}
\draw (0,0) rectangle (1,1);
\draw (1,1) rectangle (2,2);
\draw (0,1) rectangle (1,2);
\draw (0,2) rectangle (1,4);
\draw (0,4) rectangle (1,5);
\draw (1,0) rectangle (2,1);
\draw (2,0) rectangle (4,1);
\draw (4,0) rectangle (5,1);

\draw[color=red] (0.25,0) -- (0.25,5);
\draw[color=red] (1.5,0) -- (1.5,2);
\draw[color=red] (4.5,0) -- (4.5,1);

\draw[color=blue] (0,4.5) -- (1,4.5);
\draw[color=blue] (0,1.5) -- (2,1.5);
\draw[color=blue] (0,0.25) -- (5,.25);

\node at (0.25,5.3) {$\zeta_3$};
\node at (1.5,2.3) {$\zeta_2$};
\node at (4.5,1.3) {$\zeta_1$};
\node at (1.3,4.5) {$\sigma_1$};
\node at (2.3,1.5) {$\sigma_2$};
\node at (5.3,.25) {$\sigma_3$};

\draw [decorate,line width=0.5mm,
    decoration = {brace}] (-0.5,0) --  (-0.5,5) node[pos=0.5,left=10pt,black]{$K$};

\draw [decorate,line width=0.5mm,
    decoration = {brace,mirror}] (0,-0.5) --  (5,-0.5) node[pos=0.5,below=10pt,black]{$N$};

\draw [dashed] (2.5,0.5) -- (3.5,0.5);
\draw [dashed] (0.5,2.5) -- (0.5,3.5);

%\draw (5,4) node[pos=0,below=0pt,black]{$\dots$};

\end{tikzpicture}
\caption{The origami $\mathcal{O}'_{K,N}$}
\label{KNOrigami}
\end{figure}

We now state the result of this section which will be proved in the following subsections.

\begin{theorem}\label{t.H4odd-families1to6} For each family $\mathcal{F}_1$, \ldots, $\mathcal{F}_7$ the Kontsevich--Zorich monodromy of all but finitely many origamis in this family is arithmetic.
\end{theorem}

\begin{remark}\label{OurBasis}
  Observe that for the origamis $\mathcal{O}$ in the seven families  the horizontal and the vertical directions are both 3-cylinder directions. Let $\sigma_1, \sigma_2, \sigma_3$ and $\zeta_1,\zeta_2, \zeta_3$ be the waist curves of the horizontal and the vertical cylinder decompositions, respectively. By \cite{LN} their union $\tilde{B} = (\sigma_i)\cup(\zeta_i)$  is  a basis of $H_1(\mathcal{O},\mathbb{\Z})$. We will use this basis frequently in the proof of Theorem~\ref{t.H4odd-families1to6}
\end{remark}

%\begin{theorem}\label{t.H4odd-families1to6} There exists an integer $k_0\geq 1$ such that the Kontsevich--Zorich monodromy of an origami $\mathcal{O}$ is arithmetic whenever:
%\begin{itemize}
%    \item[(i)] $\mathcal{O}\in\mathcal{F}_{3n+8}$ and $n=20k-2$ with $k\geq k_0$, or 
%    \item[(ii)] $\mathcal{O}\in\mathcal{F}_{3n+10}$ and $n=4k-3$ with $k\geq k_0$, or    
%    \item[(iii)] $\mathcal{O}\in\mathcal{F}_{3n+12}$ and $n=144k-3$ with $k\geq k_0$, or 
%    \item[(iv)] $\mathcal{O}\in\mathcal{F}_{6n+14}$ and $n=4k-2$ with $k\geq k_0$, or 
%    \item[(v)] $\mathcal{O}\in\mathcal{F}_{6n+18}$ and $n=3k-1$ with $k\geq k_0$, or
%    \item[(vi)] $\mathcal{O}\in\mathcal{F}_{6n+22}$ and $n=4k-1$ with $k\geq k_0$.
%    \item[(vii)] $\mathcal{O}\in\mathcal{F'}_{N,K}$ and $N = 3n, K = 5n$ with $k\geq k_0$.
%\end{itemize}
%\end{theorem}

%In this context, we will show in Subsection \ref{ss.H4odd-family7} below the following statement: 
%
%\begin{theorem}\label{t.H4odd-family7} The Kontsevich--Zorich monodromy of $\mathcal{O}'_{3n,5n}$ is arithmetic for all but finitely many choices of  $n\in\mathbb{N}$. 
%\end{theorem}

%\begin{remark}
  %The 7 families described above are disjoint. Indeed, all families of the form $\mathcal{F}_N$ are disjoint by construction since no two elements in any of these has the same number of squares. On the other hand, $\mathcal{O'}_{3n,5n}$ is an origami having $K+N=8n$ squares, and no origami in Table~\ref{tabla2} has a number of squares which is a multiple of 8.
%\end{remark}

\subsection{Zariski-density for the families $\mathcal{F}_1$, \ldots, $\mathcal{F}_7$}

In this section we study the Zariski-density for the origamis in our seven families.  

\begin{proposition}\label{ZDensitiyForTheSeven}
  The Kontsevich-Zorich monodromies of all but finitely many origamis in the seven families $\mathcal{F}_1$, \ldots, $\mathcal{F}_7$ are Zariski dense.
\end{proposition}

\begin{proof}
  For the origamis in the families $\F_1$, \ldots, $\F_6$ the statement follows from \cite{MMY} as follows. Let $\mathcal{O}$ be an origami in one of the families described by Table~\ref{tabla1} tiled by $N$ squares. Proposition 7.5 in~\cite{MMY} implies that for $N$ sufficiently large the Kontsevich--Zorich monodromy of $\mathcal{O}$ has a Galois-pinching element. Proposition 7.3 in~[\emph{Ibid.}] says that $\mathcal{O}$ decomposes into the direction $(3,1)$ in two cylinders. Now the claim follows from Proposition~\ref{prop:density}.
  
For the origamis in the family $\F_7$ we firstly detect that $(1,1)$ is a two-cylinder direction. We find in this case a  Galois-pinching element in the Kontsevich--Zorich monodromy of $\mathcal{O}' = \mathcal{O}'_{K,N}$ for $K = 3n$, $N = 5n$ and $n$ large by explicit computations as follows. Let us consider the Dehn multitwists $\Thor$ and $\Tver$ in horizontal and in vertical direction. We read off from Figure~\ref{KNOrigami} that the strength of $\Thor$ is $L = 2N$ with multiplicities $k_1 = 2N(K-2)$, $k_2 = N$ and $k_3 = 2$. For $\Tver$ we obtain the strength $L = 2K$ and multiplicities $k_1 = 2K(N-2)$, $k_2 = K$ and  $k_3 = 2$. We use the basis  $\tilde{B} = (\sigma_i)\cup(\zeta_i)$ of $H_1(\mathcal{O},\mathbb{\Z})$ (cf. Remark~\ref{OurBasis}) and the basis 
  \[
  \{e_1 = \sigma_2 - 2\sigma_1,~ e_2 = \sigma_3 - N\sigma_1, ~ e_3 = \zeta_2 - 2\zeta_1, ~ e_4 = \zeta_3 - K\zeta_1\}
  \]
of $H_1^{(0)}(\mathcal{O}'_{K,N},\mathbb{Q})$ computed in (\ref{BasisEforF7}). Using the intersection numbers between horizontal and vertical core curves which we read off from Figure~\ref{KNOrigami} we obtain for example:
  \[
  \Omega(e_3,\sigma_3) = \Omega(\zeta_2,\sigma_3) - 2\Omega(\zeta_1,\sigma_3) = -1 + 2 = 1.
  \]
Similarly, we obtain
  \[
  \Omega(e_3,\sigma_1) = 0 \quad \mbox{and} \quad \Omega(e_3, \sigma_2) = -1
  \]
  %from Figure~\ref{KNOrigami} and using
From  (\ref{eq:action-multitwist-homology}) we then compute for the action of $\Thor$ for example:% and $\Tver$ on $H_1^{(0)}(\mathcal{O}'_{K,N},\mathbb{Q})$:
  \begin{equation}\label{ThorOne3}
  %\begin{array}{}
    \Thor: e_3 \mapsto e_3 + 2N(K- 2) \cdot 0 \cdot \sigma_1 + N\cdot (-1) \cdot \sigma_2 + 2\cdot 1 \cdot \sigma_3 = e_3 - Ne_1 + 2e_2
  %\end{array}
  \end{equation}
    We determine the remaining intersection numbers:
   \begin{IEEEeqnarray*}{rClcrClcrCl}
    \Omega(e_4,\sigma_1) &=& -1, &\qquad& \Omega(e_4,\sigma_2)  &=& - 1,  &\qquad& \Omega(e_4,\sigma_3) &=& -1+K,\\
    \Omega(e_1,\zeta_1)  &=& 0,  &&       \Omega(e_1,\zeta_2)   &=&  1,    &&      \Omega(e_1,\zeta_3)  &=& -1,\\
    \Omega(e_2,\zeta_1)  &=& 1,  &&       \Omega(e_2,\zeta_2)   &=& 1,    &&       \Omega(e_2,\zeta_3)  &=& 1-N
   \end{IEEEeqnarray*}
With similar computations as in (\ref{ThorOne3}) we obtain for the transformation matrices $A$ and $B$ of the action of $\Thor$ and $\Tver$ on $H_1^{(0)}(\mathcal{O}'_{K,N},\mathbb{Q})$ with respect to the basis $\{e_1,e_2,e_3,e_4\}$:
    \[
    A = \begin{pmatrix} 1&0&-N& -N\\ 0&1&2&2(K-1)\\0&0&1&0\\0&0&0&1\end{pmatrix} , \quad 
    B = \begin{pmatrix} 1&0&0&0\\0&1&0&0\\K&K&1&0\\-2&2(1-N)&0&1\end{pmatrix} 
    \]
Using Lemma \ref{lemma:galois-pinching-criterion} we show that $C = A^{-1}B$ is Galois-pinching if $n$ is large enough.
Using the Computer Algebra System Mathematica we compute the characteristic polynomial of $C$:
  \begin{align*}
   & x^4+(6N+K(6-5N)-8)x^3+2(7-6N+2K^2(N-2)N\\
  +& K(-4N^2+13N-6))x^2+(6N+K(6-5N)-8)x+1.
  \end{align*}
By taking $K=3n$ and $N=5n$, this polynomial becomes $x^4+ax^3+bx^2+ax+1$ where
$$a= -75n^2+48n-8 \quad \textrm{and} \quad b=2(450n^4-480n^3+195n^2-48n+7).$$
Note that the quantities $t=-a-4$ and $d=b+2a+2$ are positive for all $n$ sufficiently large. Furthermore, using again Mathematica we obtain the following factorization of the discriminants $\Delta_1$ and $\Delta_2$ and of their product $\Delta_1\Delta_2$:
\begin{align*}
          \Delta_1&=  a^2-4b+8=2025n^4-3360n^3+1944n^2-384n+16,\\
          \Delta_2&=  (b+2+2a)(b+2-2a)\\
                  &=  240n^2(5n-2)(3n-2)(225n^4-240n^3+135n^2-48n+8),\\
  \Delta_1\Delta_2&=  240n^2(5n-2)(3n-2)(225n^4-240n^3+135n^2-48n+8)\\
                  &\quad (2025n^4-3360n^3+1944n^2-384n+16)
\end{align*}

%and 
%$$\Delta_1\Delta_2=240n^2(5n-2)(3n-2)(225n^4-240n^3+135n^2-48n+8)(2025n^4-3360n^3+1944n^2-384n+16)$$
These are polynomial functions of $n$ taking positive values for all $n$ sufficiently large such that their square-free parts have degrees greater or equal four. 
As it is explained in \cite[\S6.7]{MMY}, these facts imply that $C = A^{-1}B$ is a Galois-pinching matrix for all $n$ sufficiently large (thanks to Siegel's theorem).
%Now it follows from Proposition~\ref{prop:density} that $\KoZoMon$ is Zariski dense in  $\Sp_{\Omega}(H_1^{(0)},\R)$.
\end{proof}

%The following lemma takes care of Zariski-density for all but finitely many elements in the families described by Table~\ref{tabla1}:

%\begin{lemma}
%Let $\mathcal{O}$ be an origami in one of the families described by Table~\ref{tabla1} and tiled by $N$ squares. Then for $N$ sufficiently large the Kontsevich--Zorich monodromy of $\mathcal{O}$ is Zariski-de%nse in $\SP(H_1^{(0)}(\mathcal{O},\R))$.
%\label{freezariskidense}
%\end{lemma}
%%\commf{The proof of this lemma has been rewritten to take into account the modifications made to Section~\ref{ss.Prasad-Rapinchuk-MMY}.}
%\begin{proof}
%  Proposition 7.5 in~\cite{MMY} implies that for $N$ sufficiently large the Kontsevich--Zorich monodromy of $\mathcal{O}$ has a Galois-pinching element. Proposition 7.3 in~[\emph{Ibid.}] says that direction i%n the plane determined by the vector $(3,1)$ is a 2-cylinder direction. Now the claim follows from Proposition~\ref{prop:density}.
%%  Since we are in the context of Origamis, there exists an affine Dehn multitwist $D$ in this direction. The restriction $D|H_1^{(0)}(\mathcal{O},\mathbb{Q})$ defines a transvection which has associated an u%nipotent matrix $B\neq \Id$ and $(B-\Id)(\mathbb{R}^4)$ has dimension 1. Zariski density follows from Lemma~\ref{lemma:criteria-zariski-density} and Remark~\ref{rmk:criteria-zariski-density}.
%\end{proof}

\begin{remark}
    Observe that the proof of density for the families $\F_1,\ldots, \F_6$ in the first part of Proposition~\ref{ZDensitiyForTheSeven} works  more generally for all but finitely many origamis in Table~\ref{tabla1}. Assuming the  Delecroix-Leli\`evre conjecture (see Section~\ref{ss:Delecroix-Lelievre-conjecture}) this implies Zariski-density for the majority of the origamis in the whole stratum (that is, all but finitely many origamis outside the Prym locus).
\end{remark}

\subsection{Proof of arithmeticity for our seven families in $\mathcal{H}^{odd}(4)$}\label{TheSteps}

It turns out that we can proceed in a unified way to show arithmeticity for the Kontsevich-Zorich monodromy $\KoZoMon$ of an origami $\mathcal{O}$ in our families in $\mathcal{H}^{odd}(4)$ after we have shown Zariski density. We proceed in the following steps:\\

%We use two different methods: 
%\begin{itemize}
%    \item \emph{Monodromies for origamis in $\mathcal{H}(4)$}. Here we use Lemma~\ref{lemma:criteria-zariski-density}, for which we need to certify that the monodromy in question has a Galois-pinching matrix.  Depending on the monodrom%y, the existence of a Galois-pinching element will follow from either Lemma~\ref{lemma:galois-pinching-criterion} or Proposition 7.5 in~\cite{MMY} (see the next Section for more details).
%    \item \emph{Monodromies for origamis in genus 4}. In this case the criterion we use is the following result of Detinko, Flannery and Hulpke, (see~\cite{Detinko2018ZariskiDA} Prop. 3.7) : 
%    \begin{theorem}
%        \label{Thm:ZarDenTrans0}
%Suppose that a subgroup \(H\leq \Sp(2n,\mathbb{Z})\) contains a transvection \(t\in H\). Then \(H\) is Zariski dense if and only if the normal closure \(\langle t\rangle^H\) of \(t\) in \(H\) is absolutely irreducible.
%\end{theorem}
%    The verification of this criterion for the family of genus 4 Origamis we study is computer assisted. See Section~\ref{ssec:zariski-density-genus4} for more details.
%\end{itemize}

%Once the Zariski density of the monodromy of the origami $\mathcal{O}$ in question is established, we show (in all cases) arithmeticity by applying the Sing-Venkataramana criterion as stated in Theorem~\ref{thm:sing-venkataramana}. This is achieved in the following way: 

%\begin{itemize}

%\item [(A)]
\textbf{(A) Basis of  \boldmath{$H_1^{(0)}(\mathcal{O},\mathbb{Z})$}:}\\
  Using the basis $\tilde{B}$ of $H_1(\mathcal{O},\mathbb{Z})$ from Remark~\ref{OurBasis}, we produce a basis $\{e_1, e_2, e_3, e_4\}$ of the non-tautological part $H_1^{(0)}(\mathcal{O},\mathbb{Q})$ in the following way. First remark that $H_1^{(0)}(\mathcal{O},\mathbb{Q})$ is the kernel of the holonomy map $\hol:  H_1(\mathcal{O},\mathbb{Q}) \to \R^2$ given by $\hol(\gamma)=\int_\gamma\omega$, where $\mathcal{O}=(M,\omega)$. From Figures~\ref{familyMMY} and~\ref{KNOrigami} we observe that we have $\hol(\sigma_i) = (l(\sigma_i),0)$ and $\hol(\zeta_i) = (0,l(\zeta_i))$. Here $l(\sigma_i)$ and $l(\zeta_i)$ are the respective lengths. Furthermore, we have as specific feature of the origamis in our seven families that there is a horizontal and a vertical cylinder core curve with holonomy vector ${1 \choose 0}$ and ${0 \choose 1}$, respectively.\\ 
  For the family $\F_7$ this is the case for $\sigma_1$ and $\zeta_1$. A short linear algebra computation shows that then - for the family $\F_7$ - the following is  a basis of the kernel of $\hol$ and thus of $H_1^{(0)}(\mathcal{O},\mathbb{Q})$:
  \begin{equation}
      \begin{split}
          \begin{IEEEeqnarraybox}[][c]{lCrClCr} \label{BasisEforF7}
            e_1 &=&  \sigma_2 - l(\sigma_2)\cdot \sigma_1,  &\qquad&   e_2  &=& \sigma_3 - l(\sigma_3)\cdot \sigma_1,\\
            e_3 &=&  \zeta_2 - l(\zeta_2)\cdot \zeta_1,     &      &   e_4  &=& \zeta_3 - l(\zeta_3)\cdot \zeta_1     
          \end{IEEEeqnarraybox}
      \end{split}
  \end{equation}
  As a matter of fact (\ref{BasisEforF7}) is a $\Z$-basis of  $H_1^{(0)}(\mathcal{O},\mathbb{Z})$. For the families $\F_1$ to $\F_6$ the core curves $\sigma_1$ and $\zeta_2$ have length 1 and we therefore obtain  for family $\F_1$ to $\F_6$ the following
  basis of $H_1^{(0)}(\mathcal{O},\mathbb{Z})$:
  \begin{equation}
      \begin{split}
        \begin{IEEEeqnarraybox}[][c]{rClcrCl}\label{BasisEforF1to6} 
          e_1 &=& \sigma_2 - l(\sigma_2)\cdot \sigma_1,&\qquad&  e_2 &=& \sigma_3 - l(\sigma_3)\cdot \sigma_1,\\
          e_3 &=& \zeta_1 - l(\zeta_1)\cdot \zeta_2,   &      &  e_4 &=& \zeta_3 - l(\zeta_3)\cdot \zeta_2          
        \end{IEEEeqnarraybox}
      \end{split}
  \end{equation}
  We finally obtain the transformation matrix $BC_1$ i.e., the matrix such that  for  any vector $x \in  H_1^{(0)}(\mathcal{O},\mathbb{Q})$ we obtain the coordinate vector $x_{\tilde{B}}$ with respect to the basis $\tilde{B} = \{\sigma_1, \ldots, \zeta_3\}$ from the coordinate vector $x_B$ of $x$ with respect to the basis $B = \{e_1,e_2,e_3,e_4\}$ as
  \begin{equation}\label{BaseChange}
    x_{\tilde{B}} = BC_1\cdot x_B.
  \end{equation}    
  %$\pi_{\star}(\sigma_i) = {1 \choose 0}$ and $\pi_\{\zeta_j\}$ have derivative ${0 \choose 1}$.\\

\textbf{(B) The three 2-cylinder directions:}\\ We find three different 2-cylinder directions $\theta_1$, $\theta_2$ and $\theta_3$. Let $\{\alpha_1,\alpha_2\}$, $\{\beta_1,\beta_2\}$ and $\{\gamma_1,\gamma_2\}$ be the waist curves of the cylinder decompositions of $\mathcal{O}$ corresponding to $\theta_1$, $\theta_2$ and $\theta_3$ respectively. \\

\textbf{(C) Computation of the Dehn multitwists and of \boldmath{$X_1$}, \boldmath{$X_2$} and \boldmath{$X_3$}:}\\
  For each 2-cylinder direction $\theta_n$ we denote by $T_n$ the corresponding affine multitwist (see Section~\ref{cyn-decompos-transvec}) and by $\widetilde{D}_n$ and $D_n$ the actions of $T_n$ on $H_1(\mathcal{O},\mathbb{Q})$ and on $H_1^{(0)}(\mathcal{O},\mathbb{Q})$ respectively.\\ For each $T_n$ $(n \in \{1,2,3\})$ we compute the multiplicities $k_i$ for each cylinder using Remark~\ref{rem:multiplicites}. For this we need the combinatorial lengths $f_i$ and the heights $h_i$ with $i \in \{1,2\}$ for both cylinder in direction $\theta$. However, it suffices to know the quotients $\frac{f_1}{f_2}$ and $\frac{h_1}{h_2}$. Therefore, rather than giving the heights $h_1$ and $h_2$, we list natural  numbers $c_1$, $c_2$ such that $\frac{h_1}{h_2} = \frac{c_1}{c_2}$ and obtain the multiplicities from the equation
  \begin{equation}
    \label{OurMultiplicities}
    \frac{k_1}{k_2} = \frac{c_1}{c_2}\cdot \frac{f_2}{f_1}.
  \end{equation}
  We then express $\widetilde{D}_n$  using Equation (\ref{eq:action-multitwist-homology}).
  %      For example, once the explicit expression for $\widetilde{D}_n$ has been obtained, $D_n$ can be deduced from the intersection data of $\{\alpha_1,\alpha_2\}$, $\{\beta_1,\beta_2\}$ and $\{\gamma_1,\gamma_2\}$ with $(\sigma_i)\cup(\zeta_i)$ and the formula expressing $(e_j)$ as linear combinations of this basis of $H_1(\mathcal{O},\mathbb{Q})$.
  We finally compute generators $X_1$, $X_2$ and $X_3$ in $H_1^{(0)}(M,\mathbb{Q})$ as generators of the one-dimensional image of $D_1 - \Id$, $D_2 - \Id$ and $D_3 - \Id$ in $H_1^{(0)}(M,\mathbb{Q})$. It follows from the proof of Lemma~\ref{lemma:transvection} that $X_i$ can be chosen as a rational multiple of
  \begin{equation}\label{ConstructX}
    Y = k_2\,w_2 - \frac{c_2}{c_1}\,k_1\,w_1
  \end{equation}
  where $w_1$ and $w_2$ are the core curves of the cylinders in the respective direction. We then have again by the proof of Lemma~\ref{lemma:transvection} for $x \in H_1^{(0)}(\mathcal{O},\mathbb{Q})$:
  \begin{equation}\label{DnForThe7Families}
    D_n(x) = x + \Omega(x,w_2)\cdot Y
  \end{equation}
  %There is no canonical procedure to do this. For example, if $C_{1,n}$ and $C_{2,n}$ are the cylinders of \emph{rational} widths $w_{1,n}$ and $w_{2,n}$ in the decomposition defined by $\theta_n$, then the following formula:
  %    $$
  %    \begin{array}{c}
  %     X_1= w_{2,1}\alpha_1-w_{1,1}\alpha_2\hspace{2cm} X_2=w_{2,2}\beta_1-w_{1,2}\beta_2\\
  %    \\
  %     X_3 = w_{2,3}\gamma_1-w_{1,3}\gamma_2
  %    \end{array}
  %    $$
  %    produce elements in $H_1^{(0)}(M,\mathbb{Q})$. However, the widths of $C_{1,n}$ and $C_{2,n}$ are not always rational. Thus we may use other methods to define the vectors $X_i$. For further information, please refer to the next two chapters.
 From (\ref{DnForThe7Families}) we have that $(D_n-\Id)(H_1^{(0)}(\mathcal{O},\mathbb{Z}))=\mathbb{Z}Y$, which is one of the hypothesis of Theorem~\ref{thm:sing-venkataramana}.\\

\textbf{(D) Proof of linear independence of \boldmath{$X_1$}, \boldmath{$X_2$} and \boldmath{$X_3$} and explicit computation of the affine Dehn multitwists \boldmath{$D_1$}, \boldmath{$D_2$} and \boldmath{$D_3$}:}\\
  For $n=1,2,3$, we compute an explicit expression for $D_n$, the action of $T_n$ on $H_1^{(0)}(\mathcal{O},\mathbb{Q})$ as described in the following. We first summarize the intersection numbers of the cylinder core curves in the three chosen 2-cylinder directions $\theta_1$, $\theta_2$ and $\theta_3$ and the core curves of the horizontal and the vertical directions. We furthermore write down the fundamental matrix $\tilde{G}$ of the intersection form $\Omega$ with respect to the basis $\tilde{B} = \{\sigma_1, \sigma_2, \sigma_3, \zeta_1, \zeta_2, \zeta_3\}$. We then obtain the coordinate vectors $x^{(1)}$, \ldots, $x^{(6)}$ of the core curves $c_1 = \alpha_1$, \ldots, $c_6 = \gamma_2$ for the directions $\theta_1$, $\theta_2$ and $\theta_3$ from the equation:
  \begin{equation}\label{coordinates}
      (x^{(i)})^t \cdot \tilde{G} = (\Omega(c_i,\sigma_1), \ldots, \Omega(c_i,\zeta_3))
  \end{equation}
  Hence we can express the core curves $\alpha_1$, \ldots, $\gamma_2$ in terms of the basis $\tilde{B} = \{\sigma_1, \ldots, \zeta_3\}$. Using this we express the elements $X_1$, $X_2$, $X_3$ that we have computed in Step (C) in terms of the same basis. We then use the base change matrix $BC_1$ computed in Step (A) in order to express $X_1$, $X_2$, $X_3$ wit respect to the basis $B = \{e_1,e_2,e_3,e_4\}$ chosen in Step~(A). From this we see that $X_1$, $X_2$ and $X_3$ are linearly independent.  Furthermore, we compute the intersection numbers of the elements $e_i$ in the basis $B$ and the core curves $\alpha_2$, $\beta_2$ and $\gamma_2$. We then obtain an explicit expression for $D_i$ from (\ref{DnForThe7Families}) using that the $Y$ in the formula equals $X_i$. We finally compute $\Omega(X_i,X_j)$ to determine whether two elements of $\{X_1,X_2,X_3\}$ satisfy $\Omega(X_i,X_j)\neq 0$, which is one of the hypothesis of Theorem~\ref{thm:sing-venkataramana}.\\

\textbf{(E) Action of the affine Dehn multitwists on the subspace \boldmath{W}}:\\
  Let $W$ be the $\mathbb{Q}$-vector subspace generated by $\{X_1,X_2,X_3\}$. We choose an annihilator element $e\in W$ as suitable multiple of:
  \begin{equation}
    \label{annihilatoreq}
    \Omega(X_2,X_3)\cdot X_1 + \Omega(X_3,X_1)\cdot X_2+\Omega(X_1,X_2)\cdot X_3.
    %e=-\frac{\Omega(X_3,X_2)}{\Omega(X_1,X_2)}X_1-\frac{\Omega(X_3,X_1)}{\Omega(X_2,X_1)}X_2+X_3.
  \end{equation}
  Then, by (\ref{NiceFormUnipotent}) the elements in the unipotent radical of $\SP(W)$ in the ordered basis $\{X_i,X_j,e\}$  are of the form:
  \begin{equation}
    \label{rad-uni-matrix0}
    \left(\begin{array}{ccc}1&0&0\\0&1&0\\x&y&1\end{array}\right),
  \end{equation} for $x,y \in \mathbb{Q}$.
  We now compute the restrictions of $D_n$ to $W$ in the basis $\{X_i,X_j,e\}$. \\
  
\textbf{(F) Detection of an element in the unipotent radical:}\\ We finally find an appropriate word in the letters $D_1$, $D_2$, $D_3$ and their inverses which produces a matrix of the form (\ref{rad-uni-matrix0}). 
%\end{itemize}

\begin{remark}
  Out of all the steps above, Step (F) is likely the most challenging one in the proof of arithmeticity once Zariski density has been established. In this step the approach is entirely heuristic in all cases.
\end{remark}

\subsection{Arithmeticity for the  family \boldmath{$\mathcal{F}_1$}}\label{ss.H4odd-family1}

We follow steps (A) to (F) from the procedure described in Section~\ref{TheSteps} in order to show that for all but finitely many origamis in $\F_1$ the Kontsevich-Zorich monodromy is arithmetic. We denote in the following by $\mathcal{O}_n$ the origami $\mathcal{O}_{3n+8,S_N}$ in the family $\mathcal{F}_{1}$ with $3n+8$ squares and monodromy group $S_N$. \\

%\begin{itemize}
\textbf{Step (A).}
  As in Section~\ref{TheSteps} $\sigma_1,\sigma_2,\sigma_3,\zeta_1,\zeta_2,\zeta_3$ are the core curves of the horizontal and vertical cylinders and we fix the basis
  \[
  \tilde{B}  = 
  \{\tilde{b}_1 = \sigma_1,~ \tilde{b}_2 = \sigma_2, ~\tilde{b}_3 = \sigma_3,~ \tilde{b}_4 = \zeta_1,~\tilde{b}_5 = \zeta_2,~ \tilde{b}_6 = \zeta_3\}
  \]
  of the homology $H_1(\mathcal{O}_n;\mathbb{Q})$. We read off from  Figure~\ref{familyMMY}, Figure~\ref{3n+8fam} and Table~\ref{tabla1}:
  \begin{IEEEeqnarray*}{rClCrClCrCl}
    l(\sigma_1) &=& V_1 = 1,       &\qquad&  l(\sigma_2)    &=& V_1 + V_3 + V_2 = 3n + 3,   &\qquad&    l(\sigma_3) &=&  V_2 = 2,\\
    l(\zeta_1)  &=& H_1 + H_3 = 2, &&       l(\zeta_2)      &=& H_3 = 1,                    &&          l(\zeta_3)  &=& H_3 + H_2 = 3
  \end{IEEEeqnarray*}
  \begin{figure}[!htbp]
    \centering
    \begin{tikzpicture}
      \draw (0,0) rectangle (1,1) node[pos=0.5,left=12pt,black]{\textcolor{purple}{$\sigma_2$}} node[pos=0.5,below=12pt,black]{\textcolor{purple}{$\zeta_1$}};
      \draw (0,1) rectangle (1,2)  node[pos=0.5,left=12pt,black]{\textcolor{purple}{$\sigma_1$}};
      \draw (1,0) rectangle (2,1) node[pos=0.5,below=12pt,black]{\textcolor{purple}{$\zeta_2$}} ;
      \draw (2,0) rectangle (3,1);
      \draw (3,0) rectangle (4,1);
      \draw (4,0) rectangle (5,1);
      \draw (5,0) rectangle (6,1);
      \draw (4,0) rectangle (5,-1) node[pos=0.5,left=12pt,black]{\textcolor{purple}{$\sigma_3$}} ;
      \draw (5,0) rectangle (6,-1);
      \draw (4,-1) rectangle (5,-2)node[pos=0.5,below=12pt,black]{\textcolor{purple}{$\zeta_3$}};
      \draw (5,-1) rectangle (6,-2);
      
      \draw [dashed] (2.2,0.7) -- (2.9,0.7);
      \draw [line,blue] (0,1.5) -- (1,1.5) ;
      \draw [line,red] (0,0.5) -- (6,0.5) ;
      \draw [line,yellow] (4,-0.5) -- (6,-0.5) ;
      \draw [line,green] (0.5,0) -- (0.5,2) ;
      \draw [line,orange] (1.5,0) -- (1.5,1) ;
      \draw [line,purple] (4.5,-2) -- (4.5,1) ;
      
    \end{tikzpicture}
    \caption{Prototype for the family $\F_1$.%\commref{\sout{You change convention on the colors of sigma and zeta in each figure, please pick one}}
    }
    \label{3n+8fam}
  \end{figure}
  By Equation (\ref{BasisEforF1to6}) we have:
  %The following formulae give a general recipe to produce a basis for $H_1^{(0)}(\mathcal{O},\mathbb{Q})$ when $\mathcal{O}$ is an origami in one of the families given by Table~\ref{tabla1}:
  %  \begin{equation}
  %      \label{eq:basis-e_i}
  %      \begin{array}{c}
  %          e_1=V_1\sigma_2-(V_3+V_1+V_2)\sigma_1 \hspace{2cm} e_2=V_1\sigma_3-V_2\sigma_1  \\
  %          e_3=H_3\zeta_1-(H_3+H_1)\zeta_2 \hspace{2cm} e_4=H_3\zeta_3-(H_3+H_2)\zeta_2
  %      \end{array}
  %  \end{equation}
  %  In the case of the family $\mathcal{F}_{3n+8}$ we have 
  %  $V_3=3n,V_1=1,V_2=2, H_1=H_3=1,H_2=2$ and thus:
  
\begin{equation}
    \begin{split}
        \begin{IEEEeqnarraybox}[][c]{rClCrCl} \label{BasisEForFamily1}
            e_1 &=& \sigma_2-(3n+3)\sigma_1,  & \qquad &   e_2 &=& \sigma_3-2\sigma_1\\
            e_3 &=& \zeta_1-2\zeta_2,         &&           e_4 &=& \zeta_3-3\zeta_2       
        \end{IEEEeqnarraybox}
    \end{split}
\end{equation}
  
  This gives us the transformation matrix (cf. \ref{BaseChange}):
  \[
  BC_1 = 
  \begin{pmatrix}-(3n+3)& -2 & 0 &0\\1 &0&0&0\\ 0&1&0&0\\0&0&1&0\\ 0&0&-2&-3\\ 0&0&0&1\end{pmatrix}
  \]
%    with the property that for  any vector x  we obtain the coordinate vector $x_{\tilde{B}}$ with respect to the basis $\tilde{B} = \{\sigma_1, \ldots, \zeta_3\}$ from the coordinate vector $x_B$ of $x$ with respect to the basis $B = \{e_1,e_2,e_3,e_4\}$ as
%    \begin{equation}\label{BaseChange}
%      x_{\tilde{B}} = BC_1\cdot x_B.
  %    \end{equation}
\textbf{Step (B).}
Let $\theta_1$, $\theta_2$ and $\theta_3$ be the directions defined by the vectors $(1,2)$, $(-1,2)$ and $(1,3)$ respectively. As depicted in Figures~\ref{fig:cylinder3n+8(1,2)-bis},~\ref{fig:cylinder3n+8(-1,2)-bis} and \ref{fig:cylinder3n+8(1,3)-bis}, each of them is a 2-cylinder direction. We call the  core curves of the white cylinders  $\alpha_1$, $\beta_1$ and $\gamma_1$ and the  core curves of the ruled cylinders $\alpha_2$, $\beta_2$ and $\gamma_2$, respectively. \\

\textbf{Step (C).}
    For each cylinder direction $\theta_i$ (with $i \in \{1,2,3\}$),  we read off   from  Figures~\ref{fig:cylinder3n+8(1,2)-bis},~\ref{fig:cylinder3n+8(-1,2)-bis} and \ref{fig:cylinder3n+8(1,3)-bis} the combinatoral lengths $f_1$ and $f_2$ of the two cylinders and the ratio $\frac{c_1}{c_2}$ of their heights (cf. Equation~\ref{OurMultiplicities} and explanation above). Recall that the degree of the origami is $N = 3n+8$. Observe from Figures~\ref{fig:cylinder3n+8(1,2)-bis} and Figure~\ref{fig:cylinder3n+8(-1,2)-bis} that for $\theta_1$ and  $\theta_2$ the ruled cylinder has combinatorial length 1 and the white one has combinatorial length $N-1 = 3n + 7$. In Figure~\ref{fig:cylinder3n+8(1,3)-bis} however  the ruled cylinder has combinatorial length 2 and the other cylinder has combinatorial length  $3n + 4$. In Figures~\ref{fig:cylinder3n+8(1,2)-bis} and Figure~\ref{fig:cylinder3n+8(-1,2)-bis} both cylinders have same height, whereas in Figure~\ref{fig:cylinder3n+8(1,3)-bis} the height of the ruled cylinder is double the heigth of the non-ruled cylinder. We then obtain the multiplicites $k_1$ and $k_2$ of the corresponding affine Dehn multitwist from Equation~\ref{OurMultiplicities}.  The corrsponding data are listed in Table~\ref{table:multitwistdata3n+8}.
    
\begin{table}[!htbp]
      \centering
      \begin{tabular}{|c|c |c| c|} 
        \hline
        &$\theta_1$ & $\theta_2$ & $\theta_3$ \\ [0.5ex] 
        \hline
        $c_1$ & 1 & 1 & 1  \\ 
        $c_2$ & 1 & 1 & 2 \\
        $f_1$ & 3n+7 &  3n+7 & 3n+4\\
        $f_2$ & 1 &  1&  2 \\
        $k_1$ &  1 & 1  & 1\\
        $k_2$ &  3n+7 & 3n+7 & 3n+4\\
        \hline
      \end{tabular}
      \caption{Data for multitwist in directions $\theta_1$, $\theta_2$, $\theta_3$ for the origami $\mathcal{O}_{3n+8,S_N}$ }
      \label{table:multitwistdata3n+8}
\end{table}
\begin{figure}[!htbp]
	\centering
	\noindent\begin{subfigure}[b]{0.3\textwidth}
		\centering
		%tikzpicture 1
		\begin{tikzpicture}[scale=0.4]
          %% squares
	    \draw (0,0) rectangle (4,1);
	     \draw (6,0) rectangle (8,1);
	    \draw (6,0) rectangle (8,-1);
	    \draw (6,-1) rectangle (8,-2);
	    \draw (0,1) rectangle (1,2);
	    \draw (1,0) -- (1,1);
	    \draw (2,0) -- (2,1);
	    \draw (3,0) -- (3,1);
        \draw (7,1) -- (7,-2);
        %% Punkte
        \draw [dashed] (4.8,0.5) -- (5.2,0.5);
        \draw [decorate,line width=0.5mm,decoration={brace,mirror}] (1,-0.5) --  (5.8,-0.5) node[pos=0.5,below=10pt,black]{$3n$};
        \draw [decorate, decoration={snake}] (4.5,0) -- (4.5,1);
        \draw [decorate, decoration={snake}] (5.5,0) -- (5.5,1);
        \draw (4,0) -- (4.5,0);
        \draw (4,1) -- (4.5,1);
        \draw (5.5,0) -- (6,0);
        \draw (5.5,1) -- (6,1); 
        %% hatching
        \draw[pattern color=blue, pattern = north east lines] (0,0) -- (0.5,0) -- (1,1) -- (1,2) -- (0,0) -- cycle;
	    \draw[pattern color=blue, pattern = north east lines] (0,1) -- (0.5,2) -- (0,2) -- (0,1) -- cycle;
        %% cylinder
	    \draw[thick, color=red] (0,0) -- (1,2);
	    \draw[thick, color=red] (0.5,0) -- (1,1);
	    \draw[thick, color=red] (0,1) -- (0.5,2);
	    \draw[thick, color=red] (1,0) -- (1.5,1);
	    \draw[thick, color=red] (1.5,0) -- (2,1);
	    \draw[thick, color=red] (2,0) -- (2.5,1);
	    \draw[thick, color=red] (2.5,0) -- (3,1);
	    \draw[thick, color=red] (3,0) -- (3.5,1);
	    \draw[thick, color=red] (3.5,0) -- (4,1);
	    \draw[thick, color=red] (6,0) -- (6.5,1);
	    \draw[thick, color=red] (6,-1) -- (7,1);
	    \draw[thick, color=red] (6,-2) -- (7.5,1);
	    \draw[thick, color=red] (6.5,-2) -- (8,1);
	    \draw[thick, color=red] (7,-2) -- (8,0);
	    \draw[thick, color=red] (7.5,-2) -- (8,-1);
        \end{tikzpicture}
		\caption{}\label{fig:cylinder3n+8(1,2)-bis}
	\end{subfigure}%
	\hfill
	\begin{subfigure}[b]{0.3\textwidth}
		\centering
		%tikzpicture 2
		\begin{tikzpicture}[scale=0.4]
         %% squares
	    \draw (0,0) rectangle (4,1);
	    \draw (6,0) rectangle (8,1);
	    \draw (6,0) rectangle (8,-1);
	    \draw (6,-1) rectangle (8,-2);
	    \draw (0,1) rectangle (1,2);
	    \draw (1,0) -- (1,1);
	    \draw (2,0) -- (2,1);
	    \draw (3,0) -- (3,1);
        \draw (7,1) -- (7,-2);
        %% Punkte
        \draw [dashed] (4.8,0.5) -- (5.2,0.5);
        \draw [decorate,line width=0.5mm,decoration={brace,mirror}] (1,-0.5) --  (5.8,-0.5) node[pos=0.5,below=10pt,black]{$3n$};
        \draw [decorate, decoration={snake}] (4.5,0) -- (4.5,1);
        \draw [decorate, decoration={snake}] (5.5,0) -- (5.5,1);
        \draw (4,0) -- (4.5,0);
        \draw (4,1) -- (4.5,1);
        \draw (5.5,0) -- (6,0);
        \draw (5.5,1) -- (6,1); 
        %% hatching
        %
        \draw[pattern color=blue, pattern = north east lines] (0.5,0) -- (0,1) -- (0,2) -- (1,0) -- (0.5,0) -- cycle;
	    \draw[pattern color=blue, pattern = north east lines] (1,1) -- (0.5,2) -- (1,2) -- (1,1) -- cycle;
        %% cylinder
	    \draw[thick, color=red] (0.5,0) -- (0,1);
	    \draw[thick, color=red] (1,0) -- (0,2);
	    \draw[thick, color=red] (1.5,0) -- (0.5,2);
	    \draw[thick, color=red] (2,0) -- (1.5,1);
	    \draw[thick, color=red] (2.5,0) -- (2,1);
	    \draw[thick, color=red] (3,0) -- (2.5,1);
	    \draw[thick, color=red] (3.5,0) -- (3,1);
	    \draw[thick, color=red] (4,0) -- (3.5,1);
	    \draw[thick, color=red] (6.5,-2) -- (6,-1);
	    \draw[thick, color=red] (7,-2) -- (6,0);
	    \draw[thick, color=red] (7.5,-2) -- (6,1);
	    \draw[thick, color=red] (8,-2) -- (6.5,1);
	    \draw[thick, color=red] (8,-1) -- (7,1);
	    \draw[thick, color=red] (8,0) -- (7.5,1);
        \end{tikzpicture}
		\caption{} \label{fig:cylinder3n+8(-1,2)-bis}
	\end{subfigure}%
	\hfill
	\begin{subfigure}[b]{0.3\textwidth}
		\centering
		%tikzpicture 3
		\begin{tikzpicture}[scale=0.4]
        %% squares
	    \draw (0,0) rectangle (4,1);
	    \draw (6,0) rectangle (8,1);
	    \draw (6,0) rectangle (8,-1);
	    \draw (6,-1) rectangle (8,-2);
	    \draw (0,1) rectangle (1,2);
	    \draw (1,0) -- (1,1);
	    \draw (2,0) -- (2,1);
	    \draw (3,0) -- (3,1);
        \draw (7,1) -- (7,-2);
        % Punkte
        \draw [dashed] (4.8,0.5) -- (5.2,0.5);
        \draw [decorate,line width=0.5mm,decoration={brace,mirror}] (1,-0.5) --  (5.8,-0.5) node[pos=0.5,below=10pt,black]{$3n$};
        \draw [decorate, decoration={snake}] (4.5,0) -- (4.5,1);
        \draw [decorate, decoration={snake}] (5.5,0) -- (5.5,1);
        \draw (4,0) -- (4.5,0);
        \draw (4,1) -- (4.5,1);
        \draw (5.5,0) -- (6,0);
        \draw (5.5,1) -- (6,1); 
        % hatching
        \draw[pattern color=blue, pattern = north east lines] (6,-2) -- (7,1) -- (6.333,1) -- (6,0) -- (6,-2) -- cycle;
	    \draw[pattern color=blue, pattern = north east lines] (6.333,-2) -- (7,-2) -- (8,1) -- (7.333,1)--(6.333,-2) -- cycle;
	    \draw[pattern color=blue, pattern = north east lines] (7.333,-2) -- (8,-2) -- (8,0) --(7.333,-2) -- cycle;
        % cylinder
	    \draw[thick, color=red] (0,1) -- (.333,2);
	    \draw[thick, color=red] (0,0) -- (0.666,2);
	    \draw[thick, color=red] (0.333,0) -- (1,2);
	    \draw[thick, color=red] (0.666,0) -- (1,1);
	    \draw[thick, color=red] (1,0) -- (1.333,1);
	    \draw[thick, color=red] (1.333,0) -- (1.666,1);
	    \draw[thick, color=red] (1.666,0) -- (2,1);
	    \draw[thick, color=red] (2,0) -- (2.333,1);
	    \draw[thick, color=red] (2.333,0) -- (2.666,1);
	    \draw[thick, color=red] (2.666,0) -- (3,1);
	    \draw[thick, color=red] (3,0) -- (3.333,1);
	    \draw[thick, color=red] (3.333,0) -- (3.666,1);
	    \draw[thick, color=red] (3.666,0) -- (4,1);
	    \draw[thick, color=red] (6,0) -- (6.333,1);
	    \draw[thick, color=red] (6,-2) -- (7,1);
	    \draw[thick, color=red] (6.333,-2) -- (7.333,1);
	    \draw[thick, color=red] (7,-2) -- (8,1);
	    \draw[thick, color=red] (7.333,-2) -- (8,0);
        \end{tikzpicture}
		\caption{} \label{fig:cylinder3n+8(1,3)-bis}
	\end{subfigure}
	\caption{Cylinder decomposition in direction $(1,2)$, $(-1,2)$ and $(1,3)$ of the origami $\mathcal{O}_{3n+8,S_N}$. Here $\alpha_1$, $\beta_1$ and $\gamma_1$ are the core curves of the white cylinders in the respective decomposition.}
\end{figure}
   Using (\ref{eq:action-multitwist-homology}) and Table~\ref{table:multitwistdata3n+8} we obtain for the action  $\widetilde{D}_i$ on $H_1(\mathcal{O}_n;\mathbb{Q})$ of the affine Dehn multitwist $T_i$ in direction $\theta_i$: 
    \begin{align}
        \begin{split}\label{eq:widetilde-D_i-3n+8}
        \widetilde{D}_1 =& \Id+\Omega(\cdot,\alpha_1)\,\alpha_1+(3n+7)\Omega(\cdot,\alpha_2)\,\alpha_2\\
        \widetilde{D}_2 =& \Id+\Omega(\cdot,\beta_1)\,\beta_1+(3n+7)\Omega(\cdot,\beta_2)\,\beta_2\\
        \widetilde{D}_3 =& \Id+ \Omega(\cdot,\gamma_1)\,\gamma_1+(3n+4)\Omega(\cdot,\gamma_2)\,\gamma_2
        \end{split}
    \end{align}
   Using (\ref{ConstructX}) and Table~\ref{table:multitwistdata3n+8} we define the vectors
    \begin{equation}\label{XVectorsForFamily1}
      X_1 = (3n+7)\alpha_2-\alpha_1, \quad X_2 = (3n+7)\beta_2-\beta_1, \quad X_3 = (3n+4)\gamma_2 - 2\gamma_1
    \end{equation}

 \textbf{Step (D).}
    The following table summarizes the intersection data of the core curves of the cylinders in the three 2-cylinder directions and the basis $\tilde{B}$, see Figures~\ref{3n+8fam},~\ref{fig:cylinder3n+8(1,2)-bis},~\ref{fig:cylinder3n+8(-1,2)-bis} and \ref{fig:cylinder3n+8(1,3)-bis}:
    \begin{table}[!htbp]
      \centering
      \begin{tabular}{|c|c|c|c|c|c|c|} 
        \hline
        $\Omega$ & $\alpha_1$ & $\alpha_2$ & $\beta_1$ & $\beta_2$ & $\gamma_1$ & $\gamma_2$ \\ 
        \hline
        $\sigma_1$ & 1 & 1 & 1 & 1 & 3 & 0 \\ 
        $\sigma_2$ & $6n+5$ & 1 & $6n+5$ & 1 & $9n+5$ & 2 \\
        $\sigma_3$ & 4 & 0 & 4 & 0 & 2  & 2 \\
        $\zeta_1$ & -1 & -1 & 1 & 1 & -2 & 0 \\
        $\zeta_2$ & -1 & 0 & 1 & 0 & -1 &  0 \\ 
        $\zeta_3$ & -3 & 0 & 3 & 0 & -1 & -1 \\ [1ex] 
        \hline
      \end{tabular}
      \caption{Intersection data of cylinder core curves and the basis of homology fixed in step (A).}
      \label{intersections3n+8}
    \end{table}
    We furthermore read off from Figure~\ref{3n+8fam}  the fundamental matrix $\tilde{G} = (\Omega(\tilde{b}_i, \tilde{b}_j))$ of the intersection form $\Omega$ with respect to the basis
    $\tilde{B}$ fixed in step (A) and determine its inverse:
    \begin{align}\label{fundamentalMatrix}
    \tilde{G}= 
    \begin{pmatrix}
      0 & 0 & 0 & 1 & 0 & 0\\
      0 & 0 & 0 & 1 & 1 & 1\\
      0 & 0 & 0 & 0 & 0 & 1\\
      -1 & -1 &  0 & 0 & 0 & 0\\
      0  & -1 &  0 & 0 & 0 & 0 \\
      0  & -1 & -1 & 0 & 0 & 0\\
    \end{pmatrix},~~
    \tilde{G}^{-1}=
    \begin{pmatrix} 
      0 & 0 & 0 &-1 & 1 & 0\\
      0 & 0 & 0 & 0 &-1 & 0\\
      0 & 0 & 0 & 0 & 1 &-1\\
      1 & 0 & 0 & 0 & 0 & 0\\
     -1 & 1 &-1 & 0 & 0 & 0 \\
      0 & 0 & 1 & 0 & 0 & 0 
    \end{pmatrix}
    \end{align}
     Observe that the matrix $\tilde{G}$ is the same for all families $\mathcal{F}_i$, $i=1,\ldots,6$. We now compute the coordinate vectors $x^{(1)}$ of $\alpha_1$ from (\ref{coordinates}):
    \begin{align*}
    (x^{(1)})^t = & (-1 \,,\, -6n-5\,,\,-4\,,\,  1 \,,\, 1\,,\,3 \,)\cdot \tilde{G}^{-1}\\
                = & (\, 0 \,,\,  1 \,,\, 2 \,,\, 1 \,,\, 6n \,,\, 4\,)
    \end{align*}
    Hence $\alpha_1 = \sigma_2 + 2\,\sigma_3 + \zeta_1 + 6n\,\zeta_2 + 4\,\zeta_3$. Similarly we compute $\alpha_2$, \ldots, $\gamma_2$ and obtain altogether:
    \begin{equation}\label{CCforF1}
        \begin{array}{lcl}
        \alpha_1 = \sigma_2 + 2\sigma_3 + \zeta_1 + 6n\zeta_2 + 4\zeta_3,   && \alpha_2 = \sigma_1 + \zeta_1\\
        \beta_1  =  -\sigma_2 - 2\sigma_3 + \zeta_1 +6n\zeta_2 +4\zeta_3,   && \beta_2  = -\sigma_1 + \zeta_1 \\
        \gamma_1 =\sigma_1 + \sigma_2 + 3\zeta_1  + 9n\zeta_2  + 2\zeta_3,  && \gamma_2 =      \sigma_3  + 2\zeta_3 
        \end{array}
    \end{equation}
    Using (\ref{CCforF1}) we express $X_1$, $X_2$ and $X_3$ defined in (\ref{XVectorsForFamily1}) as linear combination of the elements $\sigma_1$, \ldots, $\zeta_3$:
    \begin{align}
      \begin{split}\label{XiViaBtilde}
        X_1 &  = (3n+7)\sigma_1  -  \sigma_2 -      2\sigma_3 + (3n+6)\zeta_1 -  6n\zeta_2 - 4\zeta_3\\
        X_2 &  = -(3n+7)\sigma_1 +  \sigma_2 +      2\sigma_3 + (3n+6)\zeta_1 -  6n\zeta_2 - 4\zeta_3\\
        X_3 &  = -2\sigma_1 - 2\sigma_2 + (3n+4)\sigma_3 -      6\zeta_1 - 18n\zeta_2 + (6n+4)\zeta_3
      \end{split}
    \end{align}
    Now, we compute the coordinate vectors of $X_1$, $X_2$ and $X_3$ with respect to the basis $B = \{e_1,\ldots, e_4\}$ by solving the linear equation system (\ref{BaseChange}) and hence obtain: 
    % \[
    % \begin{array}{lcl}
    %   \Omega(\alpha_1,e_1) &=& \Omega(\alpha_1, \sigma_2) - (3n+3)\cdot \Omega(\alpha_1, \sigma_1) = -(6n+5) - (3n+3)\cdot(-1) = -3n-2\\
    %   \Omega(\alpha_1,e_2) &=& \Omega(\alpha_1, \sigma_3) - 2\cdot \Omega(\alpha_1, \sigma_1) = 
    %   \Omega(\)
    % \end{array}\]
  \begin{align}
       \begin{split}\label{eq:X_i-e_i-3n+8}
        X_1 & =-e_1 - 2e_2    + (3n + 6)e_3 - 4e_4\\
        X_2 &= e_1 + 2e_2    + (3n + 6)e_3 - 4e_4\\
        X_3 &=-2e_1+(3n+4)e_2 -        6e_3 + (6n + 4)e_4
      \end{split}
  \end{align}
From (\ref{eq:X_i-e_i-3n+8}) we obtain that $X_1$, $X_2$, $X_3$ are linearly independent.\\
From Table~\ref{intersections3n+8} we now compute the intersection data between the basis vectors $e_i$ and the core curves using Equation (\ref{BasisEForFamily1}):
\begin{align}\label{IntersectionNumbers2}
  \begin{split}
    \Omega(e_1,\alpha_2) &= \Omega(\sigma_2,\alpha_2) - (3n+3)\cdot\Omega(\sigma_1,\alpha_2) = 1 -(3n+3)\cdot 1 = -3n-2\\
    \Omega(e_2,\alpha_2) &= \Omega(\sigma_3,\alpha_2) - 2\cdot\Omega(\sigma_1,\alpha_2) = 0 - 2 = -2\\
    \Omega(e_3,\alpha_2) &= \Omega(\eta_1,\alpha_2) - 2\cdot\Omega(\zeta_2,\alpha_2) = -1 -2\cdot 0 = -1 \\
    \Omega(e_4,\alpha_2) &= \Omega(\zeta_3,\alpha_2) - 2\cdot\Omega(\zeta_2,\alpha_2) = 0\\[2mm]
    \Omega(e_1,\beta_2)  &= \Omega(\sigma_2,\beta_2) - (3n+3)\Omega(\sigma_1,\beta_2) = 1 - (3n+3)\cdot 1 = -3n-2\\
    \Omega(e_2,\beta_2)  &= \Omega(\sigma_3,\beta_2) - 2\Omega(\sigma_1,\beta_2) = 0 - 2 = -2\\
    \Omega(e_3,\beta_2)  &= \Omega(\zeta_1,\beta_2) - 2\Omega(\zeta_2,\beta_2) = 1 - 0 = 1\\
    \Omega(e_4,\beta_2)  &= \Omega(\zeta_3,\beta_2) - 2\Omega(\zeta_2,\beta_2) = 0 - 0 = 0\\[2mm]
    \Omega(e_1,\gamma_2) &= \Omega(\sigma_2,\gamma_2) - (3n+3)\Omega(\sigma_1,\gamma_2) = 2 - (3n+3)\cdot 0 = 2\\
    \Omega(e_2,\gamma_2) &= \Omega(\sigma_3,\gamma_2) - 2\Omega(\sigma_1,\gamma_2) = 2 - 2\cdot 0 = 2\\
    \Omega(e_3,\gamma_2) &= \Omega(\zeta_1,\gamma_2) - 2\Omega(\zeta_2,\gamma_2) = 0 - 0 = 0\\
    \Omega(e_4,\gamma_2) &= \Omega(\zeta_3,\gamma_2) - 3\Omega(\zeta_2,\gamma_2) = -1 - 3\cdot 0 = -1\\      
  \end{split}
\end{align}
The preceding equations determine $D_i$ with respect to the basis $B$ by~(\ref{DnForThe7Families}) - with $Y = X_i$.\\

    We now compute the intersection numbers for the $X_i$'s using (\ref{XiViaBtilde}) and the fundamental matrix $\tilde{G}$ from (\ref{fundamentalMatrix}) :
      \begin{align}
        \begin{split}\label{intersectionXi}
          \Omega(X_1,X_2) =& 2(3n+6)(3n+8)\\
          \Omega(X_1,X_3) =& -2(3n+8)\\
          \Omega(X_2,X_3) =& 10(3n+8)
        \end{split}
      \end{align}

 \textbf{Step (E).}
  Recall that $W$ is the subspace of $H_1^{(0)}(\mathcal{O}_n,\Q)$ generated by $X_1$, $X_2$ and $X_3$.\\
    We first compute the intersection numbers for $X_1$ and $X_3$ with some of the core curves of our three chosen 2-cylinder directions using  (\ref{eq:X_i-e_i-3n+8}) and (\ref{IntersectionNumbers2}):
    \begin{align}
      \begin{split}\label{IntersectionsBCoreCurves}
        \Omega(X_3,\alpha_2) &= -2\Omega(e_1,\alpha_2) + (3n+4)\Omega(e_2,\alpha_2) - 6\Omega(e_3,\alpha_2) + (6n+4)\Omega(e_4,\alpha_2)\\
                             &= 2\cdot(3n+2) + (3n+4)\cdot(-2) + 6 + 0 = 2\\
        \Omega(X_1,\gamma_2) &= -\Omega(e_1,\gamma_2) - 2\Omega(e_2,\gamma_2) + (3n+6)\Omega(e_3,\gamma_2) - 4\Omega(e_4,\gamma_2)\\
                             &= -2-4+0+4 = -2\\
        \Omega(X_1,\beta_2)  &= -\Omega(e_1,\beta_2) - 2\Omega(e_2,\beta_2) + (3n+6)\Omega(e_3,\beta_2) - 4\Omega(e_4,\beta_2)\\
                             &= (3n+2) + 4 + (3n+6)  + 0 = 6(n+2)\\
        \Omega(X_3,\beta_2)  &= -2\Omega(e_1,\beta_2) + (3n+4)\Omega(e_2,\beta_2) - 6\Omega(e_3,\beta_2) + (6n+4)\Omega(e_4,\beta_2)\\
                             &= 2\cdot(3n+2) - 2\cdot(3n+4) - 6 + 0 = -10
      \end{split}
    \end{align}
     By (\ref{intersectionXi}) and (\ref{annihilatoreq}) we choose the  annihilator element as
     \begin{equation}\label{TheAnnihilator}
       e=5X_1+X_2+(3n+6)X_3. %% changed the signs!!
     \end{equation}
    
     Finally, we compute the transformation matrix of $D_n|_{W}$ ($n \in \{1,2,3\}$) with respect to the basis $\{X_1,X_3,e\}$. Since $e$ is an annihilator, $D_n(e)=e$ for all $n=3$.
     Furthermore, we have by definition of $X_i$ that $D_i(X_i) = X_i$. 
     Direct computations using (\ref{IntersectionsBCoreCurves}), (\ref{DnForThe7Families}) and (\ref{TheAnnihilator}) show that:
      \begin{align*} %% Attention: changed the sign!!!
        D_1(X_3) =& X_3 + \Omega(X_3,\alpha_2)\cdot X_1 = X_3 + 2X_1\\ 
        D_3(X_1) =& X_1 + \Omega(X_1,\gamma_2)\cdot X_3 = X_1 - 2X_3\\
        D_2(X_1) =& X_1 + \Omega(X_1,\beta_2) \cdot X_2 = X_1 + 6(n+2)X_2\\
              =& X_1 + 6(n+2)\cdot(e-5X_1 - (3n+6)X_3)\\
              =&  -(30n+59)X_1-18(n+2)^{2}X_3+6(n+2)e\\
        % D_2(X_1) =& (30(n+2)+1)X_1-18(n+2)^{2}X_3+6(n+2)e\\
        % D_2(X_1) =& (-30(n+2)+1)X_1+18(n+2)^{2}X_3-6(n+2)e\\
   \end{align*}
    As well as:
    \begin{align*}
     D_2(X_3) =& X_3 + \Omega(X_3,\beta_2)\cdot X_2 = X_3 - 10X_2\\
              =& X_3 - 10e + 50X_1 + 10(3n+6)X_3 \\
              =&  50X_1 + (30n+61)X_3 - 10e 
     %...50X_1 + (-30n -59)X_3 + 10e $$
     %D_2(X_3) =& -50X_1+(30(n+2)+1))X_3-10e $$
    \end{align*}
    %In the last two equations we used that $X_2 = e -5X_1 - (3n+6)X_3$ by (\ref{TheAnnihilator}).
    
    Hence the transformation matrices $A_n$ representing $D_n|W$ in the basis $X_1,~X_3,~e$ are: 
    \[
    A_1=
    \begin{pmatrix}
    1 & 2 & 0 \\
    0 & 1 & 0 \\
    0 & 0 & 1\\
    \end{pmatrix},  \quad
    A_3=
    \begin{pmatrix}
    1 & 0 & 0 \\
    -2 & 1 & 0 \\
    0  & 0 & 1\\
    \end{pmatrix}, \quad
    A_2=
    \begin{pmatrix}
    -(30n+59) & 50 & 0 \\
    -18(n+2)^{2} & 30n + 61 & 0 \\
    6(n+2) & -10 & 1\\
    \end{pmatrix}.
    \]
%      $$D_1=\begin{pmatrix}
%    1 & -2 & 0 \\
%    0 & 1 & 0 \\
%    0 & 0 & 1\\
%    \end{pmatrix},  \hspace{1cm} D_3=\begin{pmatrix}
%    1 & 0 & 0 \\
%    4 & 1 & 0 \\
%0 & 0 & 1\\
%    \end{pmatrix},$$ 
%    
%    $$D_2=\begin{pmatrix}
%    -30(n+2)+1 & -50 & 0 \\
%    18(n+2)^{2} & 30(n+2)+1 & 0 \\
%    6(n+2) & 10 & 1\\
%    \end{pmatrix}. $$  

\emph{Step (F)}.
    We set $n=10k-2$  with $k \in \mathbb{N}$ and show that
    \[A_1^{-25}A_3^{3k}A_2 A_3^{-3k}\]
    %$$D_3^{-3k}D_2 D_3^{3k}D_1^{-25}$$ 
    is a nontrivial element in the unipotent radical of $\Sp(W)$.\\
    Firstly, observe that $A_3$ is the elementary matrix $L_{2,1}(-2)$, i.e. the matrix such that the entries in the diagonal are all 1, the entry at position $(2,1)$ is -2 and all other entries are 0.
    Hence $A_3^{3k} = L_{2,1}(-6k)$ and  $A_3^{-3k} = L_{2,1}(6k)$. One has for the elementary matrix $L_{2,1}(T)$ and an arbitrary matrix
    $A = \left(\begin{smallmatrix}x&y&z\\ u&v&w\\r&s&t\end{smallmatrix}\right)$ that
    \begin{equation}\label{TheConjugate}
      L_{2,1}(T)\cdot A \cdot L_{2,1}(-T) =
      \begin{pmatrix} 
      x - Ty & y & z\\ Tx+u - T(Ty+v) & Ty + v & Tz+w\\ \star & \star & t
      \end{pmatrix}
    \end{equation}
    Here $\star$ symbolyzes arbitrary numbers.
    Choosing $A = A_2$ this gives:% we in particular have that $z = 0 = w$ and $y = 50$.\\
    \[\begin{array}{rclrclrclrcl}
      x &=& -(30n+59),&
      y &=& 50,&
      z &=& 0, &
      u &=& -18(n+2)^2,\\
      v &=& 30n+61,& %-(30n + 59) = - (30(20k - 2) + 59))\\  &=& - (600k - 60 + 59) = - 50T + 1\\ 
      w &=& 0,&
      t &=& 1,&
      Tz + w &=& 0.
      \end{array}\] 
%    \begin{alignat*}{4}
%      x &= -(30n+59),&
%      y &= 50,\\
%      z &= 0, &
%      u &= -18(n+2)^2,\\
%      v &= 30n+61,& %-(30n + 59) = - (30(20k - 2) + 59))\\  &=& - (600k - 60 + 59) = - 50T + 1\\ 
%      w &= 0,\\
%      t &= 1,&
%      Tz + w &= 0.
%    \end{alignat*}
    We choose $T = -6k$ and use that $n=10k-2$ and thus $3(n+2) = -5T$. It follows that:
    \[
    (x - Ty) - 1 =  - (30n+59) - 50T - 1 = (-10)\cdot(3(n+2)+5T)= 0\\
    \]
   In particular the relation $x - Ty = 1$ holds. Furthermore we conclude
   \begin{align*}
      Ty + v =  50T + 30(n+2) + 1 = 50T - 50T + 1 = 1.
    \end{align*}
   On the other hand we obtain the relation
   \[\begin{array}{l}
   Tx + u - T(Ty+v) =  Tx + u - T = T(x-1) + u \\
   \begin{array}{lcl}
     &=&  T(-30n-60) - 18(n+2)^2
                       \; = \;   -30T(n+2) -6\cdot 3(n+2)^2 \\
                       &=&  -30T(n+2) + 6\cdot 5T (n+2) \; = \;  0.
   \end{array}
   \end{array}\]
   % Change of style:
   % \begin{align*}
   % Tx + u - T(Ty+v) = & Tx + u - T = T(x-1) + u \\
   %                  = & T(-30n-60) - 18(n+2)^2\\
   %                  = & -30T(n+2) -6\cdot 3(n+2)^2 \\
   %                    = & -30T(n+2) + 6\cdot 5T (n+2) \\
   %                    = & 0.
   % \end{align*}
    Hence it follows from (\ref{TheConjugate}) that
    \[A_3^{3k}A_2 A_3^{-3k} = \begin{pmatrix} 1 & y & 0\\ 0 & 1 & 0\\ \star &\star&1 \end{pmatrix}.\]
    Finally, we obtain:
    \[A_1^{-25}A_3^{-3k}A_2 A_3^{3k} = L_{1,2}(-50) = \begin{pmatrix} 1 &y-50& 0\\ 0&1&0\\\star&\star &1\end{pmatrix}\]
      Using that $y = 50$ and (\ref{rad-uni-matrix0}) we obtain that the matrix is in the  unipotent radical of $\Sp(W)$.\\
      %\end{itemize}
      
We conclude that for $n = 10k-2$ for the origami $\mathcal{O}_n$ in the family $\mathcal{F}_{1}$ we have that the Kontsevich-Zorich monodromy is arithemtic if $k$ is large enough.
Hence the proof of Theorem~\ref{t.H4odd-families1to6} is complete  for the family $\mathcal{F}_1$ .

\subsection{Arithmeticity for the family \boldmath{$\mathcal{F}_2$}}\label{ss.H4odd-family2}

To avoid unnecessary repetitions, we condense the rest of our calculations. For this we will use tables and use the same notations as in the preceding section. The computations follow the procedure described in Section~\ref{TheSteps}.

\textbf{Step (A)} You can directly read off the basis $B$ (see (\ref{BasisEforF1to6})) and the matrix $BC_1$ (see (\ref{BaseChange})) from Figure~\ref{familyMMY} and Table~\ref{tabla1}.

\textbf{Step (B).} The vectors defining the chosen 2-cylinder directions $\{\theta_1,\theta_2,\theta_3\}$ are in the first row of Table~\ref{multitwistdataintersections3n+10} (A). We have set the parameter $n$ for $\mathcal{O}_n$ in the family $\mathcal{F}_2$ as $n=2k-1$, in order to obtain a 2-cylinder decomposition for directions $\theta_1$ and $\theta_2$. These decompositions are illustrated in Figures~\ref{fig:cylinder3n+10(2,1)},~\ref{fig:cylinder3n+10(-2,1)} and \ref{fig:cylinder3n+10(1,1)}. In Table~\ref{multitwistdataintersections3n+10} (A) we indicate below for each 2-cylinder direction the names of the corresponding core curves.

\textbf{Step (C).} We computed Table~\ref{multitwistdataintersections3n+10} (B) using Figures~\ref{fig:cylinder3n+10(2,1)},~\ref{fig:cylinder3n+10(-2,1)} and \ref{fig:cylinder3n+10(1,1)} as well as (\ref{OurMultiplicities}). We now obtain the vectors $\{X_1,X_2,X_3\}$ from (\ref{ConstructX}) as linear combination of $\alpha_1$, \ldots, $\gamma_2$ and the Dehn multitwists $D_n$ from (\ref{DnForThe7Families}).

%and the vectors $\{X_1,X_2,X_3\}$ we use (\ref{OurMultiplicities}), (\ref{ConstructX}) and the data in Table~\ref{multitwistdataintersections3n+10} (B). 

\textbf{Step (D).} Table~\ref{multitwistdataintersections3n+10} (A) contains the intersection data between the core curves $\{\alpha_1,\cdots,\gamma_2\}$ and the elements of the basis $\widetilde{B}$. Recall that $\tilde{G}$ is given in (\ref{fundamentalMatrix}). As carried out in Section~\ref{ss.H4odd-family1} we use these data to compute the coordinates of $X_1$,$X_2$ and $X_3$ first with respect to the basis $\tilde{B}$ and then with respect to the basis $B$. The first three rows of Table ~\ref{coefbaseandAction3n+10} show the coefficients of $X_i$ in terms of $B$. We further compute the action of the transvections $D_1$, $D_2$ and $D_3$ on $B$. We have $D_i(e_j)=e_j + r_{i,j}X_i$. The last three rows of Table ~\ref{coefbaseandAction3n+10} show the $r_{i,j}$'s.
 %We also use Table~\ref{multitwistdataintersections3n+10} (A) to compute $\Omega(X_1,X_3)$, $\Omega(X_2,X_3)$ and $\Omega(X_1,X_2)$. In particular, we verify
Finally, with similar computations as in \ref{fundamentalMatrix} we obtain in the same way that $\Omega(X_1,X_3)\neq 0$.

\textbf{Step (E).} Using Table~\ref{multitwistdataintersections3n+10} (A) and (\ref{annihilatoreq}) we obtain the annihilator $e$. With the analog computations as in section~\ref{ss.H4odd-family1} we compute the action of $D_n$, $n=1,2,3$ on $\{X_1,X_3\}$. The results are summarized in the first two rows of Table~\ref{MatricesAndWord3n+10}. These allow us to compute the matrices $A_1$, $A_2$ and $A_3$ corresponding to the action of $D_1$, $D_2$ and $D_3$ on the basis $\{X_1,X_3,e\}$. We now set the parameter $k$ introduced in Step $(B)$ above as $k=2l-1$.
     
\textbf{Step (F).} The word giving an element in the unipotent radical of $\Sp(W)$ is given by $A_{3}^{l}A_{2} A_{3}^{-l}A_{1}^{-9}$. See Table~\ref{MatricesAndWord3n+10}.\\
%\end{itemize}

This proves item (ii) of Theorem~\ref{t.H4odd-families1to6}. For the remaining families we just present the corresponding tables, since we can proceed in the same way to obtain the results.

\begin{table}[!htbp]
\begin{subtable}{.6\textwidth}
\centering
\begin{tabular}{|c|c |c| c| c| c| c|} 
 \hline
 \multirow{2}{*}{$\Omega$} & \multicolumn{2}{c|} {$\theta_1\parallel(2,1)$} & \multicolumn{2}{c|}{$\theta_2\parallel(-2,1)$} &  \multicolumn{2}{c|}{$\theta_3\parallel(1,1)$} \\
 \cline{2-7}
 &$\alpha_1$ & $\alpha_2$ & $\beta_1$ & $\beta_2$ & $\gamma_1$ & $\gamma_2$ \\ [0.5ex] 
 \hline
 $\sigma_1$ & 0 & 1 & 0 & 1 & 1 & 0 \\ 
 $\sigma_2$ & $3k$ & $3k$ & $3k$ & $3k$ & $6k-1$ & 1   \\
 $\sigma_3$ & 1 & 1 & 1& 1 & 1  & 1 \\
 $\zeta_1$ & -1 & -3 & 1 & 3 & -2 & 0 \\
 $\zeta_2$ & -1 & -1 & 1 & 1 & -1 &  0 \\ 
 $\zeta_3$ & -4 & -4 & 4 & 4 & -2 & -2 \\ [1ex] 
 \hline
\end{tabular}
\caption{}\label{A-multitwistdataintersections3n+10}
\end{subtable}
\begin{subtable}{.38\textwidth}
\centering
\begin{tabular}{|c|c |c| c|} 
 \hline
 &$\theta_1$ & $\theta_2$ & $\theta_3$ \\ [0.5ex] 
 \hline
 $c_1$ & 1 & 1 & 1  \\ 
 $c_2$ & 1 & 1 & 1    \\
 $f_1$ & $3k+4$ & $3k+4$  & $6k+3$  \\
 $f_2$ & $3k+5$ & $3k+5$ & 4 \\
 $k_1$ & $ 3k+5 $ & $3k+5$ & $ 4 $\\
 $k_2$ & $ 3k+4 $ & $ 3k+4 $ & $ 6k+3 $\\
 \hline
\end{tabular}
    \label{B-multitwistdataintersections3n+10}
    \caption{}
    \end{subtable}
\caption{For the family $\mathcal{F}_2$: In (A): intersection numbers for core curves in the directions $\{\theta_1,\theta_2,\theta_3\}$ and elements of $\tilde{B}$.
  In (B): multiplicities and combinatorial lengths in directions $\{\theta_1,\theta_2,\theta_3\}$.}
\label{multitwistdataintersections3n+10}
\end{table}

% Change of style:
%\begin{table}[!htbp]
%\begin{subtable}{\linewidth}
%\centering
%\begin{tabular}{|c|c |c| c| c| c| c|} 
% \hline
% \multirow{2}{*}{$\Omega$} & \multicolumn{2}{c|} {$\theta_1\parallel(2,1)$} & \multicolumn{2}{c|}{$\theta_2\parallel(-2,1)$} &  \multicolumn{2}{c|}{$\theta_3\parallel(1,1%)$} \\
% \cline{2-7}
% &$\alpha_1$ & $\alpha_2$ & $\beta_1$ & $\beta_2$ & $\gamma_1$ & $\gamma_2$ \\ [0.5ex] 
% \hline
% $\sigma_1$ & 0 & 1 & 0 & 1 & 1 & 0 \\ 
% $\sigma_2$ & $3k$ & $3k$ & $3k$ & $3k$ & $6k-1$ & 1   \\
% $\sigma_3$ & 1 & 1 & 1& 1 & 1  & 1 \\
% $\zeta_1$ & -1 & -3 & 1 & 3 & -2 & 0 \\
% $\zeta_2$ & -1 & -1 & 1 & 1 & -1 &  0 \\ 
% $\zeta_3$ & -4 & -4 & 4 & 4 & -2 & -2 \\ [1ex] 
% \hline
%\end{tabular}
%\caption{}\label{A-multitwistdataintersections3n+10}
%\end{subtable}
%\begin{subtable}{\linewidth}
%\centering
%\begin{tabular}{|c|c |c| c|} 
% \hline
% &$\theta_1$ & $\theta_2$ & $\theta_3$ \\ [0.5ex] 
% \hline
% $c_1$ & 1 & 1 & 1  \\ 
% $c_2$ & 1 & 1 & 1    \\
% $f_1$ & $3k+4$ & $3k+4$  & $6k+3$  \\
% $f_2$ & $3k+5$ & $3k+5$ & 4 \\
% $k_1$ & $ 3k+5 $ & $3k+5$ & $ 4 $\\
% $k_2$ & $ 3k+4 $ & $ 3k+4 $ & $ 6k+3 $\\
% \hline
%\end{tabular}
%    \label{B-multitwistdataintersections3n+10}
%    \caption{}
%    \end{subtable}
%    \caption{In (A) you can find for 2-cylinder directions $\{\theta_1,\theta_2,\theta_3\}$ and the intersections between core curves and elements of the basis $\tilde{B}%$. In (B) we have the multiplicities and combinatorial lengths in the three directions $\{\theta_1,\theta_2,\theta_3\}$ for the family $\mathcal{F}_2$.}
%\label{multitwistdataintersections3n+10}
%\end{table}
%%
%%
\begin{table}[!htbp]
\centering
\begin{tabular}{|c|c |c| c| c|} 
 \hline
 &$e_1$ & $e_2$ & $e_3$ & $e_4$ \\ [0.5ex] 
 \hline
 $X_1$ & -1 & -3 & $3k+3$ & -1 \\ 
 $X_2$ & 1 & 3 & $3k+3$  & -1  \\
 $X_3$ & -4 &  $12k+2$ & -4 &$6k-1$\\
 $D_{1}$ & $-3k$ & -1 & -1 & 0\\ 
 $D_{2}$ & $-3k$ & -1 & 1 & 0 \\
 $D_{3}$ & 1 & 1 & 0 & -2 \\ [1ex] 
 \hline
\end{tabular}
\caption{Coefficients for the base in $H_1^{(0)}(M,\mathbb{Q})$ and coefficients for the action of the associated transvection on the elements of the basis for the family $\mathcal{F}_2$.}
\label{coefbaseandAction3n+10}
\end{table}
\begin{table}[!htbp]
\centering
\begin{tabular}{|c|c |c|} 
 \hline
 &$X_1$ & $X_3$  \\ [0.5ex] 
 \hline
 $D_{1}$ & $X_1$ & $X_3+2X_1$  \\ 
 $D_{2}$ & $(-18(k+1)+1)X_1-18(k+1)^{2}X_3+6(k+1)e$ & $18X_1+(18(k+1)+1))X_3-6e$    \\
 $D_{3}$ & $X_1-2X_3$ &  $X_3$ \\
 \hline
 Annihilator & \multicolumn{2}{c|}{$e=3X_1+X_2+3(k+1)X_3$}\\
 \hline
 Word & \multicolumn{2}{c|}{$k=2l-1, A_{3}^{l}A_{2} A_{3}^{-l}A_{1}^{-9}$ } \\ [1ex] 
 \hline
\end{tabular}
\caption{Action of each transvection on the ordered basis $X_1,X_3,e$ and nontrivial element $\Sp(W)$ for the family $\mathcal{F}_2$.}
\label{MatricesAndWord3n+10}
\end{table}

\begin{figure}[!h]
	\centering
	\noindent\begin{subfigure}[b]{0.3\textwidth}
		\centering
		%tikzpicture 1
		% Cylinder decomposition in direction (2,1)
		\begin{tikzpicture}[scale=0.45]
        %% squares
	    \draw (0,0) rectangle (4,1);
	    \draw (6,0) rectangle (8,1);
	    \draw (6,0) rectangle (8,-1);
	    \draw (6,-1) rectangle (8,-2);
	    \draw (6,-2) rectangle (8,-3);
	    \draw (0,1) rectangle (1,2);
	    \draw (1,0) -- (1,1);
	    \draw (2,0) -- (2,1);
	    \draw (3,0) -- (3,1);
        \draw (7,1) -- (7,-3);
        %% Punkte
        \draw [dashed] (4.8,0.5) -- (5.2,0.5);
        \draw [decorate,line width=0.5mm,decoration={brace,mirror}] (1,-0.5) --  (5.8,-0.5) node[pos=0.5,below=10pt,black]{$6k-3$};
        \draw [decorate, decoration={snake}] (4.5,0) -- (4.5,1);
        \draw [decorate, decoration={snake}] (5.5,0) -- (5.5,1);
        \draw (4,0) -- (4.5,0);
        \draw (4,1) -- (4.5,1);
        \draw (5.5,0) -- (6,0);
        \draw (5.5,1) -- (6,1); 
        %% hatching
        \draw[pattern color=blue, pattern = north east lines] (0,0) -- (2,1) -- (1,1) -- (0,0.5) -- (0,0) -- cycle;
    	\draw[pattern color=blue, pattern = north east lines] (1,0) -- (2,0) -- (4,1) -- (3,1) -- (1,0) -- cycle;
        \draw[pattern color=blue, pattern = north east lines] (3,0) -- (4,0) -- (4,0.5) -- (3,0) -- cycle;	
        \draw[pattern color=blue, pattern = north east lines] (6,0) -- (8,1) -- (7,1) -- (6,0.5) -- (6,0) -- cycle;
	    \draw[pattern color=blue, pattern = north east lines] (7,-3) -- (8,-3) -- (8,-2.5) -- (7,-3) -- cycle;
	    \draw[pattern color=blue, pattern = north east lines] (6,-3) -- (8,-2) -- (8,-1.5) -- (6,-2.5) -- (6,-3) -- cycle;
	    \draw[pattern color=blue, pattern = north east lines] (6,-2) -- (8,-1) -- (8,-0.5) -- (6,-1.5) -- (6,-2) -- cycle;
	    \draw[pattern color=blue, pattern = north east lines] (6,-1) -- (8,0) -- (8,0.5) -- (6,-0.5) -- (6,1) -- cycle;
	    %% cylinder
	    \draw[thick, color=red] (0,0) -- (2,1);
	    \draw[thick, color=red] (1,0) -- (3,1);
	    \draw[thick, color=red] (2,0) -- (4,1);
	    \draw[thick, color=red] (3,0) -- (4,0.5);
	    \draw[thick, color=red] (6,0.5) -- (7,1);
	    \draw[thick, color=red] (6,0) -- (8,1);
	    \draw[thick, color=red] (0,0.5) -- (1,1);
	    \draw[thick, color=red] (0,1) -- (1,1.5);
	    \draw[thick, color=red] (0,1.5) -- (1,2);
	    \draw[thick, color=red] (6,-0.5) -- (8,0.5);
	    \draw[thick, color=red] (6,-1) -- (8,0);
	    \draw[thick, color=red] (6,-1.5) -- (8,-0.5);
	    \draw[thick, color=red] (6,-2) -- (8,-1);
	    \draw[thick, color=red] (6,-2.5) -- (8,-1.5);
	    \draw[thick, color=red] (6,-3) -- (8,-2);
	    \draw[thick, color=red] (7,-3) -- (8,-2.5);
	    \end{tikzpicture}
		\caption{}\label{fig:cylinder3n+10(2,1)}
	\end{subfigure}%
	\hfill
	\begin{subfigure}[b]{0.3\textwidth}
		\centering
		%tikzpicture 2
		% Second cylinder decomposition (-2,1)
        \begin{tikzpicture}[scale=0.45]
    	\draw (0,0) rectangle (4,1);
    	\draw (6,0) rectangle (8,1);
    	\draw (6,0) rectangle (8,-1);
    	\draw (6,-1) rectangle (8,-2);
    	\draw (6,-2) rectangle (8,-3);
    	\draw (0,1) rectangle (1,2);
    	\draw (1,0) -- (1,1);
    	\draw (2,0) -- (2,1);
    	\draw (3,0) -- (3,1);
        \draw (7,1) -- (7,-3);
        %% Punkte
        \draw [dashed] (4.8,0.5) -- (5.2,0.5);
        \draw [decorate,line width=0.5mm,decoration={brace,mirror}] (1,-0.5) --  (5.8,-0.5) node[pos=0.5,below=10pt,black]{$6k-3$};
        \draw [decorate, decoration={snake}] (4.5,0) -- (4.5,1);
        \draw [decorate, decoration={snake}] (5.5,0) -- (5.5,1);
        \draw (4,0) -- (4.5,0);
        \draw (4,1) -- (4.5,1);
        \draw (5.5,0) -- (6,0);
        \draw (5.5,1) -- (6,1); 
        %% hatching
   	    \draw[pattern color=blue, pattern = north east lines] (1,0) -- (0,0.5) -- (0,1) -- (2,0) -- (1,0) -- cycle;
	    \draw[pattern color=blue, pattern = north east lines] (3,0) -- (1,1) -- (2,1) -- (4,0) -- (3,0) -- cycle;
        \draw[pattern color=blue, pattern = north east lines] (4,0.5) -- (4,1) -- (3,1) -- (4,0.5) -- cycle;
        \draw[pattern color=blue, pattern = north east lines] (8,0.5) -- (8,1) -- (7,1) -- (8,0.5) -- cycle;	
	    \draw[pattern color=blue, pattern = north east lines] (8,0) -- (6,1) -- (6,0.5) -- (8,-0.5) -- (8,0) -- cycle;
	    \draw[pattern color=blue, pattern = north east lines] (8,-1) -- (6,0) -- (6,-0.5) -- (8,-1.5) -- (8,-1) --cycle;
	    \draw[pattern color=blue, pattern = north east lines] (8,-2) -- (6,-1) -- (6,-1.5) -- (8,-2.5) -- (8,-2) -- cycle;
	    \draw[pattern color=blue, pattern = north east lines] (8,-3) -- (6,-2) -- (6,-2.5) -- (7,-3) -- (8,-3) -- cycle;
        %% cylinder
	    \draw[thick, color=red] (1,0) -- (0,0.5);
    	\draw[thick, color=red] (2,0) -- (0,1);
    	\draw[thick, color=red] (3,0) -- (0,1.5);
    	\draw[thick, color=red] (1,1.5) -- (0,2);
	    \draw[thick, color=red] (4,0) -- (2,1);
	    \draw[thick, color=red] (4,0.5) -- (3,1);
	    \draw[thick, color=red] (8,0.5) -- (7,1);
    	\draw[thick, color=red] (8,0) -- (6,1);
    	\draw[thick, color=red] (8,-0.5) -- (6,0.5);
    	\draw[thick, color=red] (8,-1) -- (6,0);
    	\draw[thick, color=red] (8,-1.5) -- (6,-0.5);
    	\draw[thick, color=red] (8,-2) -- (6,-1);
    	\draw[thick, color=red] (8,-2.5) -- (6,-1.5);
    	\draw[thick, color=red] (8,-3) -- (6,-2);
    	\draw[thick, color=red] (7,-3) -- (6,-2.5);
	    \end{tikzpicture}
		\caption{}\label{fig:cylinder3n+10(-2,1)}
	\end{subfigure}%
	\hfill
	\begin{subfigure}[b]{0.3\textwidth}
		\centering
		%tikzpicture 3
		% First cylinder decomposition (1,2)
        \begin{tikzpicture}[scale=0.45]
        %% squares
        \draw (0,0) rectangle (4,1);
        \draw (6,0) rectangle (8,1);
    	\draw (6,0) rectangle (8,-1);
    	\draw (6,-1) rectangle (8,-2);
    	\draw (6,-2) rectangle (8,-3);
    	\draw (0,1) rectangle (1,2);
	    \draw (1,0) -- (1,1);
	    \draw (2,0) -- (2,1);
	    \draw (3,0) -- (3,1);
        \draw (7,1) -- (7,-3);
        %% Punkte
        \draw [dashed] (4.8,0.5) -- (5.2,0.5);
        \draw [decorate,line width=0.5mm,decoration={brace,mirror}] (1,-0.5) --  (5.8,-0.5) node[pos=0.5,below=10pt,black]{$6k-3$};
        \draw [decorate, decoration={snake}] (4.5,0) -- (4.5,1);
        \draw [decorate, decoration={snake}] (5.5,0) -- (5.5,1);
        \draw (4,0) -- (4.5,0);
        \draw (4,1) -- (4.5,1);
        \draw (5.5,0) -- (6,0);
        \draw (5.5,1) -- (6,1); 
        %% hatching
        \draw[pattern color=blue, pattern = north east lines] (6,0) -- (6,-1) -- (8,1) -- (7,1) -- (6,0) -- cycle;
	    \draw[pattern color=blue, pattern = north east lines] (6,-2) -- (6,-3) -- (8,-1) -- (8,0) -- (6,-2) -- cycle;
        \draw[pattern color=blue, pattern = north east lines] (7,-3) -- (8,-3) -- (8,-2) -- (7,-3) -- cycle;	
        %% cylinder
	    \draw[thick, color=red] (0,1) -- (1,2);
	    \draw[thick, color=red] (0,0) -- (1,1);
	    \draw[thick, color=red] (1,0) -- (2,1);
	    \draw[thick, color=red] (2,0) -- (3,1);
	    \draw[thick, color=red] (3,0) -- (4,1);
	    \draw[thick, color=red] (6,0) -- (7,1);
	    \draw[thick, color=red] (6,-1) -- (8,1);
	    \draw[thick, color=red] (6,-2) -- (8,0);
	    \draw[thick, color=red] (6,-3) -- (8,-1);
	    \draw[thick, color=red] (7,-3) -- (8,-2);
	    \end{tikzpicture}
		\caption{}\label{fig:cylinder3n+10(1,1)}
	\end{subfigure}
	\caption{Cylinder decomposition in direction \((2,1)\), $(-2,1)$ and $(1,1)$ of origamis in \(\mathcal{F}_2\). Here $\alpha_2$, $\beta_2$ and $\gamma_2$ are the core curves of the white cylinders in the respective decomposition.}
\end{figure}

%\begin{table}[H]
%\centering
%\begin{tabular}{|c|c |c| c| c| c| c|} 
% \hline
% \multirow{2}{*}{$\Omega$} & \multicolumn{2}{c|} {$\theta_1\parallel(2,1)$} & %\multicolumn{2}{c|}{$\theta_2\parallel(-2,1)$} &  \multicolumn{2}{c|}{$\theta_3\parallel(1,1)$} \\
% \cline{2-7}
% &$\alpha_1$ & $\alpha_2$ & $\beta_1$ & $\beta_2$ & $\gamma_1$ & $\gamma_2$ \\ [0.5ex] 
% \hline
% $\sigma_1$ & 0 & 1 & 0 & 1 & 1 & 0 \\ 
% $\sigma_2$ & $3k$ & $3k$ & $3k$ & $3k$ & $6k-1$ & 1   \\
% $\sigma_3$ & 1 & 1 & 1& 1 & 1  & 1 \\
% $\zeta_1$ & -1 & -3 & 1 & 3 & -2 & 0 \\
% $\zeta_2$ & -1 & -1 & 1 & 1 & -1 &  0 \\ 
% $\zeta_3$ & -4 & -4 & 4 & 4 & -2 & -2 \\ [1ex] 
% \hline
%\end{tabular}
%\hspace{1cm}
%\begin{tabular}{|c|c |c| c|} 
% \hline
% &$X_1$ & $X_2$ & $X_3$ \\ [0.5ex] 
% \hline
% $c_1$ & 1 & 1 & 1  \\ 
% $c_2$ & 1 & 1 & 1    \\
% $f_1$ & $3k+4$ &  1 & 2\\
% $f_2$ & $3k+5$ &  $6k+8$&  $6k+1$ \\ $k_1$ & $ 3k+5 $ & $ 6k+8 $ & $ 6k+1 $\\
% $k_2$ & $ 3k+4 $ & $ 1 $ & $ 2 $\\
% \hline
%\end{tabular}
%\caption{Intersections forms between the waist curves of the horizontal and vertical cylinders and data for multitwist for the family  $\mathcal{F}_{3n+10}$}
%\label{X-multitwistdataintersections3n+10}
%\end{table}

%\clearpage

\subsection{Arithmeticity for the family \boldmath{$\mathcal{F}_3$}}\label{ss.H4odd-family3}

We set the parameter $n$ defining the family as $n=2k-1$ so that the directions $\theta_1$, $\theta_2$ and $\theta_2$ give us two-cylinder decompositions. At the end we set $k=36l-1$ to find the appropiate word in the matrices $A_1$, $A_2$ and $A_3$. We collected the necessary information in Table \ref{multitwistdataintersections3n+12}, Table \ref{coefbaseandAction3n+12} and Table \ref{MatricesAndWord3n+12}. The cylinder decompositions we used for this family are illustrated in Figure
\ref{fig:cylinder3n+12(2,1)}, Figure \ref{fig:cylinder3n+12(1,2)} and Figure \ref{fig:cylinder3n+12(1,5)}.
\begin{table}[!htbp]
\begin{subtable}{0.6\textwidth}
\centering
\begin{tabular}{|c|c |c| c| c| c| c|} 
 \hline
 \multirow{2}{*}{$\Omega$} & \multicolumn{2}{c|} {$\theta_1\parallel(2,1)$} & \multicolumn{2}{c|}{$\theta_2\parallel(1,2)$} &  \multicolumn{2}{c|}{$\theta_3\parallel(1,5)$} \\
 \cline{2-7}
 &$\alpha_1$ & $\alpha_2$ & $\beta_1$ & $\beta_2$ & $\gamma_1$ & $\gamma_2$ \\ [0.5ex] 
 \hline
 $\sigma_1$ & 1 & 0 & 1 & 1 & 5 & 0 \\ 
 $\sigma_2$ & $3k$ & $3k$ & $12k-1$ & 1 & $30k-8$ & 2   \\
 $\sigma_3$ & 1 & 1 & 4 & 0 & 2  & 2 \\
 $\zeta_1$ & -3 & -1 & -1 & -1 & -2 & 0 \\
 $\zeta_2$ & -1 & -1 & -1 & 0 & -1 &  0 \\ 
 $\zeta_3$ & -5 & -5 & -5 & 0 & -1 & -1 \\ [1ex] 
 \hline
\end{tabular}
\caption{}
\end{subtable}
\begin{subtable}{0.38\textwidth}
\centering
\begin{tabular}{|c|c |c| c|} 
 \hline
 &$\theta_1$ & $\theta_2$ & $\theta_3$ \\ [0.5ex] 
 \hline
 $c_1$ & 1 & 1 & 1  \\ 
 $c_2$ & 1 & 1 & 4    \\
 $f_1$ &  $3k+5$ &  $6k+8$&  $6k+1$\\
 $f_2$ & $3k+4$ &  1 & 2\\
 $k_1$ & $ 3k+4$ & $6k+8 $ & 1\\
 $k_2$ & $3k+5$ &  $6k+8$&  $2(6k+1)$ \\
 \hline
\end{tabular}
\caption{}
\end{subtable}
\caption{Intersection data for the waist curves of the horizontal and vertical cylinders and data for multitwist for the family $\mathcal{F}_3$.}
\label{multitwistdataintersections3n+12}
\end{table}

%FIGURE
\begin{figure}[!h]
	\centering
	\noindent\begin{subfigure}[b]{0.3\textwidth}
		\centering
		%tikzpicture 1
		% First cylinder decomposition (1,2)
        \begin{tikzpicture}[scale=0.45]
        %% squares
	    \draw (0,0) rectangle (4,1);
	    \draw (6,0) rectangle (8,1);
	    \draw (6,0) rectangle (8,-1);
	    \draw (6,-1) rectangle (8,-2);
	    \draw (6,-2) rectangle (8,-3);
	    \draw (6,-3) rectangle (8,-4);
	    \draw (0,1) rectangle (1,2);
	    \draw (1,0) -- (1,1);
	    \draw (2,0) -- (2,1);
	    \draw (3,0) -- (3,1);
        \draw (7,1) -- (7,-4);
        %% Punkte
        \draw [dashed] (4.8,0.5) -- (5.2,0.5);
        \draw [decorate,line width=0.5mm,decoration={brace,mirror}] (1,-0.5) --  (5.8,-0.5) node[pos=0.5,below=10pt,black]{$6k-3$};
        \draw [decorate, decoration={snake}] (4.5,0) -- (4.5,1);
        \draw [decorate, decoration={snake}] (5.5,0) -- (5.5,1);
        \draw (4,0) -- (4.5,0);
        \draw (4,1) -- (4.5,1);
        \draw (5.5,0) -- (6,0);
        \draw (5.5,1) -- (6,1); 
        %% hatching
        \draw[pattern color=blue, pattern = north east lines] (0,0) -- (2,1) -- (1,1) -- (0,0.5) -- (0,0) -- cycle;
	    \draw[pattern color=blue, pattern = north east lines] (1,0) -- (2,0) -- (4,1) -- (3,1) -- (1,0) -- cycle;
        \draw[pattern color=blue, pattern = north east lines] (3,0) -- (4,0) -- (4,0.5) -- (3,0) -- cycle;	
        \draw[pattern color=blue, pattern = north east lines] (6,0) -- (8,1) -- (7,1) -- (6,0.5) -- (6,0) -- cycle;
		\draw[pattern color=blue, pattern = north east lines] (6,-1) -- (8,0) -- (8,0.5) -- (6,-0.5) -- (6,-1) -- cycle;
		\draw[pattern color=blue, pattern = north east lines] (6,-2) -- (8,-1) -- (8,-0.5) -- (6,-1.5) -- (6,-2) -- cycle;
		\draw[pattern color=blue, pattern = north east lines] (6,-3) -- (8,-2) -- (8,-1.5) -- (6,-2.5) -- (6,-3) -- cycle;
		\draw[pattern color=blue, pattern = north east lines] (6,-4) -- (8,-3) -- (8,-2.5) -- (6,-3.5) -- (6,-4) -- cycle;
		\draw[pattern color=blue, pattern = north east lines] (7,-4) -- (8,-4) -- (8,-3.5) -- (7,-4) -- cycle;
	    %% cylinder
	    \draw[thick, color=red] (0,0) -- (2,1);
	    \draw[thick, color=red] (0,0.5) -- (1,1);
	    \draw[thick, color=red] (0,1) -- (1,1.5);
	    \draw[thick, color=red] (0,1.5) -- (1,2);
	    \draw[thick, color=red] (1,0) -- (3,1);
	    \draw[thick, color=red] (2,0) -- (4,1);
	    \draw[thick, color=red] (3,0) -- (4,0.5);
	    \draw[thick, color=red] (6,0.5) -- (7,1);
	    \draw[thick, color=red] (6,0) -- (8,1);
	    \draw[thick, color=red] (6,-0.5) -- (8,0.5);
	    \draw[thick, color=red] (6,-1.5) -- (8,-0.5);
	    \draw[thick, color=red] (6,-2) -- (8,-1);
	    \draw[thick, color=red] (6,-2.5) -- (8,-1.5);
	    \draw[thick, color=red] (6,-3) -- (8,-2);
	    \draw[thick, color=red] (6,-3.5) -- (8,-2.5);
	    \draw[thick, color=red] (6,-4) -- (8,-3);
	    \draw[thick, color=red] (7,-4) -- (8,-3.5);
	    \end{tikzpicture}
		\caption{}\label{fig:cylinder3n+12(2,1)}
	\end{subfigure}%
	\hfill
	\begin{subfigure}[b]{0.3\textwidth}
		\centering
		%tikzpicture 2
		% Second cylinder decomposition (1,2)
        \begin{tikzpicture}[scale=0.45]
	    \draw (0,0) rectangle (4,1);
	    \draw (6,0) rectangle (8,1);
    	\draw (6,0) rectangle (8,-1);
    	\draw (6,-1) rectangle (8,-2);
    	\draw (6,-2) rectangle (8,-3);
    	\draw (6,-3) rectangle (8,-4);
    	\draw (0,1) rectangle (1,2);
    	\draw (1,0) -- (1,1);
	    \draw (2,0) -- (2,1);
    	\draw (3,0) -- (3,1);
        \draw (7,1) -- (7,-4);
        %% Punkte
        \draw [dashed] (4.8,0.5) -- (5.2,0.5);
        \draw [decorate,line width=0.5mm,decoration={brace,mirror}] (1,-0.5) --  (5.8,-0.5) node[pos=0.5,below=10pt,black]{$6k-3$};
        \draw [decorate, decoration={snake}] (4.5,0) -- (4.5,1);
        \draw [decorate, decoration={snake}] (5.5,0) -- (5.5,1);
        \draw (4,0) -- (4.5,0);
        \draw (4,1) -- (4.5,1);
        \draw (5.5,0) -- (6,0);
        \draw (5.5,1) -- (6,1); 
        %% hatching
        \draw[pattern color=blue, pattern = north east lines] (0,0) -- (1,2) -- (1,1) -- (0.5,0) -- (0,0) -- cycle;
	    \draw[pattern color=blue, pattern = north east lines] (0,1) -- (0.5,2) -- (0,2) -- (0,1) -- cycle;
	    %% cylinder
    	\draw[thick, color=red] (0,0) -- (1,2);
    	\draw[thick, color=red] (0,1) -- (0.5,2);
    	\draw[thick, color=red] (0.5,0) -- (1,1);
    	\draw[thick, color=red] (1,0) -- (1.5,1);
    	\draw[thick, color=red] (1.5,0) -- (2,1);
    	\draw[thick, color=red] (2,0) -- (2.5,1);
    	\draw[thick, color=red] (2.5,0) -- (3,1);
    	\draw[thick, color=red] (3,0) -- (3.5,1);
    	\draw[thick, color=red] (3.5,0) -- (4,1);
    	\draw[thick, color=red] (6,0) -- (6.5,1);
    	\draw[thick, color=red] (6,-1) -- (7,1);
    	\draw[thick, color=red] (6,-2) -- (7.5,1);
    	\draw[thick, color=red] (6,-3) -- (8,1);
    	\draw[thick, color=red] (6,-4) -- (8,0);
    	\draw[thick, color=red] (6.5,-4) -- (8,-1);
    	\draw[thick, color=red] (7,-4) -- (8,-2);
    	\draw[thick, color=red] (7.5,-4) -- (8,-3);
		\end{tikzpicture}
		\caption{}\label{fig:cylinder3n+12(1,2)}
	\end{subfigure}%
	\hfill
	\begin{subfigure}[b]{0.3\textwidth}
		\centering
		%tikzpicture 3
		% Third cylinder decomposition (1,5)
        \begin{tikzpicture}[scale=0.45]
        %% squares
        \draw (0,0) rectangle (4,1);
    	\draw (6,0) rectangle (8,1);
    	\draw (6,0) rectangle (8,-1);
    	\draw (6,-1) rectangle (8,-2);
    	\draw (6,-2) rectangle (8,-3);
    	\draw (6,-3) rectangle (8,-4);
    	\draw (0,1) rectangle (1,2);
    	\draw (1,0) -- (1,1);
    	\draw (2,0) -- (2,1);
    	\draw (3,0) -- (3,1);
        \draw (7,1) -- (7,-4);
        %% Punkte
        \draw [dashed] (4.8,0.5) -- (5.2,0.5);
        \draw [decorate,line width=0.5mm,decoration={brace,mirror}] (1,-0.5) --  (5.8,-0.5) node[pos=0.5,below=10pt,black]{$6k-3$};
        \draw [decorate, decoration={snake}] (4.5,0) -- (4.5,1);
        \draw [decorate, decoration={snake}] (5.5,0) -- (5.5,1);
        \draw (4,0) -- (4.5,0);
        \draw (4,1) -- (4.5,1);
        \draw (5.5,0) -- (6,0);
        \draw (5.5,1) -- (6,1); 
        %% hatching
        \draw[pattern color=blue, pattern = north east lines] (6,0) -- (6.2,1) -- (7,1) -- (6,-4) -- (6,0) -- cycle;
	    \draw[pattern color=blue, pattern = north east lines] (6.2,-4) -- (7.2,1) -- (8,1) -- (7,-4) -- (6.2,-4) -- cycle;
		\draw[pattern color=blue, pattern = north east lines] (7.2,-4) -- (8,0) -- (8,-4) -- (7.2,-4) -- cycle;
	    %% cylinder
    	\draw[thick, color=red] (0,1) -- (0.2,2);
        \draw[thick, color=red] (0,0) -- (0.4,2);
	    \draw[thick, color=red] (0.2,0) -- (0.6,2);
	    \draw[thick, color=red] (0.4,0) -- (0.8,2);
	    \draw[thick, color=red] (0.6,0) -- (1,2);
	    \draw[thick, color=red] (0.8,0) -- (1,1);
	    \draw[thick, color=red] (1,0) -- (1.2,1);
	    \draw[thick, color=red] (1.2,0) -- (1.4,1);
	    \draw[thick, color=red] (1.4,0) -- (1.6,1);
	    \draw[thick, color=red] (1.6,0) -- (1.8,1);
	    \draw[thick, color=red] (1.8,0) -- (2,1);
	    \draw[thick, color=red] (2,0) -- (2.2,1);
	    \draw[thick, color=red] (2.2,0) -- (2.4,1);
	    \draw[thick, color=red] (2.4,0) -- (2.6,1);
	    \draw[thick, color=red] (2.6,0) -- (2.8,1);
	    \draw[thick, color=red] (2.8,0) -- (3,1);
	    \draw[thick, color=red] (3,0) -- (3.2,1);
	    \draw[thick, color=red] (3.2,0) -- (3.4,1);
	    \draw[thick, color=red] (3.4,0) -- (3.6,1);
	    \draw[thick, color=red] (3.6,0) -- (3.8,1);
	    \draw[thick, color=red] (3.8,0) -- (4,1);
	    \draw[thick, color=red] (6,0) -- (6.2,1);
        \draw[thick, color=red] (6,-4) -- (7,1);
        \draw[thick, color=red] (6.2,-4) -- (7.2,1);
	    \draw[thick, color=red] (7,-4) -- (8,1);
	    \draw[thick, color=red] (7.2,-4) -- (8,0);
        \end{tikzpicture}
		\caption{}\label{fig:cylinder3n+12(1,5)}
	\end{subfigure}
	\caption{Cylinder decomposition in direction \((2,1)\), $(1,2)$ and $(1,5)$ of origamis in \(\mathcal{F}_3\). Here $\alpha_1$, $\beta_1$ and $\gamma_1$ are the core curves of the white cylinders in the respective decomposition.}
\end{figure}

\begin{table}[!htbp]
\centering
\begin{tabular}{|c|c |c| c| c|} 
 \hline
 &$e_1$ & $e_2$ & $e_3$ & $e_4$ \\ [0.5ex] 
 \hline
 $X_1$ & 1 & 4 & $-(3k+4)$ & 1 \\ 
 $X_2$ & -1 & -4 & $6k+7$  & -4  \\
 $X_3$ & -4 &  $2(6k+1)$ & -20 &$24k-4$\\
 $D_{1}$ & $3k$ & 1 & 1 & 0\\ 
 $D_{2}$ & $-6k+1$ & -2 & -1 & 0 \\
 $D_{3}$ & 2 & 2 & 0 & -1 \\ [1ex] 
 \hline
\end{tabular}
\caption{Coefficients for the basis in $H_1^{(0)}(M,\mathbb{Q})$ and coefficients for the action of the associated transvection on the elements of the basis for the family $\mathcal{F}_3$}
\label{coefbaseandAction3n+12}
\end{table}

\begin{table}[!htbp]
\centering
\begin{tabular}{|c|c |c|} 
 \hline
 &$X_2$ & $X_3$  \\ [0.5ex] 
 \hline
 $D_{1}$ & $(\dfrac{-9(k+1)}{2}+1)X_2+\dfrac{-3}{2}(k+1)^{2}X_3-3(k+1)e$ & $27X_2+(\dfrac{9(k+1)}{2}+1)X_3-18e$  \\ 
 $D_{2}$ & $X_2$ & $X_3+12X_2$    \\
 $D_{3}$ & $X_2-6X_3$ &  $X_3$ \\
 \hline
 Annihilator & \multicolumn{2}{c|}{$e=X_1+\dfrac{3}{2}X_2+(\dfrac{k+1}{4})X_3$}\\
 \hline
 Word & \multicolumn{2}{c|}{$k=36l-1, A_{3}^{l}A_{1}^{12} A_{3}^{-l}A_{2}^{-27}$ } \\ [1ex] 
 \hline
\end{tabular}
\caption{Action of each transvection on the ordered basis $X_2,X_3,e$ and nontrivial element $\Sp(W)$ for the family $\mathcal{F}_3$}
\label{MatricesAndWord3n+12}
\end{table}

%This establishes item (iii) of Theorem \ref{t.H4odd-families1to6}.

%FAMILY F4

%\clearpage 

\subsection{Arithmeticity for the family \boldmath{$\mathcal{F}_4$}}\label{ss.H4odd-family4} In this case we obtain a 2-cylinder decomposition for any $n\in\mathbb{N}$ and hence no additional parameter is needed in Step (B). Only one change of parameter is needed to find the appropriate word on the matrices $A_1$, $A_2$ and $A_3$. This is indicated in Table \ref{MatricesAndWord6n+14}. The remaining tables for this family are Table \ref{multitwistdataintersections6n+14} and Table \ref{coefbaseandAction6n+14}. The cylinder decompositions for this family can be seen in Figure \ref{fig:cylinder6n+14(2,1)}, Figure \ref{fig:cylinder6n+14(1,1)} and Figure \ref{fig:cylinder6n+14(-2,1)}

\begin{table}[!htbp]
\centering
\begin{subtable}{\textwidth}
\centering
\begin{tabular}{|c|c |c| c| c| c| c|} 
 \hline
 \multirow{2}{*}{$\Omega$} & \multicolumn{2}{c|} {$\theta_1\parallel(-2,1)$} & \multicolumn{2}{c|}{$\theta_2\parallel(1,1)$} &  \multicolumn{2}{c|}{$\theta_3\parallel(2,1)$} \\
 \cline{2-7}
 &$\alpha_1$ & $\alpha_2$ & $\beta_1$ & $\beta_2$ & $\gamma_1$ & $\gamma_2$ \\ [0.5ex] 
 \hline
 $\sigma_1$ & 0 & 1 & 0 & 1 & 0 & 1 \\ 
 $\sigma_2$ & $3n+3$ & $3n+3$ & 1 & $6n+5$ & $3n+3$ & $3n+3$  \\
 $\sigma_3$ & 1 & 1 & 1 & 1 & 1  & 1 \\
 $\zeta_1$ & 1 & 5 & 0 & -3 & -1 & -5 \\
 $\zeta_2$ & 1 & 1 & 0 & -1 & -1 & -1 \\ 
 $\zeta_3$ & 4 & 4 & 2 & -2 & -4 & -4 \\ [1ex] 
 \hline
\end{tabular}
\caption{}
\end{subtable}
\hfill
\begin{subtable}{\textwidth}
\centering
\begin{tabular}{|c|c |c| c|} 
 \hline
 &$\theta_1$ & $\theta_2$ & $\theta_3$ \\ [0.5ex] 
 \hline
 $c_1$ & 1 & 1 & 1  \\ 
 $c_2$ & 1 & 1 & 1    \\
 $f_1$ & $3n+6$ &  4 & $3n+6$\\
 $f_2$ & $3n+8$ &  $6n+10$&  $3n+8$ \\ 
 $k_1$ & $3n+8 $ & $3n+5 $ & $ 3n+8$\\
 $k_2$ & $3n+6$ & $2 $ & $ 3n+6$\\
 \hline
\end{tabular}
\caption{}
\end{subtable}
\caption{Intersections forms between the waist curves of the horizontal and vertical cylinders and data for multitwist for the family $\mathcal{F}_4$.}
\label{multitwistdataintersections6n+14}
\end{table}

%FIGURE

\begin{figure}[!htbp]
	\centering
	\noindent\begin{subfigure}[b]{0.3\textwidth}
		\centering
		%tikzpicture 1
		% First cylinder decomposition (2,1)
        \begin{tikzpicture}[scale=0.45]
        %% squares
	    \draw (0,0) rectangle (3,1);
	    \draw (5,0) rectangle (7,1);
	    \draw (5,0) rectangle (7,-1);
	    \draw (5,-1) rectangle (7,-2);
    	\draw (5,-2) rectangle (7,-3);
    	\draw (0,1) rectangle (1,2);
    	\draw (0,2) rectangle (1,3);
    	\draw (1,0) -- (1,1);
    	\draw (2,0) -- (2,1);
    	\draw (3,0) -- (3,1);
        \draw (6,1) -- (6,-3);
        %% Punkte
        \draw [dashed] (3.8,0.5) -- (4.2,0.5);
        \draw [decorate,line width=0.5mm,decoration={brace,mirror}] (1,-0.5) --  (4.8,-0.5) node[pos=0.5,below=10pt,black]{$6n+3$};
        \draw [decorate, decoration={snake}] (3.5,0) -- (3.5,1);
        \draw [decorate, decoration={snake}] (4.5,0) -- (4.5,1);
        \draw (3,0) -- (3.5,0);
        \draw (3,1) -- (3.5,1);
        \draw (4.5,0) -- (5,0);
        \draw (4.5,1) -- (5,1); 
        %% hatching
        \draw[pattern color=blue, pattern = north east lines] (0,0) -- (2,1) -- (1,1) -- (0,0.5) -- (0,0) -- cycle;
	    \draw[pattern color=blue, pattern = north east lines] (1,0) -- (2,0) -- (3,0.5) -- (3,1) -- (1,0) -- cycle;
	    \draw[pattern color=blue, pattern = north east lines] (5,0) -- (7,1) -- (6,1) -- (5,0.5) -- (5,0) -- cycle;
	    \draw[pattern color=blue, pattern = north east lines] (5,-1) -- (7,0) -- (7,0.5) -- (5,-0.5) -- (5,-1) -- cycle;
	    \draw[pattern color=blue, pattern = north east lines] (5,-2) -- (7,-1) -- (7,-0.5) -- (5,-1.5) -- (5,-2) -- cycle;
	    \draw[pattern color=blue, pattern = north east lines] (5,-3) -- (7,-2) -- (7,-1.5) -- (5,-2.5) -- (5,-3) -- cycle;
	    \draw[pattern color=blue, pattern = north east lines] (6,-3) -- (7,-3) -- (7,-2.5) -- (6,-3) -- cycle;
	    %% cylinder
	    \draw[thick, color=red] (0,0) -- (2,1);
	    \draw[thick, color=red] (0,0.5) -- (1,1);
    	\draw[thick, color=red] (0,1) -- (1,1.5);
    	\draw[thick, color=red] (0,1.5) -- (1,2);
    	\draw[thick, color=red] (0,2) -- (1,2.5);
    	\draw[thick, color=red] (0,2.5) -- (1,3);
    	\draw[thick, color=red] (1,0) -- (3,1);
	    \draw[thick, color=red] (2,0) -- (3,0.5);
	    \draw[thick, color=red] (5,0.5) -- (6,1);
    	\draw[thick, color=red] (5,0) -- (7,1);
	    \draw[thick, color=red] (5,-0.5) -- (7,0.5);
	    \draw[thick, color=red] (5,-1) -- (7,0);
	    \draw[thick, color=red] (5,-1.5) -- (7,-0.5);
	    \draw[thick, color=red] (5,-2) -- (7,-1);
	    \draw[thick, color=red] (5,-2.5) -- (7,-1.5);
	    \draw[thick, color=red] (5,-3) -- (7,-2);
	    \draw[thick, color=red] (6,-3) -- (7,-2.5);
	    \end{tikzpicture}
		\caption{}\label{fig:cylinder6n+14(2,1)}
	\end{subfigure}%
	\hfill
	\begin{subfigure}[b]{0.3\textwidth}
		\centering
		%tikzpicture 2
		% Second cylinder decomposition (1,1)
        \begin{tikzpicture}[scale=0.45]
        %% squares
        \draw (0,0) rectangle (3,1);
	    \draw (5,0) rectangle (7,1);
	    \draw (5,0) rectangle (7,-1);
	    \draw (5,-1) rectangle (7,-2);
	    \draw (5,-2) rectangle (7,-3);
	    \draw (0,1) rectangle (1,2);
	    \draw (0,2) rectangle (1,3);
	    \draw (1,0) -- (1,1);
	    \draw (2,0) -- (2,1);
	    \draw (3,0) -- (3,1);
        \draw (6,1) -- (6,-3);
        %% Punkte
        \draw [dashed] (3.8,0.5) -- (4.2,0.5);
        \draw [decorate,line width=0.5mm,decoration={brace,mirror}] (1,-0.5) --  (4.8,-0.5) node[pos=0.5,below=10pt,black]{$6n+3$};
        \draw [decorate, decoration={snake}] (3.5,0) -- (3.5,1);
        \draw [decorate, decoration={snake}] (4.5,0) -- (4.5,1);
        \draw (3,0) -- (3.5,0);
        \draw (3,1) -- (3.5,1);
        \draw (4.5,0) -- (5,0);
        \draw (4.5,1) -- (5,1); 
        %% hatching
        \draw[pattern color=blue, pattern = north east lines] (5,0) -- (6,1) -- (7,1) -- (5,-1) -- (5,0) -- cycle;
	   	\draw[pattern color=blue, pattern = north east lines] (5,-2) -- (7,0) -- (7,-1) -- (5,-3) -- (5,-2) -- cycle;
        \draw[pattern color=blue, pattern = north east lines] (6,-3) -- (7,-3) -- (7,-2) -- (6,-3) -- cycle;
	    %% cylinder
	    \draw[thick, color=red] (0,0) -- (1,1);
	    \draw[thick, color=red] (0,1) -- (1,2);
	    \draw[thick, color=red] (0,2) -- (1,3);
	    \draw[thick, color=red] (1,0) -- (2,1);
	    \draw[thick, color=red] (2,0) -- (3,1);
        \draw[thick, color=red] (5,0) -- (6,1);
	    \draw[thick, color=red] (5,-1) -- (7,1);
	    \draw[thick, color=red] (5,-2) -- (7,0);
	    \draw[thick, color=red] (5,-3) -- (7,-1);
	    \draw[thick, color=red] (6,-3) -- (7,-2);
		\end{tikzpicture}
		\caption{}\label{fig:cylinder6n+14(1,1)}
	\end{subfigure}%
	\hfill
	\begin{subfigure}[b]{0.3\textwidth}
		\centering
		%tikzpicture 3
		% Third cylinder decomposition (-2,1)
        \begin{tikzpicture}[scale=0.45]
	    \draw (0,0) rectangle (3,1);
	    \draw (5,0) rectangle (7,1);
	    \draw (5,0) rectangle (7,-1);
	    \draw (5,-1) rectangle (7,-2);
	    \draw (5,-2) rectangle (7,-3);
	    \draw (0,1) rectangle (1,2);
	    \draw (0,2) rectangle (1,3);
	    \draw (1,0) -- (1,1);
	    \draw (2,0) -- (2,1);
	    \draw (3,0) -- (3,1);
	    \draw (6,1) -- (6,-3);
        %% Punkte
        \draw [dashed] (3.8,0.5) -- (4.2,0.5);
        \draw [decorate,line width=0.5mm,decoration={brace,mirror}] (1,-0.5) --  (4.8,-0.5) node[pos=0.5,below=10pt,black]{$6n+3$};
        \draw [decorate, decoration={snake}] (3.5,0) -- (3.5,1);
        \draw [decorate, decoration={snake}] (4.5,0) -- (4.5,1);
        \draw (3,0) -- (3.5,0);
        \draw (3,1) -- (3.5,1);
        \draw (4.5,0) -- (5,0);
        \draw (4.5,1) -- (5,1); 
        %% hatching
        \draw[pattern color=blue, pattern = north east lines] (1,0) -- (2,0) -- (0,1) -- (0,0.5) -- (1,0) -- cycle;
	    \draw[pattern color=blue, pattern = north east lines] (3,0) -- (3,0.5) -- (2,1) -- (1,1) -- (3,0) -- cycle;
	    \draw[pattern color=blue, pattern = north east lines] (5,1) -- (7,0) -- (7,-0.5) -- (5,0.5) -- (5,1) -- cycle;
	    \draw[pattern color=blue, pattern = north east lines] (5,0) -- (7,-1) -- (7,-1.5) -- (5,-0.5) -- (5,0) -- cycle;
	   	\draw[pattern color=blue, pattern = north east lines] (5,-1) -- (7,-2) -- (7,-2.5) -- (5,-1.5) -- (5,-1) -- cycle;
	    \draw[pattern color=blue, pattern = north east lines] (5,-2) -- (7,-3) -- (6,-3) -- (5,-2.5) -- (5,-2) -- cycle;
	    \draw[pattern color=blue, pattern = north east lines] (7,1) -- (6,1) -- (7,0.5) -- (7,1) -- cycle;
	    %% cylinder
	    \draw[thick, color=red] (1,0) -- (0,0.5);
	    \draw[thick, color=red] (2,0) -- (0,1);
	    \draw[thick, color=red] (1,1) -- (0,1.5);
	    \draw[thick, color=red] (1,1.5) -- (0,2);
    	\draw[thick, color=red] (1,2) -- (0,2.5);
	    \draw[thick, color=red] (1,2.5) -- (0,3);
	    \draw[thick, color=red] (3,0) -- (1,1);
	    \draw[thick, color=red] (2,1) -- (3,0.5);
	    \draw[thick, color=red] (7,0.5) -- (6,1);
	    \draw[thick, color=red] (7,0) -- (5,1);
	    \draw[thick, color=red] (7,-0.5) -- (5,0.5);
	    \draw[thick, color=red] (7,-1) -- (5,0);
    	\draw[thick, color=red] (7,-1.5) -- (5,-0.5);
	    \draw[thick, color=red] (7,-2) -- (5,-1);
	    \draw[thick, color=red] (7,-2.5) -- (5,-1.5);
    	\draw[thick, color=red] (7,-3) -- (5,-2);
    	\draw[thick, color=red] (6,-3) -- (5,-2.5);
		\end{tikzpicture}
		\caption{}\label{fig:cylinder6n+14(-2,1)}
	\end{subfigure}
	\caption{Cylinder decomposition in direction \((2,1)\), $(1,1)$ and $(-2,1)$ of origamis in \(\mathcal{F}_4\). Here $\alpha_2$, $\beta_2$ and $\gamma_2$ are the core curves of the white cylinder in the respective decomposition.}
\end{figure}

\begin{table}[!htbp]
\centering
\begin{tabular}{|c|c |c| c| c|} 
 \hline
 &$e_1$ & $e_2$ & $e_3$ & $e_4$ \\ [0.5ex] 
 \hline
 $X_1$ & 2 & 6 & $3n+6$ & -2 \\ 
 $X_2$ & 2 & $-(6n+8)$ & 2  & $-(3n+3)$  \\
 $X_3$ & -2 &  -6 & $3n+6$ & -2\\
 $D_{1}$ & $-(3n+3)$ & -1 & 2 & 0\\ 
 $D_{2}$ & -1 & -1 & 0 & 2 \\
 $D_{3}$ &  $-(3n+3)$  & -1 & -2 & 0 \\ [1ex] 
 \hline
\end{tabular}
\caption{Coefficients for the basis in $H_1^{(0)}(M,\mathbb{Q})$ and coefficients for the action of the associated transvection on the elements of the basis for the family $\mathcal{F}_4$.}
\label{coefbaseandAction6n+14}
\end{table}
\begin{table}[!htbp]
\centering
\begin{tabular}{|c|c |c|} 
 \hline
 &$X_2$ & $X_3$  \\ [0.5ex] 
 \hline
 $D_{1}$ & $(1+36(n+2))X_2-18X_3+6e$ & $36(n+2)^2X_2+(1-36(n+2  ))X_3+12(n+2)e$  \\ 
 $D_{2}$ & $X_2$ & $X_3+4X_2$     \\
 $D_{3}$ & $X_2-2X_3$ &  $X_3$ \\
 \hline
 Annihilator & \multicolumn{2}{c|}{$e=X_1-(6n+12)X_2+3X_3$}\\
 \hline
 Word & \multicolumn{2}{c|}{$n=2k-2, A_{2}^{k}A_{1} A_{2}^{-k}A_{3}^{-9}$ } \\ [1ex] 
 \hline
\end{tabular}
\caption{Action of each transvection on the ordered basis $X_2,X_3,e$ and nontrivial element $\Sp(W)$ for the family $\mathcal{F}_4$}
\label{MatricesAndWord6n+14}
\end{table}

%This shows item (iv) of Theorem \ref{t.H4odd-families1to6}.

%\clearpage

\subsection{Arithmeticity for the family \boldmath{$\mathcal{F}_5$}}\label{ss.H4odd-family5}

Again, we do not  need an additional parameter in Step (B) to get a 2-cylinder decomposition. Only one change of parameter is needed to find the appropriate word in the matrices $A_1$, $A_2$ and $A_3$. This is indicated in Table \ref{MatricesAndWord6n+18}. The remainig tables for this family are Table \ref{multitwistdata6n+18} and Table \ref{coefbaseandAction6n+18}.
The cylinder decompositions we used for this family can be seen in Figure \ref{fig:cylinder6n+18(2,1)}, Figure \ref{fig:cylinder6n+18(-1,1)} and Figure \ref{fig:cylinder6n+18(1,1)}.
\begin{table}[!htbp]
\centering
\begin{subtable}{0.6\textwidth}
\centering
\begin{tabular}{|c|c |c| c| c| c| c|} 
 \hline
 \multirow{2}{*}{$\Omega$} & \multicolumn{2}{c|} {$\theta_1\parallel(2,1)$} & \multicolumn{2}{c|}{$\theta_2\parallel(1,1)$} &  \multicolumn{2}{c|}{$\theta_3\parallel(-1,1)$} \\
 \cline{2-7}
 &$\alpha_1$ & $\alpha_2$ & $\beta_1$ & $\beta_2$ & $\gamma_1$ & $\gamma_2$ \\ [0.5ex] 
 \hline
 $\sigma_1$ & 0 & 1 & 0 & 1 & 0 & 1 \\ 
 $\sigma_2$ & $3n+3$ & $3n+3$ & 1 & $6n+5$ & 1 &  $6n+5$  \\
 $\sigma_3$ & 1 & 1 & 1 & 1 & 1  & 1 \\
 $\zeta_1$ & 1 & 5 & 0 & -3 & 0 & 3 \\
 $\zeta_2$ & 1 & 1 & 0 & -1 & 0 & 1 \\ 
 $\zeta_3$ & 6 & 6 & -3 & -3 & 3 & 3 \\ [1ex] 
 \hline
\end{tabular} 
\caption{}
\end{subtable}
\hfill
\begin{subtable}{0.38\textwidth}
\centering
\begin{tabular}{|c|c |c| c|} 
 \hline
 &$\theta_1$ & $\theta_2$ & $\theta_3$ \\ [0.5ex] 
 \hline
 $c_1$ & 1 & 1 & 1  \\ 
 $c_2$ & 1 & 1 & 1    \\
 $f_1$ & $3n+8$ &  6 & 6\\
 $f_2$ & $3n+10$ &  $6n+12$&  $6n+12$ \\ 
 $k_1$ & $ 3n+10 $ & $ n+2 $ & $ n+2 $\\
 $k_2$ & $ 3n+8 $ & $ 1 $ & $ 1 $\\
 \hline
\end{tabular}
\caption{}
\end{subtable}
\caption{Intersections forms between the waist curves of the horizontal and vertical cylinder and data for multitwist for the family  $\mathcal{F}_5$}
\label{multitwistdata6n+18}
\end{table}

\begin{table}[!htbp]
\centering
\begin{tabular}{|c|c |c| c| c|} 
 \hline
 &$e_1$ & $e_2$ & $e_3$ & $e_4$ \\ [0.5ex] 
 \hline
 $X_1$ & 2 & 10 & $3n+8$ & -2 \\ 
 $X_2$ & 1 & $-(3n+4)$ & 1  & $-(n+1)$  \\
 $X_3$ & -1 &  $3n+4$ & 1 & $-(n+1)$ \\
 $D_{1}$ & $-(3n+3)$ & -1 & 2 & 0\\ 
 $D_{2}$ & -1 & -1 & 0 & 3 \\
 $D_{3}$ &  -1  & -1 & 0 & -3 \\ [1ex] 
 \hline
\end{tabular}
\caption{Coefficients for the basis in $H_1^{(0)}(M,\mathbb{Q})$ and coefficients for the action of the associated transvection on the elements of the basis for the family $\mathcal{F}_5$.}
\label{coefbaseandAction6n+18}
\end{table}
\begin{table}[!htbp]
\centering
\begin{tabular}{|c|c |c|} 
 \hline
 &$X_1$ & $X_3$  \\ [0.5ex] 
 \hline
 $D_{1}$ & $X_1$ & $X_3+X_1$  \\ 
 $D_{2}$ & $(1+18(n+1))X_1-54X_3-18e$ & $6(n+1)^2X_1+(1-18(n+1))X_3-6(n+1)e$     \\
 $D_{3}$ & $X_1-6X_3$ &  $X_3$ \\
 \hline
 Annihilator & \multicolumn{2}{c|}{$e=(n+1)X_1+X_2-3X_3$}\\
 \hline
 Word & \multicolumn{2}{c|}{$n=3k-1, A_{1}^{k}A_{2} A_{1}^{-k}A_{3}^{-9}$ } \\ [1ex] 
 \hline
\end{tabular}
\caption{Action of each transvection on the ordered basis $X_1,X_3,e$ and nontrivial element $\Sp(W)$ for the family $\mathcal{F}_5$.}
\label{MatricesAndWord6n+18}
\end{table}

%This yields item (v) of Theorem \ref{t.H4odd-families1to6}.

\begin{figure}[!htbp]
	\centering
	\noindent\begin{subfigure}[b]{0.3\textwidth}
		\centering
		%tikzpicture 1
		% First cylinder decomposition (1,2)
        \begin{tikzpicture}[scale=0.45]
        %% squares
	    \draw (0,0) rectangle (3,1);
	    \draw (5,0) rectangle (7,1);
	    \draw (5,0) rectangle (7,-1);
	    \draw (5,-1) rectangle (7,-2);
	    \draw (5,-2) rectangle (7,-3);
	    \draw (0,1) rectangle (1,2);
	    \draw (0,2) rectangle (1,3);
	    \draw (5,-3) rectangle (7,-4);
        \draw (5,-4) rectangle (7,-5);
	    \draw (1,0) -- (1,1);
	    \draw (2,0) -- (2,1);
	    \draw (3,0) -- (3,1);
        \draw (6,1) -- (6,-5);
        %% Punkte
        \draw [dashed] (3.8,0.5) -- (4.2,0.5);
        \draw [decorate,line width=0.5mm,decoration={brace,mirror}] (1,-0.5) --  (4.8,-0.5) node[pos=0.5,below=10pt,black]{$6n+3$};
        \draw [decorate, decoration={snake}] (3.5,0) -- (3.5,1);
        \draw [decorate, decoration={snake}] (4.5,0) -- (4.5,1);
        \draw (3,0) -- (3.5,0);
        \draw (3,1) -- (3.5,1);
        \draw (4.5,0) -- (5,0);
        \draw (4.5,1) -- (5,1); 
        %% hatching
        \draw[pattern color=blue, pattern = north east lines] (0,0) -- (2,1) -- (1,1) -- (0,0.5) -- (0,0) -- cycle;
	    \draw[pattern color=blue, pattern = north east lines] (1,0) -- (2,0) -- (3,0.5) -- (3,1) -- (1,0) -- cycle;
        \draw[pattern color=blue, pattern = north east lines] (5,0) -- (7,1) -- (6,1) -- (5,0.5) -- (5,0) -- cycle;
		\draw[pattern color=blue, pattern = north east lines] (5,-1) -- (7,0) -- (7,0.5) -- (5,-0.5) -- (5,-1) -- cycle;
	    \draw[pattern color=blue, pattern = north east lines] (5,-2) -- (7,-1) -- (7,-0.5) -- (5,-1.5) -- (5,-2) -- cycle;
		\draw[pattern color=blue, pattern = north east lines] (5,-3) -- (7,-2) -- (7,-1.5) -- (5,-2.5) -- (5,-3) -- cycle;
		\draw[pattern color=blue, pattern = north east lines] (5,-4) -- (7,-3) -- (7,-2.5) -- (5,-3.5) -- (5,-4) -- cycle;
		\draw[pattern color=blue, pattern = north east lines] (5,-5) -- (7,-4) -- (7,-3.5) -- (5,-4.5) -- (5,-5) -- cycle;
		\draw[pattern color=blue, pattern = north east lines] (6,-5) -- (7,-5) -- (7,-4.5) -- (6,-5) -- cycle;
	    %% cylinder
	    \draw[thick, color=red] (0,0) -- (2,1);
	    \draw[thick, color=red] (0,0.5) -- (1,1);
	    \draw[thick, color=red] (0,1) -- (1,1.5);
	    \draw[thick, color=red] (0,1.5) -- (1,2);
	    \draw[thick, color=red] (0,2) -- (1,2.5);
    	\draw[thick, color=red] (0,2.5) -- (1,3);
	    \draw[thick, color=red] (1,0) -- (3,1);
	    \draw[thick, color=red] (2,0) -- (3,0.5);
		\draw[thick, color=red] (5,0.5) -- (6,1);
	    \draw[thick, color=red] (5,0) -- (7,1);
	    \draw[thick, color=red] (5,-0.5) -- (7,0.5);
	    \draw[thick, color=red] (5,-1.5) -- (7,-0.5);
	    \draw[thick, color=red] (5,-1) -- (7,0);
	    \draw[thick, color=red] (5,-2) -- (7,-1);
	    \draw[thick, color=red] (5,-2.5) -- (7,-1.5);
	    \draw[thick, color=red] (5,-3) -- (7,-2);
	    \draw[thick, color=red] (5,-3.5) -- (7,-2.5);
	    \draw[thick, color=red] (5,-4) -- (7,-3);
	    \draw[thick, color=red] (5,-4.5) -- (7,-3.5);
	    \draw[thick, color=red] (5,-5) -- (7,-4);
	    \draw[thick, color=red] (6,-5) -- (7,-4.5);
		\end{tikzpicture}
		\caption{}\label{fig:cylinder6n+18(2,1)}
	\end{subfigure}%
	\hfill
	\begin{subfigure}[b]{0.3\textwidth}
		\centering
		%tikzpicture 2
		% Second cylinder decomposition (-1,1)
        \begin{tikzpicture}[scale=0.45]
	    \draw (0,0) rectangle (3,1);
	    \draw (5,0) rectangle (7,1);
	    \draw (5,0) rectangle (7,-1);
	    \draw (5,-1) rectangle (7,-2);
    	\draw (5,-2) rectangle (7,-3);
	    \draw (0,1) rectangle (1,2);
	    \draw (0,2) rectangle (1,3);
		\draw (5,-3) rectangle (7,-4);
        \draw (5,-4) rectangle (7,-5);
	    \draw (1,0) -- (1,1);
	    \draw (2,0) -- (2,1);
	    \draw (3,0) -- (3,1);
        \draw (6,1) -- (6,-5);
        %% Punkte
        \draw [dashed] (3.8,0.5) -- (4.2,0.5);
        \draw [decorate,line width=0.5mm,decoration={brace,mirror}] (1,-0.5) --  (4.8,-0.5) node[pos=0.5,below=10pt,black]{$6n+3$};
        \draw [decorate, decoration={snake}] (3.5,0) -- (3.5,1);
        \draw [decorate, decoration={snake}] (4.5,0) -- (4.5,1);
        \draw (3,0) -- (3.5,0);
        \draw (3,1) -- (3.5,1);
        \draw (4.5,0) -- (5,0);
        \draw (4.5,1) -- (5,1); 
        %% hatching
        \draw[pattern color=blue, pattern = north east lines] (5,1) -- (7,-1) -- (7,0) -- (6,1) -- (5,1) -- cycle;
		\draw[pattern color=blue, pattern = north east lines] (5,0) -- (7,-2) -- (7,-3) -- (5,-1) -- (5,0) -- cycle;
		\draw[pattern color=blue, pattern = north east lines] (5,-2) -- (7,-4) -- (7,-5) -- (5,-3) -- (5,-2) -- cycle;
    	\draw[pattern color=blue, pattern = north east lines] (6,-5) -- (5,-4) -- (5,-5) -- (6,-5) -- cycle;
	    %% cylinder
	    \draw[thick, color=red] (1,0) -- (0,1);
	    \draw[thick, color=red] (1,2) -- (0,3);
	    \draw[thick, color=red] (2,0) -- (0,2);
	    \draw[thick, color=red] (3,0) -- (2,1);
	    \draw[thick, color=red] (7,0) -- (6,1);
	    \draw[thick, color=red] (7,-1) -- (5,1);
	    \draw[thick, color=red] (7,-2) -- (5,0);
	    \draw[thick, color=red] (7,-3) -- (5,-1);
	    \draw[thick, color=red] (7,-4) -- (5,-2);
	    \draw[thick, color=red] (7,-5) -- (5,-3);
	    \draw[thick, color=red] (6,-5) -- (5,-4);
		\end{tikzpicture}
		\caption{}\label{fig:cylinder6n+18(-1,1)}
	\end{subfigure}%
	\hfill
	\begin{subfigure}[b]{0.3\textwidth}
		\centering
		%tikzpicture 3
		% Third cylinder decomposition in direction (1,1)
		\begin{tikzpicture}[scale=0.45]
        %% squares
	    \draw (0,0) rectangle (3,1);
	    \draw (5,0) rectangle (7,1);
	    \draw (5,0) rectangle (7,-1);
	    \draw (5,-1) rectangle (7,-2);
	    \draw (5,-2) rectangle (7,-3);
	    \draw (0,1) rectangle (1,2);
	    \draw (0,2) rectangle (1,3);
	    \draw (5,-3) rectangle (7,-4);
        \draw (5,-4) rectangle (7,-5);
	    \draw (1,0) -- (1,1);
	    \draw (2,0) -- (2,1);
	    \draw (3,0) -- (3,1);
        \draw (6,1) -- (6,-5);
        %% Punkte
        \draw [dashed] (3.8,0.5) -- (4.2,0.5);
        \draw [decorate,line width=0.5mm,decoration={brace,mirror}] (1,-0.5) --  (4.8,-0.5) node[pos=0.5,below=10pt,black]{$6n+3$};
        \draw [decorate, decoration={snake}] (3.5,0) -- (3.5,1);
        \draw [decorate, decoration={snake}] (4.5,0) -- (4.5,1);
        \draw (3,0) -- (3.5,0);
        \draw (3,1) -- (3.5,1);
        \draw (4.5,0) -- (5,0);
        \draw (4.5,1) -- (5,1); 
        %% hatching
        \draw[pattern color=blue, pattern = north east lines] (5,0) -- (6,1) -- (7,1) -- (5,-1) -- (5,0) -- cycle;
	    \draw[pattern color=blue, pattern = north east lines] (5,-2) -- (7,0) -- (7,-1) -- (5,-3) -- (5,-2) -- cycle;
	    \draw[pattern color=blue, pattern = north east lines] (5,-4) -- (7,-2) -- (7,-3) -- (5,-5) -- (5,-4) -- cycle;
        \draw[pattern color=blue, pattern = north east lines] (6,-5) -- (7,-5) -- (7,-4) -- (6,-5) -- cycle;
	    %% cylinder
	    \draw[thick, color=red] (0,0) -- (1,1);
	    \draw[thick, color=red] (0,1) -- (1,2);
	    \draw[thick, color=red] (0,2) -- (1,3);
	    \draw[thick, color=red] (1,0) -- (2,1);
	    \draw[thick, color=red] (2,0) -- (3,1);
	    \draw[thick, color=red] (5,0) -- (6,1);
	    \draw[thick, color=red] (5,-1) -- (7,1);
	    \draw[thick, color=red] (5,-2) -- (7,0);
	    \draw[thick, color=red] (5,-3) -- (7,-1);
	    \draw[thick, color=red] (5,-4) -- (7,-2);
	    \draw[thick, color=red] (5,-5) -- (7,-3);
	    \draw[thick, color=red] (6,-5) -- (7,-4);
		\end{tikzpicture}
		\caption{}\label{fig:cylinder6n+18(1,1)}
	\end{subfigure}
	\caption{Cylinder decomposition in direction \((2,1)\) $(-1,1)$ and $(1,1)$ of origamis in \(\mathcal{F}_5\). Here $\alpha_2$, $\beta_2$ and $\gamma_2$ are the core curves of the white cylinders in the respective decomposition.}
\end{figure}

%\clearpage

\subsection{Arithmeticity for the family \boldmath{$\mathcal{F}_6$}}\label{ss.H4odd-family6} 
The tables for this family are Table \ref{multitwistdata6n+22}, Table \ref{coefbaseandAction6n+22} and Table \ref{MatricesAndWord6n+22}. The cylinder decompositions we used for this family are in direction $(2,1)$, $(1,1)$ and $(-2,1)$ and can be seen in Figure \ref{fig:cylinder6n+22(2,1)}, Figure \ref{fig:cylinder6n+22(-2,1)} and Figure \ref{fig:cylinder6n+22(1,1)}.
\begin{table}[!htbp]
\centering
\begin{subtable}{0.6\textwidth}
\centering
\begin{tabular}{|c|c |c| c| c| c| c|} 
 \hline
 \multirow{2}{*}{$\Omega$} & \multicolumn{2}{c|} {$\theta_1\parallel(2,1)$} & \multicolumn{2}{c|}{$\theta_2\parallel(1,1)$} &  \multicolumn{2}{c|}{$\theta_3\parallel(-1,1)$} \\
 \cline{2-7}
 &$\alpha_1$ & $\alpha_2$ & $\beta_1$ & $\beta_2$ & $\gamma_1$ & $\gamma_2$ \\ [0.5ex] 
 \hline
 $\sigma_1$ & 0 & 1 & 0 & 1 & 0 & 1 \\ 
 $\sigma_2$ & $3n+3$ & $3n+3$ & 1 & $6n+5$ & 1 &  $6n+5$  \\
 $\sigma_3$ & 1 & 1 & 1 & 1 & 1  & 1 \\
 $\zeta_1$ & 1 & 5 & 0 & -3 & 0 & 3 \\
 $\zeta_2$ & 1 & 1 & 0 & -1 & 0 & 1 \\ 
 $\zeta_3$ & 8 & 8 & -4 & -4 & 4 & 4 \\ [1ex] 
 \hline
\end{tabular}
\caption{}
\end{subtable}
\hfill
\begin{subtable}{0.38\textwidth}
\centering
\begin{tabular}{|c|c |c| c|} 
 \hline
 &$\theta_1$ & $\theta_2$ & $\theta_3$ \\ [0.5ex] 
 \hline
 $c_1$ & 1 & 1 & 1  \\ 
 $c_2$ & 1 & 1 & 1    \\
 $f_1$ & $3n+10$ &  8 & 8\\
 $f_2$ & $3n+12$ &  $6n+14$&  $6n+14$ \\ 
 $k_1$ & $ 3n+12$ & $ 3n+7 $ & $3n+7 $\\
 $k_2$ & $ 3n+10$ & $ 4 $ & $4 $\\
 \hline
\end{tabular}
\caption{}
\end{subtable}
\caption{Intersections forms between the waist curves of the horizontal and vertical cylinders and data for multitwist for the family $\mathcal{F}_6$.}
\label{multitwistdata6n+22}
\end{table}

\begin{table}[!htbp]
\centering
\begin{tabular}{|c|c |c| c| c|} 
 \hline
 &$e_1$ & $e_2$ & $e_3$ & $e_4$ \\ [0.5ex] 
 \hline
 $X_1$ & 2 & 14 & $3n+10$ & -2 \\ 
 $X_2$ & 4 & $-(12n+16)$ & 4  & $-(3n+3)$  \\
 $X_3$ & -4 &  $12n+16$ & 4 & $-(3n+3)$ \\
 $D_{1}$ & $-(3n+3)$ & -1 & 2 & 0\\ 
 $D_{2}$ & -1 & -1 & 0 & 4 \\
 $D_{3}$ &  -1  & -1 & 0 & -4 \\ [1ex] 
 \hline
\end{tabular}
\caption{Coefficients for the basis in $H_1^{(0)}(M,\mathbb{Q})$ and coefficients for the action of the associated transvection on the elements of the basis for the family $\mathcal{F}_6$.}
\label{coefbaseandAction6n+22}
\end{table}

\begin{table}[!htbp]
\centering
\begin{tabular}{|c|c |c|} 
 \hline
 &$X_1$ & $X_3$  \\ [0.5ex] 
 \hline
 $D_{1}$ & $X_1$ & $X_3+4X_1$  \\ 
 $D_{2}$ & $(1+72(n+1))X_1-72X_3-24e$ & $72(n+1)^2X_1+(1-72(n+1))X_3-24(n+1)e$     \\
 $D_{3}$ & $X_1-8X_3$ &  $X_3$ \\
 \hline
 Annihilator & \multicolumn{2}{c|}{$e=3(n+1)X_1+X_2-3X_3$}\\
 \hline
 Word & \multicolumn{2}{c|}{$n=4k-1, A_{1}^{k}A_{2} A_{1}^{-k}A_{3}^{-9}$ } \\ [1ex] 
 \hline
\end{tabular}
\caption{Action of each transvection on the ordered basis $X_1,X_3,e$ and nontrivial element $\Sp(W)$ for the family $\mathcal{F}_6$.}
\label{MatricesAndWord6n+22}
\end{table}

\begin{figure}[!htbp]
	\centering
	\noindent\begin{subfigure}[b]{0.3\textwidth}
		\centering
		%tikzpicture 1
		% First cylinder decomposition (2,1)
		\begin{tikzpicture}[scale=0.45]
        %% squares
	    \draw (0,0) rectangle (3,1);
	    \draw (5,0) rectangle (7,1);
	    \draw (5,0) rectangle (7,-1);
	    \draw (5,-1) rectangle (7,-2);
	    \draw (5,-2) rectangle (7,-3);
	    \draw (0,1) rectangle (1,2);
	    \draw (0,2) rectangle (1,3);
    	\draw (5,-3) rectangle (7,-4);
        \draw (5,-4) rectangle (7,-5);
        \draw (5,-5) rectangle (7,-6);
        \draw (5,-6) rectangle (7,-7);
	    \draw (1,0) -- (1,1);
	    \draw (2,0) -- (2,1);
	    \draw (3,0) -- (3,1);
        \draw (6,1) -- (6,-7);
        % Punkte
        \draw [dashed] (3.8,0.5) -- (4.2,0.5);
        \draw [decorate,line width=0.5mm,decoration={brace,mirror}] (1,-0.5) --  (4.8,-0.5) node[pos=0.5,below=10pt,black]{$6n+3$};
        \draw [decorate, decoration={snake}] (3.5,0) -- (3.5,1);
        \draw [decorate, decoration={snake}] (4.5,0) -- (4.5,1);
        \draw (3,0) -- (3.5,0);
        \draw (3,1) -- (3.5,1);
        \draw (4.5,0) -- (5,0);
        \draw (4.5,1) -- (5,1); 
        % hatching
        \draw[pattern color=blue, pattern = north east lines] (0,0) -- (2,1) -- (1,1) -- (0,0.5) -- (0,0) -- cycle;
	    \draw[pattern color=blue, pattern = north east lines] (1,0) -- (2,0) -- (3,0.5) -- (3,1) -- (1,0) -- cycle;
        \draw[pattern color=blue, pattern = north east lines] (5,0) -- (7,1) -- (6,1) -- (5,0.5) -- (5,0) -- cycle;
		\draw[pattern color=blue, pattern = north east lines] (5,-0.5) -- (5,-1) -- (7,0) -- (7,0.5) -- (5,-0.5) -- cycle;
		\draw[pattern color=blue, pattern = north east lines] (5,-2) -- (7,-1) -- (7,-0.5) -- (5,-1.5) -- (5,-2) -- cycle;
		\draw[pattern color=blue, pattern = north east lines] (5,-3) -- (7,-2) -- (7,-1.5) -- (5,-2.5) -- (5,-3) -- cycle;
		\draw[pattern color=blue, pattern = north east lines] (5,-4) -- (7,-3) -- (7,-2.5) -- (5,-3.5) -- (5,-4) -- cycle;
		\draw[pattern color=blue, pattern = north east lines] (5,-5) -- (7,-4) -- (7,-3.5) -- (5,-4.5) -- (5,-5) -- cycle;
		\draw[pattern color=blue, pattern = north east lines] (5,-6) -- (7,-5) -- (7,-4.5) -- (5,-5.5) -- (5,-6) -- cycle;
		\draw[pattern color=blue, pattern = north east lines] (5,-7) -- (7,-6) -- (7,-5.5) -- (5,-6.5) -- (5,-7) -- cycle;
		\draw[pattern color=blue, pattern = north east lines] (6,-7) -- (7,-7) -- (7,-6.5) -- (6,-7) -- cycle;
	    %% cylinder
	    \draw[thick, color=red] (0,0) -- (2,1);
	    \draw[thick, color=red] (0,0.5) -- (1,1);
	    \draw[thick, color=red] (0,1) -- (1,1.5);
	    \draw[thick, color=red] (0,1.5) -- (1,2);
	    \draw[thick, color=red] (0,2) -- (1,2.5);
	    \draw[thick, color=red] (0,2.5) -- (1,3);
	    \draw[thick, color=red] (1,0) -- (3,1);
	    \draw[thick, color=red] (2,0) -- (3,0.5);
		\draw[thick, color=red] (5,0.5) -- (6,1);
	    \draw[thick, color=red] (5,0) -- (7,1);
	    \draw[thick, color=red] (5,-0.5) -- (7,0.5);
	    \draw[thick, color=red] (5,-1) -- (7,0);
	    \draw[thick, color=red] (5,-1.5) -- (7,-0.5);
	    \draw[thick, color=red] (5,-2) -- (7,-1);
	    \draw[thick, color=red] (5,-2.5) -- (7,-1.5);
	    \draw[thick, color=red] (5,-3) -- (7,-2);
	    \draw[thick, color=red] (5,-3.5) -- (7,-2.5);
	    \draw[thick, color=red] (5,-4) -- (7,-3);
	    \draw[thick, color=red] (5,-4.5) -- (7,-3.5);
	    \draw[thick, color=red] (5,-5) -- (7,-4);
	    \draw[thick, color=red] (5,-5.5) -- (7,-4.5);
	    \draw[thick, color=red] (5,-6) -- (7,-5);
	    \draw[thick, color=red] (5,-6.5) -- (7,-5.5);
	    \draw[thick, color=red] (5,-7) -- (7,-6);
	    \draw[thick, color=red] (6,-7) -- (7,-6.5);
		\end{tikzpicture}
		\caption{}\label{fig:cylinder6n+22(2,1)}
	\end{subfigure}%
	\hfill
	\begin{subfigure}[b]{0.3\textwidth}
		\centering
		%tikzpicture 2
		% Second cylinder decomposition (1,-2)
        \begin{tikzpicture}[scale=0.45]
	    \draw (0,0) rectangle (3,1);
	    \draw (5,0) rectangle (7,1);
	    \draw (5,0) rectangle (7,-1);
	    \draw (5,-1) rectangle (7,-2);
    	\draw (5,-2) rectangle (7,-3);
	    \draw (0,1) rectangle (1,2);
	    \draw (0,2) rectangle (1,3);
	    \draw (5,-3) rectangle (7,-4);
        \draw (5,-4) rectangle (7,-5);
        \draw (5,-5) rectangle (7,-6);
        \draw (5,-6) rectangle (7,-7);
    	\draw (1,0) -- (1,1);
    	\draw (2,0) -- (2,1);
    	\draw (3,0) -- (3,1);
        \draw (6,1) -- (6,-7);
        %% Punkte
        \draw [dashed] (3.8,0.5) -- (4.2,0.5);
        \draw [decorate,line width=0.5mm,decoration={brace,mirror}] (1,-0.5) --  (4.8,-0.5) node[pos=0.5,below=10pt,black]{$6n+3$};
        \draw [decorate, decoration={snake}] (3.5,0) -- (3.5,1);
        \draw [decorate, decoration={snake}] (4.5,0) -- (4.5,1);
        \draw (3,0) -- (3.5,0);
        \draw (3,1) -- (3.5,1);
        \draw (4.5,0) -- (5,0);
        \draw (4.5,1) -- (5,1); 
        %% hatching
        \draw[pattern color=blue, pattern = north east lines] (5,1) -- (7,-1) -- (7,0) -- (6,1) -- (5,1) -- cycle;
	    \draw[pattern color=blue, pattern = north east lines] (5,0) --  (7,-2) -- (7,-3) -- (5,-1) -- (5,0) -- cycle;
	    \draw[pattern color=blue, pattern = north east lines] (5,-2) -- (7,-4) -- (7,-5) -- (5,-3) -- (5,-2) -- cycle;
	    \draw[pattern color=blue, pattern = north east lines] (5,-4) -- (7,-6) -- (7,-7) -- (5,-5) -- (5,-4) -- cycle;
        \draw[pattern color=blue, pattern = north east lines] (6,-7) -- (5,-6) -- (5,-7) -- (6,-7) -- cycle;
	    %% cylinder
    	\draw[thick, color=red] (1,0) -- (0,1);
    	\draw[thick, color=red] (1,2) -- (0,3);
    	\draw[thick, color=red] (2,0) -- (0,2);
    	\draw[thick, color=red] (3,0) -- (2,1);
    	\draw[thick, color=red] (7,0) --  (6,1);
    	\draw[thick, color=red] (7,-1) -- (5,1);
    	\draw[thick, color=red] (7,-2) -- (5,0);
    	\draw[thick, color=red] (7,-3) -- (5,-1);
    	\draw[thick, color=red] (7,-4) -- (5,-2);
    	\draw[thick, color=red] (7,-5) -- (5,-3);
    	\draw[thick, color=red] (7,-6) -- (5,-4);
    	\draw[thick, color=red] (7,-7) -- (5,-5);
    	\draw[thick, color=red] (6,-7) -- (5,-6);
		\end{tikzpicture}
		\caption{}\label{fig:cylinder6n+22(-2,1)}
	\end{subfigure}%
	\hfill
	\begin{subfigure}[b]{0.3\textwidth}
		\centering
		%tikzpicture 3
		% Thir cylinder decomposition in direction (1,1)
		\begin{tikzpicture}[scale=0.45]
        %% squares
	    \draw (0,0) rectangle (3,1);
	    \draw (5,0) rectangle (7,1);
	    \draw (5,0) rectangle (7,-1);
	    \draw (5,-1) rectangle (7,-2);
	    \draw (5,-2) rectangle (7,-3);
	    \draw (0,1) rectangle (1,2);
	    \draw (0,2) rectangle (1,3);
	    \draw (5,-3) rectangle (7,-4);
        \draw (5,-4) rectangle (7,-5);
        \draw (5,-5) rectangle (7,-6);
        \draw (5,-6) rectangle (7,-7);
    	\draw (1,0) -- (1,1);
    	\draw (2,0) -- (2,1);
    	\draw (3,0) -- (3,1);
        \draw (6,1) -- (6,-7);
        %% Punkte
        \draw [dashed] (3.8,0.5) -- (4.2,0.5);
        \draw [decorate,line width=0.5mm,decoration={brace,mirror}] (1,-0.5) --  (4.8,-0.5) node[pos=0.5,below=10pt,black]{$6n+3$};
        \draw [decorate, decoration={snake}] (3.5,0) -- (3.5,1);
        \draw [decorate, decoration={snake}] (4.5,0) -- (4.5,1);
        \draw (3,0) -- (3.5,0);
        \draw (3,1) -- (3.5,1);
        \draw (4.5,0) -- (5,0);
        \draw (4.5,1) -- (5,1); 
        %% hatching
        \draw[pattern color=blue, pattern = north east lines] (5,0) -- (6,1) -- (7,1) -- (5,-1) -- (5,0) -- cycle;
	    \draw[pattern color=blue, pattern = north east lines] (5,-2) -- (7,0) -- (7,-1) -- (5,-3) -- (5,-2) -- cycle;
	    \draw[pattern color=blue, pattern = north east lines] (5,-4) -- (7,-2) -- (7,-3) -- (5,-5) -- (5,-4) -- cycle;
	    \draw[pattern color=blue, pattern = north east lines] (5,-6) -- (7,-4) -- (7,-5) -- (5,-7) -- (5,-6) -- cycle;
        \draw[pattern color=blue, pattern = north east lines] (6,-7) -- (7,-7) -- (7,-6) -- (6,-7) -- cycle;
	    %% cylinder
	    \draw[thick, color=red] (0,0) -- (1,1);
	    \draw[thick, color=red] (0,1) -- (1,2);
	    \draw[thick, color=red] (0,2) -- (1,3);
	    \draw[thick, color=red] (1,0) -- (2,1);
	    \draw[thick, color=red] (2,0) -- (3,1);
	    \draw[thick, color=red] (5,0) -- (6,1);
	    \draw[thick, color=red] (5,-1) -- (7,1);
	    \draw[thick, color=red] (5,-2) -- (7,0);
	    \draw[thick, color=red] (5,-3) -- (7,-1);
	    \draw[thick, color=red] (5,-4) -- (7,-2);
	    \draw[thick, color=red] (5,-5) -- (7,-3);
    	\draw[thick, color=red] (5,-6) -- (7,-4);
	    \draw[thick, color=red] (5,-7) -- (7,-5);
	    \draw[thick, color=red] (6,-7) -- (7,-6);
        \end{tikzpicture}
		\caption{}\label{fig:cylinder6n+22(1,1)}
	\end{subfigure}
	\caption{Cylinder decomposition in direction \((2,1)\), $(-1,1)$ and $(1,1)$ of origamis in \(\mathcal{F}_6\). Here $\alpha_2$, $\beta_2$ and $\gamma_2$ are the core curve of the white cylinders in the respective decomposition.}
\end{figure}

%\clearpage

\subsection{Arithmeticity for the family \boldmath{$\mathcal{F}_7$}}\label{ss.H4odd-family7}%{sec:ArithmOfF7}

We now consider the origamis $\mathcal{O}'_{K,N}$ with $K = 3n$ and $N = 5n$ in the family $\mathcal{F}_7$, see Figure~\ref{KNOrigami}. In this case  - different as for the families $\mathcal{F}_1$,\ldots, $\mathcal{F}_7$ - we obtain the basis $B$ of $H_1^{(0)}(\mathcal{O}'_{K,N},\mathbb{Q})$ from Equation (\ref{BasisEforF7}). Also the transformation matrix $BC_1$ and the fundamental matrix $\tilde{G}$ differ in this case. 
Here we obtain for $B$:
\begin{equation}
  \begin{split}
    \begin{IEEEeqnarraybox}[][c]{rClCrCl}\label{BasisEForFamily7}
      e_1 &=& \sigma_2-2\sigma_1,&\qquad & e_2 &=& \sigma_3-5n\sigma_1\\
      e_3 &=& \zeta_2-2\zeta_1,  &      &  e_4 &=& \zeta_3-3n\zeta_1    
    \end{IEEEeqnarraybox}  
  \end{split}
\end{equation}
Then the corresponding transformation matrix (cf. \ref{BaseChange}) is:
  \[
  BC_1 = \begin{pmatrix}-2& -5n & 0 &0\\1 &0&0&0\\ 0&1&0&0\\0&0&-2&-3n\\ 0&0&1&0\\ 0&0&0&1\end{pmatrix}
  \]
The fundamental matrix of the intersection form in terms of the basis $\sigma_i,\xi_i$ is:
\begin{equation}
    \tilde{G} = 
    \begin{pmatrix}
     0 & 0 & 0 & 0 & 0 & 1\\
      0 & 0 & 0 & 0 & 1 & 1\\
      0 & 0 & 0 & 1 & 1 & 1\\
      0 & 0 &  -1 & 0 & 0 & 0\\
      0  & -1 &  -1 & 0 & 0 & 0 \\
      -1  & -1 & -1 & 0 & 0 & 0\\
    \end{pmatrix}
\end{equation}
Now, the tables for this family are:    
%The data for obtaining the word is the following: 
%(1,1)=\theta1, (1,-1)=\theta_2 (1,3)=\theta_
\begin{table}[!htbp]
\centering
\begin{subtable}{0.6\textwidth}
\centering
\begin{tabular}{|c|c |c| c| c| c| c|} 
 \hline
 \multirow{2}{*}{$\Omega$} & \multicolumn{2}{c|} {$\theta_1\parallel(1,1)$} & \multicolumn{2}{c|}{$\theta_2\parallel(1,-1)$} &  \multicolumn{2}{c|}{$\theta_3\parallel(1,3)$} \\
 \cline{2-7}
 &$\alpha_1$ & $\alpha_2$ & $\beta_1$ & $\beta_2$ & $\gamma_1$ & $\gamma_2$ \\ [0.5ex] 
 \hline
$\sigma_1$ & 1  & 0   &  0 & 1 & 1 & 2    \\ 
$\sigma_2$ & 1  & 1   &  1 & 1 & 1 & 5   \\
$\sigma_3$ & 1  & $5n-1$ &  1 &$5n-1$& 1 &$5n-1$    \\
$\zeta_1$ &0  & -1  & 0 & 1 & 0 & -1   \\
$\zeta_2$ &-1 & -1  & 1 & 1 & 0 &  -2    \\ 
$\zeta_3$ &$-(3n-1)$&-1& 1 & $3n-1$ & $-n$ & $-2n$
  \\ [1ex] 
 \hline
\end{tabular}
\caption{}
\end{subtable}
\hfill
\begin{subtable}{0.38\textwidth}
\centering
\begin{tabular}{|c|c |c| c|} 
 \hline
 &$\theta_1$ & $\theta_2$ & $\theta_3$ \\ [0.5ex] 
 \hline
$c_1$ & 1 & 1 & 1  \\ 
$c_2$ & 1 & 1 & 1 \\
$f_1$ & $3n$ &  2 & $n$\\
$f_2$ & $5n$ &  $8n-2$&  $7n$ \\
$k_1$ &  $5n$ & $8n-2$  & $7n$\\
$k_2$ &  $3n$ & 2 & $n$\\
 \hline
\end{tabular}
\caption{}
\end{subtable}
\caption{Intersections forms between the waist curves of the horizontal and vertical cylinders and data for multitwist for the family $\mathcal{F}_7$.}
\label{multitwistdataF7}
\end{table}
\begin{table}[!htbp]
\centering
\begin{tabular}{|c|c |c| c| c|} 
 \hline
 &$e_1$ & $e_2$ & $e_3$ & $e_4$ \\ [0.5ex] 
 \hline
 $X_1$ & $-5n$ & $3n$ & $3n$ & $-5n$ \\ 
 $X_2$ & $8n-2$ & -2 & $-8n+2$ & 2  \\
 $X_3$ & $n$ &  $n$ & $3n$ & $-5n$ \\
 $D_{1}$ & 1 & $5n-1$ & 1 & $3n-1$\\ 
 $D_{2}$ & -1 & -1 & -1 & -1 \\
 $D_{3}$ &  1  & $5n-1$ & 0 & $n$ \\ [1ex] 
 \hline
\end{tabular}
\caption{Coefficients for the basis in $H_1^{(0)}(M,\mathbb{Q})$ and coefficients for the action of the associated transvection on the elements of the basis for the family $\mathcal{F}_7$.}
\label{coefbaseandActionF7}
\end{table}
\begin{table}[!htbp]
\centering
\begin{tabular}{|c|c |c|} 
 \hline
 &$X_1$ & $X_2$  \\ [0.5ex] 
 \hline
 $D_{1}$ & $X_1$ & $X_2-4nX_1$  \\ 
 $D_{2}$ & $X_1+4nX_2$ & $X_2$     \\
 $D_{3}$ & $X_1+n(5n-4)^2X_2-n(5n-4)e$ &  $X_2$ \\
 \hline
 Annihilator & \multicolumn{2}{c|}{$e=(5n-4)X_2-2X_3$}\\
 \hline
 Word & \multicolumn{2}{c|}{$A_{3}^{4}A_{2}^{-(5n-4)^2}$ } \\ [1ex] 
 \hline
\end{tabular}
\caption{Action of each transvection on the ordered basis $X_1,~X_3,~e$ and nontrivial element $\Sp(W)$ for the family $\mathcal{F}_7$.}
\label{MatricesAndWordF7}
\end{table}

\begin{figure}[!htbp]
	\centering
	\noindent\begin{subfigure}[b]{0.3\textwidth}
		\centering
		\begin{tikzpicture}[scale=0.6]
			\draw (0,0) rectangle (1,1);
			\draw (1,1) rectangle (2,2);
			\draw (0,1) rectangle (1,2);
			\draw (0,2) rectangle (1,4);
			\draw (0,4) rectangle (1,5);
			\draw (1,0) rectangle (2,1);
			\draw (2,0) rectangle (4,1);
			\draw (4,0) rectangle (5,1);

			\draw[pattern color=blue, pattern = north east lines] (0,0) -- (1,0) -- (2,1) -- (2,2) -- (0,0) -- cycle;
			\draw[pattern color=blue, pattern = north east lines] (0,1) -- (1,2) -- (0,2) -- (0,1) -- cycle;
     		\draw[pattern color=blue, pattern = north east lines] (0,4) -- (1,5) -- (0,5) -- (0,4) -- cycle;
    		\draw[pattern color=blue, pattern = north east lines] (0,4) -- (1,4) -- (1,5) -- (0,4) -- cycle;
     
    		\draw[thick, color=red] (0,4) -- (1,5);
    		\draw[thick, color=red] (0,1) -- (1,2);
			\draw[thick, color=red] (0,0) -- (2,2);
			\draw[thick, color=red] (1,0) -- (2,1);
			\draw[thick, color=red] (4,0) -- (5,1);
      		
      		\draw[dashed] (2.5,0.5) -- (3.5,0.5);
			\draw[dashed] (0.5,2.5) -- (0.5,3.5);
		\end{tikzpicture}
		\caption{}\label{fig:F7Dir11}
	\end{subfigure}%
	\hfill
	\begin{subfigure}[b]{0.3\textwidth}
		\centering
		\begin{tikzpicture}[scale=0.6]
			\draw (0,0) rectangle (1,1);
			\draw (1,1) rectangle (2,2);
			\draw (0,1) rectangle (1,2);
			\draw (0,2) rectangle (1,4);
			\draw (0,4) rectangle (1,5);
			\draw (1,0) rectangle (2,1);
			\draw (2,0) rectangle (4,1);
			\draw (4,0) rectangle (5,1);

			\draw[pattern color=blue, pattern = north east lines] (1,0) -- (0,1) -- (0,2) -- (2,0) -- (1,0) -- cycle;
			\draw[pattern color=blue, pattern = north east lines] (2,1) -- (1,2) -- (2,2) -- (2,1) -- cycle;
     
   			\draw[thick, color=red] (1,0) -- (0,1);
    		\draw[thick, color=red] (2,0) -- (0,2);
			\draw[thick, color=red] (2,1) -- (1,2);
			\draw[thick, color=red] (1,4) -- (0,5);
			\draw[thick, color=red] (5,0) -- (4,1);

			\draw[dashed] (2.5,0.5) -- (3.5,0.5);
			\draw[dashed] (0.5,2.5) -- (0.5,3.5);
		\end{tikzpicture}
		\caption{}\label{fig:F7Dir1-1}
	\end{subfigure}%
	\hfill
	\begin{subfigure}[b]{0.3\textwidth}
		\centering
		\begin{tikzpicture}[scale=0.6]
			\draw (0,0) rectangle (1,1);
			\draw (1,1) rectangle (2,2);
			\draw (0,1) rectangle (1,2);
			\draw (0,2) rectangle (1,4);
			\draw (0,4) rectangle (1,5);
			\draw (1,0) rectangle (2,1);
			\draw (2,0) rectangle (4,1);
			\draw (4,0) rectangle (5,1);

			\draw[pattern color=blue, pattern = north east lines] (0,0) -- (0.66,2) -- (1,2) -- (0.33,0) -- (0,0) -- cycle;
			\draw[pattern color=blue, pattern = north east lines] (0.66,4) -- (1,4) -- (1,5) -- (0.66,4) -- cycle;
     		\draw[pattern color=blue, pattern = north east lines] (0,4) -- (0.33,5) -- (0,5) -- (0,4) -- cycle;
   
   			\draw[thick, color=red] (0,4) -- (0.33,5);
    		\draw[thick, color=red] (0.33,4) -- (0.66,5);
			\draw[thick, color=red] (0.66,4) -- (1,5);
			\draw[thick, color=red] (0,1) -- (0.33,2);
    		\draw[thick, color=red] (0,0) -- (0.66,2);
			\draw[thick, color=red] (0.33,0) -- (1,2);
	  		\draw[thick, color=red] (0.66,0) -- (1.33,2);
    		\draw[thick, color=red] (1,0) -- (1.66,2);
			\draw[thick, color=red] (1.33,0) -- (2,2);
			\draw[thick, color=red] (1.66,0) -- (2,1);
	  
   		 	\draw[thick, color=red] (4,0) -- (4.33,1);
    		\draw[thick, color=red] (4.33,0) -- (4.66,1);
			\draw[thick, color=red] (4.66,0) -- (5,1);
	
			\draw[dashed] (2.5,0.5) -- (3.5,0.5);
			\draw[dashed] (0.5,2.5) -- (0.5,3.5);
		\end{tikzpicture}
		\caption{}\label{fig:F7Dir13}
	\end{subfigure}
	\caption{In the three figures \ref{fig:F7Dir11}, \ref{fig:F7Dir1-1} and \ref{fig:F7Dir13} we can see the cylinder decompositions in direction $(1,1)$, $(1,-1)$ and $(1,3)$ for the family $\mathcal{F}_7$, where $\alpha_2$, $\beta_2$ and $\gamma_2$ are the core curves of the white cylinders of the respective decomposition.}
\end{figure}

%\clearpage

%%%%%%%%%%%%%%%%%%%%%%%%%%%%%%%%%%%%%%%%%%%%%%%%%%%%%%%%%%%%%%%%%%%%%%%%%%%%%
%%%%%%%%%%%%%%%%% Section 4    %%%%%%%%%%%%%%%%%%%%%%%%%%%%%%%%%%%%%%%%%%%%%%
%%%%%%%%%%%%%%%%%%%%%%%%%%%%%%%%%%%%%%%%%%%%%%%%%%%%%%%%%%%%%%%%%%%%%%%%%%%%%

\section{Arithmeticity of Kontsevich--Zorich monodromies in $\mathcal{H}^{hyp}(4)$}\label{s.H4hyp}

%\subsection{One-parameter family of origamis in $\mathcal{H}^{hyp}(4)$}

In this section we find a one-parameter family of origamis in $\mathcal{H}^{hyp}(4)$ whose Kontsevich-Zorich monodromy is arithmetic.

Let $\mathcal{P}_{N}$ be the origami associated to the following pair of permutations (cf. Figure~\ref{FigureF8}):
$$h=(1)(2, \dots, N+1)\quad \textrm{and} \quad v=(1,3,2)(4,5,\dots,N+1).$$

\begin{figure}[!htbp]
\centering
\begin{tikzpicture}
  \draw (0.5,1.5) node {1};
  \draw (0.5,.5)  node {2};
  \draw (1.5,.5)  node {3};
  \draw (2.5,.5)  node {4};
  \draw (3.5,.5)  node {5};
  \draw (6.5,.5)  node {N+1};
 \draw (0,0) rectangle (1,1);
\draw (0,1) rectangle (1,2);
\draw (1,0) rectangle (2,1);

\draw (2,0) rectangle (3,1);
\draw (3,0) rectangle (4,1);
\draw (4,0) rectangle (6,1);
\draw (6,0) rectangle (7,1);

\draw (0.5,2) node {$\backslash$};
\draw (1.5,0) node {$\backslash$};

\draw (0.5,0) node {$\backslash \backslash$};
\draw (1.5,1) node {$\backslash \backslash$};

\draw (2.5,1) node {$\times$};
\draw (3.5,0) node {$\times$};

\draw (2.5,0) node {$\wr$};
\draw (6.5,1) node {$\wr$};

\draw [decorate,line width=0.5mm,
    decoration = {brace,mirror}] (0,-0.5) --  (7,-0.5) node[pos=0.5,below=10pt,black]{$N$};

\draw [dashed] (4.5,0.5) -- (5.5,0.5);

%\draw (5,4) node[pos=0,below=0pt,black]{$\dots$};

\end{tikzpicture}
\caption{The origami $\mathcal{P}_{N}$} \label{FigureF8}
\end{figure}

Note that $\mathcal{P}_{N}\in\mathcal{H}^{hyp}(4)$ has HLK invariant $(1,[2,2,2])$ or $(0,[3,3,1])$ depending if $N$ is odd or even. Furthermore, observe that $\mathcal{P}_N$ decomposes into two cylinders in the horizontal and in the vertical direction, rather than in three cylinders as in Section~\ref{s.H4odd}. However,  $(-1,1)$ always is a 3-cylinder direction. Moreover, if $N = 2n$ is even then $(1,1)$ is a further 3-cylinder direction. Thus, as before, it follows from  \cite{LN} that their waist curves generate the homology of $\mathcal{P}_{N}\in\mathcal{H}^{hyp}(4)$. We denote the waist curves in direction $(1,1)$ by $\sigma_1$, $\sigma_2$, $\sigma_3$ such that their combinatorial lengths are $1$, $n+1$ and $n-1$, respectively. We denote the waist curves in direction $(-1,1)$ by $\zeta_1$, $\zeta_2$ and $\zeta_3$ such that their combinatorial lengths are 1, 2 and 2, respectively, and $\zeta_2$ runs through the square labeled with 1. Then, we obtain the basis
\begin{equation}\label{TildeBForF8}
  \tilde{B} = \{\sigma_1,\sigma_2,\sigma_3,\zeta_1,\zeta_2,\zeta_3\} 
\end{equation}
of $H_1(\mathcal{P}_{N},\QQ)$.
%Also, $\mathcal{P}_N$ always decomposes into three cylinders in the direction $(-1,1)$ whose waist curves \textcolor{purple}{$\zeta_1$, $\zeta_2$ and $\zeta_3$} \commref{It is confusing it use these notations here, they where for horizontal and vertical cylinders before !} (resp.) have holonomy vectors $(-1,1)$, $2(-1,1)$, $2(-1,1)$ (resp.) and intersect the squares labeled $1$, $2$ and $4$ (resp.). Furthermore, when $N=2n$ is even, $\mathcal{P}_N$ decomposes into three cylinders in the direction $(1,1)$ whose waist curves \textcolor{purple}{$\sigma_1$, $\sigma_2$ and $\sigma_3$} have holonomy vectors $(1,1)$, $(n+1)(1,1)$, $(n-1)(1,1)$.
We further choose as base of $H_1^{(0)}(\mathcal{P}_{N},\mathbb{Q})$ the basis $B = \{e_1,\ldots, e_4\}$ with
\begin{equation}\label{BForF8}
  e_1 = \sigma_2-(n+1)\sigma_1, \quad e_2 = \sigma_3-(n-1)\sigma_1, \quad e_3 = \quad \zeta_2-2\zeta_1, \quad e_4 = \zeta_3-2\zeta_1.
\end{equation}

\begin{proposition}\label{ZdenisityInHyp}
  For $n$ large enough, the Kontsevich-Zorich monodromy of $\mathcal{P}_{2n}$ is Zariski dense in $\SP(H_1^{(0)}(\mathcal{P}_{2n},\mathbb{R}))$.
\end{proposition}
%\subsubsection{Zariski denseness of the Kontsevich--Zorich monodromy of $\mathcal{P}_{2n}$}
%The inverse moduli of the three cylinders in direction $(1,1)$ are $\mu_1^{-1} = 1$, $\mu_2^{-1} = n+1$ and $\mu_3^{-1} = n-1$; those in direction $(-1,1)$ are $\mu'_1$
%By inspecting the intersections between the cycles $\sigma_i$ and $\zeta_j$
\begin{proof}
We use again Lemma~\ref{lemma:galois-pinching-criterion} and assume the notations from there. First, we compute the action of the horizontal and the vertical Dehn multitwist on $H_1^{(0)}(\mathcal{P}_{2n},\mathbb{Q})$.
Observe that the intersection numbers $\Omega(\sigma_i,\zeta_i)$ with $i,j\in\{1,2,3\}$ are:
\[
\begin{array}{|c|c|c|c|}
  \hline
  \Omega(\sigma_i,\zeta_j)&\zeta_1&\zeta_2&\zeta_3\\
  \hline
  \sigma_1 &0&1&1\\
  \sigma_2 &1&3&2\\
  \sigma_3 &1&0&1\\
  \hline
\end{array}
\]
The inverse moduli of the three cylinders in direction $(1,1)$ are $\mu_1^{-1} = 1$, $\mu_2^{-1} = n+1$ and $\mu_3^{-1} = n-1$. Hence the Dehn multitwist $\Tdiag$ in directions $(1,1)$ is of strength $L = (n-1)(n+1)$ with multiplicities $k_1 = n^2-1$, $k_2 = n-1$ and $k_3 = n+1$. Hence $\Tdiag$ 
acts by (\ref{eq:action-multitwist-homology}) via:
\begin{align*}
  \Tdiag \colon &  \zeta_1 \mapsto \zeta_1 - (n-1)\sigma_2 - (n+1)\sigma_3,\\
                &  \zeta_2 \mapsto \zeta_2 - (n^2-1)\sigma_1 - 3(n-1)\sigma_2,\\
                &  \zeta_3 \mapsto \zeta_3 - (n^2-1)\sigma_1 - 2(n-1)\sigma_2 - (n+1)\sigma_3
\end{align*}
For direction $(-1,1)$ we obtain the inverse moduli $\mu_1^{-1} = 1/(N-3)$, $\mu_2^{-1} = 2$ and $\mu_3^{-1} = 2$.  Hence the corresponding Dehn multitwist $\Tadiag$
is of strength $2$ with multiplicities $k_1 = 2(N-3)$, $k_2 = 1$ and $k_3 = 1$. $\Tadiag$ 
acts by (\ref{eq:action-multitwist-homology}) via:
\begin{align*}
  \Tadiag \colon &  \sigma_1 \mapsto \sigma_1 + \zeta_2 + \zeta_3,\\
                 &  \sigma_2 \mapsto \sigma_2 + 2(N-3)\zeta_1 + 3\zeta_2 + 2 \zeta_3,\\
                 &  \sigma_3 \mapsto \sigma_3 + 2(N-3)\zeta_1 + \zeta_3
\end{align*}

%From this we see that the Dehn multitwists in directions $(1,1)$ and $(-1,1)$ are of with strengths $(n-1)(n+1)$ and $2$ in the basis $\{\sigma_1,\sigma_2, \sigma_3,\zeta_1,\zeta_2,\zeta_3\}$ are $\widetilde{A}(\sigma_i)=\sigma_i$, $\widetilde{B}(\zeta_i)=\zeta_i$ for $1\leq i\leq 3$, and
%$$\widetilde{A}(\zeta_1) = \zeta_1+(n-1)\sigma_2+(n+1)\sigma_3, \quad \widetilde{A}(\zeta_2) = \zeta_2 + (n^2-1)\sigma_1+3(n-1)\sigma_2,$$
%$$\widetilde{A}(\zeta_3) = \zeta_3 + (n^2-1)\sigma_1 + 2(n-1)\sigma_2 + (n+1)\sigma_3, \quad \widetilde{B}(\sigma_1) = \sigma_1+\zeta_2+\zeta_3,$$
%$$\widetilde{B}(\sigma_2) = \sigma_2+2\zeta_1+3\zeta_2+2\zeta_3, \quad \widetilde{B}(\sigma_3) = \sigma_3+2\zeta_1+\zeta_3.$$

Thus, the matrices $A$ and $B$ of the restrictions of these actions to $H_1^{(0)}(\mathcal{P}_{2n},\mathbb{Q})$ with respect to the basis $B = \{b_1, \ldots, b_4\}$ are
$$A=\left(\begin{array}{cccc}1&0&1-n&0\\0&1&2(n+1)&n+1\\0&0&1&0\\0&0&0&1\end{array}\right), \quad B=\left(\begin{array}{cccc}1&0&0&0\\0&1&0&0\\ -n+2&-n+1&1&0 \\ -n+1 & -n+2 & 0 & 1\end{array}\right).$$

The characteristic polynomial of $A^{-1}B$ is $x^4+ax^3+bx^2+ax+1$ with $a=-2(n^2+n-1)$ and $b=2n^3+n^2+2n-3$. Hence, $t=-a-4$ and $d=b+2a+2$ are positive for all $n$ sufficiently large. Furthermore, 
\begin{align*}
  \Delta_1 &= 4(n^4-2n^2-4n+6) \quad \text{and} \\
  \Delta_2 &=  (2n^3 + 5n^2 + 6n - 5)(2n - 3)(n + 1)(n - 1),
\end{align*}
are factorizations of the polynomial into irreducible factors in $\Z[X]$ computed with Mathematica. It follows that the square-free parts of the polynomials $\Delta_1(n)$, $\Delta_2(n)$ and $\Delta_1(n)\cdot\Delta_2(n)$ are of degree greater or equal to three, hence by \cite[Prop.~6.18]{MMY} $\Delta_1$, $\Delta_2$ and $\Delta_3$ are squares at most for finitely many $n$. Therefore,   by Lemma~\ref{lemma:galois-pinching-criterion}, $A^{-1}B$ is Galois-pinching for  $n$ sufficiently large. Since $\mathcal{P}_{2n}$ decomposes into two cylinders for example in directions $(1,0)$, Proposition~\ref{prop:density} implies the claim.
\end{proof}

%Now, we obtain the desired statement.

\begin{theorem}\label{t.H4hyp-family1} The Kontsevich--Zorich monodromy of $\mathcal{P}_{2n}$ is arithmetic for all but finitely many choices of  $n\in\mathbb{N}$. 
\end{theorem}
\begin{proof}
  By Proposition~\ref{ZdenisityInHyp} we may assume -- by choosing $n$ large enough -- that the  Kontsevich-Zorich monodromy of $\mathcal{P}_{2n}$ is Zariski dense in $\SP(H_1^{(0)}(\mathcal{P}_{2n},\mathbb{R}))$. We apply again the strategy described in  Section~\ref{TheSteps} with some modifications. We use the notations from \ref{TheSteps}.\\[2mm]  
%  \begin{itemize}
\textbf{Step (A).}
We choose the bases $\tilde{B} = \{\sigma_1,\ldots,\zeta_3\}$ of $H_1(\mathcal{P}_{N},\Q)$  and  $B = \{e_1,\ldots, e_4\}$ of $H_1^{(0)}(\mathcal{P}_{N},\Q)$ from  (\ref{TildeBForF8}) and (\ref{BForF8}).\\[2mm]
\textbf{Step (B).}
 We choose the two-cylinder directions  $\theta_1 = (1,0)$, $\theta_2 = (0,1)$ and $\theta_3 = (2N,1)$.\\[2mm]
\textbf{Step (C).}
  For the three directions $\theta_1$, $\theta_2$, $\theta_3$ we obtain the characteristic numbers shown in Table~\ref{NumbersF8} on the left side
%  \begin{table}[H]
%  \begin{tabular}{|c|c |c| c|} 
% \hline
% &$\theta_1$ & $\theta_2$ & $\theta_3$ \\ [0.5ex] 
% \hline
% $c_1$ & 1 & 1 & 1  \\ 
% $c_2$ & 1 & 1 & 1 \\
% $f_1$ & $1$ &  3 & $N-2$\\
% $f_2$ & $N$ &  $N-2$&  $3$ \\
% $k_1$ &  $N$ & $N-2$  & $3$\\
% $k_2$ &  $1$ & 3 & $N-2$\\
% \hline
%\end{tabular}
%\caption{Data for multitwists  in direction $\theta_1$, $\theta_2$, $\theta_3$ for $\mathcal{P}_{N}$}
%\label{multitwistdataF8}
%\end{table}
  and hence as generators $X_i$ of the images of $D_i - \Id$ for $i \in \{1,2,3\}$:
  \begin{equation}\label{XiForF8}
  %\begin{array}{l}
    X_1 = \alpha_2 - N\alpha_1, \quad X_2 = 3\beta_2 - (N-2)\beta_1, \quad X_3 = 3\gamma_1 - (N-2)\gamma_2  
    %\end{array}
  \end{equation}

\begin{table}[!htbp]
\centering
\begin{subtable}{\textwidth}
\centering
\begin{tabular}{|c|c |c| c| c| c| c|} 
 \hline
 \multirow{2}{*}{$\Omega$} & \multicolumn{2}{c|} {$\theta_1\parallel(1,0)$} & \multicolumn{2}{c|}{$\theta_2\parallel(0,1)$} &  \multicolumn{2}{c|}{$\theta_3\parallel(2N,1)$} \\
 \cline{2-7}
 &$\alpha_1$ & $\alpha_2$ & $\beta_1$ & $\beta_2$ & $\gamma_1$ & $\gamma_2$ \\ [0.5ex] 
 \hline
$\alpha_1$ & 0  & 0      &  1 & 0   & 0   & 1    \\ 
$\alpha_2$ & 0  & 0      &  2 & N-2 & N-2 & 2   \\
$\beta_1$  &-1  &-2      &  0 & 0   & $-4N+8$ & $-(2N+8)$    \\
$\beta_2$  &0   & -(N-2) &  0 & 0 & $-(N-2)(2N-4)$ & -4(N-2)   \\ [1ex] 
%$\gamma_1$ &-1 & -1  & 1 & 1 & 0 &  0    \\ 
%$\gamma_2$ &$-(3n-1)$&-1& 1 & $3n-1$ & $-n$ & $-2n$
%  \\ [1ex] 
 \hline
\end{tabular}
\caption{}
\end{subtable}
\hfill
\begin{subtable}{\textwidth}
\centering
\begin{tabular}{|c|c|c|c|} 
 \hline
 &$\theta_1$ & $\theta_2$ & $\theta_3$ \\ [0.5ex] 
 \hline
 $c_1$ & 1 & 1 & 1  \\ 
 $c_2$ & 1 & 1 & 1 \\
 $f_1$ & $1$ &  3 & $N-2$\\
 $f_2$ & $N$ &  $N-2$&  $3$ \\
 $k_1$ &  $N$ & $N-2$  & $3$\\
 $k_2$ &  $1$ & 3 & $N-2$\\
 \hline
      \end{tabular}
\caption{}
\end{subtable}
\caption{Tables for the proof of Theorem~\ref{t.H4hyp-family1}: (A) Data for multitwists in direction $\theta_1$, $\theta_2$, $\theta_3$. (B) Intersection numbers of waist curves.}
\label{NumbersF8}
\end{table}

\textbf{Step (D).}
Since for the origami $\mathcal{P}_{N}$ we nicely see the intersection numbers of the waist curves in direction $\theta_1$, $\theta_2$ and $\theta_3$ %\Omega(\gamma_i,\gamma'_j)$ (with $\gamma_i,\gamma'_j \in \{\alpha_1,\alpha_2,\beta_1,\beta_2,\gamma_1,\gamma_2\}$)  directly
from the image of the origami (cf. Table~\ref{NumbersF8}, right side), we proceed slightly different than in Section~\ref{s.H4odd}.
%    \begin{table}[H]
%\begin{center}
%\begin{tabular}{|c|c |c| c| c| c| c|} 
% \hline
% \multirow{2}{*}{$\Omega$} & \multicolumn{2}{c|} {$\theta_1\parallel(1,0)$} & \multicolumn{2}{c|}{$\theta_2\parallel(0,1)$} &  \multicolumn{2}{c|}{$\theta_3\parallel(2N,1)$} \\
% \cline{2-7}
% &$\alpha_1$ & $\alpha_2$ & $\beta_1$ & $\beta_2$ & $\gamma_1$ & $\gamma_2$ \\ [0.5ex] 
% \hline
%$\alpha_1$ & 0  & 0      &  1 & 0   & 0   & 1    \\ 
%$\alpha_2$ & 0  & 0      &  2 & N-2 & N-2 & 2   \\
%$\beta_1$  &-1  &-2      &  0 & 0   & $-4N+8$ & $-(2N+8)$    \\
%$\beta_2$  &0   & -(N-2) &  0 & 0 & $-(N-2)(2N-4)$ & -4(N-2)   \\ [1ex] 
%%$\gamma_1$ &-1 & -1  & 1 & 1 & 0 &  0    \\ 
%%$\gamma_2$ &$-(3n-1)$&-1& 1 & $3n-1$ & $-n$ & $-2n$
%%  \\ [1ex] 
% \hline
%\end{tabular}
%\caption{Intersection numbers of waist curves in direction $\theta_1$, $\theta_2$, $\theta_3$ for $\mathcal{P}_{N}$}
%\label{IntersectionNumbersF8}
%\end{center}
%    \end{table}
Using Table~\ref{NumbersF8} and (\ref{XiForF8}) we directly obtain that
\begin{align*}
    \Omega(X_1,X_2) &= (N+1)(N-2), \\
    \Omega(X_2,X_3) &= -2(N-2)^2(N+1),\\
    \Omega(X_3,X_1) &= -(N+1)(N-2).  
\end{align*}
By (\ref{eq:action-multitwist-homology}), we further obtain for the corresponding  Dehn multitwists $T_1$, $T_2$, $T_3$:
\begin{IEEEeqnarray*}{llcllcl}
    T_1\colon& \beta_1 &\mapsto& \beta_1 - N\alpha_1 - 2\alpha_2,  &\beta_2 &\mapsto& \beta_2 -(N-2)\alpha_2,\\
    &\gamma_1 &\mapsto& \gamma_1 - (N-2)\alpha_2, &\gamma_2 &\mapsto& \gamma_2 - N\alpha_1 - 2\alpha_2\\
    T_2\colon& \alpha_1 &\mapsto& \alpha_1 + (N-2)\beta_1, &\alpha_2 &\mapsto& \alpha_2  + 2(N-2)\beta_1 + 3(N-2)\beta_2\\
    & \gamma_1 &\mapsto& \gamma_1 + 2(N-2)^2[2\beta_1 + 3\beta_2], &\gamma_2 &\mapsto& \gamma_2 + (N-2)[(2N+8)\beta_1 + 12\beta_2]\\
    T_3\colon& \alpha_1 &\mapsto& \alpha_1 + (N-2)\gamma_2, &\alpha_2 &\mapsto& \alpha_2 + (N-2)[3\gamma_1 + 2 \gamma_2]\\
       & \beta_1 &\mapsto& \beta_1 - (N-2)[12\gamma_1 + (2N+8)\gamma_2], \quad&\beta_2 &\mapsto& \beta_2 - (N-2)^2[6\gamma_1 + 4\gamma_2]
\end{IEEEeqnarray*}
$X_1$, $X_2$ and $X_3$ are linearly independent. Namely, $r_1X_1 + r_2X_2 + r_3X_3 = 0$ implies - by taking the intersection product with $X_1$ and $X_2$ - that $r_3 = -r_2$ and $r_1 = 2(N-2)r_2$ for $N > 2$.The intersection product with $\zeta_1$ in the equation $2(N-2)\,r_2\,X_1 + r_2\,X_2 - r_2\,X_3$ leads to $r_2 = 0$. Recall for this that the class $\zeta_1$ is represented e.g. by the cycle in direction $(-1,1)$ from the lower edge of Square 5 to the upper edge of Square 4 in Figure~\ref{FigureF8} and use Table~\ref{INZetaF8}. \\
 \begin{table}[H]
    %\begin{center}
\[
\begin{array}{|c|cccccc|}
  \hline
        \Omega(\zeta_1,\cdot)& \alpha_1&\alpha_2&\beta_1&\beta_2&\gamma_1&\gamma_2\\
        \hline
        \zeta_1 & 0& -1 & 0 & -1 &2N-4&-4 \\
        \hline
      \end{array}
\]
\caption{Intersection numbers of $\zeta_i$ and waist curves in direction $\theta_i$}
\label{INZetaF8}
\end{table}
 
 \textbf{Step (E).} We obtain $e = 2(N-2)X_1 + X_2 - X_3$.   
     The transformation matrices of the actions $D_1$, $D_2$ and $D_3$ with respect to the basis $\{X_1,X_2,e\}$ of $W = <X_1,X_2,X_3>$ are: 
\begin{align*}
    A_1 &= \begin{pmatrix} 1&-(N-2)&0\\0&1&0\\0&0&1\end{pmatrix}, \quad A_2 =\begin{pmatrix}1&0&0\\N-2 &1&0\\0&0&1\end{pmatrix}\quad \text{and} \\[2mm]
    A_3 &= \begin{pmatrix} 1 + 2(N-2)^2&-4(N-2)^3&0\\N - 2&1 - 2(N-2)^2&0\\ -(N-2) &2(N-2)^2&1\end{pmatrix}
\end{align*}
In particular,
  \[
  A_1^2A_3^{-1}A_1^{-2}A_2 = 
  \begin{pmatrix}
    1&0&0\\0&1&0\\N-2 &0&1 
  \end{pmatrix}
  \]
is the desired element.
% is a non-trivial element of the unipotent radical of $\textrm{Sp}_{\Omega}(W)$.         
\end{proof}

%%.TEX FILES INPUT HERE%%

%%%%%%%%%%%%%%%%%%%%%%%%%%%%%%%%%%%%%%%%%%%%%%%%%%%%%%%%%%%%%%%%%%%%%%%%%%%%%%55
%%%% Section 5
%%%%%%%%%%%%%%%%%%%%%%%%%%%%%%%%%%%%%%%%%%%%%%%%%%%%%%%%%%%%%%%%%%%%%%%%%%%%%%%%%
%%%%%%%%%%%%%%%%%%%%%%%%%%%%%%%%%%%%%%%%%%%%%%%%%%%%%%%%%%%%%%%%%%%%%%%%%%%%%%%%%%%%

%%% Section 5
\section{Arithmeticity of certain Stairs origamis in genus four}\label{s.H6}

Let \(N\geq 4\) and \(M=4+2m\) with \(m\geq 0\). In this section
we consider the origami \(\mathcal{O}_{N,M}\) of degree $N+M+2$ associated to the pair of permutations \(h,v \in\mbox{Sym}(\{1,\dots,N+M+2\})\), where

\begin{align*}
h=&(1,2,3\dots,N)(N+1,N+2,N+3)(N+4,N+5)(N+6)\dots(N+M+2)\\
v=&(1,\ N+1,\ N+4,\ N+6,\dots, N+M+2)(2,\ N+2,\ N+5)(3,\ N+3)(4)\dots(N).
\end{align*}

\begin{figure}[htbp]
\begin{center}
\begin{tikzpicture}[scale=0.6]
%% squares
	\draw (0,0) rectangle (6,1);
	\draw (0,1) rectangle (3,2);
	\draw (0,2) rectangle (2,3);
	\draw (0,3) rectangle (1,4);
	\draw (0,0) rectangle (1,8);
	\draw (1,0) -- (1,3);
	\draw (2,0) -- (2,2);
	\draw (3,0) -- (3,1);
	\draw (5,0) -- (5,1);
	\draw (0,6) -- (1,6);
	\draw (0,7) -- (1,7);
	
%% Punkte
    %\draw [dashed] (3.5,0.5) -- (4.5,0.5);
    %\draw [dashed] (0.5,4.5) -- (0.5,5.5);
    
%% Schleifen N, M
    \draw [decorate,line width=0.5mm,decoration={brace}] (-0.5,0) --  (-0.5,8) node[pos=0.5,left=10pt,black]{$M$};
    \draw [decorate,line width=0.5mm,decoration={brace,mirror}] (0,-1.1) --  (6,-1.1) node[pos=0.5,below=10pt,black]{$N$};

%% waist curves horizontal cylinder
    \draw[thick,color=green] (0,0.3)--(3.4,0.3);
    \draw[dashed,thick,color=green] (3.5,0.3) -- (4.5,0.3);
	\draw[thick,color=green, ->] (4.6,0.3)--(6,0.3);
	\node[color=green, right] (sigma6) at (6,0.3) {\(\sigma_N\)};
	\draw[thick,color=green, ->] (0,1.3)--(3,1.3);
	\node[color=green, right] (sigma2) at (3,1.3) {\(\sigma_3\)};
	\draw[thick,color=green, ->] (0,2.3)--(2,2.3);
	\node[color=green, right] (sigma3) at (2,2.3) {\(\sigma_2\)};
	\draw[thick,color=green, ->] (0,6.3)--(1,6.3);
	\node[color=green, right] (sigma1) at (1,6.3) {\(\sigma_1\)};
	
%% waist curves vertical cylinder
	\draw[thick,color=red] (0.3,0)--(0.3,4.4);
	\draw[dashed,thick,color=red] (0.3,4.5) -- (0.3,5.5);
	\draw[thick, color=red,->] (0.3,5.6) -- (0.3,8);
	\node[color=red, below] (zeta1) at (0.3,0) {\(\zeta_M\)};
	\draw[thick,color=red, ->] (1.3,0)--(1.3,3);
	\node[color=red, below] (zeta2) at (1.3,0) {\(\zeta_3\)};
	\draw[thick,color=red, ->] (2.3,0)--(2.3,2);
	\node[color=red, below] (zeta3) at (2.3,0) {\(\zeta_2\)};
	\draw[thick,color=red, ->] (5.3,0)--(5.3,1);
    \node[color=red,below] (zeta4) at (5.3,0) {\(\zeta_1\)};
\end{tikzpicture}
\end{center}
\caption{Origami \(\mathcal{O}_{N,M}\) with horizontal waist curves \(\sigma_1,\sigma_2,\sigma_3,\sigma_N\) and vertical waist curves \(\zeta_1,\zeta_2,\zeta_3,\zeta_M\)
  %\commref{\sout{draw a brace from top to botom with number M, just as in the brace $2m$.}}
}
\label{fig:OrigamiNM}
\end{figure}

We will show for finitely many of the origamis \(\mathcal{O}_{N,M}\) that their Kontsevich-Zorich monodromy is arithmetic. 

\begin{theorem}\label{KoZoGenus4}
  For all \(N = 3m+4\) and \(M = 2m+4\) with \(m \in  \{0,\ldots, 50\}\)  the Kontsevich-Zorich monodromy $\KoZoMon$ of  \(\mathcal{O}_{N,M}\) is arithmetic.
\end{theorem}

The main difficulty for this sequence of origamis is to prove that the Kontsevich-Zorich monodromy $\KoZoMon$ is Zariski-dense. We use in this case a different approach than we used in Section~\ref{s.H4odd} and \ref{s.H4hyp} for the origamis in genus 3. This approach is described in Section~\ref{ssec:zariski-density-genus4}. Our proof is computer-aided and we need the explicit description of several elements in $\KoZoMon$ which we develop in Section~\ref{favour}. In Section~\ref{ssec:zariski-density-genus4} we prove the density of the Kontsevich-Zorich monodromy of the origamis considered in Proposition~\ref{KoZoGenus4}. Finally, in Section~\ref{ArithmeticInGenus4} we show arithmeticity with similar methods as in Section~\ref{s.H4odd} and Section~\ref{s.H4hyp}.

\subsection{Our favorite Dehn twists}\label{favour} 
In this subsection we use cylinder decompositions of the origami \(\mathcal{O}_{N,M}\) in several directions to construct Dehn multitwists. We will compute their transformation matrices with respect to a basis $B^{(0)}$ of \(H_1^{(0)}(\mathcal{O}_{N,M},\Z)\). This will be the cornerstone for the arguments in Section~\ref{ssec:zariski-density-genus4} and  \ref{ArithmeticInGenus4}.

%as subgroups of \(\mbox{Sp}_\Omega(H_1^{(0)}(\mathcal{O}_{N,M},\mathbb{Q}))\) with respect to the intersection form $\Omega$.- where These twists act on the homology \(H_1(\mathcal{O}_{N,M},\mathbb{Q})\) and preserve the intersection form \(\Omega\). With this procedure we can find elements in $\mbox{Sp}_\Omega(H_1^{(0)}(\mathcal{O}_{N,M},\mathbb{Q}))$, which will be the cornerstone of the arguments in this section.

The waist curves \(\sigma_1,\sigma_2,\sigma_3,\sigma_N\) of the four maximal horizontal cylinders and the four waist curves \(\zeta_1,\zeta_2,\zeta_3,\zeta_M\) of the four maximal vertical cylinders form a basis \(B\) of \(H_1(\mathcal{O}_{N,M},\mathbb{Q})\) (see Figure \ref{fig:OrigamiNM}). We have the following holonomy vectors for these waist curves:
\begin{IEEEeqnarray*}{lCrClCrClCrClCr}
\mbox{hol}(\sigma_1)&=&(1,0), &\quad& \mbox{hol}(\sigma_2)&=&(2,0), &\quad & 
\mbox{hol}(\sigma_3)&=&(3,0), &\quad& \mbox{hol}(\sigma_N)&=&(N,0)\\
\mbox{hol}(\zeta_1) &=&(0,1), &\quad& \mbox{hol}(\zeta_2) &=&(0,2), &\quad&
\mbox{hol}(\zeta_3) &=&(0,3), &\quad& \mbox{hol}(\zeta_M) &=&(0,M).
\end{IEEEeqnarray*}
We have that \(B^{(0)}=\{\Sigma_1,\Sigma_2,\Sigma_N,Z_1,Z_2,Z_M\}\) is a basis of the non-tautological part \(H^{(0)}(\mathcal{O}_{N,M},\mathbb{Q})\), where
\begin{IEEEeqnarray*}{lCrClCrClCr}
\Sigma_1 &=&\sigma_2-2\sigma_1, & \quad & \Sigma_2&=&\sigma_3-3\sigma_1, &\quad& \Sigma_N &=& \sigma_N-N\sigma_1, \\
    Z_1  &=&\zeta_2-2\zeta_1,   &       &      Z_2&=&\zeta_3-3\zeta_1,   &\quad&    Z_M   &=&\zeta_M-M\zeta_1.
\end{IEEEeqnarray*}
That \(B^{(0)}\) is indeed a basis for example  follows  from the fact that the fundamental matrix \(\tilde{G}\) of the intersection form $\Omega$ with respect to \(B^{(0)}\) (see \ref{FMatrixInGenus4}) is regular:
%. Namely, the symplectic intersection-form \(\Omega\) has the following representation-matrix on \(H^{(0)}(\mathcal{O}_{N,M},\mathbb{Q})\) with respect to \(B^{(0)}\):
\begin{displaymath}\label{FMatrixInGenus4}
\tilde{G} =  \left(\begin{array}{cccccc}
%         0&N+M+2& 0& 0& 0& 0& 0& 0 \\
% -N-M-2& 0& 0& 0& 0& 0& 0& 0 \\
                 0& 0&         0& 0& 1& -1 \\
                 0& 0&         0&1&  1&-2 \\
                 0& 0&         0& -1&-2& -N-M+1 \\
                 0& -1&        1& 0& 0& 0 \\
                -1&-1&         2& 0& 0& 0 \\
                 1 &2& N+M-1& 0& 0& 0 
  \end{array}\right)
\end{displaymath}
We will use the  Dehn multitwists along the waist curves of the cylinders in the directions \((1,1)\), \((1,-1)\), \((1,2)\), \((1,-2)\) as well as the horizontal and vertical direction.  In the following we compute their actions on \(H_1^{(0)}(\mathcal{O}_{N,M},\mathbb{Q})\).
%With cylinder decompositions of \(\mathcal{O}_{N,M}\) in directions \((1,1)\), \((1,-1)\), \((1,2)\), \((1,-2)\) as well as the horizontal and vertical direction we construct Dehn multitwists along the waist curves of the cylinders. We compute in the following their actions on \(H_1^{(0)}(\mathcal{O}_{N,M},\mathbb{Q}))\).
%These twists induce actions on \(H_1^{(0)}(\mathcal{O}_{N,M},\mathbb{Q}))\), which we present in the following.

\begin{figure}[htbp]
	\centering
	\noindent\begin{subfigure}[b]{0.5\textwidth}
		\centering
		% Second cylinder decomposition (1,-1)
\begin{tikzpicture}[scale=0.6]
%% squares
	\draw (0,0) rectangle (3,1);
	\draw (5,0) rectangle (6,1);
	\draw (0,1) rectangle (3,2);
	\draw (0,2) rectangle (2,3);
	\draw (0,3) rectangle (1,4);
	\draw (0,0) rectangle (1,4);
	\draw (0,6) rectangle (1,8);
	\draw (1,0) -- (1,3);
	\draw (2,0) -- (2,2);
	\draw (3,0) -- (3,1);
	\draw (5,0) -- (5,1);
	\draw (0,6) -- (1,6);
	\draw (0,7) -- (1,7);
%% Punkte
    \draw [dashed] (3.8,0.5) -- (4.2,0.5);
    \draw [dashed] (0.5,4.8) -- (0.5,5.2);
%% Schleifen N, M
    \draw [decorate,line width=0.5mm,decoration={brace}] (-0.5,4) --  (-0.5,8) node[pos=0.5,left=10pt,black]{$2m$};
    \draw [decorate,line width=0.5mm,decoration={brace,mirror}] (0,-1.1) --  (6,-1.1) node[pos=0.5,below=10pt,black]{$N$};
%% Zigzag
    \draw [decorate, decoration={snake}] (0,4.5) -- (1,4.5);
    \draw [decorate, decoration={snake}] (0,5.5) -- (1,5.5);
    \draw (0,4) -- (0,4.5);
    \draw (1,4) -- (1,4.5);
    \draw (0,5.5) -- (0,6);
    \draw (1,5.5) -- (1,6);
    \draw [decorate, decoration={snake}] (3.5,0) -- (3.5,1);
    \draw [decorate, decoration={snake}] (4.5,0) -- (4.5,1);
    \draw (3,0) -- (3.5,0);
    \draw (3,1) -- (3.5,1);
    \draw (4.5,0) -- (5,0);
    \draw (4.5,1) -- (5,1);
%% hatching
    \draw[pattern color=blue, pattern = north east lines] (1,0) -- (0,1) -- (0,3) -- (3,0) -- (1,0) --
     cycle;
    \draw[pattern color=blue, pattern = north east lines] (1,3) -- (2,3) -- (2,2) -- (1,3) --
     cycle;
    \draw[pattern color=blue, pattern = north east lines] (2,2) -- (3,2) -- (3,1) -- (2,2) --
     cycle;    
%% cylinder
    \draw[thick, color=red] (0,8) -- (1,7);
    \draw[thick, color=red] (0,7) -- (1,6);
    \draw[thick, color=red] (0,4) -- (1,3);
    \draw[thick, color=red] (0,3) -- (3,0);
    \draw[thick, color=red] (0,2) -- (2,0);
    \draw[thick, color=red] (0,1) -- (1,0);
    \draw[thick, color=red] (1,3) -- (3,1);
    \draw[thick, color=red] (6,0) -- (5,1);  
\end{tikzpicture}
		\caption{}\label{fig:cylinder(1,-1)}
	\end{subfigure}%
	\hfill
	\begin{subfigure}[b]{0.5\textwidth}
		\centering
% First cylinder decomposition (1,1)
\begin{tikzpicture}[scale=0.6]
%% squares
	\draw (0,0) rectangle (3,1);
	\draw (5,0) rectangle (6,1);
	\draw (0,1) rectangle (3,2);
	\draw (0,2) rectangle (2,3);
	\draw (0,3) rectangle (1,4);
	\draw (0,0) rectangle (1,4);
	\draw (0,6) rectangle (1,8);
	\draw (1,0) -- (1,3);
	\draw (2,0) -- (2,2);
	\draw (3,0) -- (3,1);
	\draw (5,0) -- (5,1);
	\draw (0,6) -- (1,6);
	\draw (0,7) -- (1,7);
%% Punkte
    \draw [dashed] (3.8,0.5) -- (4.2,0.5);
    \draw [dashed] (0.5,4.8) -- (0.5,5.2);
%% Schleifen N, M
    \draw [decorate,line width=0.5mm,decoration={brace}] (-0.5,4) --  (-0.5,8) node[pos=0.5,left=10pt,black]{$2m$};
    \draw [decorate,line width=0.5mm,decoration={brace,mirror}] (0,-1.1) --  (6,-1.1) node[pos=0.5,below=10pt,black]{$N$};
%% Zigzag
    \draw [decorate, decoration={snake}] (0,4.5) -- (1,4.5);
    \draw [decorate, decoration={snake}] (0,5.5) -- (1,5.5);
    \draw (0,4) -- (0,4.5);
    \draw (1,4) -- (1,4.5);
    \draw (0,5.5) -- (0,6);
    \draw (1,5.5) -- (1,6);    
    \draw [decorate, decoration={snake}] (3.5,0) -- (3.5,1);
    \draw [decorate, decoration={snake}] (4.5,0) -- (4.5,1);
    \draw (3,0) -- (3.5,0);
    \draw (3,1) -- (3.5,1);
    \draw (4.5,0) -- (5,0);
    \draw (4.5,1) -- (5,1);  
%% hatching
    \draw[pattern color=blue, pattern = north east lines] (0,1) -- (2,3) -- (1,3) -- (0,2) -- (0,1) --
     cycle;
    \draw[pattern color=blue, pattern = north east lines] (1,0) -- (2,0) -- (3,1) -- (3,2) -- (1,0) --
     cycle;
%% cylinder
    \draw[thick, color=red] (0,7) -- (1,8);
    \draw[thick, color=red] (0,6) -- (1,7);
    \draw[thick, color=red] (0,3) -- (1,4);
    \draw[thick, color=red] (0,2) -- (1,3);
    \draw[thick, color=red] (0,1) -- (2,3);
    \draw[thick, color=red] (0,0) -- (2,2);
    \draw[thick, color=red] (1,0) -- (3,2);
    \draw[thick, color=red] (2,0) -- (3,1);
    \draw[thick, color=red] (5,0) -- (6,1);
\end{tikzpicture}
		\caption{}\label{fig:cylinder(1,1)}
	\end{subfigure}%
	\caption{Cylinder decomposition in direction \((1,-1)\) and $(1,1)$ of the origami \(\mathcal{O}_{N,M}\). Here $\chi_1$ and $\delta_1$ are the waist curves of the blue cylinders.}
\end{figure}
For the direction \((1,1)\) we have a decomposition into two maximal cylinders of equal height with waist curves \(\delta_1\) of length \(3\) and \(\delta_2\) of length \(N+M-1\) (see Figure \ref{fig:cylinder(1,1)}). In direction \((1,-1)\) we also have  a decomposition into two maximal cylinders (see Figure \ref{fig:cylinder(1,-1)}) again of equal height. We denote the waist curve of length \(5\) by \(\chi_1\) and the waist curve of length \(N+M-3\) by \(\chi_2\). By (\ref{eq:action-multitwist-homology}), the associated Dehn multitwists along the waist curves of these maximal cylinders then act as the transvections \(D_\delta\) and \(D_\chi\) on \(H_1^{(0)}(\mathcal{O}_{N,M},\mathbb{Q})\) given by:
%the following mapping rules:
\begin{align} \label{DehnTwists1}
    D_\delta\colon      & v\longmapsto v +(N+M-1)\ \Omega(v,\delta_1)\,\delta_1+3\  \Omega(v,\delta_2)\,\delta_2\\
    D_\chi   \colon     & v\longmapsto v + (N+M-3)\ \Omega(v,\chi_1)\,\chi_1 + 5\ \Omega(v,\chi_2)\,\chi_2
    \label{DehnTwists2}
  \end{align}
  As in the previous sections we count intersection points of the curves \(\sigma_i,\sigma_N, \zeta_i,\zeta_M\) \((i=1,2,3)\) with the waist curves of the cylinders:
  %\commref{\sout{In the following, make tables like in section 3.6}}

\begin{table}[htbp]
\centering
\begin{tabular}{|c|c|c|c|c|}
\hline
$\Omega$    & $\delta_1$  & $\delta_2$  & $\chi_1$ & $\chi_2$   \\ [0.5ex]
\hline
$\sigma_1$  &  $0$        &  $1$        & $0$      & $-1$       \\ 
$\sigma_2$  &  $1$        &  $1$        & $-1$     & $-1$       \\
$\sigma_3$  &  $1$        &  $2$        & $-2$     & $-1$       \\
$\sigma_N$  &  $1$        &  $N-1$      & $-2$     & $-(N-2)$   \\
$\zeta_1 $  &  $0$        &  $-1$       & $ 0$     & $-1$       \\
$\zeta_2 $  &  $-1$       &  $-1$       & $-1$     & $-1$       \\
$\zeta_3 $  &  $-1$       &  $-2$       & $-2$     & $-1$       \\
$\zeta_M $  &  $-1$       &  $-(M-1)$   & $-2$     & $-(M-2)$   \\
\hline
\end{tabular}
\caption{Number of intersection points between the waist curves $\delta_1,\,,\delta_2$ of the cylinders in direction $(1,1)$ and the waist curves $\chi_1,\,\chi_2$ of the cylinders in direction $(1,-1)$ with the elements \(\sigma_i,\sigma_N, \zeta_i,\zeta_M\) \((i=1,2,3)\).}
\end{table}

As next step we compute the matrix representations \( M_\delta^{(0)}\) and \(M_\chi^{(0)}\) of the transvections \(D_\delta, D_\chi\in\mbox{Sp}_\Omega(H_1^{(0)}(\mathcal{O}_{N,M},\mathbb{Q}))\) with respect to \(B^{(0)}\). The following two elements of the homology will be essential for this:
\begin{align*}
  \Delta  &=(N+M-1)\cdot \delta_1-3\cdot\delta_2 \in H_1^{(0)}(\mathcal{O}_{N,M},\mathbb{Z})\\
  X       &=(N+M-3)\cdot \chi_1-5\cdot \chi_2 \in H_1^{(0)}(\mathcal{O}_{N,M},\mathbb{Z})
\end{align*}
They both have zero holonomy and are thus elements of the non-tautological part  \(H_1^{(0)}(\mathcal{O}_{N,M},\mathbb{Z})\).
The intersection numbers obtained above allow us by a simple but longish computation to determine the coefficients of $\Delta$ and $X$ with respect to the basis \(B\) of $H_1(\mathcal{O}_{N,M},\mathbb{Q})$. We then convert them to coefficients with respect to the basis $B^{(0)}$ of $H_1^{(0)}(\mathcal{O}_{N,M},\mathbb{Q})$ and obtain:
\begin{equation}
    \begin{split}
        \begin{IEEEeqnarraybox}[][c]{rCll}\label{DeltaX}
   \Delta &=&-& 3\,\Sigma_1+(N+M-1)\,\Sigma_2-3\,\Sigma_N\\
          & &-& 3\,Z_1+(N+M-1)\,Z_2-3\,Z_M,\\
        X &=& & (N+M-3)\,\Sigma_1+(N+M-3)\,\Sigma_2-5\,\Sigma_N\\
          & &-& (N+M-3)\,Z_1-(N+M-3)\,Z_2+5\,Z_M
        \end{IEEEeqnarraybox}
   \end{split}
\end{equation}

Now, using (\ref{DehnTwists1}) and (\ref{DehnTwists2}) we obtain:
\begin{equation}
    \begin{split}
        \begin{IEEEeqnarraybox}[][c]{rClCrClCrCl}   \label{DdeltaDchi}
        D_\delta(\Sigma_1)&=&\Sigma_1+\Delta, &\quad&
        D_\delta(\Sigma_2)&=&\Sigma_2+\Delta, &\quad&
        D_\delta(\Sigma_N)&=& \Sigma_N+\Delta,\\
        D_\delta(Z_1)     &=& Z_1-\Delta, &\quad& 
        D_\delta(Z_2)     &=& Z_2-\Delta, &\quad&
        D_\delta(Z_M)     &=& Z_M-\Delta,\\[2mm]
        D_\chi(\Sigma_1)  &=& \Sigma_1-X, &\quad& 
        D_\chi(\Sigma_2)  &=& \Sigma_2-2X,&\quad& 
        D_\chi(\Sigma_N)  &=& \Sigma_N-2X,\\
        D_\chi(Z_1)       &=& Z_1-X,  &\quad&
        D_\chi(Z_2)       &=& Z_2-2X, &\quad& 
        D_\chi(Z_M)       &=& Z_M-2X\\
      \end{IEEEeqnarraybox}
    \end{split}
\end{equation}
This finally leads to the matrix representations $M_\delta^{(0)},~M_\chi^{(0)}\in\mathbb{R}^{6\times 6}$ which we have included in the file AddToArticleBKKMNSVW.g (cf. \cite{KattlerGIT2023}). It can be printed with the functions BigMatrix1 and BigMatrix2.
%\begin{displaymath}
%  \begin{array}{rl}
%  M_\delta^{(0)}=&
%    \begin{footnotesize}
%    \left(\begin{array}{cccccc}
%    -2        & -3     &     -3 &   3    &  3     & 3    \\
%    N+M-1     & N+M    & N+M-1  & -N-M+1 & -N-M+1 & -N-M+1\\
%    -3        & -3     &  -2    &   3    &  3     & 3     \\
%    -3        & -3     &  -3    &   4    &  3     & 3\\
%   N+M-1      & N+M-1  & N+M-1  & -N-M+1 & -N-M+2 & -N-M+1\\
%    -3        & -3     &  -3    &   3    &  3     & 4
%    \end{array}\right)
%    \end{footnotesize},\\
%&  \\
%  M_\chi^{(0)}=& 
%    \begin{footnotesize}
%    \left(\begin{array}{cccccc}
%    -M-N+4 & -2M-2N+6 & -2M-2N+6 & -M-N+3 & -2M-2N+6 & -2M-2N+6\\
%    -M-N+3 & -2M-2N+7 & -2M-2N+6 & -M-N+3 & -2M-2N+6 & -2M-2N+6\\
%     5     &  10      &  11      &  5     &  10      & 10 \\
%     M+N-3 & 2M+2N-6   & 2M+2N-6 &  M+N-2 & 2M+2N-6  & 2M+2N-6\\
%     M+N-3 & 2M+2N-6   & 2M+2N-6 &  M+N-3 & 2M+2N-5  & 2M+2N-6\\
%    -5     &  -10      & -10     &  -5    & -10      & -9
%    \end{array}\right).
%    \end{footnotesize}
%  \end{array}
%\end{displaymath}
%
%
\begin{figure}[htbp]
	\centering
	\noindent\begin{subfigure}[b]{0.5\textwidth}
		\centering
		% First cylinder decomposition (1,2)
\begin{tikzpicture}[scale=0.6]
%% squares
	\draw (0,0) rectangle (3,1);
	\draw (5,0) rectangle (6,1);
	\draw (0,1) rectangle (3,2);
	\draw (0,2) rectangle (2,3);
	\draw (0,3) rectangle (1,4);
	\draw (0,0) rectangle (1,4);
	\draw (0,6) rectangle (1,8);
	\draw (1,0) -- (1,3);
	\draw (2,0) -- (2,2);
	\draw (3,0) -- (3,1);
	\draw (5,0) -- (5,1);
	\draw (0,6) -- (1,6);
	\draw (0,7) -- (1,7);
%% Punkte
    \draw [dashed] (3.8,0.5) -- (4.2,0.5);
    \draw [dashed] (0.5,4.8) -- (0.5,5.2);
%% Schleifen N, M
    \draw [decorate,line width=0.5mm,decoration={brace}] (-0.5,4) --  (-0.5,8) node[pos=0.5,left=10pt,black]{$2m$};
    \draw [decorate,line width=0.5mm,decoration={brace,mirror}] (0,-1.1) --  (6,-1.1) node[pos=0.5,below=10pt,black]{$N$};
%% Zigzag
    \draw [decorate, decoration={snake}] (0,4.5) -- (1,4.5);
    \draw [decorate, decoration={snake}] (0,5.5) -- (1,5.5);
    \draw (0,4) -- (0,4.5);
    \draw (1,4) -- (1,4.5);
    \draw (0,5.5) -- (0,6);
    \draw (1,5.5) -- (1,6);    
    \draw [decorate, decoration={snake}] (3.5,0) -- (3.5,1);
    \draw [decorate, decoration={snake}] (4.5,0) -- (4.5,1);
    \draw (3,0) -- (3.5,0);
    \draw (3,1) -- (3.5,1);
    \draw (4.5,0) -- (5,0);
    \draw (4.5,1) -- (5,1); 
%% hatching
    \draw[pattern color=blue, pattern = north east lines] (0,6) -- (1,8) -- (0.5,8) -- (0,7) -- (0,6) -- cycle;
	\draw[pattern color=blue, pattern = north east lines] (0.5,6) -- (1,6) -- (1,7) -- (0.5,6) -- cycle;	
	\draw[pattern color=blue, pattern = north east lines] (0,2) -- (1,4) -- (0.5,4) -- (0,3) -- (0,2) --cycle;
	\draw[pattern color=blue, pattern = north east lines] (0.5,0) -- (1,0) -- (2,2) -- (2,3) -- (0.5,0) -- cycle;   
%% cylinder
    \draw[thick, color=red] (0,3) -- (0.5,4);
	\draw[thick, color=red] (0,2) -- (1,4);   
    \draw[thick, color=red] (0,7) -- (0.5,8);
    \draw[thick, color=red] (0,6) -- (1,8);
	\draw[thick, color=red] (0.5,6) -- (1,7);
	\draw[thick, color=red] (0.5,6) -- (1,7);
	\draw[thick, color=red] (0,1) -- (1,3);
	\draw[thick, color=red] (0,0) -- (1.5,3);
	\draw[thick, color=red] (0.5,0) -- (2,3);
	\draw[thick, color=red] (1,0) -- (2,2);
	\draw[thick, color=red] (1.5,0) -- (2.5,2);
	\draw[thick, color=red] (2,0) -- (3,2);
	\draw[thick, color=red] (2.5,0) -- (3,1);
	\draw[thick, color=red] (5,0) -- (5.5,1);
	\draw[thick, color=red] (5.5,0) -- (6,1);			
\end{tikzpicture}
		\caption{}\label{fig:cylinder(1,2)}
	\end{subfigure}%
	\hfill
	\begin{subfigure}[b]{0.5\textwidth}
		\centering
		% Second cylinder decomposition (1,-2)
\begin{tikzpicture}[scale=0.6]
%% squares
	\draw (0,0) rectangle (3,1);
	\draw (5,0) rectangle (6,1);
	\draw (0,1) rectangle (3,2);
	\draw (0,2) rectangle (2,3);
	\draw (0,3) rectangle (1,4);
	\draw (0,0) rectangle (1,4);
	\draw (0,6) rectangle (1,8);
	\draw (1,0) -- (1,3);
	\draw (2,0) -- (2,2);
	\draw (3,0) -- (3,1);
	\draw (5,0) -- (5,1);
	\draw (0,6) -- (1,6);
	\draw (0,7) -- (1,7);
%% Punkte
    \draw [dashed] (3.8,0.5) -- (4.2,0.5);
    \draw [dashed] (0.5,4.8) -- (0.5,5.2);
%% Schleifen N, M
    \draw [decorate,line width=0.5mm,decoration={brace}] (-0.5,4) --  (-0.5,8) node[pos=0.5,left=10pt,black]{$2m$};
    \draw [decorate,line width=0.5mm,decoration={brace,mirror}] (0,-1.1) --  (6,-1.1) node[pos=0.5,below=10pt,black]{$N$};
%% Zigzag
    \draw [decorate, decoration={snake}] (0,4.5) -- (1,4.5);
    \draw [decorate, decoration={snake}] (0,5.5) -- (1,5.5);
    \draw (0,4) -- (0,4.5);
    \draw (1,4) -- (1,4.5);
    \draw (0,5.5) -- (0,6);
    \draw (1,5.5) -- (1,6);
    \draw [decorate, decoration={snake}] (3.5,0) -- (3.5,1);
    \draw [decorate, decoration={snake}] (4.5,0) -- (4.5,1);
    \draw (3,0) -- (3.5,0);
    \draw (3,1) -- (3.5,1);
    \draw (4.5,0) -- (5,0);
    \draw (4.5,1) -- (5,1); 
%% hatching
    \draw[pattern color=blue, pattern = north east lines] (0,8) -- (1,6) -- (1,7) -- (0.5,8) -- (0,8) -- cycle;
	\draw[pattern color=blue, pattern = north east lines] (0,7) -- (0.5,6) -- (0,6) -- (0,7) -- cycle;	
	\draw[pattern color=blue, pattern = north east lines] (0,4) -- (2,0) -- (2.5,0) -- (0.5,4) -- (0,4) --cycle;
	\draw[pattern color=blue, pattern = north east lines] (2,2) -- (3,0) -- (3,1) -- (2.5,2) -- (2,2) --cycle;
	\draw[pattern color=blue, pattern = north east lines] (0,0) -- (0.5,0) -- (0,1) -- (0,0) -- cycle;	
	\draw[pattern color=blue, pattern = north east lines] (5,0) -- (6,0) -- (6,1) -- (5,1) -- (5,0)-- cycle;	   
%% cylinder
    \draw[thick, color=red] (0,8) -- (1,6);
	\draw[thick, color=red] (0.5,8) -- (1,7);
	\draw[thick, color=red] (0,7) -- (0.5,6);
	\draw[thick, color=red] (0,1) -- (0.5,0);
	\draw[thick, color=red] (0,2) -- (1,0);
	\draw[thick, color=red] (0,3) -- (1.5,0);
	\draw[thick, color=red] (0,4) -- (2,0);
	\draw[thick, color=red] (0.5,4) -- (2.5,0);
	\draw[thick, color=red] (1.5,3) -- (3,0);
	\draw[thick, color=red] (2.5,2) -- (3,1);
	\draw[thick, color=red] (5,1) -- (5.5,0);
	\draw[thick, color=red] (5.5,1) -- (6,0);	
\end{tikzpicture}
\caption{}\label{fig:cylinder(1,-2)}
	\end{subfigure}
	\caption{Cylinder decomposition in direction $(1,2)$ and \((1,-2)\) of the origami \(\mathcal{O}_{N,M}\). Here $\gamma_1$ and $\alpha_1$ are the waist curves of the blue cylinders.}
\end{figure}
For the directions \((1,2)\) and \((1,-2)\) we assume that $M$ is even. We then again have decompositions into two cylinders of equal height (see Figure \ref{fig:cylinder(1,2)} and Figure \ref{fig:cylinder(1,-2)}). For direction \((1,2)\) the waist curves \(\gamma_1\), \(\gamma_2\) have length \(\frac{M}{2}\) and \(\frac{2N+M+4}{2}\). For direction \((1,-2)\) the waist curves \(\alpha_1\), \(\alpha_2\) have length \(N+m\) and \(m+6\). We get the following mappings \(D_\gamma\) and \(D_\alpha\):
\begin{align*}
    D_\gamma\colon      & v\longmapsto v +(2N+M+4)\ \Omega(v,\gamma_1)\gamma_1+M\  \Omega(v,\gamma_2)\gamma_2\\
    D_\alpha  \colon     & v\longmapsto v + (m+6)\ \Omega(v,\alpha_1)\alpha_1 + (N+m)\ \Omega(v,\alpha_2)\alpha_2
  \end{align*}
For the intersection points of the waist curve \(\gamma_1,\gamma_2\) and \(\alpha_1,\alpha_2\) with \(\sigma_i,\sigma_N, \zeta_i,\zeta_M\) we counted:
\begin{table}[htbp]
\centering
\begin{tabular}{|c|c|c|c|c|}
\hline
$\Omega$    & $\gamma_1$  & $\gamma_2$  & $\alpha_1$ & $\alpha_2$   \\ [0.5ex]
\hline
$\sigma_1$  &  $1$        &  $1$        & $-1$       & $-1$       \\ 
$\sigma_2$  &  $1$        &  $3$        & $-1$       & $-3$       \\
$\sigma_3$  &  $1$        &  $5$        & $-2$       & $-4$       \\
$\sigma_N$  &  $1$        &  $2N-1$     & $4-2N$     & $-4$   \\
$\zeta_1 $  &  $0$        &  $-1$       & $-1$       & $0$       \\
$\zeta_2 $  &  $0$        &  $-2$       & $-1$       & $-1$       \\
$\zeta_3 $  &  $-1$       &  $-2$       & $-1$       & $-2$       \\
$\zeta_M $  &  $-(1+m)$   &  $-(3+m)$   & $-(1+m)$   & $-(3+m)$   \\
\hline
\end{tabular}
\caption{Number of intersection points between the waist curves $\gamma_1,\,\gamma_2$ of the cylinders in direction $(1,2)$ and the waist curves $\alpha_1,\,\alpha_2$ of the cylinders in direction $(1,-2)$ with the elements \(\sigma_i,\sigma_N, \zeta_i,\zeta_M\) \((i=1,2,3)\).}
\end{table}

With these data we compute similar as above the representation matrices \(M_\gamma^{(0)}\) and \(M_\alpha^{(0)}\) of the maps \(D_\gamma,D_\alpha\) with respect to the basis \(B^{(0)}\).
The crucial elements of the non tautological part of homology are in this case:
\begin{align*}
\Gamma&=(2N+M+4)\gamma_1-M\gamma_2 \in H_1^{(0)}(\mathcal{O}_{N,M},\mathbb{Z})\\
     A&=(m+6)  \alpha_1-(N+m)\alpha_2\in H_1^{(0)}(\mathcal{O}_{N,M},\mathbb{Z})
\end{align*}
We get for the coefficients of $\Gamma$ and $A$ in the basis $B^{(0)}=\{\Sigma_1,\Sigma_2,\Sigma_n,Z_1,Z_2,Z_M\}$:
\begin{equation}
    \begin{split}
        \begin{IEEEeqnarraybox}[][c]{rCll}\label{GammaA}
            \Gamma &=&  & (2N+M+4)\,\Sigma_1-M\,\Sigma_2-M\,\Sigma_N\\
                   & &- & 2M\,Z_1-2M\,Z_2+(2N+4)\,Z_M,\\  
                 A &=&- & (N+m)\,\Sigma_1-(N+m)\Sigma_2+(m+6)\Sigma_n \\
                   & &+ & (N-6)Z_1+2(N+m)Z_2+(N-6)Z_M
        \end{IEEEeqnarraybox}
  \end{split}
 \end{equation}
Furthermore we compute
\begin{equation}
    \begin{split}
        \begin{IEEEeqnarraybox}[][c]{rClCrClCrCl}\label{USW}
        D_\gamma(\Sigma_1)&=& \Sigma_1-\Gamma,      &\quad&
        D_\gamma(\Sigma_2)&=& \Sigma_2-2\Gamma,     &\quad&
        D_\gamma(\Sigma_N)&=& \Sigma_N-(N-1)\Gamma,\\
        D_\gamma(Z_1)     &=&  Z_1,                  &\quad&
        D_\gamma(Z_2)     &=&  Z_2-\Gamma,           &\quad& 
        D_\gamma(Z_M)     &=&  Z_M-(m+1)\Gamma,\\[2mm]
        D_\alpha(\Sigma_1)&=& \Sigma_1+A,           &\quad& 
        D_\alpha(\Sigma_2)&=& \Sigma_2+A,            &\quad& 
        D_\alpha(\Sigma_N)&=& \Sigma_N-(N-4)A,\\
        D_\alpha(Z_1)     &=& Z_1+A,       &\quad& 
        D_\alpha(Z_2)     &=& Z_2+2A,      &\quad&
        D_\alpha(Z_M)     &=& Z_M+(3+m)A
        \end{IEEEeqnarraybox}
    \end{split}
\end{equation}
The matrix representations \(M_\gamma^{(0)}\) and \(M_\alpha^{(0)}\in\mathbb{R}^{6\times 6}\) on \(H_1^{(0)}(\mathcal{O}_{N,M},\mathbb{Q})\) with respect to the basis $B^{(0)}$ are included in the file AddToArticleBKKMNSVW.g (cf. \cite{KattlerGIT2023}). They can be printed with the functions BigMatrix4 and BigMatrix3.

%\begin{displaymath}
% \begin{array}{rl}
%   M_\gamma^{(0)}=&
%    \begin{footnotesize}
%    \left(\begin{array}{cccccc}
%    -M-2N-3  & -2M-4N-8  & -(N-1)(M+2N+4) &0      &-M-2N-4  &-(m+1)(M+2N+4)\\ 
%     M       &  2M+1     &  (N-1)M        &0      & M       & (m+1)M\\
%     M       &  2M       &  (N-1)M+1      &0      & M       & (m+1)M\\
%    2M       &  4M       & 2(N-1)M        &1      & 2M      & 2(m+1)M\\
%    2M       &  4M       & 2(N-1)M        &0      & 2M+1    & 2(m+1)M\\
%   -2N-4     & -4N-8     &-(N-1)(2N+4)    &0      &-2N-4    &-(m+1)(2N+4)+1
%    \end{array}\right)
%    \end{footnotesize},\\
% & \\
%    M_\alpha^{(0)}=&
%    \begin{footnotesize}
%    \left(\begin{array}{cccccc}
%    -(N+m)+1 & -(N+m)    & (N-4)(N+m)     & -(N+m)  & -2(N+m)   & -(3+m)(N+m)\\
%    -(N+m)   & -(N+m)+1  & (N-4)(N+m)     & -(N+m)  & -2(N+m)   & -(3+m)(N+m)\\
%      m+6    &   m+6     & -(N-4)(m+6)+1  &  m+6    &  2(m+6)   &  (3+m)(m+6)\\
%      N-6    &   N-6     & -(N-4)(N-6)    &  N-5    &  2(N-6)   &  (3+m)(N-6)\\
%    2(N+m)   & 2(N+m)    & -2(N-4)(N+m)   & 2(N+m)  & 4(N+m)+1  & 2(3+m)(N+m)\\
%      N-6    &  N-6      & -(N-4)(N-6)    & N-6     &  2(N-6)   &  (3+m)(N-6)+1
%    \end{array}\right).\end{footnotesize}\\
%  \end{array}
%\end{displaymath}
%
%
%
\begin{figure}[htbp]
	\centering
	\noindent\begin{subfigure}[b]{0.5\textwidth}
		\centering
		% First cylinder decomposition (1,0)
\begin{tikzpicture}[scale=0.6]
%% squares
	\draw (0,0) rectangle (3,1);
	\draw (5,0) rectangle (6,1);
	\draw (0,1) rectangle (3,2);
	\draw (0,2) rectangle (2,3);
	\draw (0,3) rectangle (1,4);
	\draw (0,0) rectangle (1,4);
	\draw (0,6) rectangle (1,8);
	\draw (1,0) -- (1,3);
	\draw (2,0) -- (2,2);
	\draw (3,0) -- (3,1);
	\draw (5,0) -- (5,1);
	\draw (0,6) -- (1,6);
	\draw (0,7) -- (1,7);
%% Punkte
    \draw [dashed] (3.8,0.5) -- (4.2,0.5);
    \draw [dashed] (0.5,4.8) -- (0.5,5.2);
%% Schleifen N, M
    \draw [decorate,line width=0.5mm,decoration={brace}] (-0.5,3) --  (-0.5,8) node[pos=0.5,left=10pt,black]{$M-3$};
    \draw [decorate,line width=0.5mm,decoration={brace,mirror}] (0,-1.1) --  (6,-1.1) node[pos=0.5,below=10pt,black]{$N$};
%% Zigzag
    \draw [decorate, decoration={snake}] (0,4.5) -- (1,4.5);
    \draw [decorate, decoration={snake}] (0,5.5) -- (1,5.5);
    \draw (0,4) -- (0,4.5);
    \draw (1,4) -- (1,4.5);
    \draw (0,5.5) -- (0,6);
    \draw (1,5.5) -- (1,6);
    \draw [decorate, decoration={snake}] (3.5,0) -- (3.5,1);
    \draw [decorate, decoration={snake}] (4.5,0) -- (4.5,1);
    \draw (3,0) -- (3.5,0);
    \draw (3,1) -- (3.5,1);
    \draw (4.5,0) -- (5,0);
    \draw (4.5,1) -- (5,1);
%% cylinder
    \draw[pattern color=yellow, pattern = north west lines] (0,2) -- (2,2) -- (2,3) -- (0,3) -- (0,2) --
     cycle;
    \draw[pattern color=green, pattern = north west lines] (0,1) -- (3,1) -- (3,2) -- (0,2) -- (0,1) --
     cycle;
	\draw[pattern color=orange, pattern = north west lines] (0,0) -- (3,0) -- (3,1) -- (0,1) -- (0,0) --
	 cycle;
	\draw[pattern color=orange, pattern = north west lines] (5,0) -- (6,0) -- (6,1) -- (5,1) -- (5,0) --
	 cycle;
\end{tikzpicture}
		\caption{}\label{fig:cylinder(1,0)}
	\end{subfigure}%
	\hfill
	\begin{subfigure}[b]{0.5\textwidth}
		\centering
		% Second cylinder decomposition (0,1)
\begin{tikzpicture}[scale=0.6]
%% squares
	\draw (0,0) rectangle (3,1);
	\draw (5,0) rectangle (6,1);
	\draw (0,1) rectangle (3,2);
	\draw (0,2) rectangle (2,3);
	\draw (0,3) rectangle (1,4);
	\draw (0,0) rectangle (1,4);
	\draw (0,6) rectangle (1,8);
	\draw (1,0) -- (1,3);
	\draw (2,0) -- (2,2);
	\draw (3,0) -- (3,1);
	\draw (5,0) -- (5,1);
	\draw (0,6) -- (1,6);
	\draw (0,7) -- (1,7);
%% Punkte
    \draw [dashed] (3.8,0.5) -- (4.2,0.5);
    \draw [dashed] (0.5,4.8) -- (0.5,5.2);
%% Schleifen N, M
    \draw [decorate,line width=0.5mm,decoration={brace}] (-0.5,0) --  (-0.5,8) node[pos=0.5,left=10pt,black]{$M$};
    \draw [decorate,line width=0.5mm,decoration={brace,mirror}] (3,-1.1) --  (6,-1.1) node[pos=0.5,below=10pt,black]{$N-3$};
%% Zigzag
    \draw [decorate, decoration={snake}] (0,4.5) -- (1,4.5);
    \draw [decorate, decoration={snake}] (0,5.5) -- (1,5.5);
    \draw (0,4) -- (0,4.5);
    \draw (1,4) -- (1,4.5);
    \draw (0,5.5) -- (0,6);
    \draw (1,5.5) -- (1,6);
    \draw [decorate, decoration={snake}] (3.5,0) -- (3.5,1);
    \draw [decorate, decoration={snake}] (4.5,0) -- (4.5,1);
    \draw (3,0) -- (3.5,0);
    \draw (3,1) -- (3.5,1);
    \draw (4.5,0) -- (5,0);
    \draw (4.5,1) -- (5,1);
%% cylinder
     \draw[pattern color=yellow, pattern = north west lines] (2,0) -- (3,0) -- (3,2) -- (2,2) -- (2,0) --
     cycle;
     \draw[pattern color=green, pattern = north west lines] (1,0) -- (2,0) -- (2,3) -- (1,3) -- (1,0) --
     cycle;
     \draw[pattern color=orange, pattern = north west lines] (0,0) -- (1,0) -- (1,4) -- (0,4) -- (0,0) --
     cycle;
     \draw[pattern color=orange, pattern = north west lines] (0,6) -- (1,6) -- (1,8) -- (0,8) -- (0,6) --
     cycle;
\end{tikzpicture}
\caption{}\label{fig:cylinder(0,1)}
	\end{subfigure}
	\caption{Cylinder decomposition in direction $(1,0)$ and \((0,1)\) of the origami \(\mathcal{O}_{N,M}\).}
\end{figure}
In the cylinder decomposition of the origami \(\mathcal{O}_{N,M}\) in horizontal and vertical direction we have in both cases four maximal cylinders with moduli \(M-3,\ 1/2,\ 1/3,\ 1/N\) for the horizontal direction and moduli \(N-3,\ 1/2,\ 1/3,\ 1/M\) for the vertical direction (see Figure \ref{fig:cylinder(1,0)} and Figure \ref{fig:cylinder(0,1)}). 
Hence we obtain two corresponding multitwists which act on \(H_1^{(0)}(\mathcal{O}_{N,M},\mathbb{Q})\) by:
\begin{align*}
D_h \colon  & w\mapsto w + 6(M-3)N\ \Omega(w,\sigma_1)\sigma_1 + 3N\ \Omega(w,\sigma_2)\sigma_2+ 2N\ \Omega(w,\sigma_3)\sigma_3 + 6\ \Omega(w,\sigma_N)\sigma_N, \\
D_v \colon  & w\mapsto w + 6(N-3)M\ \Omega(w,\zeta_1)\zeta_1 + 3M\ \Omega(w,\zeta_2)\zeta_2+ 2M\ \Omega(w,\zeta_3)\zeta_3 + 6\ \Omega(w,\zeta_M)\zeta_M.
\end{align*}
We read the intersection numbers of the $\sigma_i$ and $\zeta_j$ from Figure~\ref{fig:OrigamiNM}. A straight forward computation now gives the representation matrices \( M_h^{(0)}\) and \(M_v^{(0)}\) for the action of the horizontal and vertical twist on \(H_1^{(0)}(\mathcal{O}_{N,M},\mathbb{Q})\) with respect to \(B^{(0)}\):
\begin{displaymath}
  \begin{array}{rl}
  M_h^{(0)}=&
    \left(\begin{array}{cccccc}
    1& 0& 0&   0& -3N& -3N\\
    0& 1& 0& -2N& -2N& -2N\\
    0& 0& 1&   6&  12&  6(M-1)\\
    0& 0& 0&  1&   0& 0 \\
    0& 0& 0& 0 & 1  & 0 \\
    0& 0& 0& 0 & 0  & 1 
    \end{array}\right),\\
  M_v^{(0)}=&
    \left(\begin{array}{cccccc}
    1   &     0&       0& 0& 0& 0\\
    0   &     1&       0& 0& 0& 0\\
    0   &     0&       1& 0& 0& 0\\
    0   &    3M&      3M& 1& 0& 0\\
    2M  &    2M&      2M& 0& 1& 0\\
   -6   & -12  & -6(N-1)& 0& 0& 1
    \end{array}\right).
  \end{array}
\end{displaymath}

\subsection{Zariski density in genus four}
    \label{ssec:zariski-density-genus4}

For the proof of the Zariski-density of the  Kontsevich-Zorich monodromy of the $\mathcal{O}_{N,M}$ we will follow an approach which differs from the one in the previous sections. Let us first describe the idea before we go into the details:

The key ingredient of our arguments is the following theorem of Detinko, Flannery and Hulpke.
\begin{theorem}[Detinko, Flannery, Hulpke, \cite{Detinko2018ZariskiDA} Prop. 3.7]\label{Thm:ZarDenTrans}
Suppose that a subgroup \(H\leq \Sp(2n,\mathbb{Z})\) contains a transvection \(t\in H\). Then \(H\) is Zariski dense if and only if the normal closure \(\langle t\rangle^H\) of \(t\) in \(H\) is absolutely irreducible.
\end{theorem}

In order to use this theorem we firstly have to overcome the issue that the Kontsevich monodromy $\Gamma^{(0)}$ by definition lives in $\mbox{Sp}_\Omega(H_1^{(0)}(\mathcal{O}_{N,M},\mathbb{Z}))$  and thus per se is not a subgroup of   \(\Sp(2n,\mathbb{Z})\). We solve this problem by passing to a finite index subgroup. More precisely, we choose in the beginning a \(\mathbb{Z}\)-submodule \(\Gamma_{N,M}\) of \(H_1^{(0)}(\mathcal{O}_{N,M},\mathbb{Z})\) such that with a suitable choice of a base of \(\Gamma_{N,M}\) the intersection form \(\Omega\) restricted to \(\Gamma_{N,M}\) is a multiple of the standard symplectic form on $\Z^6$.  Thus the elements of $\mbox{Sp}_\Omega(H_1^{(0)}(\mathcal{O}_{N,M},\mathbb{Z}))$ which stabilize $\Gamma_{N,M}$ can be identified with elements of the standard symplectic group $\mbox{Sp}(6,\mathbb{Z})$.
Therefore our goal is to find a transvection \(t\) in the image \(G\) of the action 
  \[
  \mbox{Aff}^+(\mathcal{O}_{N,M})\longrightarrow \mbox{Sp}_\Omega(H_1^{(0)}(\mathcal{O}_{N,M},\mathbb{Z})),
  \]
which stabilizes the lattice \(\Gamma_{N,M}\) and to show that the normal closure \(\langle t \rangle^H\) is absolutely irreducible, where we identify 
  \[
  H=G\cap \mbox{Stab}_{\mbox{Sp}_\Omega(H_1^{(0)}(\mathcal{O}_{N,M},\mathbb{Z}))}(\Gamma_{N,M})
  \]
with a subgroup of \(\text{Sp}(6,\mathbb{Z})\). With this approach we will show Zariski-density for finitely many $N,M\in\mathbb{N}$ using a computer aided proof for the irreducibility of \(\langle t\rangle^H\).

Nevertheless it seems to be possible to find  infinitely many \(N,M\in\mathbb{N}\) such that $\mbox{Sp}_\Omega(H_1^{(0)}(\mathcal{O}_{N,M},\mathbb{Z}))$ is Zariski-dense, if we again use Galois-theoretical arguments (see Remark \ref{Rem:InfFamONM}).

\vspace*{0.5cm}

We now start going into details. Consider the \(\mathbb{Z}\)-submodule \(\Gamma_{N,M}\) of \(H_1^{(0)}(\mathcal{O}_{N,M},\mathbb{Z})\) generated by the following elements:
\begin{IEEEeqnarray*}{rClCrCl}
  c_1&=&(N+M+2)\Sigma_1,       &\quad&  c_2 &=& (N+M+2)(-2 \Sigma_1-\ \Sigma_N),\\
  c_3&=&(-1-N-M)Z_1+ Z_2+ Z_M, &\quad&  c_4 &=& Z_2,\\
  c_5&=&Z_1,                   &\quad&  c_6 &=& \Sigma_1+\Sigma_2+\Sigma_N 
\end{IEEEeqnarray*}  
The submodule \(\Gamma_{N,M}\) has finite index in \(H_1^{(0)}(\mathcal{O}_{N,M},\mathbb{Z})\) and if we restrict the symplectic intersection-form \(\Omega\) to \(\Gamma_{N,M}\) we get the following matrix representation in \(I_\Omega^C\in\mathbb{Q}^{6\times6}\) with respect to the basis \(C=\{c_1,c_2,c_3,c_4,c_5,c_6\}\):
\begin{displaymath}
I_\Omega^C=(N+M+2)\cdot
  \left(\begin{array}{cccccc}
     0      & 0        & 0        & 1          & 0          & 0\\
     0      & 0        & 0        & 0          & 1          & 0\\
     0      & 0        & 0        & 0          &  0         & 1\\
    -1      & 0        & 0        & 0          &  0         & 0 \\
     0      & -1       & 0        & 0          &  0         & 0 \\
     0      &  0       &-1        & 0          &  0         & 0
  \end{array}\right)
\end{displaymath}
Let \(G\) be the image of the action \(\mbox{Aff}^+(\mathcal{O}_{N,M})\to \mbox{Sp}_\Omega(H_1^{(0)}(\mathcal{O}_{N,M},\mathbb{Z}))\). We conclude that choosing the basis $C$ of \(\Gamma_{N,M}\) identifies the elements \(\phi\in G\) which stabilize the lattice \(\Gamma_{N,M}\) with elements of the standard symplectic group \(\mbox{Sp}(6,\mathbb{Z})\) i.e., 
  \[
    H=G\cap \mbox{Stab}_{\mbox{\tiny Sp}_\Omega(6,\mathbb{Z})}(\Gamma_{N,M})
    \leq \mbox{Sp}(6,\mathbb{Z}).
  \]
In the following we describe the elements of \(\mbox{Sp}_\Omega(H_1^{(0)}(\mathcal{O}_{N,M},\mathbb{Z}))\) which stabilize the lattice \(\Gamma_{N,M}\). More precisely, we  find conditions for their matrix representations to do so. Denote by \(C\in \mathbb{Q}^{6\times6}\) the matrix, which has as columns the coefficients of the vectors \(c_i\) \((i=1,\dots,6)\) written as a linear combination of elements in \(B^{(0)}\). Furthermore let \(C^{-1}\in \mathbb{Q}^{6\times 6}\) be the inverse of \(C\) i.e.,
\begin{IEEEeqnarray*}{rCl}
  C&=&
  \left(\begin{array}{cccccc}
    N+M+2 &-2(N+M+2) & 0      & 0 & 0 & 1\\
    0     & 0        & 0      & 0 & 0 & 1\\
    0     & -(N+M+2) & 0      & 0 & 0 & 1\\
    0     & 0        & -1-N-M & 0 & 1 & 0\\
    0     & 0        & 1      & 1 & 0 & 0\\
    0     & 0        & 1      & 0 & 0 & 0
  \end{array}\right),\\
        & & \\
  C^{-1}&=&
   \left(\begin{array}{cccccc}
   \frac{1}{M+N+2} & \frac{1}{M+N+2} & \frac{-2}{M+N+2} & 0 & 0 & 0\\
                0  & \frac{1}{M+N+2} & \frac{-1}{M+N+2} & 0 & 0 & 0\\
                0  & 0               & 0                & 0 & 0 & 1\\
                0  & 0               & 0                & 0 & 1 & -1\\
                0  & 0               & 0                & 1 & 0 & M+N+1\\
                0  & 1               & 0                & 0 & 0 & 0
\end{array}\right)
\end{IEEEeqnarray*}
An element \(\phi\in \mbox{Sp}_\Omega(H_1^{(0)}(\mathcal{O}_{N,M},\mathbb{Z}))\) stabilizes the lattice \(\Gamma_{N,M}\) if and only if \(\phi(c_i)\) is an element of \(\mbox{Span}_\mathbb{Z}(\{c_1,\dots,c_6\})\) for each \(i\in\{1,\dots, 6\}\) or equivalent
 \[
  C^{-1}\cdot D_{B^{(0)}}(\phi(c_i)))\in \mathbb{Z}^6
 \]
for every \(i\in\{1,\dots,6\}\). Here we write \(D_{B^{(0)}}(\phi(c_i))\in \mathbb{R}^6\) for the coefficients of \(\phi(c_i)\) with respect to the basis \(B^{(0)}\). 

Thus \(\phi\) stabilizes \(\Gamma_{N,M}\) if and only if for each \(i\in\{1,\dots,6\}\) the element \(D_{B^{(0)}}(\phi(c_i))\) is in both the kernels of the following two maps
  \begin{IEEEeqnarray*}{lClClCl}
    g_1\colon\mathbb{Z}^6 &\longrightarrow& \mathbb{Z}{/}(N+M+2)\,\mathbb{Z},& \quad&
    (v_1,\dots v_6)&\longmapsto&   \overline{v_1+v_2-2v_3}\\
    g_2\colon\mathbb{Z}^6&\longrightarrow& \mathbb{Z}{/}(N+M+2)\,\mathbb{Z}, & \quad&
    (v_1,\dots v_6)&\longmapsto&  \overline{v_2-v_3}.
  \end{IEEEeqnarray*}
Easy but boring calculations show that for every \(i\in\{1,\dots,6\}\) and every matrix \(M\) in the set \(\{M_\delta^{(0)},\ M_\chi^{(0)},\ M_\gamma^{(0)},\ M_\alpha^{(0)},\ M_v^{(0)}\}\), we have
 \[
   M\cdot D_{B^{(0)}}(c_i)\in\mbox{ker}(g_1)\cap\mbox{ker}(g_2).
 \] 
Hence \(D_\delta,D_\chi,D_\gamma,D_\alpha,D_v\in \mbox{Stab}_{\Gamma_{N,M}}(\mbox{Sp}_\Omega(H_1^{(0)}(\mathcal{O}_{N,M},\mathbb{Z})))\). Furthermore \((M_h^{(0)}-I_6)^2=\underline{0}\) and so we conclude with the Bernoulli-formula \((M_h^{(0)})^{N+M+2}\equiv I_6\) modular \( N+M+2\). Hence
\[
   (M_h^{(0)})^{N+M+2}\cdot D_{B^{(0)}}(c_i)\in\mbox{ker}(g_1)\cap\mbox{ker}(g_2)
 \]  

Consider the algebra \(A_{N,M}\) generated by the transvection \(t=M_\delta^{(0)}\) and the elements 
\(M^{-1}t M\in\langle t \rangle^H\) where 
 \[
  M\in\{M_\chi^{(0)}, M_\gamma^{(0)}, M_\alpha^{(0)}, M_v^{(0)},(M_h^{(0)})^{N+M+2}\}.
 \]
Here \(\langle t\rangle ^H\) denotes the normal closure of the transvection \(t\) in \(H\).
For \(N\in\{4,5,...,50\}\) and \(M = 2m + 4\) with \(m\in\{0,1,...,50\}\) we calculated with GAP \(\mbox{dim}_\mathbb{Q}(A_{N,M})=36\) for the vector space dimension of the algebra \(A_{N,M}\). This shows that \(\langle t\rangle^H\) is an absolutely irreducible group %(IV Theorem 2.10 in \cite{TamburiniClassicalgroupsGeom16}) 
and for these cases we conclude with Theorem \ref{Thm:ZarDenTrans} that \(H\) is Zariski dense in \(\mbox{Sp}_\Omega(H_1^{(0)}(\mathcal{O}_{N,M},\mathbb{C}))\).

\begin{remark}\label{Rem:InfFamONM}
It is possible to obtain the Zariski denseness of $H$ for infinitely many choices of the parameters $N$ and $M$ by developing the Galois-theoretical method in Subsection \ref{ss.Prasad-Rapinchuk-MMY}: This approach was developed later in the follow-up paper \cite{KanyMatheus23}.
\end{remark}

\subsection{Arithmeticity in genus four}\label{ArithmeticInGenus4}

Recall from Section~\ref{favour} that for each of the cylinder decompositions in direction \((1,1)\), \((1,-1)\) and \((1,2)\) we get two maximal cylinders. Their waist curves are \(\delta_1,\delta_2\in H_1(\mathcal{O}_{N,M},\mathbb{Z})\) for direction \((1,1)\), \(\chi_1,\chi_2\in H_1(\mathcal{O}_{N,M},\mathbb{Z})\) for the direction \((1,-1)\) and \(\gamma_1,\gamma_2\) for the direction \((1,2)\in H_1(\mathcal{O}_{N,M},\mathbb{Z})\). Furthermore we introduced in \ref{DeltaX} and \ref{GammaA} the following elements in the non-tautological part  \(H_1^{(0)}(\mathcal{O}_{N,M},\mathbb{Z})\):
\begin{IEEEeqnarray*}{rCll}
   \Delta &=&-& 3\,\Sigma_1+(N+M-1)\Sigma_2-3\,\Sigma_N\\
          & &-& 3\, Z_1+(N+M-1)\,Z_2-3\,Z_M,\\[1mm]
        X &=& & (N+M-3)\,\Sigma_1+(N+M-3)\,\Sigma_2-5\,\Sigma_N\\
          & &-& (N+M-3)\,Z_1-(N+M-3)\,Z_2+5\,Z_M,\\[1mm]
   \Gamma &=& & (2N+M+4)\,\Sigma_1-M\,\Sigma_2-M\,\Sigma_N\\
          & &-& 2M\,Z_1-2M\,Z_2+(2N+4)\,Z_M
\end{IEEEeqnarray*}
Set \(W=\mbox{Span}_\mathbb{Q}(\Delta,\ X,\ \Gamma)\). The vector space \(W\) has dimension \(\mbox{dim}_\mathbb{Q}(W)=3\). We set \(A=-22+4N+4M\), \(B=-6-3m\) and \(C=-12+3N-9m\). Using (\ref{DdeltaDchi}) and (\ref{USW}) we obtain that the restrictions of the transvections \(D_\delta,\ D_\chi\) and \(D_\gamma\) to $W$ have the following matrix representations with respect to the basis \(\{\Delta,\ X,\ \Gamma\}\):
\begin{displaymath}
\left(\begin{array}{ccc}1 & A & -2B \\ 0 & 1 & 0 \\ 0 & 0 & 1 \end{array}\right),\quad 
\left(\begin{array}{ccc}1 & 0 & 0 \\ -A & 1 & -2C \\ 0 & 0 & 1 \end{array}\right),\quad
\left(\begin{array}{ccc}1 & 0 & 0 \\  0 & 1 & 0    \\ B & C & 1 \end{array}\right)
\end{displaymath}
The vector \(e=-2C\Delta+2BX+A\Gamma\) is fixed by all the three elements \(D_\delta,\ D_\chi\) and \(D_\gamma\). Furthermore $\Omega(e,w)=0$ for all $w\in W$. With respect to the new basis
\(\{\Delta,\ X,\ e\}\) we get the following matrix representations for \(D_\delta,\ D_\chi\) and \(D_\gamma\):
\begin{displaymath}
  \left(\begin{array}{ccc}1 &  A & 0 \\ 0 &  1 & 0 \\ 0 & 0 & 1 \end{array}\right),\quad 
  \left(\begin{array}{ccc}1 &  0& 0  \\ -A&  1& 0  \\ 0 & 0 & 1 \end{array}\right),\quad
  \left(\begin{array}{ccc}2\frac{BC}{A}+1  &  2\frac{C^2}{A}     & 0  \\                                                                                                                                                                                                                                                                                                                                                                       
                          -2\frac{B^2}{A}  & -2\frac{BC}{A}+1    & 0 \\ 
                           \frac{B}{A}     &  \frac{C}{A}        & 1
  \end{array}\right)
\end{displaymath}
If we choose \(C=0\) or equivalent \(N=3m+4\), we have 
  \[
  \Omega(\Delta,X)=-50\,(2m^2+5m+2)<0
  \]
for all $m>0$ and the group generated by \(D_\delta|_W,~D_\chi|_W,~D_\gamma|_W\) contains a non-trivial element of the unipotent radical of the symplectic group on \(W\), namely \((D_\chi|_W)^{-2B^2}\circ (D_\gamma|_W)^{A^2}\) is represented by
  \[
    \begin{pmatrix}
    1 & 0 & 0 \\
    0 & 1 & 0 \\
    BA& 0 & 1
    \end{pmatrix}
  \]
with respect to the basis \(\{\Delta,\ X,\ e\}\). This ends the proof of Theorem \ref{t.B}. 

%%%%%%%%%%%%%%%%%%%%%%%%%%%%%%
%%%%%%%%%%%%%%%%%%%%%%%%%%%%%%%%%%
%%%%%%%%%%%%%%%%%%%%%%%%%%%%%%%%%%%%%%

%%%%%%%%%%%%%%%%%%%%%%%%%%%%%%%%%%%%%%%%%%%%%
%%%%%%% Section 6
%%%%%%%%%%%%%%%%%%%%%%%%%%%%%%%%%%%%%%%%%%%%%%%

\section{Computational results}\label{s.computationals}

In this section we present results on the Kontsevich-Zorich monodromy that we obtained from computer experiments conducted with the GAP-package \cite{origami-package} and  \cite{ModularGroup}.  
Recall from Section~\ref{KZmonOrigami} that if an origami has no non-trivial translation then its Veech group is isomorphic to the affine group. Its explicit action on homology is described in \cite[Section 3]{HM}. Using this, one can compute the Kontsevich--Zorich monodromy up to finite index. More precisely, choosing a suitable basis of  the non-tautological part $H_1^{(0)}(M,\Z)$ and a sublattice of $H_1^{(0)}(M,\Z)$, one obtains a finite index subgroup of the Kontsevich--Zorich monodromy as subgroup in $\SP(2g-2,\Z)$. The code of the computer experiments presented in this section is collected in \cite{KattlerGIT2023}.

\subsection{Origamis in genus two}

In genus two it is known by a result of M{\"o}ller (cf. Appendix~\ref{a.ArithmShadVee}) that the Kontsevich--Zorich monodromies of origamis are always arithmetic. In this case they are subgroups of $\SL(2,\ZZ) = \SP(2,\ZZ)$. We compute the Kontsevich--Zorich monodromies of origamis in genus 2 of small degree and detect a nice pattern from these computer experiments.

\subsubsection{The stratum $\mathcal{H}(2)$}

Recall that the $\SL(2,\ZZ)$-orbits of origamis of degree $n$ in $\mathcal{H}(2)$ are classified by Hubert/Leli{\`e}vre and McMullen (cf. \cite{HL},\cite{McMullenSpin}; see also \cite[Section 4]{Weitze2015} for a condensed presentation of this result) in the following way: For each $n$ there are at most 2 orbits. More precisely, if $n$ is even or 3, then there is only one orbit. If $n$ is odd and not 3, then there are 2 orbits called $A_n$ and $B_n$ distinguished by their number of integer Weierstrass points. The origamis in the orbit $A_n$ have 1 integer Weierstrass point whereas the origamis in the orbit $B_n$ have 3 integer Weierstrass points.
From the classification of the orbits it follows in particular that each orbit can be represented by an $L$-shaped origami $L(a,b)$ of degree  $n = a+b-1$, cf. \cite[Remark 4.2]{Weitze2015}. Here  $L(a,b)$ is the origami associated to the pair of permutations $h, v  \in\mbox{Sym}( \{1,\dots,n\}) $, where
$h=(1,2,3\dots,a)$ and $v=(1, a+1,\ldots, n)$, cf. Figure~\ref{fig:Lab}. If $n$ is odd then $L(a,b)$ lies in the orbit $A_n$, if $a$ and $b$ are even, it lies in $B_n$, if $a$ and $b$  are odd. Hence it suffices to study for each $n$ one or two $L$-shaped origamis of degree $n$ depending on $n$ being even or odd.
%\commref{You should make it accessible !
%Create a file in a git repository that will produce the given table. So that we can reproduce the experiment and check your code.} \textcolor{purple}{
With our computer experiments (cf.\cite[\texttt{Example2\_AKZ}]{KattlerGIT2023}) we obtain the results presented in Theorem~\ref{Thm:TheLs}.

\begin{theorem}\label{Thm:TheLs}
  Let $n$ be the number of squares of the origami $\mathcal{O} = (M,\omega)$ in $\mathcal{H}(2)$. For $n \leq 21$ we obtain for the  Kontsevich--Zorich monodromy $\KoZoMon$ of $(M,\omega)$:
\begin{itemize}
\item
  If $n$ is even, then $[\SL_2(\ZZ):\KoZoMon] = 3$ 
\item
  If $n$ is odd and $\mathcal{O}$ lies in the orbit $A_n$ or $n=3$ then $[\SL_2(\ZZ): \KoZoMon] = 1$
   % \[
   % [\SL_2(\ZZ): \KoZoMon] = 1
   % \]
\item
  If $n$ is odd and $\mathcal{O}$ lies in the orbit $B_n$ then $ [\SL_2(\ZZ): \KoZoMon] = 3$
  %  \[
  % [\SL_2(\ZZ): \KoZoMon] = 3
  %  \]
\end{itemize}
\end{theorem}

\begin{remark}
  We conjecture that the statement in Theorem~\ref{Thm:TheLs} holds for all origamis in  $\mathcal{H}(2)$ of arbitrary degree. Part of this conjecture is proved in \cite{Kattler2023} which arose as follow-up work of the article at hand.  
\end{remark}
  
\begin{figure}
\begin{center}
\begin{tikzpicture}[scale=0.8]
%% squares
	\draw (0,0) rectangle (2,1);
        \draw (0,2) rectangle (1,3);
        \draw (5,0) rectangle (6,1);
       	%\draw (6,0) rectangle (4,1); 
	\draw (1,0) -- (1,1);
        \draw [dashed] (2,0) -- (5,0);
	\draw [dashed] (2,1) -- (5,1);
        \draw [dashed] (0,1) -- (0,2);
        \draw [dashed] (1,1) -- (1,2);

        \node at (0.5,0.5) {1};
        \node at (1.5,0.5) {2};
        \node at (5.5,0.5) {a};
        \node at (0.5,2.5) {n};
       % \draw (0.5,1.5) node {1};
        
%% Punkte
    %\draw [dashed] (3.5,0.5) -- (4.5,0.5);
    %\draw [dashed] (0.5,4.5) -- (0.5,5.5);
    
%% Schleifen N, M
    \draw [decorate,line width=0.5mm,decoration={brace, mirror}] (0,-.3) --  (6,-.3) node[pos=0.5,below=10pt,black]{$a$};
    \draw [decorate,line width=0.5mm,decoration={brace}] (-.3,0) --  (-.3,3) node[pos=0.5,left=10pt,black]{$b$};
\end{tikzpicture}
\end{center}
\caption{Origami \(L(a,b)\): opposite edges are glued}
\label{fig:Lab}
\end{figure}

\subsubsection{The stratum $\mathcal{H}(1,1)$}

In the following we consider origamis $\mathcal{O}= \mathcal{O}(k,l)$ of degree $n = k+l$ given by the following permutations, cf. Figure~\ref{fig:Okl}:
\[h = (1,2,\ldots, k)(k+1,\ldots, n), \quad v = (k,k+1)\]

\begin{figure}
\begin{center}
\begin{tikzpicture}[scale=0.8]
%% squares
	\draw (0,0) rectangle (1,1);
	\draw [dashed] (1,0) -- (3,0);
	\draw [dashed] (1,1) -- (3,1);
	\draw (3,0) rectangle (4,1);
	\draw (3,1) rectangle (4,2);	
    \draw [dashed] (4,1) -- (6,1);
    \draw [dashed] (4,2) -- (6,2);
    \draw (6,1) rectangle (7,2);

    \node at (0.5,0.5) {1};
    \node at (3.5,0.5) {k};
    \node at (3.5,1.5) {k+1};
    \node at (6.5,1.5) {n};
        
%% Punkte
    %\draw [dashed] (3.5,0.5) -- (4.5,0.5);
    %\draw [dashed] (0.5,4.5) -- (0.5,5.5);
    
%% Schleifen N, M
    \draw [decorate,line width=0.5mm,decoration={brace, mirror}] (0,-.3) --  (4,-.3) node[pos=0.5,below=10pt,black]{$k$};
    \draw [decorate,line width=0.5mm,decoration={brace}] (3,2.3) --  (7, 2.3) 
    node[pos=0.5,above=10pt, black]{$l$};
\end{tikzpicture}
\end{center}
\caption{Origami \(O(k,l)\): opposite edges are glued}
\label{fig:Okl}
\end{figure}

We obtain (cf. \cite[\texttt{Example3\_AKZ}]{KattlerGIT2023}) the pattern in Table \ref{table:Index} for the index of the Kontsevich--Zorich monodromy in $\SL(2,\ZZ)$.
Observe that  \(O(k,l)\) allows a translation if and only if $k=l$, namely square $k$ is translated to square $k+1$ and vice versa and square $i$ ($i \in \{1,\ldots, k-1\}$) to $i+k+1$ and vice versa. We have to exclude these surfaces from the computations since we only consider surfaces without translations.

\begin{table}
\centering
$\begin{array}{c|cccccccccccccccccccc}
l\backslash k&2&3&4&5&6&7&8&9&10&11&12&13&14&15&16&17&18&19&20&21\\ \hline
2 &- & 1 & 1 & 1 & 6 & 1 & 1 & 1 & 3 & 1 & 1 & 1 & 6 & 1 & 1 & 1 & 3 & 1 & 1 & 1 \\
3 &1 & - & 1 & 3 & 4 & 3 & 1 & 3 & 1 & 3 & 1 & 3 & 1 & 12 & 1 & 3 & 1 & 3 & 1 & 3 \\
4 &1 & 1 & - & 1 & 1 & 1 & 1 & 1 & 1 & 1 & 6 & 1 & 1 & 1 & 1 & 1 & 1 & 1 & 3 & 1 \\
5 &1 & 3 & 1 & - & 1 & 3 & 1 & 3 & 1 & 3 & 1 & 3 & 1 & 3 & 1 & 3 & 1 & 3 & 12 & 3 \\
6 &6 & 4 & 1 & 1 & - & 1 & 1 & 1 & 6 & 1 & 4 & 1 & 3 & 1 & 1 & 1 & 6 & 1 & 1 & 4 \\
7 &1 & 3 & 1 & 3 & 1 & - & 1 & 3 & 1 & 3 & 1 & 3 & 1 & 3 & 1 & 3 & 1 & 3 & 1 & 3 \\
8 &1& 1& 1& 1& 1& 1& -& 1& 1& 1& 1& 1& 1& 1& 1& 1& 1& 1& 1& 1 \\
9 &1& 3& 1& 3& 1& 3& 1& -& 1& 3& 1& 3& 1& 3& 1& 3& 4& 3& 1& 3 \\
10 &3& 1& 1& 1& 6& 1& 1& 1& -& 1& 1& 1& 6& 12& 1& 1& 3& 1& 1& 1 \\
11 &1& 3& 1& 3& 1& 3& 1& 3& 1& -& 1& 3& 1& 3& 1& 3& 1& 3& 1& 3 \\
12&1& 1& 6& 1& 4& 1& 1& 1& 1& 1& -& 1& 1& 4& 1& 1& 1& 1& 24& 1 \\
13 &1& 3& 1& 3& 1& 3& 1& 3& 1& 3& 1& -& 1& 3& 1& 3& 1& 3& 1& 3 \\
14 &6& 1& 1& 1& 3& 1& 1& 1& 6& 1& 1& 1& -& 1& 1& 1& 6& 1& 1& 1 \\
15 &1& 12& 1& 3& 1& 3& 1& 3& 12& 3& 4& 3& 1& -& 1& 3& 1& 3& 1& 12\\
16 &1& 1& 1& 1& 1& 1& 1& 1& 1& 1& 1& 1& 1& 1& -& 1& 1& 1& 1& 1 \\
17 &1& 3& 1& 3& 1& 3& 1& 3& 1& 3& 1& 3& 1& 3& 1& -& 1& 3& 1& 3 \\
18 &3& 1& 1& 1& 6& 1& 1& 4& 3& 1& 1& 1& 6& 1& 1& 1& -& 1& 1& 1 \\
19 &1& 3& 1& 3& 1& 3& 1& 3& 1& 3& 1& 3& 1& 3& 1& 3& 1& -& 1& 3 \\
20 &1& 1& 3& 12& 1& 1& 1& 1& 1& 1& 24& 1& 1& 1& 1& 1& 1& 1& -& 1 \\
21 &1& 3& 1& 3& 4& 3& 1& 3& 1& 3& 1& 3& 1& 12& 1& 3& 1& 3& 1& - 
\end{array}$
\caption{Index of the Kontsevich-Zorich monodromy for the origami \(\mathcal{O}(k,l)\) in the stratum \(\mathcal{H}(1,1)\)}
%\commman{Best we put this table in a floating environment to avoid big gaps in the text?}
\label{table:Index}
\end{table}

\subsection{An example in genus 3} 
    \label{ssec:example-genus-3}
    In this section we study the Kontsevich--Zorich monodromy $\Gamma^{(0)}(M,\omega)$ of the origami $\mathcal{O}^\prime_{3,5} = (M,\omega) \in \mathcal{H}^{odd}(4)$ defined by the permutations $(1,2,3,4,5)(6,7)$ and $ (1,6,8)(2,7))$  (see Figure~\ref{fig:O35}), which is the smallest member of the family studied in Section \ref{ss.H4odd-family7}. From the computer experiments we obtain the results described in the following (cf.\cite[\texttt{Example1\_AKZ}]{KattlerGIT2023}).
    %The corresponding source code is given in Table xxx in \cite{KattlerGAP2023}.
    %Here we used \textit{GAP} (ToDo Include source).
%Consider the basis $B= \{\sigma_l, \sigma_m, \sigma_c, \zeta_l, \zeta_m, \zeta_c\}$ of $H_1(M,\ZZ)$ (see Figure~\ref{fig:O35}) and the basis
%  \[
%  \tilde{B} = \{ \Sigma_l = \sigma_l - 5\sigma_c,\;
%  \Sigma_m = \sigma_m - 2\sigma_c,\;
%  Z_l = \zeta_l - 3\zeta_c,\;
%  Z_m = \zeta_m - 2\zeta_c \}
%  \]
%  of $H_1^{(0)}(M,\ZZ)$ defined at the beginning of Section~\ref{ss.H4odd-family7}.
  %\commman{Hier ist was kaputt}.

\begin{figure}[htb]
\centering
\begin{tikzpicture}[scale = 0.8]
  \draw(0,0) -- (5,0);
  \draw(0,0) -- (0,3);
  \draw(1,0) -- (1,3);
  \draw(0,1) -- (5,1);
  \draw(2,0) -- (2,2);
  \draw(0,2) -- (2,2);
  \draw(3,0) -- (3,1);
  \draw(4,0) -- (4,1);
  \draw(5,0) -- (5,1);
  \draw(0,3) -- (1,3);
  \node at (.5,.5) {1};
  \node at (1.5,.5) {2};
  \node at (2.5,.5) {3};
  \node at (3.5,.5) {4};
  \node at (4.5,.5) {5};
  \node at (.5,1.5) {6};
  \node at (1.5,1.5) {7};
  \node at (.5,2.5) {8};
%\draw[->, color = red](0.5, 0) -- (0.5, 3) node[above] {$\zeta_l$};
%\draw[->, color = red](1.5, 0) -- (1.5, 2) node[above,xshift = 2mm] {$\zeta_m$};
%\draw[->, color = red](3.5, 0) -- (3.5, 1) node[above] {$\zeta_c$};
%\draw[->, color = green](0, 0.5) -- (5, 0.5) node[right] {$\sigma_l$};
%\draw[->, color = green](0, 1.5) -- (2, 1.5) node[right] {$\sigma_m$};
%\draw[->, color = green](0, 2.5) -- (1, 2.5) node[right] {$\sigma_c$};
\end{tikzpicture}
\hspace*{1cm}
\begin{tikzpicture}[scale = 0.8]
\draw(0,0) -- (5,0);
\draw(0,0) -- (0,3);
\draw(1,0) -- (1,3);
\draw(0,1) -- (5,1);
\draw(2,0) -- (2,2);
\draw(0,2) -- (2,2);
\draw(3,0) -- (3,1);
\draw(4,0) -- (4,1);
\draw(5,0) -- (5,1);
\draw(0,3) -- (1,3);
\draw[->, color = red](0.5, 0) -- (0.5, 3) node[above] {$\zeta_l$};
\draw[->, color = red](1.5, 0) -- (1.5, 2) node[above,xshift = 2mm, yshift = -1mm] {$\zeta_m$};
\draw[->, color = red](3.5, 0) -- (3.5, 1) node[above] {$\zeta_c$};
\draw[->, color = green](0, 0.5) -- (5, 0.5) node[right] {$\sigma_l$};
\draw[->, color = green](0, 1.5) -- (2, 1.5) node[right] {$\sigma_m$};
\draw[->, color = green](0, 2.5) -- (1, 2.5) node[right,yshift = 1mm, xshift = -1mm] {$\sigma_c$};
\end{tikzpicture}
\caption{The origami $\mathcal{O} = \mathcal{O}_{3,5}$: opposite edges are glued.}
\label{fig:O35}
\end{figure}

The Veech group $\VG(M,\omega)$ is an index 1020  subgroup of $\SL(2,\ZZ)$ with 102 generators.
Consider the basis $B= \{\sigma_l, \sigma_m, \sigma_c, \zeta_l, \zeta_m, \zeta_c\}$ of $H_1(M,\ZZ)$ (see Figure~\ref{fig:O35}) and the basis
  \(
  \tilde{B} = \{ \Sigma_l = \sigma_l - 5\sigma_c,\;
  \Sigma_m = \sigma_m - 2\sigma_c,\;
  Z_l = \zeta_l - 3\zeta_c,\;
  Z_m = \zeta_m - 2\zeta_c \}
  )\)  of $H_1^{(0)}(M,\ZZ)$ defined at the beginning of Section~\ref{ss.H4odd-family7}. Let
%\commman{Vorne heißt das $\rho$. Shadow taucht ja nicht mehr auf}
$\rho: \VG(M,\omega) \to SL(4,\ZZ)$ be the action of the Veech group on the non tautological part  $H_1^{(0)}(M,\ZZ)$, when
we identify $H_1^{(0)}(M,\ZZ)$ with $\ZZ^4$ according to the chosen basis $\tilde{B}$. %We want to study the image $H = \rho_{\tilde{B}}(\VG(M,\omega)) \subseteq \SL(4,\ZZ)$ \commman{Hier dann auch $\rho$} i.e., the group of all matrices obtained from the transformations in the Kontsevich--Zorich monodromy with respect to the basis $\tilde{B}$ of $H_1^{(0)}(M,\ZZ)$.
From Section~\ref{ss.H4odd-family7} we know that the image $H$ of $\rho$ is an arithmetic group. Recall that the action of the Veech group respects the intersection form $\Omega$. However, in general there is no symplectic basis of $H_1^{(0)}(M,\ZZ)$ defined over $\ZZ$. Hence in general $H$ can not be conjugated to a subgroup of the standard symplectic group $\SP(4,\ZZ)$. But we indeed can pass to a finite index subgroup of $H$ which is conjugated by a matrix $\tilde{T} \in \GL(4,\Q)$ to a subgroup $\tilde{U}$ of $\SP(4,\ZZ)$ of finite index (see below). Now, recall that $\SP(4,\ZZ)$  has the \textit{congruence subgroup property}, cf. \cite{BassMilnorSerre67}. Namely, any finite index subgroup of \(\SP(4,\ZZ)\) is a  congruence group of some level $l$, i.e. it contains $\Gamma(l) = \{A \in \Sp(4,\ZZ)|\; A \equiv \Id \mod l\}$, where $\Id$ is the identity matrix. We determine with GAP and in particular with the package \cite{Detinko2018ZariskiDA} the index and the level $l$ of the group $\tilde{U} \subseteq \SP(4,\ZZ)$ .\\ %zzz (Zitat für KG einfügen!!)

In detail, this is achieved for this example  as follows. 
Observe that the intersection form on the homology restricted to  $H_1^{(0)}(M,\ZZ)$ has the fundamental matrix $G$ given in (\ref{AllTheMatrices}) with respect to the basis $\tilde{B}$.
%We choose as Basis of the nontautological part 
%\[
%\left(  
%\Sigma_1 - 5\Sigma_3, \Sigma_2 - 2\Sigma_3, Z_1 - 3Z_3, Z_2 - 2Z_3
%\right)
%\]
%Now we use the package [], to calculate the Shadow Veech $\Gamma^{S}$ group of $\mathcal{O}$. We get a group with 152 generators.
%With respect to this basis, the Intersection form has the following matrix:
\begin{align}
\begin{split}\label{AllTheMatrices}
G =
\begin{pmatrix}
   0 &   0 &    7 & 1 \\
    0 &   0 &   1 &  -1\\
   -7 &  -1 &   0 &  0 \\
   -1 &   1 &   0 &  0\\
\end{pmatrix}, \quad
T =
\begin{pmatrix}
    1 &   1 &   0 & 0\\
    0 &   1 &   0 &  0\\
    0 &   0 &   0 &  1\\
    0 &   0 &   1 &  -7\\
\end{pmatrix},\\
\text{and} \quad G' = T^t G T =
\begin{pmatrix}
 0 & 0 & 1 & 0\\
 0 & 0 & 0 & 8\\ 
 -1 & 0 & 0 & 0\\
 0 & -8 & 0 & 0\\ 
\end{pmatrix}
\end{split}
\end{align}

The determinant of $G$ is 64. Hence we cannot find a symplectic basis of $H_1^{(0)}(M,\ZZ)$ defined over $\Z$. However, we do the  basis change given by the transformation matrix $T$ from (\ref{AllTheMatrices}) such that the fundamental matrix of the intersection form with respect to this new basis $B^\prime = (b'_1,b'_2,b'_3,b'_4)$ becomes the matrix $G'$ in (\ref{AllTheMatrices}).  
%conjugate the group $H$ into   $\Sp(4,\mathbb{Z})$ we transform the basis with the following transformation:
Define $\Lambda$ to be the lattice generated by $C = (8b'_1, b'_2, b'_3, b'_4)$. Then the intersection form on $\Lambda$ has the fundamental matrix $\tilde{G}$ in (\ref{EvenMoreMatrices}).
\begin{equation}
\tilde{G} = \begin{pmatrix}
 0 & 0 & 8 & 0\\
 0 & 0 & 0 & 8\\ 
 -8 & 0 & 0 & 0\\
 0 & -8 & 0 & 0\\ 
\end{pmatrix}, \quad
\tilde{T} = 
\begin{pmatrix}
  8&0&0&0\\
  0&1&0&0 \\
  0&0&1&0 \\
  0&0&0&1 
\end{pmatrix}
\end{equation}
\label{EvenMoreMatrices}
Observe that a matrix $A$ in $\SL(4,\ZZ)$ lies in $\SP(4,\ZZ)$ if and only if the corresponding linear transformation respects $\tilde{G}$ i.e., if and only if $A^t\cdot \tilde{G} \cdot A =\tilde{G}$.  We now have to restrict to those elements in $H'=T^{-1}H T$ which stabilize the lattice $\Lambda$ i.e., we consider
  \[
  U' = \textrm{Stab}_{H'}(\Lambda) = \{A \in H'|\; \forall x \in \Lambda: A\cdot x \in \Lambda\}
  .\]
  Computing $U' = \textrm{Stab}_{H'}(\Lambda)$ we obtain that it is a subgroup of index 48 in $H'$. Now we do the basis change described by the transformation matrix $\tilde{T}$ in (\ref{EvenMoreMatrices}) in order to express the elements of $U'$ with respect to the basis $C$. In this way we obtain $\tilde{U} = \tilde{T}^{-1}\cdot U' \cdot \tilde{T}$ as subgroup of $\SP(4,\ZZ)$. We find a transvection $Trv$ in this group. Using \cite{Detinko2018ZariskiDA} (with this transvection) we finally obtain as result that $\tilde{U}$ is a congruence subgroup of level 16 and of index 46080 in $\SP(4,\ZZ)$.

\subsection{A non Zariski-dense Kontsevich--Zorich monodromy in genus four}%zzz
In this subsection we consider the origami $\mathcal{O} = (M,\omega)$ defined as 
  \[
  \mathcal{O} = ((2,3,4)(5,7,6), (1,2,3,5,4,6,7))
  \]
  of degree 7 and genus 4 in stratum $\mathcal{H}(6)$, see Figure~\ref{fig:SpecialOrigamiGenus4}.
   Cf.\cite[\texttt{Example4\_AKZ}]{KattlerGIT2023}) for the computation of the following results.

\begin{figure}[htb]
\centering
% First cylinder decomposition (1,2)
\begin{tikzpicture}[scale=0.7]
%% squares
  \draw (0,0) rectangle ++(1,1);
  \draw (0,1) rectangle ++(1,1);
  \draw (1,1) rectangle ++(1,1);
  \draw (2,1) rectangle ++(1,1);
  \draw (2,2) rectangle ++(1,1);
  \draw (3,2) rectangle ++(1,1);
  \draw (4,2) rectangle ++(1,1);
  \draw (.5,.5) node {1};
  \draw (.5,1.5) node {2};
  \draw (1.5,1.5) node {3};
  \draw (2.5,1.5) node {4};
  \draw (2.5,2.5) node {6};
  \draw (3.5,2.5) node {5};
  \draw (4.5,2.5) node {7};        
  %\draw (.5,1.9) node {/};
  \draw (.5,1.98) node {\uproman{1}};
  \draw (1.5,.98) node {\uproman{1}};
  \draw (1.5,1.98) node {\uproman{2}};
  \draw (3.5,1.98) node {\uproman{2}};
  \draw (2.5,.98) node {\uproman{3}};
  \draw (3.5,2.98) node {\uproman{3}};
  \draw (2.5,2.98) node {\uproman{4}};
  \draw (4.5,1.98) node {\uproman{4}};
  \draw (4.5,2.98) node {\uproman{5}};
  \draw (.5,-.02) node {\uproman{5}};          
\end{tikzpicture}
\caption{Origami $\mathcal{O}$ in $\mathcal{H}(6)$: edges with same labels and unlabeled opposite edges are glued}
\label{fig:SpecialOrigamiGenus4}
\end{figure}

Its Veech group $\VG(M,\omega)$ is a subgroup of index 8 in $\SL(2,\Z)$ generated by the following two parabolic matrices:
  \[
  A_1 =  \begin{pmatrix} 1&3\\0&1 \end{pmatrix} \quad \mbox{and} \quad A_2 = \begin{pmatrix} 1&0\\-1\ &1\end{pmatrix}.
  \]
The Schreier coset graph (cf.~\cite[2.3]{MKS2004})  of $\VG(M,\omega)$ with respect to the generators \(S\) and \(T\) -- as defined in \ref{TheMatrixSandT} -- of \(\SL(2,\ZZ)\) is shown in Figure~\ref{fig:cosetgraph} on the left side. Observe that $\VG(M,\omega)$ does not contain the matrix $-\Id$. Hence its image $\PSL(M,\omega)$ in $\PSL(2,\Z)$ is a subgroup of index 4. Its coset graph is shown in  Figure~\ref{fig:cosetgraph} on the right side. $\PSL(M,\omega)$ has two cusps of width 3 and width 1, respectively. They correspond to the $\overline{T}$-orbits where $\overline{T}$ is the image of $T$ in $\PSL(2,\Z)$. 
\begin{figure}[htb]
\centering
\begin{tikzpicture}[scale=0.7]
   %Vertices of the first graph:
  \node[shape=circle,draw=black] (A) at (0,0) {8};
  \node[shape=circle,draw=black] (B) at (4,0) {7};
  \node[shape=circle,draw=black] (C) at (2,1) {6};
  \node[shape=circle,draw=black] (D) at (1.5,2.5) {3};
  \node[shape=circle,draw=black] (E) at (2.5,2.5) {2};
  \node[shape=circle,draw=black] (F) at (2,4) {1} ;
  \node[shape=circle,draw=black] (G) at (0,5) {5};
  \node[shape=circle,draw=black] (H) at (4,5) {4};
  %Vertices of the second graph:
  \node[shape=circle,draw=black] (I) at (8,1) {5/7};
  \node[shape=circle,draw=black] (J) at (12,1) {4/8};
  \node[shape=circle,draw=black] (K) at (10,2.5) {1/6};
  \node[shape=circle,draw=black] (L) at (10,4.5) {2/3};

  %Edges of the first graph:
  \path [->] (C) edge node[left] {} (A);
  \path [->] (B) edge node[left] {} (C);
  \path [->] (A) edge node[left] {} (B);
  \path [->, dashed] (C) edge node[left] {} (D);
  \path [->, dashed] (E) edge node[right] {} (C);
  \path [->, dashed] (D) edge node[left] {} (F);
  \path [->, dashed] (F) edge node {} (E);
  \path [->](F) edge node {} (H);
  \path [->](H) edge node {} (G);
  \path [->](G) edge node {} (F);    
  \path [->,dashed, bend right=30](G) edge node {} (A);
  \path [->,dashed, bend right=30](B) edge node {} (H);
  \path [->,dashed, bend right=20](A) edge node {} (B);
  \path [->,dashed, bend right=20](H) edge node {} (G);
  \path [->, loop left](D) edge node {} (D);
  \path [->, loop right](E) edge node {} (E);
  %Edges of the second graph:
  \path [->] (I) edge node {} (K);
  \path [->] (K) edge node {} (J);
  \path [->] (J) edge node {} (I);
  \path [->, loop above] (L) edge node {} (L);
  \path [<->, dashed, bend left = 20] (I) edge node {} (K);
  \path [<->, dashed, bend left = 20] (K) edge node {} (J);
  \path [<->, dashed, bend left = 20] (J) edge node {} (I);
  \path [<->, dashed] (K) edge node {} (L);
\end{tikzpicture}
\caption{The Schreier coset graph of $\VG(M,\omega)$ in $\SL(2,\Z)$ (left side) and of $\PSL(M,\omega)$ in $\PSL(2,\ZZ)$ (right side). The dashed arrows show the action of $S$, the non-dashed arrows the action of $T$, cf.(~\ref{TheMatrixSandT})}
  \label{fig:cosetgraph}
\end{figure}

We denote in the following by $e_i$ the lower edge of the square in $\mathcal{O}$ labeled with $i$ and with $e_{i+7}$ the left edge of the square labeled with $i$. Then
%\commman{absetzen?}
  \[
  B = \{e_1, e_3, e_4, e_5, e_7, e_8, e_9, e_{13}\}
  \]
forms a basis of $H_1(M,\Z)$. The two generators $A_1$ and $A_2$ of the Veech group $\VG(M,\omega)$ act on $H_1(M,\Z)$ with respect to this basis by the two matrices:
  \[
  D_1 = \begin{pmatrix}1&0&0&0&0&3&1&1\\0&1&0&0&0&0&1&0\\0&0&1&0&0&0&1&1\\0&0&0&1&0&0&0&0\\0&0&0&0&1&0&0&1\\0&0&0&0&0&1&0&0\\0&0&0&0&0&0&1&0\\0&0&0&0&0&0&0&1\end{pmatrix}, \quad 
  D_2 = \begin{pmatrix} 1&1&1&1&1&-1&-1&0\\0&0&0&0&0&1&0&0\\0&1&1&1&0&0&-1&0\\0&0&-1&-1&0&0&1&0\\0&-1&0&0&0&0&1&0\\0&0&0&0&-1&0&0&1\\-1&0&0&-1&0&1&0&0\\0&-1&-1&0&0&0&1&0\end{pmatrix}
   \]
    
The non tautological part $H_1^{(0)}(M,\Z)$ of the homology  has the following basis:
   \[
   \tilde{B} = \{v_1 = e_1 - e_7,\; v_2 = e_3 - e_7,\; v_3 = e_4 - e_7,\; v_4 = e_5 - e_7,\; v_5 = e_8 - e_{13},\; v_6 = e_9 - e_{13}\}
   \]
 The action of  $A_1$ and $A_2$ on  $H_1^{(0)}(M,\Z)$ with respect to $\tilde{B}$ is then given by the following two matrices:
   \[
   E_1 = \begin{pmatrix}1&0&0&0&2&0\\0&1&0&0&0&1\\0&0&1&0&-1&0\\0&0&0&1&0&0\\0&0&0&0&1&0\\0&0&0&0&0&1\end{pmatrix}, 
   \quad
   E_2 = \begin{pmatrix}0&0&0&0&-1&-1\\0&0&0&0&1&0\\0&1&1&1&0&-1\\0&0&-1&-1&0&1\\1&1&1&1&-1&-1\\-1&0&0&-1&1&0\end{pmatrix}
   \]
Hence for this example the Kontsevich--Zorich monodromy is isomorphic to the subgroup of $\SL(6,\ZZ)$ generated by $E_1$ and $E_2$.
A computation with GAP gives us that the $\Q$-algebra $\Q(C_1,C_2)$ generated by $E_1$ and $E_2$ has dimension 18 and thus not the full dimension. We conclude that the Kontsevich--Zorich monodromy is not Zariski-dense in this case.
%     Numerical computations of the Lyapunov exponents with \cite{}(todo: cite sage-package) suggest that the Lyapunov exponents are $1, \frac{3}{7}, \frac{3}{7}, 0, 0, -\frac{3}{7}, -\frac{3}{7}, -1$. By a result of F\"urstenberg and Margulis the Lyapunov exponents would be positive and would have multiplicity 1 if the the shadow Veech group was dense. Therefore we conjecture that for this example the shadow Veech group is not dense.
     %To our knowledge this is an example of smallest degree among the origamis without non-trivial automorphisms, i.e. without affine homeomorphisms up to the identity whose dervative is $I$ or $-I$, which has this property.
%All origamis of smaller degree are of genus $g \leq 3$.

%%%%%%%%%%%%%%%%%%%%%%%%%%%%%%%%%%%%%%%%%%%%
%%%%%%%%%% Appendix
%%%%%%%%%%%%%%%%%%%%%%%%%%%%%%%%%%%%%%%%%%%%

%%%%%%%%%%%%%%%%%%%%%%%%%%%%%%%%%%%%%%%%%%%
%%%%%%%%%%%%%%%%%%%%%%%%%%%%%%%%%%%%%%%%%%%%%%
%%%%%%%%%%%%%%%%%%%%%%%%%%%%%%%%%%%%%%%%%%%%%%%%

\appendix

\section{Example of Kontsevich--Zorich monodromy in the Prym locus of $\mathcal{H}^{odd}(4)$}\label{a.examplePrym}

Recall that an origami $(M,\omega)$ in the Prym locus of $\mathcal{H}^{odd}(4)$ has Kontsevich--Zorich monodromy included in $\textrm{Sp}(H_1^+)\times \textrm{Sp}(H_1^-)\simeq \SL(2,\mathbb{Z})\times \SL(2,\mathbb{Z})$ because the affine homeomorphisms of $(M,\omega)$ respect the splitting $H_1^{(0)}(M,\mathbb{Q}) = H_1^+\oplus H_1^-$ associated to the eigenspaces (of the eigenvalues $\pm 1$) of the anti-automorphism of $(M,\omega)$ (see, e.g., \cite{LN}). In particular, the Kontsevich--Zorich monodromy of an origami in the Prym locus of $\mathcal{H}^{odd}(4)$ is not Zariski dense in $\textrm{Sp}(H_1^{(0)}(M,\mathbb{R}))\simeq \textrm{Sp}(4,\mathbb{R})$, but we can still ask about the arithmeticity of Kontsevich--Zorich monodromies in this context. The answer to this question is not clear in general. 

For example, let us consider the case of the origami $\mathcal{E}_5$ associated to the permutations $h=(1,2)(3)(4,5)$ and $v=(1)(2,4,3)(5)$ to disclose the kind of question one finds by studying this locus.

\begin{figure}[!ht]
\centering
\tikzset{every picture/.style={line width=0.75pt}} %set default line width to 0.75pt        

\begin{tikzpicture}[x=0.75pt,y=0.75pt,yscale=-0.4,xscale=0.4]

%Shape: Rectangle [id:dp009682212010275304] 
\draw   (100,8) -- (189.5,8) -- (189.5,97.33) -- (100,97.33) -- cycle ;
%Shape: Rectangle [id:dp31526888189602764] 
\draw   (189.5,8) -- (279,8) -- (279,97.33) -- (189.5,97.33) -- cycle ;
%Shape: Rectangle [id:dp8810539047984844] 
\draw   (189.5,97.33) -- (279,97.33) -- (279,186.67) -- (189.5,186.67) -- cycle ;
%Shape: Rectangle [id:dp7955501123545097] 
\draw   (189.5,186.67) -- (279,186.67) -- (279,276) -- (189.5,276) -- cycle ;
%Shape: Rectangle [id:dp6175447050113175] 
\draw   (279,186.67) -- (368.5,186.67) -- (368.5,276) -- (279,276) -- cycle ;
\end{tikzpicture}
\caption{The origami in the Prym locus}
\end{figure}

By using SageMath, one has that the Veech group of $\mathcal{E}_5$ is an index 10 subgroup of $\SL(2,\mathbb{Z})$ generated by the matrices
$$\left(\begin{array}{cc}1&2\\0&1\end{array}\right), \left(\begin{array}{cc}5&-2\\3&-1\end{array}\right), \left(\begin{array}{cc}-4&3\\-7&5\end{array}\right).$$
Since
$$\left(\begin{array}{cc}5&-2\\3&-1\end{array}\right) = \left(\begin{array}{cc}1&2\\0&1\end{array}\right)\left(\begin{array}{cc}-1&0\\3&-1\end{array}\right) \quad \textrm{and} \quad \left(\begin{array}{cc}5&-2\\3&-1\end{array}\right)\left(\begin{array}{cc}-4&3\\-7&5\end{array}\right)= \left(\begin{array}{cc}-6&5\\-5&4\end{array}\right)$$
the Veech group of $\mathcal{E}_5$ is also generated by
$$\left(\begin{array}{cc}1&2\\0&1\end{array}\right), \left(\begin{array}{cc}1&0\\3&1\end{array}\right), \left(\begin{array}{cc}-1&0\\0&-1\end{array}\right), \left(\begin{array}{cc}6&-5\\5&-4\end{array}\right).$$
Observe that $\mathcal{E}_5$ has three horizontal cylinders with waist curves $\sigma_1$, $\sigma_0$, $\sigma_2$ with holonomies $(2,0)$, $(1,0)$, $(2,0)$, and three vertical cylinders with waist curves $\zeta_1$, $\zeta_0$, $\zeta_2$ with holonomies $(0,1)$, $(0,3)$, $(0,1)$, so that $H_1^{(0)}(\mathcal{E}_5,\mathbb{Q})$ has a basis consisting of $\Sigma_i=\sigma_i-2\sigma_0$, $Z_i=3\zeta_i-\zeta_0$ for $i=1,2$. Moreover, $-\Id$ acts on $H_1^{(0)}(\mathcal{E}_5,\mathbb{Q})$ by $\Sigma_i\mapsto -\Sigma_{3-i}$, $Z_i\mapsto-Z_{3-i}$ for $i=1,2$, so that
  \[
  H_1^{(0)}(\mathcal{E}_5,\mathbb{Q}) = H_1^+\oplus H_1^-
  \]
where $H_1^+$ is generated by $\Sigma^+=\Sigma_1-\Sigma_2$, $Z^+=Z_1-Z_2$, and $H_1^-$ is spanned by $\Sigma^-=\Sigma_1+\Sigma_2$ and $Z^-=Z_1+Z_2$. A direct computation reveals that the matrices $A$ and $B$ of the actions of $\left(\begin{array}{cc}1&2\\0&1\end{array}\right)$ and $\left(\begin{array}{cc}1&0\\3&1\end{array}\right)$ on the basis $\{\Sigma^+,Z^+, \Sigma^-,Z^-\}$ of $H_1^+\oplus H_1^-$ are
  \[
  A=\left(\begin{array}{cccc}1&3&0&0\\0&1&0&0\\0&0&1&1\\0&0&0&1\end{array}\right) \quad \textrm{and} \quad  B=\left(\begin{array}{cccc}1&0&0&0\\1&1&0&0\\0&0&1&0\\0&0&1&1\end{array}\right).
  \]
Moreover, the matrix $\left(\begin{array}{cc}6&-5\\5&-4\end{array}\right)$ acts trivially on $H_1^{(0)}(M,\mathbb{Q})$ because it induces a Dehn multitwist in the one-cylinder direction $(1,1)$. 

The action of $A$ and $B$ restricted to $H_1^-$ generate a copy of $\SL(2,\mathbb{Z})$ since the matrices 
 \[
 T=\left(\begin{array}{cc}1&1\\0&1\end{array}\right)\quad \text{and} \quad 
 S= \left(\begin{array}{cc}1&0\\1&1\end{array}\right)
 \]
are generators. Similarly, $A$ and $B$ restricted to $H_1^+$ generate a group $\Gamma$ which is a copy of the finite-index\footnote{The fact that $\Gamma$ has finite-index in $\SL(2,\mathbb{Z})$ is a general feature of the origamis in the Prym locus of $\mathcal{H}^{odd}(4)$: in fact, the arguments (due to M\"oller) in Appendix \ref{a.ArithmShadVee} can also be used to check that \emph{both} projections of the Kontsevich--Zorich monodromy to $\textrm{Sp}(H_1^+)$ and $\textrm{Sp}(H_1^-)$ have finite index in $\SL(2,\mathbb{Z})$.} subgroup $\Gamma_1(3)$ of $\SL(2,\mathbb{Z})$.

However, we have not further investigated how the group spanned by $A$ and $B$ sits inside the product $\Gamma \times \SL(2,\mathbb{Z})$. 

\section{Kontsevich--Zorich monodromy for genus two origamis}\label{a.ArithmShadVee}

The goal of this section is to prove Theorem \ref{Thm:ArithmShadVee} (an unpublished result of Martin M\"oller), stating that the Kontsevich--Zorich monodromy of an origami in genus \(g=2\) has finite index in \(\mbox{SL}(2,\mathbb{Z})\).

\subsection{Preliminaries} 

\subsubsection{Local systems}
The next Theorem is very important for the study of the representations in our context. For a proof see for example \cite{voisinHodge2003} Section 3.1.1.
\begin{theorem}\label{thm:monodromrep}
Let \(R\) be a ring and let \(X\) be a path-connected, locally simply connected topological space with a base point \(x\in X\). Then there is an equivalence between the category of \(R\)-local systems on \(X\) and the category of \(R\)-modules with \(\pi_1(X,x)\)-left action, given by the functor
  \[
    \mathbb{L} \longmapsto \mathbb{L}_x,
  \]
where  \(\mathbb{L}_x\) denotes the stalk of the \(R\)-local system \(\mathbb{L}\) at the base point \(x\in X\).
\end{theorem}
The mapping on \(\mathbb{L}_x\), induced by the left action of \(\pi_1(X,x)\), is called \textit{monodromy representation}.

\subsubsection{Translation structures} 

Let \(X\) be a compact Riemann surface of genus \(g\) with finitely many marked points \(\Sigma\subset X\). A \textit{translation structure} on \(X{\setminus}\Sigma\) is determined by an atlas \((V,\phi)\) with an open covering \(V=(V_i)_{i\in I} \) of \(X{\setminus}\Sigma\) and with charts \(\phi_i\colon V_i\to \mathbb{C}\) such that the transition maps \(\phi_{i,j}\colon \mathbb{C}\to \mathbb{C}\) are of the form
  \[\phi_{i,j}(z_i)=z_j+c_{i,j}\]
on the intersection \(V_i\cap V_j\).

Denote by \(\Omega T_g\) the bundle over the Teichm\"uller space \(T_g\) whose points parametrize pairs \((X,\omega)\) of a compact, marked Riemann surface \(X\) together with a non-zero holomorphic 1-form \(\omega\) on $X$.

For a point \((X,\omega)\in\Omega T_g\), let \(Z(\omega)\subset X\) be the set of zeros of \(\omega\). We can define a translation chart on \(X{\setminus}Z(\omega)\) in the following way: Choose a point \(x\in X{\setminus}Z(\omega)\), now for every simply connected \(U\subset X{\setminus}Z(\omega)\) define a map
 \[
 \phi_U\colon U \longrightarrow \mathbb{C},\quad y \longmapsto \int_x^y \omega.
 \]
Then \((U,\phi_U)\) is one of the translation charts.

On the other hand given a compact Riemann surface \(X\) of genus \(g\) with a finite set of points \(\Sigma\subset X\) and a translation atlas \((V,\phi_i)\) for \(X{\setminus}\Sigma\), we can pull back the holomorphic 1-form \(dz\) on \(\mathbb{C}\) via the charts \((V_i,\phi_i)\) to get a holomorphic 1-form \(\omega{'}\) on \(X{\setminus}\Sigma\). It is now easy to extend \(\omega{'}\) to a holomorphic 1-form \(\omega\) on \(X\).

The following proposition is standard and plausible considering the last arguments:
\begin{proposition}
On compact Riemann surfaces, translation structures are in one-to-one correspondence with holomorphic 1-forms which are not identically zero.
\end{proposition}

\subsubsection{Teichm\"uller curves}

We first want to give the definition and construction of Teichm\"uller curves. As additional literature we can recommend \cite{Lochak05Oacimsoc} as well as \cite{McMullenBilliard} section 2 and 3 and for a more intense approach have a look in \cite{TMCMMlect} section 2 and 3.

Let \(S\) be a compact hyperbolic Riemann surface of genus \(g\). We denote by \(T_g(S)\) the \textit{Teichm\"uller space} of compact Riemann surfaces \(X\) of genus \(g\geq 2\) with markings \(m\colon S \to X\). We write \(\Gamma_g(S)\) for the mapping class group of \(T_g(S)\) and  \(M_g(S)\) for the \textit{moduli space of compact Riemann surfaces of genus \(g\)}. In most of the cases we are not interested in the base point \((S,\text{id})\) of \(T_g(S)\) and write \(T_g\) for \(T_g(S)\) respectively \(\Gamma_g\) and \(M_g\) for \(\Gamma_g(S)\) and \(M_g(S)\).

We will now explain how to construct Teichm\"uller curves from certain points \((X,\omega)\in \Omega T_g\). We can define an \(\mbox{SL}(2,\mathbb{R})\)-action on \(\Omega T_g\) in the following way:
Given \(A={\tiny\begin{pmatrix} a & b \\ c & d \end{pmatrix}} \in \mbox{SL}(2,\mathbb{R})\) and \((X,\omega)\in\Omega T_g\) consider the harmonic 1-form
  \begin{displaymath}
    \omega_A= \begin{pmatrix} 1 & i \end{pmatrix} 
              \begin{pmatrix} a & b \\ c & d \end{pmatrix}
              \begin{pmatrix} \mbox{Re}(\omega) \\ \mbox{Im}(\omega)\end{pmatrix}
   \end{displaymath}
on \(X\). 

We can equip \(X\) with a new complex structure, with respect to which \(\omega_A\) is again a holomorphic 1-form. This complex structure delivers a new Riemann surface \(X_A\) and we define \(A.(X,\omega)=(X_A,\omega_A)\in \Omega T_g\). 

The fibers of the projection \(\Omega T_g\to T_g\) are stabilized by \(\mbox{SO}_2(\mathbb{R})\) and for \((X,\omega)\in \Omega T_g\) the projection of the orbit \(\Delta=\mbox{SL}(2,\mathbb{R}).(X,\omega)\subset \Omega T_g\) to \(T_g\) is an embedding. Thus we get for every translation surface \((X,\omega)\in \Omega T_g\) a map
  \begin{displaymath}
   g_\omega\colon \mathbb{H}=\mbox{SL}(2,\mathbb{R}){/}\mbox{SO}_2(\mathbb{R}) \longrightarrow T_g,
  \end{displaymath}
which is a geodesic embedding for the Teichm\"uller metric on \(T_g\), see \cite{earle1997teichmuller}.
By composing \(g_\omega\) with the projection map \(\pi_g\colon T_g\to M_g\), we get a map
  \begin{displaymath}
   f_\omega\colon \mathbb{H}=\mbox{SL}(2,\mathbb{R}){/}\mbox{SO}_2(\mathbb{R}) \longrightarrow M_g.
  \end{displaymath}
The global stabilizer of the action of \(\Gamma_g\) on \(\Delta=g_\omega(\mathbb{H})\subset T_g\) is the group \(\mbox{Aff}^+(X,\omega)\) of orientation preserving diffeomorphisms, which are affine with respect to the translation structure defined by \(\omega\) (see \cite[Theorem 1]{earle1997teichmuller}). We denote the pointwise stabilizer of the action of \(\Gamma_g\) on \(\Delta\) by $H(X,\omega)$. The quotient  \(\mbox{Aff}^+(X,\omega){/}H(X,\omega)\) can be identified with the group
  \[
   \mbox{SL}(X,\omega)=\mbox{Stab}(f_\omega)=\{A\in \mbox{Aut}(\mathbb{H})\mid f_\omega(A\,t)= f_\omega(t)\ \forall t\in \mathbb{H}\}.
  \]
We want to point out that $R\cdot \text{SL}(X,\omega)\cdot R$ coincides with the projective Veech group of \((X,\omega)\), where \(R=\text{diag}(1,-1)\) (\cite{McMullenBilliard} Prop. 3.2). Since \(g_{\omega}\) is injective we have an isomorphism \(\mathbb{H}{/}\mbox{SL}(X,\omega)\simeq\Delta{/}\mbox{Aff}^+(X,\omega)\).
The map $f_\omega\colon \mathbb{H}\to M_g$ clearly factors through its stabilizer and we call 
  \[
  j_\omega\colon\mathbb{H}{/}\mbox{SL}(X,\omega)\to M_g
  \]
or \(C_1=\Delta{/}\mbox{Aff}^+(X,\omega)\) a \textit{Teichmüller curve} if one of the following equivalent statements is true: %(\cite{McMullenBilliard}, section 2)
\begin{enumerate}
  \item[(i)]
  The stabilizer group \(\mbox{SL}(X,\omega)\subset \mbox{Aut}(\mathbb{H})\) is a lattice.
  \item[(ii)]
  The manifold \(\Delta{/}\mbox{Aff}^+(X,\omega)\) has finite volume. %or equivalent has finitely many
  %cusps \commref{\sout{This formulation is confusing}} and no funnel ends or flaring ends \commref{\sout{funnel ends is more usual}}.
\end{enumerate}
In this case the map \(j_\omega\colon C_1\to M_g\) is proper and generically injective. Its image \(j_\omega(C_1)\subset M_g\) is an algebraic curve, whose normalization is \(C_1\) (see \cite{McMullenBilliard} section 2). If the curve \(C_1\) was constructed from a pair \((X,\omega)\in \Omega T_g\), we say that \((X,\omega)\) \textit{generates the Teichmüller curve \(C_1\)}. The construction made above is visualized in the following diagram:
\begin{center}
\begin{tikzpicture}[scale=1.5]
\draw(-1,0) node[left]{$\mathbb{H}$};
\draw[->](-1,0) -- node[above]{$g_\omega$} (1,0) node[right]{$\Delta$};
\draw[->](1.3,0) -- (3,0) node[right]{$T_g$};
\draw[->](3.2,-0.2) -- node[right]{$\pi_g$} (3.2,-1.2) node[below]{$M_g$};
\draw[->](1.2,-0.2) -- (1.2,-1.2) node[below]{$C_1=\Delta{/}\text{Aff}^+(X,\omega)$};
\draw[->](-1.2,-0.2) -- (-1.2,-1.2) node[below]{$\mathbb{H}{/}\text{SL}(X,\omega)$};
\draw[->](-0.55,-1.4) -- (0.15,-1.4);
\draw[->](2.3,-1.4) -- node[above]{$j_\omega$} (3,-1.4);
\end{tikzpicture}
\end{center}
\begin{remark}\label{Rem:Veefiind}
If \((X,\omega)=\mathcal{O}\) defines an origami, the group \(\mbox{SL}(X,\omega)\) is a finite index subgroup of \(\mbox{SL}(2,\mathbb{Z})\) (see \cite{GuJu00}). This implies that \(\mbox{SL}(X,\omega)\) is a lattice in \(\mbox{Aut}(\mathbb{H})\) and hence every origami defines a Teichmüller curve. We will call a Teichmüller curve, which comes from an origami, a \textit{origami curve}.
\end{remark} 

\subsubsection{Family of curves}\label{sect:famteicur}

We recall the construction of the \textit{family of curves coming from a Teichmüller curve} as in section 1.4 of \cite{MoellerVar06} or section 3.1 of \cite{TMCMMlect}.

Let \(j_\omega\colon C_1\to M_g(S)\) be a Teichmüller curve, which comes from a pair \((X,\omega)\in \Omega T_g(S)\). Let \(M_g^{[3]}=T_g(S)/\Gamma_g^{[3]}\) be the moduli space of curves with level-\(3\) structure. Here \(\Gamma_g^{[3]}\) is the kernel of the action of \(\Gamma_g(S)\) on \(H^1(S,\mathbb{Z}/3\mathbb{Z})\). We have that \(\Gamma_g^{[3]}\leq \Gamma_g\) is a torsion free finite index subgroup. Furthermore \(M_g^{[3]}\) is a fine moduli space. Hence there is a universal family \(f^{[3]}\colon\mathcal{X}^{[3]}_\text{univ}\to M_g^{[3]}\) over \(M_g^{[3]}\). 

Let \(\Gamma_1\) be the stabilizer of \(\Delta=g_\omega(\mathbb{H})\) for the action of \(\Gamma_g^{[3]}\) on \(T_g(S)\) and define \(C_1^{[3]}=\Delta/\Gamma_1\). The inclusion \(\Delta\hookrightarrow T_g(S)\) induces a map \(C_1^{[3]}\rightarrow M_g^{[3]}\) on the quotients. The moduli space \(M_g^{[3]}\) admits a universal family \(f^{[3]}\colon\mathcal{X}^{[3]}_\text{univ}\to M_g^{[3]}\), which we can pull back via \(C_1^{[3]}\rightarrow M_g^{[3]}\) to get a family of curves \(\mathcal{X}^{[3]}_{C_1}\to C_1^{[3]}\). 

We can now pass to a finite index subgroup \(\Gamma\leq\Gamma_1\), such that the pull back of the universal family via the map \(C=\Delta{/}\Gamma\to M_g^{[3]}\) delivers a family of curves \(f\colon \mathcal{X}\to C\), which can be completed to a stable family \(\overline f\colon\overline{\mathcal{X}}\to \overline C\) over a smooth completion (smooth compactification) \(\overline C=\overline{\Delta{/}\Gamma}\) of \(C\). 

This implies that monodromies around the cusps \(\partial C=\overline C \setminus C\) are unipotent. We call such a family \(f\colon \mathcal{X} \to C\) a \textit{family of curves coming from a Teichmüller curve}.

The whole situation is visualized in the following diagram.
\begin{center}
\begin{tikzpicture}[scale=1.5]
\draw(0,0) node[left]{$C_1=\Delta{/}\text{Aff}^+(X,\omega)$};
\draw(1,0) node[right]{$M_g$};
\draw(1,1) node[right]{$M_g^{[3]}$};
\draw(0,1) node[left]{$C_1^{[3]}=\Delta{/}\Gamma_1$};
\draw[->](0,0) -- (1,0); 
\draw[->](0,1) -- (1,1); 
\draw[<-](1.2,0.2) -- (1.2,0.8);
\draw[<-](-0.5,0.2) -- (-0.5,0.8);
\draw[<-](-0.5,1.2)-- node[right]{$f^{[3]}$} (-0.5,2.7) node[above]{$\mathcal{X}_{C_1}^{[3]}$};
\draw[<-](1.2,1.2)-- node[right]{$f^{[3]}$} (1.2,2.7) node[above]{$\mathcal{X}_\text{univ}^{[3]}$};
\draw[->](-0.2,2.9) -- (0.88,2.9);
\draw[<-](-0.8,2.9) -- (-2.3,2.9) node[left]{$\mathcal{X}$};
\draw[->](-2.5,2.7)-- node[right]{$f$} (-2.5,1.8) node[below]{$C=\Delta{/}\Gamma$};
\draw[->](-2.2,1.4) -- (-1.5,1);
\draw[->] (-2.7,2.9) -- (-3.7,2.9) node[left]{$\overline{\mathcal{X}}$};
\draw[->] (-3.9,2.7) --node[right]{$\overline f$} (-3.9,1.8) node[below]{$\overline C$};
\draw[<-] (-3.7,1.6) -- (-3.1,1.6);
\end{tikzpicture}
\end{center}

\begin{remark}\label{Rem:PropertC}
We want to record two very important properties of the finite index group \(\Gamma\leq \mbox{Aff}^+(X,\omega)\) from above:
\begin{enumerate}
  \item The group \(\Gamma\) is torsion free and hence we can identify it with the fundamental group
  \(\pi_1(C,c)\) of the curve \(C=\Delta/\Gamma\).
  \item The local monodromy of \(\Gamma\) around the cusps \(\overline C\setminus C\) is unipotent.
\end{enumerate}
\end{remark}

\subsubsection{Variations of Hodge structures}\label{sect:VarOfHS}
Let \(C_1\to M_g\) be a Teichmüller curve generated by the pair \((X,\omega)\in \Omega T_g\) and let \(f\colon \mathcal{X}\to C=\Delta/\Gamma\) be the associated family of curves, constructed as in section \ref{sect:famteicur}. We can associate to $f\colon \mathcal{X} \to C$ the local system \(\mathbb{V}_\mathbb{Z}=R^1f_*\mathbb{Z}_\mathcal{X}\), where \(\mathbb{Z}_{\mathcal{X}}\) is the constant sheaf of stalk \(\mathbb{Z}\) on \(\mathcal{X}\). Due to Theorem \ref{thm:monodromrep} we have a monodromy representation for the local system $\mathbb{V}_\mathbb{Z}$ which we can describe as follows. Fix a base point \(c\in C(\mathbb{C})\). Without loss of generality we can assume that the fiber of \(f\) over \(c\) is the Riemann surface \(X\) from above and thus \((\mathbb{V}_\mathbb{Z})_c=H^1(X,\mathbb{Z})\). We identify $\pi_1(C,c)$ with the torsion free subgroup \(\Gamma\leq\mbox{Aff}^+(X,\omega)\) as before. Hence every $a\in \pi(C,c)$ is giving us an affine homeomorphism $\varphi_a\in \Gamma$. Then the monodromy operation of $a$ on $H^1(X,\mathbb{Z})$ is given by the pull back map $\varphi_a^*$ defined by $\varphi_a\in \Gamma$ (c.f. \cite[Section 3.1.2]{voisinHodge2003}).

The bundle
  \[
 \mathcal{V}^{1,0}=R^0f_*(\Omega_{\mathcal{X}{/}C}^1)= f_* \Omega_{\mathcal{X}/C}^1
  \]
is the $(1,0)$-part of a Hodge filtration of weight one on the holomorphic bundle \(\mathcal{V}=R^1f_*\mathbb{Z}_\mathcal{X}\otimes_\mathbb{Z} \mathcal{O}_{C}\), which induces the Hodge decomposition on its fibers (see \cite{DeligneTrdeGrif69})

In Remark \ref{Rem:PropertC} we already recorded that the monodromy of \(\pi_1(C,c)\) is locally unipotent around the cusps \(\overline C\setminus C\) of \(C\). Next we want to state a result of Deligne, where we need this fact.

Let \(\mathbb{V}\) be a \(\mathbb{C}\)-local system of rank $k$ on the punctured unit disc \(\mathbb{D}^*=\mathbb{D}{\setminus}\{0\}\), which has unipotent monodromy representation around \(0\). Let \((\mathcal{V}=\mathbb{V}\otimes_{\mathbb{C}} \mathcal{O}_{\mathbb{D}^*},\nabla)\) be the corresponding vector bundle with flat (in particular holomorphic) connection 
  \[
  \nabla\colon \mathcal{V}\longrightarrow \Omega_{\mathbb{D}^*}^1\otimes_{\mathcal{O}_{\mathbb{D}^*}} \mathcal{V}.
  \]
  
\begin{proposition}[Deligne \cite{DeligneEqDif70}]\label{Prop:DelExt} %II Prop. 5.2
In the situation above there is a unique extension \((\mathcal{V}_\text{ext},\nabla_\text{ext})\) of \((\mathcal{V},\nabla)\) on \(\mathbb{D}\), where \(\mathcal{V}_\text{ext}\) is a locally free \(\mathcal{O}_\mathbb{D}[x^{-1}]\)-module ( with $x$ a coordinate of $\mathbb{D}$) and 
  \[
  \nabla_\text{ext}\colon \mathcal{V}_\text{ext}\longrightarrow 
  \Omega_{\mathbb{D}}\otimes_{\mathcal{O}_\mathbb{D}} \mathcal{V}_\text{ext}
  \]
is a regular, meromorphic connection.
\end{proposition}

Since our curve \(C=\Delta/\Gamma\) has only finitely many cusps \(\partial C=\overline C \setminus C\), we can use Proposition \ref{Prop:DelExt} pointwise, to extend the bundle \(\mathcal{V}=R^1f_*\mathbb{Z}_\mathcal{X}\otimes_\mathbb{Z} \mathcal{O}_{C}\) to a bundle \(\mathcal{V}_\text{ext}\) on the smooth completion \(\overline C\) of \(C\). Furthermore we deduce that the Gauss-Manin connection \(\nabla\) corresponding to \(\mathbb{V}_\mathbb{Z}\otimes_\mathbb{Z}\mathbb{C}\) extends to a regular meromorphic connection \(\nabla_\text{ext}\) at the cusps  \(\partial C\) of $C$.

The family \(f\colon \mathcal{X}\to C\) extends to a family of stable curves \(\overline f\colon \overline{\mathcal X }\to \overline C\) (compare section \ref{sect:famteicur}) and the bundle
  \(
  \mathcal{V}^{1,0}_\text{ext}= f_* \Omega_{\overline{\mathcal X}/\overline C}^1
  \)
extends the bundle $\mathcal{V}^{1,0}$ on $\overline{C}$.
Thus \((\mathbb{V}_\mathbb{Z},\mathcal{V}^{1,0}_\text{ext})\) is a \textit{variation of Hodge structure} (VHS) in the sense of \cite{MoellerVar06} section 2.

We now want to give a polarization for the variation of Hodge structure \((\mathbb{V}_\mathbb{Z},\mathcal{V}^{1,0})\) from above (compare \cite{Kappes16LyapOfBQ} section 2.2).
On \(H^1(X,\mathbb{C})\simeq H^1_{\text{dR}}(X,\mathbb{C})\) we have the natural polarization by the cup-product pairing\footnote{The cup-product pairing \((\cdot,\cdot)\) on \(H^1(X,\mathbb{C})\) is Poincaré dual to the intersection-pairing on \(H_1(X,\mathbb{C})\).} \((\alpha,\beta)=\int_X \alpha\wedge \beta\).
The pairing \((\cdot,\cdot)\) on \(H^1(X,\mathbb{C})\) induces a positive definite hermitian form
  \[
   H\colon H^1(X,\mathbb{C})\times H^1(X,\mathbb{C})\longrightarrow \mathbb{C},
   \quad H(\alpha,\beta)=\int\limits_X \alpha\wedge *\beta,
  \]
for which the Hodge decomposition is orthogonal (here \(*\) denotes the Hodge star operator). 

The cup product pairings on the fibers of 
\(\mathbb{V}_\mathbb{C}=\mathbb{V}_{\mathbb{Z}}\otimes_{\mathbb{Z}} {\mathbb{C}}\) 
glue together to a locally constant bilinear map \(Q\colon \mathbb{V}_\mathbb{C}\otimes \mathbb{V}_\mathbb{C}\to \mathbb{C}_C\), where \(\mathbb{C}_C\) is the constant sheaf of stalk \(\mathbb{C}\) on the curve \(C\). 
The map \(Q\) induces a locally constant hermitian form \(\psi(v,w)=i/2\cdot Q(v,\overline w)\) on \(\mathbb{V}_\mathbb{C}\), for which the decomposition
  \[
  \mathbb{V}_\mathbb{Z}\otimes_\mathbb{Z}\mathcal{O}_C=\mathcal{V}^{1,0} \oplus \mathcal{V}^{0,1}
  \]
is orthogonal. For \(v\in \mathbb{V}_\mathbb{R}\) we can find an element \(w\in \mathcal{V}^{1,0}\) with \(v=\mbox{Re}(w)\) and define the \textit{Hodge Norm} as \(\|v\|=\sqrt{\psi(w,w)}\).

From Deligne's semisimplicity Theorem (\cite{DeligneUnThe87} 1.11-1.12 and Prop. 1.13) M\"oller deduced the following splitting Theorem of the local system and polarized VHS associated to a Teichm\"uller curve.

\begin{theorem}[Möller, \cite{MoellerVar06} Prop. 2.4 or \cite{TMCMMlect} Th.5.5]\label{Thm:splitVHS}
Let \(K=K(X,\omega)\) be the trace field of \((X,\omega)\in \Omega T_g\).
The polarized VHS \((\mathbb{V}_\mathbb{Z}=R^1f_*\mathbb{Z}_\mathcal{X},\mathcal{V}^{1,0},Q)\) associated to the family of curves \(f\colon \mathcal{X}\to C\) splits over \(\mathbb{Q}\) into two subsystems
  \begin{displaymath}
    \mathbb{V}_\mathbb{Q}= \mathbb{W}_\mathbb{Q}\oplus \mathbb{M}_\mathbb{Q}
  \end{displaymath}
where \(\mathbb{M}_\mathbb{Q}\) carries a polarized \(\mathbb{Q}\)-VHS of weight one and the local system \(\mathbb{W}_\mathbb{Q}\) splits over $F$, the Galois closure of the trace field $K$, as
  \begin{displaymath}
   \mathbb{W}_F=\underset{\sigma\in \mbox{Gal}(F|\mathbb{Q}){/}\mbox{Gal}(F|K)}{\bigoplus}~\mathbb{L}^\sigma,
  \end{displaymath}
such that each of the Galois-conjugate rank two subsystems \(\mathbb{L}^\sigma\) carries a polarized \(F\)-VHS of weight one. The sum of these VHS gives back the original VHS on \(\mathbb{V}_\mathbb{C}\).
\end{theorem}

\begin{remark}\label{Rem:TrivSubsys}
The subsystem \(\mathbb{L}^{\Id}\) comes from the standard action of \(\pi_1(C,c)\) on the \(\mbox{Aff}^+(X,\omega)\)-invariant subspace 
  \[
   \langle \mbox{Re}(\omega), \mbox{Im}(\omega)\rangle_\mathbb{R}\subset H^1(X,\mathbb{R}).
  \]
M\"oller showed that the subsystem of this subrepresentation is defined over a number field \(K_1\subset \mathbb{R}\) which has degree at most two over the trace field \(K=K(X,\omega)\) (\cite{MoellerVar06}, Lemma 2.2).
\end{remark}

\subsubsection{Kontsevich-Zorich cocycle and Lyapunov exponents for Teichm\"uller curves}

We want to introduce the Kontsevich-Zorich cocycle and the Lyapunov exponents in the context of Teichm\"uller curves. For references see \cite{Bouw2005TeichmullerCT} section 8 or \cite{filip2015quaternionic} section 2.4. 

Let \((X,\omega)\in \Omega T_g\) be renormalized such that it has area one. Furthermore let \(j\colon C_1\to M_g\) be the Teichm\"uller curve generated by \((X,\omega)\in \Omega T_g\) and \(f\colon \mathcal{X}\to C\) the associated family of curves as described in section \ref{sect:famteicur}. We have the \(\mathbb{R}\)-local system \(\mathbb{V}_\mathbb{R}=R^1f_*\mathbb{R}_\mathcal{X}\) and the corresponding real \(C^{\infty}\)-bundle \(\mathcal{V}=\mathbb{V}_\mathbb{R}\otimes_\mathbb{R} C^\infty_{C,\mathbb{R}}\). 

For every \(t\in \mathbb{R}\) set \(g_t=\mbox{diag}(e^t,e^{-t})\in \mbox{SL}(2,\mathbb{R})\). The flow of \(g_t\) on the Teichm\"uller disk \(\Delta\) induces a flow on the curve \(C\), which lifts to a flow on the bundle \(\mathcal{V}\) by parallel transport along paths. This flow is called the \textit{Kontsevich-Zorich cocycle} and we denote it by \(G_t^{\text{KZ}}(X,\omega)\).

The bundle \(\mathcal{V}\) carries a metric, which comes from the Hodge-norm induced by \(H\) on the fibers of \(\mathcal{V}\) (compare section \ref{sect:VarOfHS}). 
The Haar-measure \(\lambda\) on \(\mbox{SL}(2,\mathbb{R})\) induces a finite measure \(\mu_M\) on \(\Omega M_g\) with support \(M\), where \(M\) is a lift of the Teichm\"uller curve \(C_1\) to \(\Omega M_g\). The Haar-measure \(\lambda\) is of course \(\mbox{SL}(2,\mathbb{R})\)-invariant and ergodic with respect to the flow \(g_t\) \((t\in \mathbb{R})\) and the measure \(\mu_M\) inherits these two properties. Hence we can apply Oseledec's Theorem on \((\Omega M_g,\ \mu_M)\), the flow \(G_t^{\text{KZ}}(X,\omega)\) and the bundle \(\mathcal{V}\) to get \(2g\) Lyapunov exponents
  \[
  1=\lambda_1\geq \lambda_2\geq\dots\geq\lambda_g\geq 0\geq -\lambda_g\geq\dots\geq-\lambda_2\geq-\lambda_1=-1
  \]
symmetric to the origin. 

From Theorem \ref{Thm:splitVHS} we know that in genus \(g=2\) the VHS over a Teichmüller curve splits over \(\mathbb{Q}\) into two direct summands of rank two. We can apply Oseledec's Theorem to each of the summands individually. The full set of Lyapunov exponents is the union of the Lyapunov exponents of the two summands. 
In \cite{Bouw2010DifferentialEA} Bouw and Möller computed the Lyapunov spectrum of a Teichmüller curve in genus \(g=2\). We want to state their result in the following Proposition:

\begin{proposition}[Bouw, Möller, \cite{Bouw2010DifferentialEA}, Corollary 2.4]\label{Prop:LyaExpg=2}
Let \(C_1\) be a Teichm\"uller curve in genus \(g=2\) generated by the translation surface \((X,\omega)\). The positive Lyapunov exponents are
\begin{displaymath}
(\lambda_1,\lambda_2)
=\left\{
\begin{array}{rl} 
(1,\,1/3) & \mbox{if}~ (X,\omega) \in \mathcal{H}(2), \\
(1,\,1/2) & \mbox{if}~ (X,\omega) \in \mathcal{H}(1,1).
\end{array} \right.
\end{displaymath}
\end{proposition}  

\subsubsection{Period mapping}\label{sect:periodmap}

Good references for the next subsection are \cite{carlson2017period} and \cite{Kappes2011} section 7.3. First of all we want to repeat the construction of the period mapping as in \cite{DeligneTrdeGrif69} 6.1-6.2.

Let \(B\) be a connected complex manifold (in particular smooth) of genus \(g\in \mathbb{N}\) and let \((\mathbb{V}_\mathbb{Z},\mathcal{V}^{1,0},Q)\) be a pure polarized VHS of weight one on \(B\), where \(\mathbb{V}_\mathbb{Z}\) is a local system of rank \(2d\). We denote \(\mathbb{V}=\mathbb{V}_\mathbb{Z}\otimes_\mathbb{Z} \mathbb{C}\).
Fix a base point \(b\in B\) and a universal cover 
  \[
   \pi\colon \tilde B \longrightarrow B.
  \]
By pull back we get a VHS \((\pi^{-1}\mathbb{V},\pi^*\mathcal{V}^{1,0},\pi^*Q)\) of weight one on \(\tilde B\) polarized by \(\pi^* Q\). The inverse-image sheaf is by continuation along paths isomorphic to the constant sheaf of stalk \(\mathbb{V}_b\). 

Let \(\tilde{b}\in \pi^{-1}(b)\) and let \(\varphi_{\tilde b}\colon (\pi^{-1}\mathbb{V})_{\tilde b} \to \mathbb{V}_b\) be the canonical isomorphism.
For every \(z\in \tilde B\) we can construct isomorphisms between the stalk \((\pi^{-1}\mathbb{V})_z\) and \((\pi^{-1}\mathbb{V})_{\tilde b}\) by transporting germs along a path \(c\) connecting \(z\) and \(\tilde b\). Since all paths are homotopic this induces a well defined isomorphism
  \[
   \phi_{z,\tilde b}\colon (\pi^{-1}\mathbb{V})_z \longrightarrow (\pi^{-1}\mathbb{V})_{\tilde b}.
  \]
The bundle \(\pi^* \mathcal{V}^{1,0}\) singles out a subspace \(\tilde{W}_z\subset (\pi^{-1}\mathbb{V})_z\) for every \(z\in \tilde B\). We denote
  \[
   W_z= \varphi_{\tilde b}\circ \phi_{z,\tilde b}(\tilde{W}_z)\subset \mathbb{V}_b.
  \]
We define the map \(P\colon\tilde B\to\mbox{Grass}(d, \mathbb{V}_b)\) by \(P(z)=W_z\).
By construction, for every \(z\in \tilde B\) the subspace \(\tilde{W}_z\subset (\pi^{-1}\mathbb{V})_z\) obeys Riemann bilinear relations with respect to the polarization \((\pi^*Q)_z\) i.e., 
  \[
   (\pi^*Q)_z(u,w)=0\quad \mbox{and} \quad i\cdot (\pi^*Q)_z(w,\overline w)>0
  \]
for every \(u,w\in \tilde{W}_z\). 

The polarization \(\pi^*Q\colon \pi^{-1}\mathbb{V}\otimes_\mathbb{C} \pi^{-1}\mathbb{V}\to \C\) is locally constant and since we constructed the isomorphisms \(\phi_{z,\tilde b}\colon (\pi^{-1}\mathbb{V})_z \to (\pi^{-1}\mathbb{V})_{\tilde b}\) by continuation along paths, all the images
\(\phi_{z,\tilde b}(\tilde{W}_z)\) obey the Riemann bilinear relations with respect to \((\pi^*Q)_{\tilde b}\). The image \(W_z= \varphi_{\tilde b}\circ \phi_{z,\tilde b}(\tilde{W}_z)\) of \(z\in \tilde B\) under the mapping \(P\colon\tilde B \to\mbox{Grass}(d,\mathbb{V}_b)\) is hence an element of the \textit{period domain}
  \begin{displaymath}  
   \mbox{Per}((\mathbb{V}_\mathbb{Z})_b, Q_b)\subset \mbox{Grass}(d, \mathbb{V}_b),
  \end{displaymath}
the set of \(d\)-dimensional subspaces of \(\mathbb{V}_b\) which obey the Riemann bilinear relations with respect to the form \(Q_b\). The mapping
  \[
   P\colon\tilde B\longrightarrow \mbox{Per}((\mathbb{V}_\mathbb{Z})_b, Q_b),\quad z \longmapsto W_z
  \]
is called \textit{period mapping}. In the next Proposition we want to collect two very important properties of the period mapping from above.

\begin{proposition}\label{Prop:PerMap}
Let \(P\colon \tilde B \to \mbox{Per}((\mathbb{V}_\mathbb{Z})_b, Q_b)\) be the period mapping associated to a pure polarized VHS of weight one on a complex connected manifold \(B\) with fixed base point \(b\in B\).
Theorem \ref{thm:monodromrep} implies, that there is a monodromy representation \(\rho\colon \pi(B,b)\to \mbox{G}_{Q_b}((\mathbb{V}_\mathbb{Z})_b)\)\footnote{Here we write \(G_{Q_b}((\mathbb{V}_\mathbb{Z})_b)\) for the subgroup of elements in \(\mbox{GL}((\mathbb{V}_\mathbb{Z})_b)\) which are orthogonal with respect to \(Q_b\).} associated to the local system \(\mathbb{V}_\mathbb{Z}\). Then:
\begin{enumerate}
  \item[(i)]
  The period mapping \(P\colon \tilde B \to \mbox{Per}((\mathbb{V}_\mathbb{Z})_b, Q_b)\) is  
  holomorphic (\cite{Griffiths68b}, Theorem 1.27 or \cite{DeligneTrdeGrif69} 3.4).
  \item[(ii)]
  The period mapping \(P\) is equivariant with respect to the action of \(\gamma\in \pi_1(B,b)\) on
   \(\tilde B\) by deck transformations and the action of 
   \(\rho(\gamma)\in G_{Q_b}((\mathbb{V}_\mathbb{Z})_b)\) on 
   \(\mbox{Per}((\mathbb{V}_\mathbb{Z})_b, Q_b)\) (c.f. \cite[6.2]{DeligneTrdeGrif69}) i.e., the relation
         \[
            W_{\gamma(z)}=\rho(\gamma)\, W_z.
         \]
   holds for every \(z\in \tilde B\).
\end{enumerate}
\end{proposition}

\begin{remark}
The period mapping descends to a mapping
  \begin{displaymath}
   p\colon B \longrightarrow \rho(\pi_1(B,b)){\setminus}\mbox{Per}((\mathbb{V}_\mathbb{Z})_b, Q_b)
  \end{displaymath}
which we also want to call \textit{period mapping}.
\end{remark}

\subsection{Arithmeticity of Kontsevich--Zorich monodromies in genus two}

From now on let \(\mathcal{O}=(X,\omega)\in  \Omega T_g\) be an origami of genus \(g\in \mathbb{N}\), let 
  \[
  C_1=\Delta{/}\text{Aff}^+(\mathcal{O})\longrightarrow M_g
  \]
be the origami curve generated by the pair \((X,\omega)\in \Omega T_g\) and let \(f\colon\mathcal{X}\to C\) be the associated family of curves over the origami curve \(C_1\) as in section \ref{sect:famteicur}. Furthermore choose a basepoint \(c\in C(\mathbb{C})\) such that the fiber of \(f\) over \(c\) is the Riemann surface \(X\).

For every \(\gamma\in H_1(\mathcal{O},\mathbb{Z})\) we have \(\int_\gamma \mbox{Re}(\omega)\in \mathbb{Z}\) and \(\int_\gamma \mbox{Im}(\omega)\in \mathbb{Z}\), hence we can consider \(\mbox{Re}(\omega)\) and \(\mbox{Im}(\omega)\) as elements of \(H^1(\mathcal{O},\mathbb{Z})\). 
We denote by \(H_{st}^1(\mathcal{O},\mathbb{Z})\) the submodule of \(H^1(\mathcal{O},\mathbb{Z})\) spanned by  \(\mbox{Re}(\omega)\) and \(\mbox{Im}(\omega)\) and call it the \textit{tautological part} of \(H^1(\mathcal{O},\mathbb{Z})\). The \textit{non-tautological part} \(H_{(0)}^1(\mathcal{O},\mathbb{Z})\) of \(H^1(\mathcal{O},\mathbb{Z})\) is by definition the symplectic orthogonal of \(H_{st}^1(\mathcal{O},\mathbb{Z})\) with respect to the dualized intersection form \(\Omega^*\) on \(H^1(\mathcal{O},\mathbb{R})\). We get the splitting
  \[
    H^1(\mathcal{O},\mathbb{Z})=H_\text{st}^1(\mathcal{O},\mathbb{Z})
    \oplus H_{(0)}^1(\mathcal{O},\mathbb{Z}).
  \]
Since the group of affine orientation preserving diffeomorphisms \(\mbox{Aff}^+(X,\omega)\) respects the  dualized intersection form \(\Omega^*\) on \(H^1(\mathcal{O},\mathbb{R})\) and hence the orthogonal splitting \(H_\text{st}^1(\mathcal{O},\mathbb{Z})\oplus H_{(0)}^1(\mathcal{O},\mathbb{Z})\),
we have that the action of \(\pi_1(C,c)\) on \(H^1(\mathcal{O},\mathbb{Z})\) induces two actions on the \(\mbox{Aff}^+(X,\omega)\)-invariant submodules \(H_\text{st}^1(\mathcal{O},\mathbb{Z})\) and \(H_{(0)}^1(\mathcal{O},\mathbb{Z})\):
   \begin{displaymath}
   \begin{array}{lrcl}
        &\rho_\text{st}\colon\pi_1(C,c)&\longrightarrow&\mbox{Sp}_{\Omega^*}(H_\text{st}^1(\mathcal{O},\mathbb{Z}))\simeq\mbox{SL}(2,\mathbb{Z}),\\
        &                                           &                      &           \\
        &\rho_\text{sh}\colon \pi_1(C,c)&\longrightarrow &\mbox{Sp}_{\Omega^*}(H_{(0)}^1(\mathcal{O},\mathbb{Z})).
   \end{array}
   \end{displaymath}
From Theorem \ref{thm:monodromrep} and the explanations in section \ref{sect:VarOfHS} we know that the two actions \(\rho_\text{st}\) and \(\rho_\text{sh}\) correspond to two local subsystems of \(\mathbb{V}_\mathbb{Z}=R^1f_*\mathbb{Z}_\mathcal{X}\), which we  describe in the following paragraph.

The action of \(\mbox{Aff}^+(X,\omega)\) on \(H_\text{st}^1(\mathcal{O},\mathbb{Z})\) under the identification of the span \(\langle\mbox{Re}(\omega),\mbox{Im}(\omega)\rangle_\mathbb{Z}\) with \(\mathbb{Z}^2\) is just the standard action of the derivative $D(\mbox{Aff}^+(X,\omega))\subset \mbox{SL}(2,\mathbb{Z})$ on \(\mathbb{Z}^2\). 
Hence the action \(\rho_\text{st}\) of \(\pi_1(C,c)\) on \(H_\text{st}^1(\mathcal{O},\mathbb{Z})\) is also given by the standard action. This shows that the local subsystem \(\mathbb{L}^{\Id}\) of \(R^1f_*\mathbb{R}_\mathcal{X}\) from Remark \ref{Rem:TrivSubsys} is defined over \(\mathbb{Z}\) in the origami case \(\mathcal{O}=(X,\omega)\). Thus the action \(\rho_\text{st}\) corresponds to a \(\mathbb{Z}\)-local subsystem \(\mathbb{L}_\mathbb{Z}\) of \(\mathbb{V}_\mathbb{Z}\) such that \(\mathbb{L_\mathbb{Z}}\otimes_\mathbb{Z}\mathbb{R}=\mathbb{L}^{\Id}\).

To describe the action  \(\rho_\text{sh}\colon \pi_1(C,c) \to \mbox{Sp}_{\Omega^*}(H_{(0)}^1(\mathcal{O},\mathbb{Z}))\) and the corresponding local subsystem of \(\mathbb{V}_\mathbb{Z}=R^1f_*\mathbb{Z}_\mathcal{X}\), we will use Theorem \ref{Thm:splitVHS} in combination with the following Theorem of Gutkin and Judge:

\begin{theorem}[\cite{GuJu00},Thm. 5.5 and 7.1]\label{Thm:tracefQ}
For the trace field \(K(X,\omega)\) of an origami \(\mathcal{O}=(X,\omega)\), we have \(K(X,\omega)=\mathbb{Q}\).
\end{theorem}

From Theorem \ref{Thm:splitVHS} we know that the local system \(\mathbb{V}_\mathbb{Q}=R^1f_*\mathbb{Z}_\mathcal{X} \otimes_\mathbb{Z} \mathbb{Q}\) splits over \(\mathbb{Q}\) as
  \[
   \mathbb{V}_\mathbb{Q}=\mathbb{W}_\mathbb{Q}\oplus \mathbb{M}_\mathbb{Q}.
  \]
We have a further splitting of \(\mathbb{W}_\mathbb{Q}\) over the Galois closure \(F\) of the trace field \(K(X,\omega)\) in 
  \(\mathbb{W}_F=\oplus_{\sigma}\mathbb{L}^\sigma\) with \(\sigma\in \mbox{Gal}(F|\mathbb{Q}){/}\mbox{Gal}(F|K(X,\omega))\). 
But since in the origami case \(K(X,\omega)=\mathbb{Q}\) by Theorem \ref{Thm:tracefQ}, we get \(\mathbb{W}_F=\mathbb{L}_F^{\Id}\). 
We conclude that for an origami \(\mathcal{O}\) of genus \(g\) the local system \(\mathbb{V}_\mathbb{Z}=R^1f_*\mathbb{Z}_\mathcal{X}\) splits as
  \[
  \mathbb{V}_\mathbb{Z}=\mathbb{L}_\mathbb{Z}\oplus \mathbb{U}_\mathbb{Z},
  \]
where \(\mathbb{U}_\mathbb{Z}\) is the  \(\mathbb{Z}\)-local system of rank \(2(g-1)\) corresponding to the action of \(\pi_1(C,c)\) on \(H_{(0)}^1(\mathcal{O},\mathbb{Z})\) and \(\mathbb{L}_\mathbb{Z}\) is the local system corresponding to the action of \(\Gamma\simeq\pi_1(C,c)\) on \(\mathbb{Z}^2\simeq H_\text{st}^1(\mathcal{O},\mathbb{Z})\) by the linear parts of the affine maps of \(\Gamma\leq \text{Aff}^+(X,\omega)\). Furthermore \(\mathbb{U}_\mathbb{Z}\otimes_\mathbb{Z}\mathbb{Q}\simeq\mathbb{M}_\mathbb{Q}\) and therefore we know from Theorem \ref{Thm:splitVHS} that  \(\mathbb{U}_\mathbb{Z}\) carries a polarized VHS of weight one, which we denote by \((\mathbb{U}_\mathbb{Z},\ \mathcal{U}^{1,0},Q_U)\).

\begin{theorem}\label{Thm:ArithmShadVee}
Let \(\mathcal{O}=(X,\omega)\) be an origami of genus \(g=2\) with associated family of curves \(f\colon \mathcal{X}\to C\) as in Section \ref{sect:famteicur}.
Then \(\Sp_{\Omega^*}(H_{(0)}^1(\mathcal{O},\mathbb{Z}))\) is isomoprhic to the special linear group \(\SL(2,\mathbb{Z})\) and the image of the map 
  \[
    \rho_{\text{sh}}\colon \pi_1(C,c)\longrightarrow \SL(2,\mathbb{Z})
  \]
is a finite index subgroup of \(\SL(2,\mathbb{Z})\).
\end{theorem}

\begin{proof}
First we want to point out that in genus \(g=2\) the module \(H_{(0)}^1(\mathcal{O},\mathbb{Z})\) has rank two and hence the symplectic group \(\mbox{Sp}_{\Omega^*}(H_{(0)}^1(\mathcal{O},\mathbb{Z}))\) is isomorphic to \(\mbox{SL}(2,\mathbb{Z})\). 

%\commref{\sout{$\mathbb{U}_\mathbb{Z}^\vee$, this is the only occurence of this notation, define it.}} \commman{Instead of defining $\mathbb{U}_\mathbb{Z}^\vee$, I added this part to explain how the sections $a$ and $b$ in the definition of the period map look like}
Denote by $\pi\colon \mathbb{H}\to C$ the universal cover of the curve $C$ and let $c\in C(\mathbb{C})$ be a basepoint with fiber $f^{-1}(\{c\})=X$. The pull back $\pi^{-1}\mathbb{U}_\mathbb{Z}$ of the local subsystem $\mathbb{U}_\mathbb{Z}\subset R^1f_*\mathbb{Z}_\mathcal{X}$ to the universal cover $\mathbb{H}$ of $C$ is isomorphic to the constant sheaf of stalk $(\mathbb{U}_\mathbb{Z})_c=H_{(0)}^1(\mathcal{O},\mathbb{Z})$. By Poincar\'e duality we can consider $\pi^{-1}\mathbb{U}_\mathbb{Z}$ as the constant sheaf of stalk $H_1^{(0)}(\mathcal{O},\mathbb{Z})$ as well.

The period mapping of the polarized VHS \((\mathbb{U}_\mathbb{Z},\ \mathcal{U}^{1,0},Q_U)\) is in our situation given by the following map (compare \cite{TMCMMlect} Section 4):
There exists a global section \(\omega\) of the \((1,0)\)-part of the pullback of \(\mathbb{U}_\mathbb{Z}\otimes_\mathbb{Z}\mathcal{O}_C\) to the universal cover \(\mathbb{H}\) of \(C\) (\cite{Forster2012lectures}, Theorem 30.3). Furthermore there are sections $a,b$ of $\pi^{-1}\mathbb{U}_\mathbb{Z}$ that are locally a symplectic basis adapted to $\omega$ i.e.,
  \[
  \int_{a(\tau)} \omega(\tau) \in \mathbb{H} \quad \text{and} \quad \int_{b(\tau)} \omega(\tau)=1
  \]
for every \(\tau\in \mathbb{H}\). 

The period domain is in our situation analytic isomorphic to the Siegel upper half space $\mathbb{H}$ (see Proposition 1.24 in \cite{Griffiths68}) and the period mapping is given by
  \[
  P\colon \mathbb{H} \longrightarrow \mathbb{H},\quad \tau\longmapsto \int_{a(\tau)} \omega(\tau).
  \]
By part (ii) of Proposition \ref{Prop:PerMap} the period mapping \(P\colon \mathbb{H}\to \mathbb{H}\) descends to a holomorphic map \(p\colon C\to \rho_\text{sh}(\pi_1(C,c)){\setminus}\mathbb{H}\). Recall that  $\rho_\text{sh}(\pi_1(C,c))$ acts holomorphically, properly and discontinuously on $\mathbb{H}$.

We show in the following that the map $p$ is not constant. Write $\text{deg}(\mathcal{U}^{1,0})$ and $\text{deg}(\mathcal{L}^{1,0})$ for the degrees of the line bundles $\mathcal{U}^{1,0}$ and $\mathcal{L}^{1,0}$, where $\mathcal{U}^{1,0}$ and $\mathcal{L}^{1,0}$ are the $(1,0)$-parts of the Hodge filtration of the Deligne extension of \(\mathbb{U}_\mathbb{Z}\otimes_\mathbb{Z}\mathcal{O}_C\) and $\mathbb{L}_\mathbb{Z}\otimes_\mathbb{Z}\mathcal{O}_C$ to $\overline C$. 
With Proposition \ref{Prop:LyaExpg=2} from above and Proposition 8.5 in \cite{Bouw2005TeichmullerCT} we conclude for the positive Lyapunov exponent \(\lambda_U\) corresponding to \(\mathbb{U}_\mathbb{Z}\):
  \[
  \lambda_U=\frac{\text{deg}(\mathcal{U}^{1,0})}{\text{deg}(\mathcal{L}^{1,0})}=\frac{1}{3}
  \quad \text{or} \quad
  \lambda_U=\frac{\text{deg}(\mathcal{U}^{1,0})}{\text{deg}(\mathcal{L}^{1,0})}=\frac{1}{2}
  \]
Thus $\text{deg}(\mathcal{U}^{1,0})\neq 0$, the variation of Hodge structure \((\mathbb{U}_\mathbb{Z},\ \mathcal{U}^{1,0},Q_U)\) is non-trivial and the period mapping $p\colon C\to \rho_\text{sh}(\pi_1(C,c)){\setminus}\mathbb{H}$ is non-constant and thus open. 

Next we want to show that $\rho_\text{sh}(\pi_1(C,c)){\setminus}\mathbb{H}$ has finite hyperbolic volume. By Theorem 9.5 in \cite{Griffiths70} we can extend the period mapping $p$ holomorphically to a map
  \[
  p\colon C\cup S \longrightarrow  \rho_\text{sh}(\pi_1(C,c)){\setminus}\mathbb{H}, 
  \]
where $S\subset \partial C$ denotes those cusps, for which $\rho_\text{sh}$ maps the corresponding parabolic elements in $\Gamma\simeq H_1(X,c)$ to elements of finite order in $\mbox{SL}(2,\mathbb{Z})$. From Proposition 9.11 and the proof of Theorem 9.6 in \cite{Griffiths70} we conclude that 
    \[
    p\colon C\cup S \to  \rho_\text{sh}(\pi_1(C,c)){\setminus}\mathbb{H}
    \]
is proper and since $p\colon C\cup S \to  \rho_\text{sh}(\pi_1(C,c)){\setminus}\mathbb{H}$ is also holomorphic and non-constant, it is surjective. Furthermore Theorem 9.6 in \cite{Griffiths70} implies that    
    \[
    \rho_\text{sh}(\pi_1(C,c)){\setminus}\mathbb{H}=p(C\cup S)
    \]
has finite hyperbolic volume or equivalently the group \(\rho_\text{sh}(\pi_1(C,c))\) has finite index in \(\mbox{SL}(2,\mathbb{Z})\) (see \cite[Prop. 1.31]{ShimuraGoro1971Itta}).
\end{proof}

\bibliographystyle{amsplain}

\begin{thebibliography}{99}

%\bibitem{B-G}
%{\bf J. Barrow-Green},
%\newblock{\it Poincar\'e and the three body problem},
%\newblock History of Mathematics, {\bf 11}. American Mathematical Society, Providence, RI; London Mathematical Society, London, 1997. xvi+272 pp. ISBN: 0-8218-0367-0

\bibitem{Bainbridge-Habegger-Moeller-2016}
{\bf M. Bainbridge, P. Habegger and M. M\"{o}ller, Martin},
\newblock{\it Teichm\"{u}ller curves in genus three and just likely intersections in
   ${\bf G}^n_m\times {\bf G}^n_a$},
\newblock  Publ. Math. Inst. Hautes Études Sci. 124 (2016), 1–98.

\bibitem{BassMilnorSerre67}
      {\bf H. Bass, J. Milnor and J.-P. Serre},
      \newblock {\it Solution of the congruence subgroup problem for $SL_n (n\geq 3)$ \it and $Sp_{2n} (n \geq 2)$},
      \newblock Publications Mathématiques de l'IH\'{E} 33 (1967), 59–137.
      
\bibitem{BM}
{\bf Y. Benoist and S. Miquel},
\newblock{\it Arithmeticity of discrete subgroups containing horospherical lattices},
\newblock Duke Math. J. 169 (2020), no. 8, 1485--1539.

\bibitem{birkenhake2004complex}
{\bf C. Birkenhake and H. Lange},
\newblock{\it Complex Abelian Varieties},
\newblock Grundlehren der mathematischen Wissenschaften (2004), Springer Verlag.

\bibitem{Bouw2005TeichmullerCT}
{\bf I. Bouw and M. M\"oller},
\newblock{\it Teichm\"uller curves, triangle groups, and Lyapunov exponents},
\newblock Annals of Mathematics 172 (2005), 139-185.

\bibitem{Bouw2010DifferentialEA}
{\bf I. Bouw and Martin M\"oller},
\newblock{\it Differential equations associated with nonarithmetic Fuchsian groups},
\newblock Journal of The London Mathematical Society-second Series 81 (2010), 65-90.

\bibitem{BT}
{\bf C. Brav and H. Thomas},
\newblock{\it Thin monodromy in Sp(4)},
\newblock Compos. Math. 150 (2014), no. 3, 333--343.

\bibitem{carlson2017period}
{\bf J. Carlson, S. M{\"u}ller-Stach and C. Peters},
\newblock{\it Period Mappings and Period Domains},
\newblock Cambridge Studies in Advanced Mathematics (2017), Cambridge University Press.

%\bibitem{BR}
%{\bf R. Bowen and D. Ruelle},
%\newblock{\it The ergodic theory of Axiom A flows},
%\newblock Invent. Math. {\bf 29} (1975), 181--202.

\bibitem{DeligneTrdeGrif69}
{\bf P. Deligne},
\newblock{\it Travaux de Griffiths},
\newblock S\'eminaire Bourbaki : vol. 1969/70, expos\'es 364-381 (1971), Springer Press.

\bibitem{DeligneEqDif70}
{\bf P. Deligne},
\newblock{\it Equations Diff\'erentielles \`a Points Singuliers R\'eguliers},
\newblock Lecture Notes in Mathematics 163 (1970), Springer Press.

\bibitem{DeligneUnThe87}
{\bf P. Deligne},
\newblock{\it Un th\'eor\`eme de finitude pour la monodromie},
\newblock Discrete groups in Geometry and Analysis 67 (1987), 1-19, Birkh\"auser, Progress in Math.

\bibitem{Detinko2018ZariskiDA}
{\bf A. S. Detinko, D. L. Flannery and A. Hulpke},
\newblock{\it Zariski density and computing in arithmetic groups},
\newblock Math. Comput. 87 (2018), 967-986.

\bibitem{earle1997teichmuller}
{\bf Earle, Clifford J and Gardiner, Frederick P},
\newblock{\it Teichm{\"u}ller disks and Veech’s F-structures}.
\newblock Contemp. Math. 201 (1997), 165-189.

\bibitem{origami-package}
{\bf S. Ertl, L. Junk, P. Kattler, A. Rogovskyy, A. Thevis, G. Weitze-Schmithüsen}
\newblock{\it GAP Package Origami},
\newblock{  GitHub repository, \url{https://ag-weitze-schmithusen.github.io/Origami/}}

\bibitem{F}
{\bf S. Filip},
\newblock{\it Uniformization of some weight 3 variations of Hodge structure, Anosov representations, and Lyapunov exponents}, 
\newblock preprint (2021) available at arXiv:2110.07533.

\bibitem{filip2015quaternionic}
{\bf S. Filip, G. Forni and C. Matheus},
\newblock{\it Quaternionic covers and monodromy of the Kontsevich-Zorich cocycle in orthogonal groups},
\newblock Journal of the European Mathematical Society 20 (2015).

\bibitem{Forster2012lectures}
{\bf O. Forster},
\newblock{\it Lectures on Riemann Surfaces},
\newblock Graduate Texts in Mathematics 81 (2012).

\bibitem{FF}
{\bf S. Filip and C. Fougeron},
\newblock{\it A cyclotomic family of thin hypergeometric monodromy groups in $Sp_4(\mathbb{R})$},
\newblock preprint (2021) available at arXiv:2106.09181.

\bibitem{FMS}
{\bf E. Fuchs, C. Meiri and P. Sarnak},
\newblock{\it Hyperbolic monodromy groups for the hypergeometric equation and Cartan involutions},
\newblock J. Eur. Math. Soc. (JEMS) 16 (2014), no. 8, 1617--1671. 

\bibitem{Griffiths68}
{\bf P. A. Griffiths},
\newblock{\it Periods of Integrals on Algebraic Manifolds, I. (Construction and Properties of the Modular Varieties)},
\newblock American Journal of Mathematics 90 (1968), no.2, 568--626.

\bibitem{Griffiths68b}
{\bf P. A. Griffiths},
\newblock{\it Periods of Integrals on Algebraic Manifolds, II. (Local Study of the Period Mapping)},
\newblock American Journal of Mathematics 90 (1968), no.3, 805--865.

\bibitem{Griffiths70}
{\bf P. A. Griffiths},
\newblock{\it Periods of Integrals on Algebraic Manifolds, III (Some global differential-geometric properties of the period mapping)},
\newblock Publications Mathématiques de l'IHÉS, no. 38, 125--180.


\bibitem{GuJu00}
{\bf E. Gutkin and C. Judge},
\newblock{\it Affine mappings of translation surfaces: geometry and arithmetic},
\newblock Duke Mathematical Journal 103 (2000), no.2, 191 -- 213.

\bibitem{HL}
{\bf P. Hubert and S. Leli{\`e}vre},
\newblock {\it Prime arithmetic {T}eichm\"uller discs in {$H(2)$}},
\newblock Israel J. Math. 151 (2006), 281--321.

\bibitem{HM}
{\bf P. Hubert and C. Matheus},
\newblock{\it An origami of genus 3 with arithmetic Kontsevich-Zorich monodromy},
\newblock Math. Proc. Cambridge Philos. Soc. 169 (2020), no. 1, 19--30.

\bibitem{JKZ}
{\bf F. Jouve, E. Kowalski and D. Zywina},
\newblock{\it Splitting fields of characteristic polynomials of random elements in arithmetic groups},
\newblock Israel J. Math. 193 (2013), no. 1, 263--307.

\bibitem{ModularGroup}
{\bf L. Junc and G. Weitze-Schmithüsen}
{GAP package ModularGroup}

\bibitem{KanyMatheus23}
{\bf M. Kany and C. Matheus}
\newblock{\it Arithemticity of the Kontsevich--Zorich monodromies of certain families of square-tiled surfaces II},
\newblock preprint (2023) available at arXiv:2301.06894.

\bibitem{Kappes2011}
{\bf A. Kappes},
\newblock{Monodromy Repesentations and Lyapunov exponents of Origamis},
\newblock Dissertation (2011), KIT Karlsruhe, \url{https://www.math.kit.edu/iag3/~kappes/media/kappes_andre_diss.pdf}. 

\bibitem{Kappes16LyapOfBQ}
{\bf A. Kappes and M. M\"oller},
\newblock{\it Lyapunov spectrum of ball quotients with applications to commensurability questions},
\newblock Duke Mathematical Journal 165 (2016), no.1, 1-66, Duke University Press.

\bibitem{Kattler2023}
  {\bf P. Kattler},
  \newblock{\it Index of the Kontsevich-Zorich monodromy of origamis in $\mathcal{H}(2)$},
  \newblock preprint (2023) available at arXiv:2307.00816.

\bibitem{KattlerGIT2023}
  {\bf P. Kattler},
  \newblock{\it Addendum to the article Arithmeticity of Kontsevich--Zorich monodromies of origamis},
  \newblock{File AddToArticleBKKMNSVW.g on \url{https://github.com/AG-Weitze-Schmithusen/Addenda}.}

\bibitem{KZ}
{\bf M. Kontsevich and A. Zorich},
\newblock{\it Connected components of the moduli spaces of Abelian differentials with prescribed singularities},
\newblock Invent. Math. 153 (2003), no. 3, 631--678.

\bibitem{LN}
{\bf E. Lanneau and D.-M. Nguyen},
\newblock{\it Teichm\"uller curves generated by Weierstrass Prym eigenforms in genus 3 and genus 4},
\newblock J. Topol. 7 (2014), no. 2, 475--522. 

\bibitem{Lochak05Oacimsoc}
{\bf P. Lochak},
\newblock{\it On arithmetic curves in the moduli space of curves},
\newblock Journal of the Institute of Mathematics of Jussieu 4 (2005), 443--508.


\bibitem{MKS2004}
{\bf W. Magnus, A. Karrass and D. Solitar},
\newblock {\it Combinatorial group theory},
\newblock Dover Publications, Inc., Mineola, NY (2004).


\bibitem{Masur2006}
{\bf H. Masur},
\newblock{\it Ergodic theory of translation surfaces},
\newblock{Handbook of dynamical systems. Vol. 1{B} (2006), 527--547.}

\bibitem{M2022}
{\bf C. Matheus},
\newblock{\it Three lectures on square-tiled surfaces},
\newblock Collection: Panoramas et Synthèses, (2022), vol. 58, 77--99.

\bibitem{MMY}
{\bf C. Matheus, M. M\"oller and J.-C. Yoccoz},
\newblock{\it A criterion for the simplicity of the Lyapunov spectrum of square-tiled surfaces},
\newblock Invent. Math. 202 (2015), no. 1, 333--425.

\bibitem{McMullenBilliard}
{\bf C. McMullen},
\newblock{\it Billiards and Teichmuller curves on Hilbert modular surfaces},
\newblock Journal of the American Mathematical Society 16 (2003), no.4, 857-885.

\bibitem{McMullenSpin}
{\bf C. McMullen},
\newblock{\it Teichm{\"u}ller curves in genus two: Discriminant and spin},
\newblock Math. Ann. 333 (2005), no.1, 87-130.

\bibitem{Milne}
{\bf J. S. Milne},
\newblock{\it Algebraic groups. The theory of group schemes of finite type over a field},
\newblock Cambridge Stud. Adv. Math., 170 Cambridge University Press, Cambridge (2017).

%\bibitem{Mo}
%{\bf C. G. Moreira},
%\newblock{\it There are no $C^1$-stable intersections of regular Cantor sets},
%\newblock Acta Math. {\bf 206} (2011), 311--323.

%\bibitem{MY01}
%{\bf C. G. Moreira and J.-C. Yoccoz},
%\newblock{\it Stable intersections of regular Cantor sets with large Hausdoff dimensions},
%\newblock Annals of Math. {\bf 154} (2001), 45--96.

%\bibitem{MY10}
%{\bf C. G. Moreira and J.-C. Yoccoz},
%\newblock{\it Tangences homoclines stables pour des ensembles hyperboliques de grande dimension fractale},
%\newblock Ann. Scient. \'Ec. Norm. Sup. 43 (2010), 1--68.

\bibitem{MoellerVar06}
{\bf M. M\"oller},
\newblock{\it Variations of Hodge structures of a Teichmüller curve},
\newblock Journal of the American Mathematical Society 19 (2003), no.2, 327-344. 

\bibitem{MoellerAff09}
 {\bf M. M\"oller},
 \newblock{\it Affine groups of flat surfaces},
\newblock  Handbook of Teichmüller theory. Vol. II (2009), 369–387,
IRMA Lect. Math. Theor. Phys., 13.


\bibitem{TMCMMlect}
{\bf M. M\"oller},
\newblock{\it Teichm\"uller curves, mainly from the viewpoint of algebraic geometry},
\newblock (2011),  \url{https://www.uni-frankfurt.de/50569555/PCMI.pdf}.

\bibitem{PR}
{\bf G. Prasad and A. Rapinchuk},
\newblock{\it Generic elements in Zariski-dense subgroups and isospectral locally symmetric spaces},
\newblock Thin groups and superstrong approximation, 211--252, Math. Sci. Res. Inst. Publ., 61, Cambridge Univ. Press, Cambridge, 2014.

\bibitem{IR}
{\bf I. Rivin},
\newblock{\it Large galois groups with applications to zariski density},
\newblock preprint (2015) available at arXiv:1312.3009.


%\bibitem{PV}
%{\bf J. Palis and M. Viana},
%\newblock{\it Continuity of Hausdorff dimension and limit capacity for horseshoes},
%\newblock Lecture Notes in Math. {\bf 1331} (1988), 150--160.

%\bibitem{PY94}
%{\bf J. Palis and J.-C. Yoccoz},
%\newblock{\it Homoclinic tangencies for hyperbolic sets of large Hausdorff dimension},
%\newblock Acta Math. {\bf 172} (1994), 91--136.

\bibitem{MHMproject}
{\bf C. Sabbah and C. Schnell},
\newblock{\it The MHM Project (Version 2)},
\newblock \url{https://perso.pages.math.cnrs.fr/users/claude.sabbah/MHMProject/mhm.html} .


\bibitem{Sarnak}
{\bf P. Sarnak},
\newblock{\it Notes on thin matrix groups},
\newblock Thin groups and superstrong approximation, 343--362,
Math. Sci. Res. Inst. Publ., 61, Cambridge Univ. Press, Cambridge, 2014.

\bibitem{DissGabi}
{\bf G. Schmithüsen},
\newblock{\it Veech groups of origamis},
\newblock PhD Thesis (2005), \url{http://www.math.kit.edu/iag3/~schmithuesen/media/dr.pdf}, Karlsruhe Institute of Technology - KIT.

\bibitem{ShimuraGoro1971Itta}
{\bf G. Shimura},
\newblock{Introduction to the arithmetic theory of automorphic functions},
\newblock (1971), Shoten Tokyo.

\bibitem{Singh}
{\bf S. Singh},
\newblock{\it Arithmeticity of four hypergeometric monodromy groups associated to Calabi-Yau threefolds},
\newblock Int. Math. Res. Not. IMRN 2015, no. 18, 8874--8889.

\bibitem{SV}
{\bf S. Singh and T. Venkataramana},
\newblock{\it Arithmeticity of certain symplectic hypergeometric groups},
\newblock Duke Math. J. 163 (2014), no. 3, 591--617.

\bibitem{Weitze2015}
{\bf G. Weitze-Schmith\"{u}sen}
\newblock{\it The deficiency of being a congruence group for {V}eech groups of origamis},
\newblock{Int. Math. Res. Not. IMRN (2015), no. 6, 1613--1637.}

\bibitem{voisinHodge2003}
{\bf C. Voisin},
\newblock{\it Hodge Theory and Complex Algebraic Geometry II},
\newblock Cambridge Studies in Advanced Mathematics 2 (2003), Cambridge University Press.

\bibitem{Vorobets96}
{\bf Ya B Vorobets}
\newblock{\it Planar structures and billiards in rational polygons: the Veech alternative},
\newblock{Russian Mathematical Surveys 51 (1996), no. 5, 779--817.}

\bibitem{zmiaikou2011}
{\bf D. Zmiaikou},
\newblock{\it Origamis and permutation groups},
\newblock Th\`ese, Universit{\'e} Paris Sud - Paris XI.
\newblock https://theses.hal.science/tel-00648120

\bibitem{Zorich2006}
{\bf A. Zorich},
\newblock{\it Flat surfaces},
\newblock{Frontiers in number theory, physics, and geometry. {I} (2006), 437--583.}


\end{thebibliography}

\end{document}